\documentclass{article}
\usepackage[a4paper, total={6in, 8in}]{geometry}

\usepackage{graphicx}%
\usepackage{multirow}%
\usepackage{amsmath,amssymb,amsfonts}%
\usepackage{amsthm}%
\usepackage{mathrsfs}%
\usepackage[title]{appendix}%

\usepackage{bm}
\usepackage{dsfont}

\RequirePackage[colorlinks,citecolor=blue,urlcolor=blue]{hyperref}
\usepackage[numbers]{natbib}
\usepackage[dvipsnames]{xcolor}
\usepackage{xcolor, soul} 
\usepackage[normalem]{ulem}

\usepackage{booktabs}

\newtheorem{theorem}{Theorem}[section]%
\newtheorem{proposition}{Proposition}[section]%
\newtheorem{lemma}{Lemma}[section]%
\newtheorem{corollary}{Corollary}[section]
\newtheorem{assumption}{Assumption}[section]

\newtheorem{remark}{Remark}%

\newtheorem{definition}{Definition}%

\title{Minimax estimation of the structure factor of spatial point processes}

\author{Gabriel Mastrilli $^{1, 2}$}
\date{
	$^1$ Univ Rennes, Ensai, CNRS, CREST \\
	$^2$ Inria, Paris, France.
}

\begin{document}
	\maketitle
	
		\begin{abstract}
			We investigate the problem of estimating the structure factor, or spectra, of stationary spatial point processes. In the first part, we establish a minimax lower bound for this estimation problem, using an approach tailored to second-order properties of spatial point processes. Although not the main focus, this methodology also extends naturally to a minimax lower bound for the estimation of the pair correlation function of spatial point processes. In the second part, we construct a multitaper estimator that achieves the optimal rate of convergence in squared risk. Under a Brillinger-mixing condition, we further establish a chi-square–type concentration bound. Finally, we propose a data-driven procedure for selecting the number of tapers, supported by an oracle inequality, and we demonstrate the practical effectiveness of the method through numerical experiments.
			
				~
			
				~
		\end{abstract}
	
	\section{Introduction}\label{sec:intro}

	Spectral methods for second-order inference in spatial point processes were introduced in the seminal works of Bartlett~\cite{bartlett1963spectral, bartlett1964spectral}. These methods aim to estimate the spatial analogue of the spectral density of a time series, typically Bartlett’s spectral measure, often referred to simply as the \emph{spectra}, or its normalized version, the \emph{structure factor}. While spectral density estimation for time series has been a central topic of statistical research for nearly a century (see, e.g.,~\cite{percival2020spectral} and references therein), its spatial counterpart remains comparatively underdeveloped, despite its relevance in many applications. In particular, the structure factor plays a crucial role in statistical physics, where it underpins the analysis of crystalline and disordered materials~\cite{dawson1969significance, kittel2018introduction, torquato2018hyperuniform}. In spatial statistics, following Bartlett's foundational contributions, the works~\cite{mugglestone1996practical, mugglestone1996exploratory, halliday1995framework, mugglestone2001spectral, renshaw1997spectral, renshaw2002two, jarrah2010detection} highlighted the practical value of spectral methods for analyzing spatial point patterns. However, several theoretical statistical questions had long remained only partially addressed. In the past few years, however, this line of research has received renewed and sustained attention, notably through a series of works including~\cite{rajala2023fourier, hawat2023estimating, grainger2023spectral, yang2024fourier, ding2025pseudo}, which have significantly advanced the theoretical understanding of spectral estimators for spatial point processes. Despite these contributions, the question of optimal rates of convergence for structure factor estimation remains unresolved. This paper addresses this gap by establishing minimax lower bounds and constructs an estimator that achieves these optimal rates.
	
	Specifically, we establish that the minimax optimal rate of convergence for estimating a structure factor with Hölder regularity $\beta \in (0, \infty)$, based on a single realization of a stationary point process observed within a compact window $W \subset \mathbb{R}^d$, is of order $|W|^{-\beta/(2\beta + d)}$ (see Theorem~\ref{thm:minimax_S}). Equivalently, this rate can be expressed in terms of the expected number of observed points. Identifying $|W|$ with the sample size, this convergence rate formally matches the classical nonparametric rate for estimating Hölder functions~\cite{tsybakov2003introduction}, and coincides with the minimax rates known for spectral density estimation in time series~\cite{samarov1977lower, bentkus1985rate, efromovich1998data, anevski2011monotone, kartsioukas2023spectral}. As in these settings, the smoothness parameter $\beta$ has a direct interpretation in terms of second-order correlations: larger values of $\beta$, and hence faster rates, correspond to processes with more rapidly decaying pairwise correlations.
	
	Nevertheless, although the overall structure of the proof of the minimax lower bound follows standard minimax arguments from nonparametric estimation theory, its implementation requires significant adaptation due to the different nature of the objects involved. In the time series setting, minimax lower bounds for spectral density estimation typically rely on Gaussian processes. Such Gaussian structure is not readily available for point processes. Moreover, while minimax lower bounds have been established for first-order properties of non-stationary Poisson point processes~\cite{reynaud2003adaptive, kutoyants2012statistical, reynaud2010near}, these results yield trivial bounds when applied to second-order estimation in the stationary case, since the structure factor of any stationary Poisson is constant equal to one. To overcome this, and to leverage the analytical tractability of both Gaussian fields and Poisson point processes, we consider a suitable class of Cox point processes~\cite{moller2007modern}, where the random intensity field is given by $(1 + \sin(N(x)))_{x \in \mathbb{R}^d}$, with $(N(x))_{x \in \mathbb{R}^d}$ a centered stationary Gaussian field. For this class, we verify that standard minimax techniques apply (see~\cite{tsybakov2003introduction}, Chapter~2). In addition, because this construction is specifically chosen to facilitate minimax arguments for second-order statistics of spatial point processes, it also leads to a minimax lower bound for the estimation of the pair correlation function (see Theorem~\ref{thm:minimax_g}). Although this result lies outside the main focus of the present work, we believe it is of independent interest for the spatial statistics community and worth reporting.
	
	The second part of this work is devoted to the construction of an estimator of the structure factor that attains the optimal minimax rate of convergence. This estimator relies on the principle of multitapering~\cite{thomson1982spectrum}, which we motivate and introduce in what follows. Earlier contributions to spectral inference have emphasized the benefits of refining the basic periodogram~\cite{tuckey1950sampling, blackman1960measurement, bartlett1963spectral, brillinger2002john}. A common refinement is to apply a taper to the periodogram to reduce bias. Nevertheless, the variance remains high. To mitigate this issue, one standard approach is to average multiple single-taper estimators. Two averaging schemes have been considered in the spatial statistics literature, inspired by analogous methods in time series analysis~\cite[Chapters~7–8]{percival2020spectral}. The first scheme averages single-taper estimators evaluated at distinct frequencies near the target frequency, using a fixed taper function~\cite{rajala2023fourier, yang2024fourier, ding2025pseudo}. The second scheme, which we adopt in this work, averages tapered estimators evaluated at a fixed frequency across a family of orthogonal taper functions~\cite{rajala2023fourier, grainger2023spectral}. This approach, called multitapering, is well-established in signal processing~\cite{thomson1982spectrum, walden2000unified, karnik2022thomson} and has recently proven effective for spectral analysis of spatial point processes~\cite{hawat2023estimating, rajala2023fourier, grainger2023spectral, mastrilli24estimHU}. Importantly, it enables estimation of the structure factor at a fixed frequency without requiring information from nearby frequencies. Moreover, with a fixed number of tapers, it typically has a simple asymptotic $\chi^2$-type distribution \cite{bartlett1964spectral, thomson1982spectrum, percival2020spectral, rajala2023fourier}.
	
	To study the theoretical performance of multitaper estimators and prove that they can attain the minimax rate of convergence, we first establish a non-asymptotic upper bound on their squared risk (Theorem~\ref{thm:L2_risk}). This result makes explicit how the bias–variance trade-off is governed by the number of tapers and their spatial and spectral localizations. This insight motivates the use of Hermite functions, which are simultaneously well localized in both domains and numerically easy to compute. We show that this choice leads to an estimator achieving the minimax-optimal convergence rate for structure factors with Hölder regularity $0 < \beta \leq 2$. In addition, under Brillinger-mixing assumptions, we strengthen our squared risk guarantee by deriving a non-asymptotic high-probability bound. Specifically, we show that the usual $\chi^2$-type concentration of the multitaper estimator remains valid even when the number of tapers depends on the observation window (Theorem~\ref{cor:conc_herm}).

	Finally, we propose a cross-validation procedure to select the number of tapers. This selection is local, allowing for different numbers of tapers across frequencies. The method adapts “leave-one-out” techniques, originally developed for bandwidth selection in kernel-smoothed spectral estimation~\cite{beltrato1987determining, hurvich1985data, velasco2000local, ding2025pseudo}, to the multitaper framework for spatial point processes. We provide theoretical guarantees for this procedure by establishing an oracle inequality (see Theorem \ref{thm:dd}), and illustrate its practical effectiveness through numerical experiments, whose code is available at \href{https://github.com/gabrielmastrilli/multitaper_spp}{\text{https://github.com/gabrielmastrilli/multitaper\_spp}}.
	
	The paper is organized as follows. Section~\ref{sec:notations} gathers our general notations. In Section~\ref{sec:sf_pp}, we introduce the notions that are specific to stationary point processes, in order to define the structure factor. Section \ref{sec:minimax} states minimax lower bounds for the problems of estimating the structure factor and the pair correlation function of a stationary point process. The multitaper estimator considered in this work is introduced in Section \ref{sec:mt_intro}. In this section, we derive heuristic arguments supporting it and highlighting the necessity to consider a non-asymptotic framework to discuss the choice of its parameters. Section \ref{sec:l2_risk} studies its $L^2(\mathbb{P})$-risk and proves that for Hermite tapers, the multitaper estimator is optimal for the minimax criteria. Section \ref{sec:nn_asymp} studies the concentration properties of the multitaper estimator and shows that the results of Section~\ref{sec:l2_risk} remain valid with high probability. Finally, Section~\ref{sec:dd} presents the data-driven procedure for selecting the number of tapers and investigates its numerical performance. The proofs of the minimax lower bounds are provided in Section~\ref{sec:proof_minimax}.
	
	Appendices~\ref{sec:proof_risk}, \ref{sec:proof_conc}, and~\ref{sec:dd_proof} contain the proofs corresponding to Sections~\ref{sec:l2_risk}, \ref{sec:nn_asymp}, and~\ref{sec:dd}, respectively. These proofs rely on technical lemmas presented in Appendix~\ref{sec:lem_lin_stats} concerning the cumulants of linear statistics of point processes. Appendix \ref{sec:with_intens} concerns a structure factor estimator in which the unknown intensity is replaced by its empirical estimate. Appendix~\ref{sec:herm} gathers results on Hermite functions, while Appendix~\ref{sec:cum} collects technical tools on cumulants of random variables.
	
	\section{General notations}\label{sec:notations}
	We consider the Euclidean space $\mathbb{R}^d$ of dimension $d \geq 1$. For $a, b \in \mathbb{C}^d$, $a \cdot b = \sum_{i = 1}^d a_i \overline{b_i}$ denotes the Hermitian scalar product of $\mathbb{C}^d$ between $a$ and $b$. The associated Euclidean norm is denoted by $|a|$ and the supremum norm is $|a|_{\infty} := \max_{i = 1, \dots, d} |a_i|$. For any reals $x$ and $y$, we use the notations $x \vee y := \max(x, y)$ and $x \wedge y = \min(x, y)$. For  $x \in \mathbb{R}$, $\lfloor x \rfloor$ denotes the greatest integer less than $x$. For a discrete set $X$, $|X|$ denotes its cardinal. 
	
	For $1\leq p < \infty$, we denote by $L^{p}(\mathbb{R}^d)$ the space of measurable functions $f: \mathbb{R}^d \to \mathbb{C}$ such that $\int_{\mathbb{R}^d} |f(x)|^p dx < \infty$. The $L^p$-norm of $f \in L^p(\mathbb{R}^d)$ is defined by $\|f\|_p^p:= \int_{\mathbb{R}^d} |f(x)|^p dx$. We denote by $L^{\infty}(\mathbb{R}^d)$ the space of functions $f: \mathbb{R}^d \to \mathbb{C}$ such that $\|f\|_{\infty}:= \sup_{x \in \mathbb{R}^d} |f(x)| < \infty$.  We adopt the following definition for the Fourier transform of a function $f \in L^{1}(\mathbb{R}^d)$: 
	$$\forall k \in \mathbb{R}^d,\quad \mathcal{F}[f](k):= \frac1{(2\pi)^{d/2}}\int_{\mathbb{R}^d} f(x) e^{\bm{i} k \cdot x} dx.$$ As usual, the Fourier transform is extended to $L^2(\mathbb{R}^d)$ functions thanks to the Plancherel theorem~\cite{folland2009fourier}: $\forall f_1, f_2 \in L^2(\mathbb{R}^d)$, $\langle f_1, f_2 \rangle = \langle \mathcal{F}[f_1], \mathcal{F}[f_2] \rangle$. The homogeneous Sobolev space $\dot{H}^{s}(\mathbb{R}^d)$ is the space of functions $f \in L^2(\mathbb{R}^d)$ with $\|f\|_{\dot{H}^{s}}^2 := \int_{\mathbb{R}^d} |\mathcal{F}[f](k)|^2 |k|^{2s} dk < \infty$. We denote by $\delta_0$ the Dirac mass at zero.
	
	We also denote, for $n \geq 1$, $[n] = \{1, \dots, n\}$ and $\Pi[n]$ the set of partitions of $[n]$. For a partition $\sigma \in \Pi[n]$, $|\sigma|$ denotes its number of blocks. For two partitions $\sigma$ and $\tau$ of $[n]$, where $n \geq 1$, their least upper bound (refer to Chapter 3 of~\cite{mccullagh2018tensor} or Chapter 2 of~\cite{peccati2011wiener}) is denoted $\sigma \vee \tau$. We define algebraically the cumulants. For random variables $X_1, \dots, X_m$, we denote their $m$-th cumulants as:
	\begin{equation}\label{def_kappa}
		\kappa(X_1, \dots, X_m) := \sum_{\sigma \in \Pi[n]} (-1)^{|\sigma|-1}(|\sigma|-1)! \prod_{b \in \sigma} \mathbb{E}\left[\prod_{s \in b} X_{s}\right],
	\end{equation}
	where $\prod_{b \in \sigma}$ denotes the product over all the blocks of the partition $\sigma$. Moreover, for a random variable $X$ having moments of order $m$, its $m$-th cumulants are defined as $\kappa_m(X) = \kappa(X, \dots, X)$, where $X$ appears $m$-times. 
	
	Finally, the total variation of a signed measure $\nu$ is defined by  $|\nu|  := \nu^+ + \nu^-$, where $\nu  =\nu^+ - \nu^-$ is the Jordan decomposition of $\nu$ (see, e.g., \cite{dudley2002real}, Theorem 5.6.1).
	
	\section{Notations and definitions specific to point processes}\label{sec:sf_pp}
	
	In the following, we introduce the key concepts and results concerning the structure factor in Section~\ref{sec:sf1}. Then cumulants of point processes are introduced in Section \ref{sec:def_cum}. For an in-depth treatment of point processes, we refer readers to the classic two-volume work~\cite{daley2007introduction, daley2003introduction} or to the more recent, concise exposition in~\cite{baccelli2020random} and for fundamental understanding of the power spectra of point processes, we recommend~\cite{bremaud2014fourier}.
	
	\subsection{Second order properties and structure factor of stationary point processes}\label{sec:sf1} The set of points configurations in $\mathbb{R}^d$ is defined as:
	$$\text{Conf}(\mathbb{R}^d):= \{\phi \subset \mathbb{R}^d|~ \text{For all } K \text{ compact of } \mathbb{R}^d, \text{ then } |\phi \cap K| < \infty\}.$$ 
	
	This set is endowed with the $\sigma$-algebra generated by the mapping $\phi \mapsto |\phi\cap K|$ for all compact sets $K$. A point process $\Phi$ is a random element of $\text{Conf}(\mathbb{R}^d)$. A point process $\Phi$ is called stationary if for all $x \in \mathbb{R}^d$, $\Phi + x:= \{y + x|~y \in \Phi\}$ is equal in distribution to $\Phi$. As a consequence, the intensity measure $\rho^{(1)}$ of a stationary point process $\Phi$ (defined for any subset $A$ of $\mathbb{R}^d$ by $\rho^{(1)}(A) = \mathbb{E}(|\Phi \cap A|)$), is proportional to the Lebesgue measure on $\mathbb{R}^d$: $\rho^{(1)} = \lambda dx$. The scalar $\lambda \geq 0$ is called the intensity of the point process. The second  order factorial moment measure $\rho^{(2)}$ of a point process $\Phi$ is a measure on $(\mathbb{R}^d)^2$  defined by
	\begin{equation}\label{eq:rho2}
		\rho^{(2)}(A_1\times A_2) = \mathbb{E}\Big[\sum_{x, y \in \Phi}^{\neq} \mathbf{1}_{x \in A_1, y \in A_2}\Big],
	\end{equation}
	for all $A_1, A_2$ subsets of $\mathbb{R}^d$. The symbol $\neq$ over the sum means that we consider only distinct points. To define the structure factor of a point process and ensure that it is a function and not a general measure, we introduce the following assumption.
	\begin{assumption}\label{ass_rho_leb}
		Throughout the paper, we tacitly assume that the point process $\Phi$ is stationary, with an intensity $\lambda > 0$, and that its second-order intensity measure $\rho^{(2)}$ is absolutely continuous with respect to the Lebesgue measure on $\mathbb{R}^d\times \mathbb{R}^d$. This allows us to represent it as follows:
		\begin{align*}
			\rho^{(2)}(dx,dy) = \lambda^{2} g(x - y) dx dy,
		\end{align*}
		where $g(x)$ is a function on $\mathbb{R}^d$. Additionally, we assume that $g - 1 \in L^{1}(\mathbb{R}^d) \cap L^{2}(\mathbb{R}^d)$.
	\end{assumption}
	
	The function $g : \mathbb{R}^d \to \mathbb{R}$ is known as the pair-correlation function of $\Phi$. With this established, we can proceed to define the structure factor of $\Phi$ as
	\begin{equation}\label{eq:def_S}
		\forall k \in \mathbb{R}^d,~	S(k) := 1 + \lambda \int_{\mathbb{R}^d}(g(x) - 1) e^{- \bm{i}k.x} dx.
	\end{equation}
	The previous definition follows the convention from the physics literature \cite{cowley1992electron, hansen2013theory, torquato2018hyperuniform}. Indeed, $S$ is a standardized version of $\lambda S$, which is the Bartlett spectral measure of $\Phi$~\cite{bartlett1964spectral, daley2003introduction}, called spectra in related statistical works \cite{rajala2023fourier, grainger2023spectral, yang2024fourier}. In the following, we assume $\lambda = 1$, so that the structure factor and the spectra are equal. This is not restrictive in practice, when $\lambda$ is unknown and has to be estimated, as shown in Appendix~\ref{sec:with_intens}.
	
	By application of the Campbell formula~\cite{baccelli2020random} and of the Plancherel Theorem~\cite{folland2009fourier}, we deduce the following representation of the covariance between two linear statistics of $\Phi$. For all $f_1, f_2 \in L^1(\mathbb{R}^d)\cap L^2(\mathbb{R}^d)$, we have
	\begin{equation}\label{e.prop_campbell}
		\operatorname{Cov}\Big[\sum_{x \in \Phi} f_1(x), \sum_{x \in \Phi} f_2(x)\Big] = \lambda \int_{\mathbb{R}^d} \mathcal{F}[f_1](k) \overline{\mathcal{F}[f_2]}(k) S(k) dk.
	\end{equation}
	Moreover, standard results concerning the Fourier transform of $L^1(\mathbb{R}^d)$ functions~\cite{folland2009fourier} imply that, under Assumption~\ref{ass_rho_leb}, then $S$ is bounded and continuous. 
	
	\subsection{Cumulants of stationary point processes}\label{sec:def_cum}
	
	In this section we recall the definition of the factorial cumulant measures of a simple point process and their reduced counterpart in the stationary case; see~\cite{daley2007introduction, daley2003introduction}.
	
	As a generalization of the second order factorial moment measure defined in \eqref{eq:rho2}, the factorial moment measures of order $m\geq 1,$ $\rho^{(m)}$ of a simple point process $\Phi$ is defined for any bounded measurable sets $A_1, \dots, A_m$ in $\mathbb{R}^d$ by
	\begin{align*}
		\rho^{(m)}(A_1 \times  \ldots \times A_m) = \mathbb{E}\left(\sum_{x_1, \ldots, x_m \in \Phi}^{\neq} 1\{x_1 \in A_1, \ldots, x_m \in A_m\}\right).
	\end{align*}
	Then, the $m$-th order factorial cumulant moment measure $\gamma^{(m)}$ is defined by:
	\begin{align*}
		\gamma^{(m)}(A_1\times \dots\times A_m) = \sum_{\sigma\in\Pi[m]} (-1)^{|\sigma|-1} (|\sigma| - 1)!  \prod_{b \in \sigma} \rho^{(|b|)}\left(\prod_{s \in b} A_{s}\right).
	\end{align*}
	Furthermore, when the point process $\Phi$ is  stationary, we may consider   for any $m \geq 2$ the reduced $m$-th order factorial cumulant moment measure $\gamma_{\text{red}}^{(m)}$. This is  a locally finite signed measure on $(\mathbb{R}^{d})^{(m-1)}$ that satisfies
	\begin{equation}\label{eq:def_gamma_red}
		\gamma^{(m)}\left(A_1\times \dots\times A_m\right) = \int_{A_m} \gamma_{\text{red}}^{(m)}\left(\prod_{i=1}^{m-1} (A_i - x)\right) \, dx,
	\end{equation}
	where for $i = 1, \ldots, m - 1$, $(A_i - x)$ is the translation of the set $A_i$ by $x$. 

	\section{Minimax lower bounds}\label{sec:minimax}
	
	In this section, Theorem \ref{thm:minimax_S} and Theorem \ref{thm:minimax_g} provide minimax lower bounds for the problems of estimating the structure factor and the pair correlation function of a stationary spatial point process, based on its observation over a compact window. They state that we recover the classical non parametric rate: the expected number of observed points at power $-\beta/(2\beta +d)$, when considering the problem of estimating a function of Hölder-regularity $\beta$, as defined below.
	
	\begin{definition}\label{def:Hölder}
		Let $\beta > 0$ and $L > 0$. We denote by $l$ the greatest integer strictly less than $\beta$. A function $f: \mathbb{R}^d \to \mathbb{R}$ belongs to the Hölder-class $\Theta(\beta, L)$, if $f$ is $l$-times differentiable and for all multi-indexes $i = (i_1, \dots, i_d) \in \mathbb{N}^{d}$ such that $i_1+\dots + i_d = l$, we have
		$$\forall (x, y) \in \mathbb{R}^{2d}, |\partial^{l}_{i} f(x) - \partial^{l}_{i} f(y)| \leq L |x - y|^{\beta - l}.$$ 
	\end{definition}
	
	To state the minimax lower bounds, we introduce a notion of a family of regular compact sets. For example, $\mathcal{W} = (\{R w|~w \in W_0\})_{R > 0}$, where $W_0$ is a non empty compact set, is regular.
	\begin{definition}
		A family $\mathcal{W}$ of compact sets of $\mathbb{R}^d$  is said regular if there exists $\eta > 0$ such that for all $W \in \mathcal{W}$ there exists $R_W> 0$ such that $W \subset [-R_W, R_W]^d$ and $|[-R_W, R_W]^d| \leq \eta |W|$.
	\end{definition}

	The following theorem, whose proof is postponed to Section~\ref{sec:proof_minimax}, establishes the minimax rate of convergence for estimating the structure factor of Hölder regularity $\beta$ of a point process $\Phi$, based on the observation of $\Phi \cap W$, where $W$ is a compact observation window. This result is stated in a expanding window regime, which is a standard asymptotic setting for structure factor estimation~\cite{rajala2023fourier, hawat2023estimating, yang2024fourier}. Finally, as we have assumed a unit intensity, the expected number of observed points is $\mathbb{E}[|\Phi \cap W|] = |W|$, and the minimax rate is in fact expressed in terms of this expected count.
	\begin{theorem}\label{thm:minimax_S}
		Let $h :[0, \infty) \to [0, \infty)$ be an non-identically null increasing function with $h(0) = 0$. Consider $\mathcal{W}$, a regular family of compact sets of $\mathbb{R}^d$. Let $(\beta, M, L) \in (0, \infty)^3$. For $S \in \Theta(\beta, L)$, we denote by $\mathcal{L}_{(S, M)}$ the set all laws of stationary point processes of unit intensity satisfying Assumption~\ref{ass_rho_leb}, having structure factor $S$ and such that $\max_{m \in \{2, 3, 4\}}|\gamma_{red}^{(m)}|(\mathbb{R}^{d(m-1)}) \leq M.$ Then, 
		\begin{equation}\label{eq:minimax_S}
			\inf_{W \in \mathcal{W}} \inf_{\widehat{S}} \sup_{S \in \Theta(\beta, L)} \sup_{\Phi \sim \mathcal{L}_{(S, M)}} \sup_{k_0 \in \mathbb{R}^d} \mathbb{E}\left[h\left(|W|^{\frac{\beta}{2\beta+d}}|\widehat{S}(k_0; \Phi \cap W) - S(k_0)|\right)\right] > 0,
		\end{equation}
		where $\sup_{S \in \Theta(\beta, L)}$ denotes the supremun over structure factors of regularity $\Theta(\beta, L)$ and $\inf_{\widehat{S}}$ denotes the infimum over all the estimators $\widehat{S} : \mathbb{R}^d \times \text{Conf}(\mathbb{R}^d) \to \mathbb{R}$.
	\end{theorem}
	
	Theorem~\ref{thm:minimax_S} implies that the estimator of $S(0)$ constructed in~\cite{heinrich2010estimating} is minimax when $\beta \geq d^2/2$ (see Theorem A.3 in~\cite{heinrich2010estimating}). While this manuscript was being written, the preprint~\cite{ding2025pseudo} introduced a kernel estimator of the structure factor that is minimax in dimensions $d \leq 4$ for $\beta = 2$. In Section~\ref{sec:mt}, we construct, via multitapering, a minimax estimator of $S$ for $\beta \in (0, 2]$.
	
	We comment the obtained minimax rate of convergence. It is faster for large $\beta$, as expected, since larger value of $\beta$ are related to faster decay of correlations. Indeed, for a point process $\Phi$ having pair correlation function 
	$g$, then if $x^{\beta}(g(x) -1) \in L^{1}(\mathbb{R}^d)$, i.e. if $\Phi$ has second order structure that decorrelates quickly enough, then classical Fourier arguments imply that $S \in \Theta(\beta, L)$ for some $L \geq 0$ (see, e.g.,~\cite{folland2009fourier}). To make it more concrete, for DDPs~\cite{soshnikov2000determinantal, lavancier2015determinantal} with  stationary kernel $K \in L^2(\mathbb{R}^d\times \mathbb{R}^d)$ satisfying $|K(x, y)|^2 = k_0^2(x-y)$ for all $(x, y) \in \mathbb{R}^{2d}$, this condition can be rephrased $x^{\beta}k_0(x)^2 \in L^1(\mathbb{R}^d)$. For a Cox process~\cite{moller2007modern}  with directional intensity $(Z(x))_{x \in \mathbb{R}^d}$, where $Z(x)$ is a stationary random field, it reads, $x^{\beta} \operatorname{Cov}(Z(x), Z(0)) \in L^1(\mathbb{R}^d)$.

	As outlined in the introduction, the proof of Theorem~\ref{thm:minimax_S} differs from classical minimax arguments for spectral estimation in time series, which rely on Gaussian processes~\cite{samarov1977lower, bentkus1985rate, efromovich1998data, anevski2011monotone, kartsioukas2023spectral}. Such constructions are not directly applicable to point processes, and existing minimax results for first-order estimation in non-stationary Poisson processes~\cite{reynaud2003adaptive, kutoyants2012statistical, reynaud2010near} yield trivial bounds for second-order spectra estimation for stationary spatial point processes.
	
	To overcome this, we study a Cox process $\Phi_{\text{Cox}}$ with underlying intensity field $(1 + \sin(N(x)))_{x \in \mathbb{R}^d}$, where $(N(x))_{x \in \mathbb{R}^d}$ is a centered stationary Gaussian field with covariance function $c$. The key idea is that as $\operatorname{Var}[N(0)]$ tends to zero, the distribution of $\Phi_{\text{Cox}}$ converges to that of a stationary Poisson point process with intensity 1 (see Lemma \ref{lem:chi2}). At the same time, by increasing the range of correlations of the Gaussian field $N$, we ensure that the second-order properties of $\Phi_{\text{Cox}}$ remain sufficiently distinct from those of a Poisson process. To this end, we set $c = \sigma^2 c_0(\cdot/\rho)$ as the covariance function of the Gaussian field, where $c_0$ is compactly supported, and let $\rho \to \infty$ and $\sigma \to 0$ simultaneously, with a proper scaling with respect to the volume $|W|$ of the observation window. This construction enables us to apply the $\chi^2$-divergence version of Theorem 2.2 from \cite{tsybakov2003introduction}.

	An important technical point in the proof, specific to point processes, concerns the control of higher-order cumulants. Since we let $\rho \to \infty$, meaning that the range of correlations in $\Phi_{\text{Cox}}$ diverges, it is not immediate that the total variation of the third and fourth order reduced factorial cumulant measures remains bounded. However, this control is crucial, as $L^2(\mathbb{P})$-risk upper bounds typically rely on the finiteness of these cumulants or related quantities. For Gaussian processes, this issue does not arise since joint cumulants of order higher than two vanish. In our setting, we prove that for  $\Phi_{\text{Cox}}$, we have
	$
	\max_{m \in \{2, 3, 4\}} |\gamma_{\text{red}}^{(m)}|(\mathbb{R}^{d(m-1)}) \leq M,
	$
	for some constant $M$, by leveraging the fact that when $\sigma \to 0$, the underlying intensity becomes nearly Gaussian: $1 + \sin(N(x)) \simeq 1 + N(x)$ (see Lemma~\ref{lem:mini_bril_4}).
	\medskip
	
	To conclude this section, we note that the previous proof can be adapted with minor modifications to establish a minimax lower bound for the estimation of the pair correlation function of a point process. This result is stated in the next theorem, which is proved in Section~\ref{sec:proof_minimax}. We also mention that the estimator introduced in~\cite{xu2020nonparametric} achieves the minimax convergence rate corresponding to the following theorem when $\beta \in \mathbb{N} \setminus \{0\}$ (see Lemma~3.2 in~\cite{xu2020nonparametric}).

	\begin{theorem}\label{thm:minimax_g}
		Let $(\beta, M, L) \in (0, \infty)^3$ and $h :[0, \infty) \to [0, \infty)$ be non-identically null, increasing and with $h(0) = 0$. Let $\mathcal{W}$ be a regular family of compact sets of $\mathbb{R}^d$. For $g \in \Theta(\beta, M)$, we denote by $\mathcal{L}_{(g, M)}$ the set all laws of stationary point processes of unit intensity satisfying Assumption~\ref{ass_rho_leb}, having pair correlation function $g$ and such that $\max_{m \in \{2, 3, 4\}}|\gamma_{red}^{(m)}|(\mathbb{R}^{d(m-1)}) \leq M.$ Then, 
		\begin{equation}\label{eq:minimax_g}
			\inf_{W \in \mathcal{W}} \inf_{\widehat{g}} \sup_{g \in \Theta(\beta, L)} \sup_{\Phi \sim \mathcal{L}_{(g, M)}} \sup_{x_0 \in \mathbb{R}^d} \mathbb{E}\left[h\left(|W|^{\frac{\beta}{2\beta+d}}|\widehat{g}(x_0; \Phi \cap W) - g(x_0)|\right)\right] > 0,
		\end{equation}
		where $\sup_{g \in \Theta(\beta, L)}$ denotes the supremun over pair correlation functions of regularity $\Theta(\beta, L)$ and $\inf_{\widehat{g}}$ denotes the infimum over all the estimators $\widehat{g}: \mathbb{R}^d \times \text{Conf}(\mathbb{R}^d) \to \mathbb{R}$.
	\end{theorem} 
	
	\section{Minimax estimator of the structure factor}\label{sec:mt}
	
	In this section, we construct an estimator based on multitapering that achieves the minimax rate of convergence stated in Theorem~\ref{thm:minimax_S}. Specifically, Section \ref{sec:mt_intro} introduces a multitaper estimator. Then, Section \ref{sec:l2_risk} analyses its squared risk. Finally, Section~\ref{sec:nn_asymp} establishes that the $L^2(\mathbb{P})$ minimax bound can be strengthened to a high-probability bound.	As before, we assume for simplicity that the intensity is $\lambda=1$. Appendix \ref{sec:with_intens} extends the results to practical estimator, where the unknown intensity is replaced by the standard estimator $\widehat{\lambda} = |\Phi \cap W|/|W|.$
	
	\subsection{Multitaper estimator: introduction}\label{sec:mt_intro}
	
	The multitaper estimator of the structure factor of a spatial point process is defined as follows.
	
	\begin{definition}\label{def:multi_tap}
		Let $\Phi$ be a stationary point process of $\mathbb{R}^d$,  $W$ be a compact set of $\mathbb{R}^d$ and $I$ be a subset of $\mathbb{N}^d$. Let $k_0 \in \mathbb{R}^d$ be a fixed frequency. We consider a family of orthogonal tapers functions $(f_i)_{i \in I} \in L^2(\mathbb{R}^d)^{|I|}$. The multitaper estimator of the structure factor $S$ of $\Phi$ associated to this family is defined as 
		\begin{equation}\label{eq:def_mult_tapers}
			\widehat{S}(k_0) := \frac1{|I|} \sum_{i \in I} \widehat{S}_i(k_0),
		\end{equation}
		where the single-taper estimator of $S$ is $
		\widehat{S}_i(k_0) := \left|C_i(k_0)\right|^2,$
		with 
		\begin{align}\label{eq:def_Ti_Ci}
			&	C_i(k_0) := T_i(k_0) - \mathbb{E}\left[T_i(k_0)\right],~ T_i(k_0) := \sum_{x \in \Phi \cap W} e^{- \bm{i} k_0.x} f_i(x),~ \mathbb{E}\left[T_i(k_0)\right] = \int_W f_i(x) dx.
		\end{align}
	\end{definition}
	
	\begin{remark}\label{rmk:mt_trunc}
		Unlike previous works on taper-based estimators for spatial point processes~\cite{rajala2023fourier, yang2024fourier, grainger2023spectral}, which consider tapers supported within the observation window $W$ and define $T_i(k_0)$ via sums over all points $x \in \Phi$, we consider here a more general setting where the tapers may have infinite support. However, in the following, we will require strong localization properties within the window to control the truncation error.
	\end{remark}
	
	The standard framework for analyzing the statistical properties of multitaper estimators in spatial statistics is the
	expanding-window regime, where the observation window $W$ grows to fill $\mathbb{R}^d$. In this setting, the tapers are chosen so that
	they become increasingly concentrated at the origin in the Fourier domain, i.e., $\mathcal{F}[f_i] \to \delta_0$ as $W \to \mathbb{R}^d$
	for every $i \in I$. When the number of tapers $|I|$ is kept fixed while $W$ expands, it is well established that multitaper estimators
	are asymptotically unbiased. Furthermore, under suitable mixing conditions, the linear statistics $(T_i(k_0))_{i\in I}$ (defined in
	\eqref{eq:def_Ti_Ci}) become asymptotically Gaussian and pairwise uncorrelated. As a result, one obtains the following convergence in
	distribution, for $k_0 \neq 0$
	(see \cite{brillinger1972spectral, mugglestone1996exploratory, rajala2023fourier, grainger2023spectral} for spatial point processes and~\cite{walden2000unified, karnik2022thomson}, Section~8 of \cite{percival2020spectral} for similar time series results): \begin{align}\label{heur:S_normal} \widehat{S}(k_0) \xrightarrow{d} S(k_0) \frac{\chi^2(2|I|)}{2|I|}, \quad \text{as } W \to \mathbb{R}^d. 
	\end{align}
	
	This limit shows, however, that if $|I|$ remains fixed as $W$ grows, then $\widehat{S}(k_0)$ is not consistent. This key observation,
	emphasized in the works cited above, compels us to let the number of tapers $|I|$ increase with the window size $W$. But this raises theoretical challenges that are critical for the implementation and the choice of the parameters of the multitaper estimators. For instance, increasing $|I|$ reduces variance by averaging, but at the same time it typically increases bias,
	since higher-order tapers have poorer spectral concentration. This naturally leads to a set of fundamental questions:

	\begin{itemize}
		\item How should $|I|$ scale with $W$ in order to preserve asymptotic unbiasedness?
		\item Does the convergence in distribution in \eqref{heur:S_normal} still hold when $|I|$ grows with $W$?
		\item More broadly, how does the rate of convergence depend on properties of the taper family, such as spatial support and spectral concentration?
	\end{itemize}
	
	The answers to the first and third questions are well established in the time series literature, typically under the assumption that the spectral density is twice differentiable (see, e.g.,~\cite{thomson1982spectrum, riedel1995minimum, percival2020spectral}). Recent works on spatial point processes~\cite{rajala2023fourier, grainger2023spectral} have addressed analogous questions, emphasizing the asymptotic interplay between spatial and spectral localization of the tapers. Within our minimax multitaper construction for the structure factor, we obtain a non-asymptotic $L^2(\mathbb{P})$ risk bound (Theorem~\ref{thm:L2_risk}) that encapsulates this trade-off and connects back to these classical results. By contrast, the second question appears to be less thoroughly investigated in the existing time series literature (see, e.g.,~\cite{karnik2022thomson}). In our spatial point process setting, we address this question by establishing a deviation inequality under a (Brillinger) mixing condition, which is well suited to point processes.
	
	\subsection{Multitaper estimators: mean integrated squared risk}\label{sec:l2_risk}
	
	This section establishes an upper bound on the pointwise $L^2(\mathbb{P})$-risk of the multitaper estimator of the structure factor, as stated in Theorem~\ref{thm:L2_risk}. The result holds for any orthonormal family of tapers in $ L^2(\mathbb{R}^d) $. It is further specialized in Corollary~\ref{cor:L2_risk_herm} to the case of Hermite tapers, yielding the minimax rate of convergence. Theorem~\ref{thm:L2_risk} and Corollary~\ref{cor:L2_risk_herm} are proved in Appendix~\ref{sec:proof_risk}, where the variance and the bias of the multitaper estimators are analyzed separately in Lemma~\ref{lem:var_S_hat} and Lemma~\ref{lem:bias_s_hat}.
	
	\begin{theorem}\label{thm:L2_risk}
		Let $\Phi$ be a stationary point process satisfying Assumption \ref{ass_rho_leb} with $\lambda = 1$. Let $W$ be a subset of $\mathbb{R}^d$, $I$ be a discrete subset of $\mathbb{N}^d$. We consider a family of orthonormal functions $(f_i)_{i \in I}$ of $L^2(\mathbb{R}^d)$. Assume that $S \in \Theta(\beta, L)$ with $L \in (0, \infty)$ and $\beta \in (0, 2]$ (see Definition~\ref{def:Hölder}). Then, 
		\begin{align}\label{eq:L2_risk}
			\sup_{k_0 \in \mathbb{R}^d} \mathbb{E}\left[\left|\widehat{S}(k_0) - S(k_0)\right|^2\right]^{\frac12} \leq   \frac{\sqrt{2}\|S\|_{\infty}}{|I|^{\frac12}} + \mathcal{F}\text{-}loc + B_4 + 6 \|S\|_{\infty} W\text{-}loc,
		\end{align} 
		where 
		\begin{align*}
			&\mathcal{F}\text{-}loc = \sqrt{d}L \frac1{|I|}\sum_{i \in I} \|f_i\|_{\dot{H}^{\beta/2}}^2,\quad W\text{-}loc := \left(\frac1{|I|}\sum_{i \in I} \|f_{i} \mathds{1}_{\mathbb{R}^d \setminus W}\|_2^2\right)^{\frac12},
			\\& B_4 := \left(\frac1{|I|} \sum_{i \in I} \|f_i\|_4^2\right)\left(1 + 7 |\gamma_{\text{red}}^{(2)}|(\mathbb{R}^d) + 6 |\gamma_{\text{red}}^{(3)}|(\mathbb{R}^{2d}) +|\gamma_{\text{red}}^{(4)}|(\mathbb{R}^{3d})\right)^{\frac12}.
		\end{align*}
	\end{theorem}
	
	We now discuss the quantities involved in the upper bound~\eqref{eq:L2_risk}. The term $\|S\|_{\infty}/|I|^{1/2}$ follows from the orthonormality of the taper family $(f_i)_{i \in I}$, reflecting the asymptotic pairwise decorrelation of the linear statistics $(T_i(k_0))_{i \in I}$, defined in \eqref{eq:def_Ti_Ci}, as $W \to \mathbb{R}^d$ and $\mathcal{F}[f_i] \to \delta_0$ for all $i \in I$, at a fixed number of tapers $|I|$.

	The remaining terms in~\eqref{eq:L2_risk}, namely \(\mathcal{F}\text{-}loc\), \(W\text{-}loc\), and \(B_4\), capture non-asymptotic effects. These quantities reflect, respectively, the lack of spectral concentration of the tapers around the origin, their imperfect spatial localization within the observation window \(W\), and higher-order cumulants' contributions. We analyze them below.
	
	\begin{enumerate} 
		\item \textit{$W\text{-}loc$ : border effects.} The term $W\text{-}loc$ arises because only the points of $\Phi$ within $W$ are observed. Consequently, when the taper functions $(f_i)_{i \in I}$ have too much mass outside of $W$, border effects occur. Accordingly, to ensure negligible border effects, one should consider tapers with good spatial localization inside $W$.
		
		\item \textit{$B_4$: Fourth order points correlations.} The term $B_4$ keeps track of the higher order correlations present in the underlying point process. This term is expected as the previous theorem deals with a squared risk of a second order statistics of stationary point processes. Note that this term is finite for Brillinger-mixing point processes (see Definition~\ref{def:brill_mix}), including most DPPs~\cite{heinrich2016strong, biscio2016brillinger} and Cox processes \cite{heinrich1985normal, zhu2022minimum, prokevsova2013asymptotic, jovanovic2015cumulants}. In addition, we observe a multiplicative term involving the average of the $L^4(\mathbb{R}^d)$-norm of the tapers. This term is related to the spatial localization of the tapers functions: for well-localized taper functions, this average decreases to zero as the number of tapers $|I|$ increases, providing a second argument in favor of well-localized tapers.
		
		\item \textit{$\mathcal{F}\text{-}loc$: Frequency-localization bias.} The term $\mathcal{F}\text{-}loc$ represents the bias arising from the localization of the taper functions in the Fourier domain. Since this term typically diverges as $|I| \to \infty$ at fixed $W$, it highlights the non-asymptotic trade-off between using many tapers, which reduces variance but increases bias, and using few tapers, which decreases bias but results in larger variance. In addition, its expression naturally motivates the use of an orthogonal family of tapers with small Sobolev norms, and hence with well-localized Fourier transforms. In this way, we recover, within a non-asymptotic framework, a classical requirement in taper theory~\cite{riedel1995minimum, percival2020spectral, karnik2022thomson, rajala2023fourier}.
	\end{enumerate}
	
	The previous discussion motivates the consideration of taper functions with good localization property in both the spatial and the Fourier domain. Accordingly, a balance must be found between good spatial localization, to reduce the bias $W\text{-}loc$ caused by border effects and the variance due to higher-order point correlations $B_4$, and good localization in the Fourier domain, to reduce the bias $\mathcal{F}\text{-}loc$. Common choices of families satisfying both requirements are the Slepian functions \cite{slepian1961prolate, landau1961prolate, simons2011spatiospectral}, sinusoidal tapers \cite{riedel1995minimum, walden2000unified} and the Hermite tapers \cite{bayram1996multiple, flandrin1988maximum, parks1990time, daubechies2002time}. In this work, we focus on the latter family.
	
	\begin{definition}\label{def:herm}
		For $i  = (i_1, \dots, i_d) \in \mathbb{N}^d$, the Hermite taper $\psi_i$ is defined as
		\begin{equation}\label{eq_herm}
			\forall x=(x_1,\dots,x_d)\in\mathbb R^d,~\psi_i(x) = e^{-\frac12|x|^2} \prod_{l = 1}^d H_{i_l}(x_l), 
		\end{equation}
		where  for $n \in \mathbb{N}$, $y \in \mathbb{R}$, $\displaystyle H_n(y) = \frac{(-1)^n}{(2^{n} n! \sqrt{\pi})^{\frac12}} e^{y^2} \dfrac{d^n}{dy^n} e^{-y^2}$ are the Hermite polynomials.
	\end{definition}

	The next corollary simplifies the terms appearing in Theorem \ref{thm:L2_risk} when considering scaled Hermite tapers. It states that, under appropriate choices of the parameters, the corresponding multitaper estimator is optimal regarding the minimax rate of Theorem~\ref{thm:minimax_S}, for structure factor of Hölder-regularity $\beta \leq 2$ and for square observation windows.
	
	\begin{corollary}\label{cor:L2_risk_herm}
		Let $\Phi$ be a stationary point process satisfying Assumption \ref{ass_rho_leb} with $\lambda = 1$. Assume that $S \in \Theta(\beta, L)$ with $L \in (0, \infty)$ and $\beta \in (0, 2]$ (see Definition~\ref{def:Hölder}). Let $R > 0$, $W = [-R, R]^d$ and $I = \{i \in \mathbb{N}^d|~|i|_{\infty} < i_{\max}\}$ with $i_{\max} \geq 1$. We consider, for all $i \in I$, $f_i = r^{-\frac{d}2}\psi_i(\cdot/r)$, where $\psi_i$ is the $i$-th Hermite function introduced in Definition~\ref{def:herm}~and 
		\begin{equation}\label{cond:risk_herm_rd}
			r = R/\sqrt{2 i_{\max}+ i_{\max}^{1/3+\theta}}
		\end{equation}
		where $\theta \in (0, 2/3)$. Suppose that
		\begin{equation}\label{cond:risk_herm_I}
			|I| \left(|W|^{2\beta/(2\beta + d)} \left(\frac{\|S\|_{\infty}}{L}\right)^{\frac{d}{2\beta+d}}\right)^{-1} \in (c_2, c_3),
		\end{equation}
		for some constants $0 < c_2 \leq c_3 < \infty$. Then, we have 
		\begin{align*}
			\sup_{k_0\in \mathbb{R}^d}\mathbb{E}\left[\left|\widehat{S}(k_0) - S(k_0)\right|^2\right]^{\frac12} \leq K  \left(\frac{(L^d \|S\|_{\infty}^{\beta})^{\frac1{2\beta +d}}}{|W|^{\frac{\beta}{2\beta +d}}} +  \sqrt{\frac{|\gamma_{\text{red}}^{4}|\log(|W|)^{d}}{|W|}} + e^{-c |W|^{\alpha}} \right),
		\end{align*}
		where $
		|\gamma_{\text{red}}^{4}| := 1 + 7 |\gamma_{\text{red}}^{(2)}|(\mathbb{R}^d) + 6 |\gamma_{\text{red}}^{(3)}|(\mathbb{R}^{2d}) +|\gamma_{\text{red}}^{(4)}|(\mathbb{R}^{3d})$, $\alpha = 3\beta\theta/(d(2\beta+d))$ and $0 < c, K < \infty$ are constants that do not depend on $W$, $r$ and $I$.
	\end{corollary}
	
	In the previous result, condition~\eqref{cond:risk_herm_rd} ensures exponentially negligible border effects $W\text{-}loc$ by localizing all the Hermite tapers $(\psi_i(\cdot/r))_{i \in I}$ in the observation window $W$. Then, condition~\eqref{cond:risk_herm_I}, concerning the number of tapers, balances the frequency bias term $\mathcal{F}\text{-}loc$ of order $L (|I|/|W|)^{\beta/d}$, and the asymptotic variance term $\|S\|_{\infty}|I|^{-1/2}$ (see equation \eqref{eq:risk_l2_almost_final}). It leads to the dominant term of order $(L^d \|S\|_{\infty}^{\beta})^{1/(2\beta +d)}/|W|^{\beta/(2\beta +d)}$. The term accounting for higher-order correlations, namely $|\gamma_{\mathrm{red}}^{4}|^{1/2} \log(|W|)^{d/2}/|W|^{1/2}$, vanishes at the parametric rate (up to a logarithmic factor). 
	
	\subsection{Multitaper estimators: non asymptotic fluctuations}\label{sec:nn_asymp}

	In this section, we establish in Theorem~\ref{thm:conc} a high probability version of Theorem~\ref{thm:L2_risk}. Then, in Corollary~\ref{cor:conc_herm}, we specify this bound for Hermite tapers. 
	
	The proof of Theorem~\ref{thm:conc}, deferred to Appendix~\ref{sec:proof_conc}, relies on bounding the cumulants of the multitaper estimator $\widehat{S}(k_0)$ and on showing that the orthogonality of the tapers, which at first sight affects only second-order moments, in fact propagates to higher-order cumulants. Then, the deviation inequality follows from the method of cumulants~\cite{bentkus1980exponential, saulis2012limit, doring2022method}. Accordingly, the results in this section rely on the following strong Brillinger mixing condition.
	
	\begin{assumption}\label{def:brill_mix}
		A simple stationary point process $\Phi$ is said is $(c, \gamma,q)-$Brillinger-mixing if there exist $c \geq 1, \gamma \geq 0$ and $q \geq 1$, such that for all $m \geq 2$ $$|\gamma_{\text{red}}^{(m)}|(\mathbb{R}^{d(m-1)}) \leq c ((m-1)!)^{1+\gamma} q^{m-1}.$$
	\end{assumption}
	Assumption~\ref{def:brill_mix} is standard in spatial statistic and allows for a quantitative analysis of the fluctuations of linear statistics of point processes; see, e.g., \cite{brillinger1972spectral, brillinger1994uses, heinrich1985normal, heinrich1988asymptotic, heinrich2016strong, khabou2024gaussian}. When $\gamma = 0$, it can be seen as a sub-exponential condition (see, e.g., \cite{vershynin2018high}, Chapter 2). Note that, for a stationary $\alpha$-determinantal point process with stationary kernel $K \in L^1(\mathbb{R}^d)$, previous assumption holds with $\gamma = 0$, $c = 1$ and $q = \|K\|_1$ (refer to Theorem 2 of~\cite{heinrich2016strong}). Assumption~\ref{def:brill_mix} is well-suited to our setting, since the multitaper estimator is built from linear statistics.

	The following theorem provides a high-probability bound on the fluctuations of the multitaper estimator. It highlights three distinct concentration regimes, which are detailed in the discussion following the statement.
	
	\begin{theorem}\label{thm:conc}
		Let $I$ be a discrete subset of $\mathbb{N}^d$, $W$ be a compact set of $\mathbb{R}^d$ and $(f_i)_{i \in I}$ be a family of orthonormal functions of $L^2(\mathbb{R}^d)$ that are in $L^1(\mathbb{R}^d) \cap L^{\infty}(\mathbb{R}^d)$. Assume that there exist $c_f \in (1, \infty)$, $\eta_1, \eta_2 \geq 0$ and $\rho > 0$ such that 
		\begin{align}\label{ass:cesaro_f}
			\forall s\geq 4,~\left(\frac1{|I|} \sum_{i \in I} \|f_i\|_s^2 \right)^{\frac{s}2} \leq \frac{c_f \log(|I|^{\frac1d})}{\rho^{d(s/2-1)}|I|^{s \eta_1+ \eta_2}},
		\end{align}
		\begin{align}\label{ass:cesaro_f3}
			\forall s \geq 3,~\left(\frac1{|I|} \sum_{i \in I} \|f_i\|_s^2 \right)^{\frac{s}2} \leq \frac{c_f}{\rho^{d(s/2-1)}|I|^{s \eta_1}}.
		\end{align}
		Let $\Phi$ be a stationary point process satisfying Assumption \ref{ass_rho_leb} with $\lambda= 1$. Moreover, suppose that $\Phi$ is $(c, \gamma,q)-$Brillinger-mixing (see Assumption \ref{def:brill_mix}) and that $S \in \Theta(\beta, L)$ with $L \in (0, \infty)$ and $\beta \in (0, 2]$. If $\beta > 1$, assume for all $(i, i') \in I^2$, $\mathcal{F}[f_i](k) \mathcal{F}[f_{i'}](k) =  \mathcal{F}[f_i](-k) \mathcal{F}[f_{i'}](-k)$ for all $k \in \mathbb{R}^d$. Suppose that
		\begin{align}\label{cond:rI}
			c \rho^{d} \geq (1+q)~\frac{|I|^{\eta_2}}{\log(|I|^{\frac{1}d}) } \vee 1.
		\end{align}
		Let $k_0 \in \mathbb{R}^d$. We denote
		\begin{align}
			& \Sigma :=\frac{S(k_0)}{|I|} + \sqrt{d} L\left(\frac1{|I|^2}\sum_{i \in I} \|f_i \|_{\dot{H}^{\beta}}^2\right)^{\frac12} + 4 \|S\|_{\infty} \left(\frac1{|I|}\sum_{i \in I} \|f_i \mathds{1}_{\mathbb{R}^d \setminus W}\|_2^2\right)^{\frac12} \label{eq:sigma}\\& \qquad \qquad \qquad\qquad\qquad \quad \qquad+ \frac{\|S\|_{\infty}}{|I|^{\frac12}} \wedge 2 \|S\|_{\infty} \left(\frac1{|I|}\sum_{i \in I} \|\mathcal{F}[f_i] \mathds{1}_{|\cdot| \geq |k_0|}\|_2^2\right)^{\frac12}, \nonumber
			\\& \Lambda:= \left\lceil c_f \vee \frac{6c \rho^{\frac{3d}2}|I|^{3\eta_1}}{1+q}   \left(\sup_{i \in I} \|f_i \mathds{1}_{\mathbb{R}^d \setminus W}\|_3 \sup_{i \in I} \|f_i\|_3^2 +  \sup_{i \in I} \|\mathcal{F}[f_i] \mathds{1}_{|\cdot| \geq |k_0|/3}\|_1\right) \right\rceil, \label{eq:Lambda}
			\\& \text{B}_{\infty} := \frac{ \Lambda^2 c^{1/2} (1+q)^{3/2} \log(|I|^{\frac1d})^{\frac12}}{\rho^{d/2}|I|^{2\eta_1+\eta_2/2}} \wedge \frac{ \lceil c_f \rceil^2 c^{2/3} (1+q)^{4/3} }{\rho^{d/3}|I|^{2\eta_1}}. \nonumber 
		\end{align}
		Then, for all $x \geq 0$, we have
		\begin{align}\label{eq:conc}
			\mathbb{P}[|\widehat{S}(k_0) - \mathbb{E}[\widehat{S}(k_0)]| \geq x] \leq C\max(a_1, a_2, a_3),
		\end{align}
		where $C < \infty$ is a numerical constant and
		\begin{align*}
			& a_1 := \exp\left(-\frac1{12}\frac{x^2 |I|}{S(k_0)^2 + x S(k_0)}\right),
			\quad a_2 := \exp\left(- \frac{x^2}{13^{\gamma}39 \Sigma^2 + 1.5 x^2 (x \Sigma)^{-1/(1+\gamma)}}\right),
			\\&a_3:= \exp\left(- \frac{x^2}{13^{\gamma}(71 B_{\infty})^2 + 1.5 x^2(2x B_{\infty})^{-1/(2+\gamma)}}\right).
		\end{align*}
		
	\end{theorem}
	
	We begin by discussing the assumptions of the previous theorem. Assumptions~\eqref{ass:cesaro_f} and~\eqref{ass:cesaro_f3} concern the localization properties of the tapers and are required to control high-order cumulants. The parameter $\rho$ in these conditions should be thought of as a scaling factor related to the localization of the tapers within the observation window $W$, analogous to the dilation factor $r$ in Corollary~\ref{cor:L2_risk_herm}. Assumption~\eqref{cond:rI} is of a technical nature. Additionally, the symmetry condition on the Fourier transform of the tapers is used to cancel a first-order term in the Taylor expansion of the structure factor near $k_0$.
	
	The terms $a_1$, $a_2$, and $a_3$ in \eqref{eq:conc} correspond to different concentration types: $\chi^2$, $1+\gamma$ stretched exponential, and $2+\gamma$ stretched exponential, respectively, as detailed below.
	\begin{enumerate}
		\item \textit{$\chi^2$ concentration (fixed-$|I|$ regime).}  
		The term $a_1$ reflects the concentration of the multitaper estimator in the asymptotic regime where $|I|$ is fixed and $\mathcal{F}[f_i] \to \delta_0$ for all $i \in I$ as $W \to \mathbb{R}^d$ (see the discussion following Remark~\ref{rmk:mt_trunc}). Indeed, in this setting, $\widehat{S}(k_0)$ is approximately distributed as $S(k_0)\chi^2(2|I|)/(2|I|)$ (see \eqref{heur:S_normal}). Remarkably, Corollary~\ref{cor:conc_herm} above shows that even when the number of tapers $|I|$ diverges with $W$, provided it scales appropriately, this term remains the leading contribution to the fluctuations of $\widehat{S}(k_0)$, when $k_0$ is not too close to zero.
		
		\item \textit{$1+\gamma$ stretched exponential concentration (second order decorrelation error).}  
		The term $a_2$ accounts for the fact that the linear statistics $T_i(k_0)$ (see Definition~\ref{def:multi_tap}) are only asymptotically uncorrelated. For $i_1, i_2 \in I$,
		\[
		\operatorname{Cov}\left[T_{i_1}(k_0), T_{i_2}(k_0)\right] = S(k_0)\mathds{1}_{i_1=i_2} + \mathrm{Err}(i_1, i_2),
		\]
		where the error term $\mathrm{Err}(i_1, i_2)$ captures residual non-asymptotic correlations. This non-asymptotic correlation leads to an $1+\gamma$ stretched exponential deviation component. Whether $a_2$ dominates depends on the decay of 
		\[
		\frac{\|S\|_{\infty}}{|I|^{1/2}} \wedge 2\|S\|_{\infty} \left(\frac{1}{|I|}\sum_{i \in I} \|\mathcal{F}[f_i]\mathds{1}_{|\cdot|\geq |k_0|}\|_2^2\right)^{1/2},
		\]
		which typically depends on how small $k_0$ is.
		
		\item \textit{$2+\gamma$ stretched exponential concentration (high-order correlations).}  
		The term $a_3$ captures the contribution of higher-order cumulants of the statistics $(T_i(k_0))_{i \in I}$ (see Definition~\ref{def:multi_tap}). For the discussion, assume that $\gamma = 0$. Under Brillinger-mixing, the total variation of the $m$-th order reduced factorial cumulant measure are bounded by $(m-1)!$ (up to constants). But the sequence $((m-1)!)_{m \geq 1}$ is the cumulants of a standard exponential random variable. Accordingly, the higher-order cumulants of the statistics $\widehat{S_i}(k_0) = \left|T_i(k_0) - \mathbb{E}[T_i(k_0)]\right|^2$, for $i \in I$ behave, in the non-asymptotic regime, similarly to the higher-order cumulants of squares of centered exponential random variable variables. The contribution of $a_3$ is determined by $B_\infty$, which splits into two cases. To fix the idea, in the case of the Hermite tapers (see the next corollary), if
		$
		\log(|I|^{1/d})^{1/2}/(\rho^{d/2}|I|^{2\eta_1+\eta_2/2})
		$
		dominates, then $B_{\infty}$ matches the fourth-order cumulant bound $B_4$ from Theorem~\ref{thm:L2_risk}. If instead
		$
		1/(\rho^{d/3}|I|^{2\eta_1})$
		dominates (which could occur for small $k_0$ since in that case $\Lambda$ could be large), the higher-order joint cumulants contribute more significantly than the term $B_4$ appearing in the $L^2(\mathbb{P})$-risk.
	\end{enumerate}

	The following corollary establishes that the multitaper estimator based on Hermite tapers concentrates at the minimax rate with $\chi^2$-type concentration, provided the target frequency $k_0$ is not too small. When $k_0$ is small, the estimator still achieves the minimax rate if $\beta \leq d$, but with stretched-exponential concentrations. We detail these points in the discussion following the theorem.  
	
	\begin{corollary}\label{cor:conc_herm}
		Let $\Phi$ be a stationary point process satisfying Assumption \ref{ass_rho_leb} with $\lambda = 1$.  Moreover, suppose that $\Phi$ is $(c, \gamma,q)-$Brillinger-mixing (see Assumption \ref{def:brill_mix}) and that $S \in \Theta(\beta, L)$ with $L \in (0, \infty)$ and $\beta \in (0, 2]$ (see Definition~\ref{def:Hölder}). Let $R > 0$ and $W = [-R, R]^d$. Consider $I = \{i \in \mathbb{N}^d|~|i|_{\infty} < i_{\max}\}$ with $i_{\max} \geq 1$. Let $\epsilon \in \{-1, 1\}$. If $\beta \geq 1$, we add the constraint $~\psi_i(-\cdot) = \epsilon \psi_i$ in the indexes of $I$, where $\psi_i$ is the $i$-th Hermite function introduced in Definition~\ref{def:herm}. We consider, for all $i \in I$, $f_i = r^{-d/2}\psi_i(\cdot/r)$ where
		\begin{equation}\label{cond:rd_herm1}
			r = R/\sqrt{2 i_{\max}+ i_{\max}^{1/3+\theta}},
		\end{equation}
		where $\theta \in (0, 2/3)$. Suppose that $|I| \geq 2$ and
		\begin{equation}\label{cond:I_herm1}
			|I| |W|^{-2\beta/(2\beta + d)} \in (c_2, c_3),
		\end{equation}
		for some constants $0 < c_2\leq c_3 < \infty.$ Then, there exists $v_0, c_0, c_4> 0$ and $K < \infty$ such that for $|W| \geq v_0$ and $x \geq 0$, we have with probability at least $1 - x$:
		\begin{align}\label{eq:conc_herm1}
			\sup_{|k_0| \geq |k^*_W|} \mathbb{P}\left[\left|\widehat{S}(k_0) - S(k_0)\right| \geq  x + \frac{K}{|W|^{\frac{\beta}{2\beta +d}}}\right] \leq  K\Bigg(\exp\left(-\frac{c_2^2}{12}\frac{(x |W|^{\frac{\beta}{2\beta +d}})^2}{S(k_0)^2 + x S(k_0)}\right) + \varepsilon_1\Bigg), 
		\end{align}
		where $|k^*_W|:= c_0 |W|^{-\frac{1}{2\beta +d}}$ and
		\begin{align*}
			\varepsilon_1 := \exp\left(-c_4 (x |W|^{\frac{2\beta}{2\beta +d}})^2 \wedge (x |W|^{\frac{2\beta}{2\beta +d}})^{\frac1{1+\gamma}}\right) + \exp\left(-c_4 \left( \frac{x|W|^{\frac12}}{\log(|W|)^{\frac{d}2}}\right)^2 \wedge \left(\frac{x|W|^{\frac12}}{\log(|W|)^{\frac{d}2}}\right)^{\frac1{2+\gamma}}\right).
		\end{align*}
		Moreover,
		\begin{align}\label{eq:conc_herm2}
			\sup_{|k_0| \leq |k^*_W|} \mathbb{P}\left[\left|\widehat{S}(k_0) - S(k_0)\right| \geq  x + \frac{K}{|W|^{\frac{\beta}{2\beta +d}}}\right] \leq K(A_1+A_2+A_3),
		\end{align}
		where
		\begin{align*}
			&A_1 := \exp\left(-\frac{c_2^2}{12}\frac{(x |W|^{\frac{\beta}{2\beta +d}})^2}{S(k_0)^2 + x S(k_0)}\right),
			\quad A_2 :=  \exp\left(-c_4 (x |W|^{\frac{\beta}{2\beta +d}})^2 \wedge (x |W|^{\frac{\beta}{2\beta +d}})^{\frac1{1+\gamma}}\right),
			\\&A_3 := \exp\left(-c_4 (x |W|^{\frac13})^2 \wedge (x |W|^{\frac13})^{\frac1{2+\gamma}}\right).
		\end{align*}
	\end{corollary}
	
	We now discuss the assumptions of the theorem. Conditions~\eqref{cond:rd_herm1} and~\eqref{cond:I_herm1} are identical to those of Corollary~\ref{cor:L2_risk_herm}. Condition~\eqref{cond:rd_herm1} localizes the tapers inside the observation window while Condition \eqref{cond:I_herm1} balances bias and variance. The threshold $|k^*_W|$, corresponding to the minimal allowed frequency, separates the regime of large frequencies $|k_0| \geq |k^*_W|$, where $T_i(k_0)$ and $T_i(-k_0)$ (see Definition~\ref{def:multi_tap}) asymptotically decorrelate, from the small frequency regime $|k_0| \leq |k^*_W|$, where they do not. 
	
	For large frequencies $|k_0| \geq |k^*_W|$, the remainder term $\epsilon_1$, involving the rates $|W|^{2\beta/(2\beta + d)}$ and $|W|^{1/2}/\log(|W|)^{d/2}$, becomes negligible as $|W| \to \infty$. For this situation, even if $|I|$ diverges with $W$, the bound matches the minimax rate, with the anticipated $\chi^2(2|I|)/(2|I|)$ concentration occurring at fixed $|I|$ (see the heuristic equation \eqref{heur:S_normal}). In contrast, for small frequencies $|k_0| \leq |k^*_W|$, the terms $A_1$ and $A_2$ contribute to the bound. For small deviations $x$, $A_2$, which captures the non-asymptotic correlation between tapers, dominates $A_1$. The third term, $A_3$, related to higher-order correlations between single-taper estimators, is negligible compared to $A_1$ and $A_2$ if $\beta < d$, since in this case $1/3 > \beta/(2\beta + d)$. Accordingly, in that case the multitaper estimator still concentrates at the minimax rate, but with $1+\gamma$ stretched exponential concentration. However, when $\beta = d$, $A_3$ dominates $A_1$ and $A_2$ for small $x$. But, the bound remains minimax with a $2+\gamma$ stretched exponential concentration. Since $\beta$ is assumed to be smaller than two, the Hermite multitaper estimator always concentrates at the minimax rate in dimensions $d \geq 2$, under the conditions of the previous corollary.
	
	\section{Data-driven procedure for the choice of the number of tapers}\label{sec:dd}
	
	To select the number of tapers in practice, we design a local cross-validation procedure. In Section~\ref{sec:dd_def}, we describe the approach and provide theoretical justifications. Then, in Section~\ref{sec:num}, we evaluate its performances through numerical simulations. The implementation used for these experiments is available at \href{https://github.com/gabrielmastrilli/multitaper_spp}{\text{https://github.com/gabrielmastrilli/multitaper\_spp}}.

	\subsection{Description and theoretical guarantee}\label{sec:dd_def}
	
	To select the number of tapers for an estimation of $S$ at frequency $k_0$, we aim to choose $|I|$ to minimize the MSE (mean squared error) between $\widehat{S}_I(k_0)$ and a pilot estimator $\widehat{S}_p(k_0)$. The procedure relies on the following principles.
	\begin{enumerate}
		\item Rather than considering a single frequency $k_0$, we minimize the MSE between $\widehat{S}_I(k_j)$ and $\widehat{S}_p(\tilde{k}_j)$ over a collection of frequency pairs $(k_j, \tilde{k}_j)_{j \in [N]}$, for $N \geq 1$. These frequencies are chosen such that, for all $j \in [N]$, $S(k_j) = S(\tilde{k}_j)$ and, on the other hand, such that $S(k_j)$ equals $S(k_0)$ or is at least close to it. The pairs are moreover chosen sufficiently separated to ensure that:
		\begin{itemize}
			\item $\widehat{S}_I(k_j)$ and $\widehat{S}_p(\tilde{k}_j)$ are as decorrelated as possible, and
			\item for $j \neq j'$, $(\widehat{S}_I(k_j), \widehat{S}_p(\tilde{k}_j))$ and $(\widehat{S}_I(k_{j'}), \widehat{S}_p(\tilde{k}_{j'}))$ are also as decorrelated as possible.
		\end{itemize}
		In the isotropic case, one may choose the pairs $(k_j, \tilde{k}_j)_{j \in [N]}$ as sufficiently well-spread frequencies with norm $|k_0|$.
		\item To further reduce correlations, $\widehat{S}_I$ and $\widehat{S}_p$ are computed on two subsamples obtained by a $0.5$-$0.5$ independent thinning of the original point pattern. This thinning is performed independently for each frequency pairs.
		\item The pilot estimator $\widehat{S}_p$ is chosen to have small bias, which is achieved by using a small number of tapers.
	\end{enumerate}
	
	More rigorously, let $\Phi$ be a stationary point process. We denote by $N$ the number of frequency pairs $(k_j, \widetilde{k_j})_{j = 1}^N$. Let $I \subset \mathbb{N}^d$ be a finite index set. For each $j \in [N]$, define $\Phi_1^j$ as an independent thinning of $\Phi$ with retention probability $1/2$. Let $\Phi_2^j := \Phi \setminus \Phi_1^j$ be its complement. The multitaper estimator $\widehat{S}_I(k_j)$ is computed from $\Phi_1^j$ at frequency $k_j$ and using Hermite tapers $(r^{-d/2} \psi_i(\cdot / r))_{i \in I}$ for some dilation factor $r > 0$. Similarly, let $I_p \subset \mathbb{N}^d$ be another index set and define the pilot estimator $\widehat{S}_p(\widetilde{k_j})$ computed from $\Phi_2^j$ at frequency $\widetilde{k_j}$, using Hermite tapers $(r_p^{-d/2} \psi_i(\cdot / r_p))_{i \in I_p}$, for some $r_p > 0$. Let $\mathcal{I}$ be a finite set of candidate taper index sets. We define the cross-validation criterion:
	\begin{equation}\label{eq:crit_cv}
		\widehat{\mathcal{R}}(I) := \frac{1}{N} \sum_{j=1}^N \left( \widehat{S}_I(k_j) - \widehat{S}_p(\widetilde{k_j}) \right)^2.
	\end{equation}
	Finally, for the estimation of S at frequency $k_0$, we select the taper index set $I \in \mathcal{I}$ that minimizes $\widehat{\mathcal{R}}(I)$. 
	
	\begin{remark}
		The use of the families of thinned point processes $(\Phi_1^j)_{j \in [N]}$ and $(\Phi_2^j)_{j \in [N]}$ enhances the decorrelation between $\widehat{S}_{I}(k_j)$ and $\widehat{S}_p(\widetilde{k_j})$. For instance, in the Poisson case, for each $j \in [N]$, $\Phi_1^j$ and $\Phi_2^j$ are independent, and averaging over thinnings can be shown to be beneficial~\cite{lin2024efficient}. Even in the non-Poisson case, thinning has been observed to perform well for cross-validation~\cite{cronie2024cross}. However, thinning implies that both estimators $\widehat{S}_{I}(k_j)$ and $\widehat{S}_p(\widetilde{k_j})$ target the structure factor $S_{1/2} := 1/2 + S/2$ rather than $S$ itself (see Lemma \ref{lem:cov_thin}). Nevertheless, due to the similarity between $S$ and $S_{1/2}$, the number of tapers that minimizes the risk for estimating $S_{1/2}$ is also suitable for estimating $S$.
	\end{remark}
	
	We provide with Theorem \ref{thm:dd} a theoretical justification for the proposed cross-validation procedure. To state it, we introduce the notion of $k_0$-allowed frequency pairs. 
	
	\begin{definition}\label{def:k0_sep}
		Let $k_0 \in \mathbb{R}^d$ and $N \geq 1$. The family $(k_j, \widetilde{k_j})_{1 \leq j \leq N}$ of $\mathbb{R}^d$ is said $k_0$-allowed if $S(k_j) = S(\widetilde{k_j})$ for all $j \in [N]$ and if there exists $(c_1, c_2, c_3) \in (0, \infty)^3$ such that for all  $j \in [N]$
		\begin{enumerate}
			\item \label{pts:sep}$|k_j| \geq c_1 |k_0|$, $|\widetilde{k_j}| \geq c_1 |k_0|$, $|k_j - \widetilde{k_j}| \geq c_2 |k_0|$ and $|k_j + \widetilde{k_j}| \geq c_2 |k_0|$,
			\item \label{pts:no_clust}for all $a > 0$
			\begin{align*}
				\sum_{j' = 1}^N \mathds{1}_{|k_{j} - k_{j'}| \leq a} + \mathds{1}_{|\widetilde{k_{j}} - k_{j'}| \leq a} + \mathds{1}_{|k_{j} - \widetilde{k_{j'}}| \leq a} + \mathds{1}_{|\widetilde{k_{j}} - \widetilde{k_{j'}}| \leq a}  \leq c_3 \left( N \left(\frac{a}{|k_0|}\right)^{d-1}+1\right).
			\end{align*}
		\end{enumerate}
		
	\end{definition} 
	
	The $k_0$-allowed assumption on the frequency pairs $(k_j, \widetilde{k_j})_{1 \leq j \leq N}$ ensures sufficient separation between frequencies, as stated in point~\ref{pts:sep}. This guarantees the required decorrelation between single-taper estimators evaluated at distinct frequencies. Point~\ref{pts:no_clust} ensures that the frequencies are well spread and do not form clusters. For instance, in dimension~$2$, the $k_0$-allowed assumption is satisfied by choosing 
	$
	\{(k_j, \widetilde{k}_j)\}_{j \in [N]} = \mathcal{C}(|k_0|),
	$
	where
	\begin{align}\label{eq:freq_pair}
		\mathcal{C}(q) := \left\{q\left(\cos\!\left(\tfrac{2\pi j}{N}\right), \, \sin\!\left(\tfrac{2\pi j}{N}\right)\right), 
		q\left(\sin\!\left(\tfrac{2\pi j}{N}\right), \, \cos\!\left(\tfrac{2\pi j}{N}\right)\right)\right\}_{j \in [N]}, \quad N \geq 1.
	\end{align}
	With this choice, the symmetry condition $S(k_j) = S(\widetilde{k}_j)$ for all $j \in [N]$ is satisfied whenever the structure factor is symmetric under coordinate exchange, i.e.
	$
	S(k_1, k_2) = S(k_2, k_1)  \text{ for all } (k_1, k_2) \in \mathbb{R}^2.
	$ In practice, it may be desirable to relax this strict symmetry requirement in order to increase the number of frequency pairs. A natural compromise is to allow pairs $(k_j, \widetilde{k}_j)$ such that $S(k_j)$ is close to $S(k_0)$, though not necessarily equal. This enlargement promotes greater decorrelation and reduces the variance of the criteria $\widehat{\mathcal{R}}(I)$. In particular, one may consider
	$
	\bigcup_{q \in Q} \mathcal{C}(q),
	$
	for a finite set $Q$ of positive numbers with $q$ close to $|k_0|$.
	
	The following theorem, proved in Appendix~\ref{sec:dd_proof}, shows that under suitable conditions, the taper index set selected by minimizing the empirical criterion $\widehat{\mathcal{R}}$ achieves, asymptotically, near-optimal risk for estimating the structure factor of a $1/2$-independent thinning of $\Phi$.

	\begin{theorem}\label{thm:dd}
		Assume that $d \geq 2$. Let $\Phi$ be a stationary point process satisfying Assumption \ref{ass_rho_leb} with structure factor $S \in \Theta(\beta, L)$ where $\beta \leq 2$ and $L \in (0, \infty)$. Assume that for all $2 \leq m \leq 8$:
		$|\gamma_{\text{red}}^{(m)}|(\mathbb{R}^{d(m-1)}) < \infty.$  Let $k_0 \in \mathbb{R}^d$ be a target frequency. Consider $N \geq 1$ $k_0$-allowed frequencies $(k_j, \widetilde{k_j})_{1 \leq j \leq N}$ in $\mathbb{R}^d$. We denote 
		$A_W  = |k_0| |W|^{1/(2\beta +d)}.$ Suppose that $N\geq A_W^{2d(d-1)/(d+1)}$ and that $A_W\to \infty$ as $|W| \to \infty$. We consider $$\mathcal{I} := \left\{I|~c_5 (|W|^{\frac{2\beta}{2\beta +d}}/A_W^{d(d-1)/(d+1) - \eta} \vee \log(|W|)^{2d/(3\theta)+\eta})\leq |I| \leq c_6 |W|^{\frac{2\beta}{2\beta +d}} \right\},$$
		where  $0 < c_5 \leq c_6 < \infty$, $\eta > 0$, $\theta \in (0, 2/3)$ and the tapers index sets $I$ are of the form $\{i \in \mathbb{N}^d|~|i|_{\infty} < i_{\text{max}}\}$ for some $i_{\text{max}} \geq 1$. Assume that the dilation factor $r$ of the tapers $(r^{-d/2} \psi_i(\cdot / r))_{i \in I}$ corresponding to $I \in \mathcal{I}$ satisfies
		\begin{equation}\label{eq:rI_dd}
			r = R/\sqrt{2 i_{\max}+ i_{\max}^{1/3+\theta}}.
		\end{equation} 
		Assume that the tapers index set $I_p$ of the pilot estimator belongs to $\mathcal{I}$ and satisfies $|I_p| = o(|W|^{2\beta/(2\beta +d)})$ as $|W| \to \infty$. Consider $
		\widehat{I} := \underset{I \in \mathcal{I}}{\operatorname{\arg\min}}~\widehat{\mathcal{R}}(I),$
		where $\widehat{\mathcal{R}}(I)$ is defined in \eqref{eq:crit_cv}. Then, as $|W| \to \infty$ we have
		\begin{align*}
			\frac{1}{N} \sum_{j = 1}^N  \mathbb{E}[(\widehat{S}_{\widehat{I}}(k_j) - S_{\frac12}(k_j))^2] \leq \underset{I \in \mathcal{I}}{\min} \left\{\frac{1}{N} \sum_{j = 1}^N  \mathbb{E}[(\widehat{S}_{I}(k_j) - S_{\frac12}(k_j))^2]\right\} + o(|\mathcal{I}|^{\frac12} |W|^{\frac{2\beta}{2\beta +d}}),
		\end{align*}
		where $S_{\frac12} = 1/2+S/2$.
	\end{theorem}
	
	Before turning to the assumptions of the theorem, we first clarify the guarantee it provides. When the frequency pairs are chosen so that $S(k_j) = S(k_0)$, the result ensures that the cross-validation method selects asymptotically the optimal number of tapers in the sense that it minimizes the total MSE of the estimators $\widehat{S}_{\widehat{I}}(k_j)$ of $S_{\frac12}(k_0)$:
	\begin{align}\label{def:risk_moy}
		\text{tMSE}(k_0): I \in \mathcal{I} \mapsto \sum_{j=1}^N \mathbb{E}\big[(\widehat{S}_{I}(k_j) - S_{\tfrac12}(k_0))^2\big].
	\end{align}
	
	As discussed after Definition~\ref{def:k0_sep}, it may be useful to consider a weaker requirement than $S(k_j) = S(k_0)$. For example, suppose that the target frequencies are taken from $\cup_{q \in Q} \mathcal{C}(q)$, where $Q$ is a finite set of positive numbers and $\mathcal{C}(q)$ is defined in \eqref{eq:freq_pair}. If, for each $q \in Q$, there exists a frequency $k^q$ such that $S(k) = S(k^q)$ for all $k \in \mathcal{C}(q)$, then the theorem provides a guarantee for the sum, over $q$, of the total mean squared errors $\text{tMSE}(k^q)$ defined in \eqref{def:risk_moy}. Provided that for each $q \in Q$, $S(k^q)$ does not differ too much from $S(k_0)$, this still provides a local guarantee.

	Concerning the assumption of the previous theorem, the quantity $A_W$ already appeared in Theorem~\ref{cor:conc_herm} as a threshold under which correlation between tapers associated to opposite frequencies becomes non negligible. The condition $A_W \to \infty$, together with the $k_0$-allowed assumption, ensures asymptotic decorrelation between frequency pairs by forcing the frequencies $(k_j, \widetilde{k_j})_{1 \leq j \leq N}$ to be sufficiently far from zero and well-spread. As a consequence, we expect the data-driven procedure to perform better at larger frequencies $k_0$. Regarding the taper index sets $I \in \mathcal{I}$, the lower bound on their cardinality $|I|$ prevents the selection of estimators with excessively high variance, while the upper bound excludes estimators with large bias. Lastly, the condition $|I_p| = o(|W|^{2\beta/(2\beta + d)})$ ensures that the pilot estimator $\widehat{S}_p$ has negligible bias.		
	
	The remainder term $o(|\mathcal{I}|^{1/2} |W|^{2\beta/(2\beta +d)})$ in the previous theorem is negligible compared to the minimax rate established in Theorem~\ref{thm:minimax_S}, provided that the number of candidate taper index sets $|\mathcal{I}|$ remains bounded as $|W| \to \infty$. The factor $|\mathcal{I}|^{1/2}$ arises from the fact that the proof of Theorem~\ref{thm:dd} relies on controlling the second moment of a suitable transformation of the criterion $\widehat{\mathcal{R}}(I)$ (see Section \ref{sec:dd_proof}).
	
	\subsection{Numerical study}\label{sec:num}
	
	In this section we assess the practical performance of the data-driven criterion.
	
	\subsubsection{Description of the models}
	
	\begin{table}[h!]
		\centering
		\caption{Models considered and corresponding structure factors.}
		\begin{tabular}{ccc}
			\toprule
			Model & Parameter(s) & Structure factor \\
			\midrule
			Thomas cluster & $(\alpha, \sigma^2, \mu) = (5, 0.5^2, 0.2)$ & $1 + 5\exp(-|k|^2/4)$ \\
			\midrule
			Matérn cluster & $(\alpha, \rho, \mu) = (5, 1.5, 0.2)$& $1 + 5 (J_1(1.5r|k|)/(1.5r|k|))^2$ \\
			\midrule
			Exponential LGCP & $(\alpha, \sigma^2, \mu) = (1, 0.5^2, 0)$ & $1 + e^{0.25} \int_{\mathbb{R}^2} (e^{0.5e^{-|x|}} -1)e^{-\bm{i}k\cdot x} dx$ \\
			\midrule
			Ginibre & --- & $1 - \exp(-|k|^2/4)$ \\
			\midrule
			Bessel DPP & $(\rho, \alpha) = (0.3, 1)$ & $1-0.3\ \mathbf{1}_{|k|\le 4}\left[2\arccos\left(\frac{ |k|}{4}\right)-\frac{|k|}{2}\sqrt{1-\left(\frac{|k|}{4}\right)^{2}}\right]$ \\
			\midrule
			$1/2$-perturbed lattice & $-$ & $1 - \operatorname{sinc}\left(\frac{k_1}2\right)^2 \operatorname{sinc}\left(\frac{k_2}2\right)^2\exp(-\sqrt{|k/2|})$ \\
			\bottomrule
		\end{tabular}
		\label{table:models}
	\end{table}

	We consider the models listed in Table~\ref{table:models}, all studied in dimension $d=2$ over the observation window $W = [-20, 20]^2$, with the parameters specified in the second column. These models were chosen because their structure factors exhibit distinct levels of smoothness and illustrate either clustering (Thomas, Matérn, and exponential LGCP) or repulsive behavior (Ginibre, Bessel DPP, and $1/2$-perturbed lattice). For the parameters listed in Table~\ref{table:models}, the Thomas cluster, Matérn cluster, and $1/2$-perturbed lattice processes all have unit intensity, resulting in about $1600$ points within the observation window. The exponential LGCP has intensity $\exp(0.5^2) \simeq 1.28$, producing approximately $2000$ points, whereas the Ginibre process (intensity $1/\pi$) and the Bessel DPP (intensity $0.3$) each yield about $500$ points.
	
	The Thomas and Matérn cluster point processes are Neyman–Scott models with intensity field $\Lambda(x) = \alpha \sum_{y \in \text{PP}_{\mu}} p(y - x)$, where $\alpha > 0$ and $\text{PP}_{\mu}$ is a stationary Poisson point process of intensity $\mu$~\cite{daley2003introduction}. For the Thomas, $p$ is the density of a zero-mean Gaussian distribution with variance $\sigma^2$. For the Matérn, $p$ is the density of a uniform random variable over $B(0, \rho)$. The Ginibre and the Bessel DPP are determinantal point processes. The kernel of the Ginibre is $(x, y) \mapsto e^{-x \overline{y} - |x|^2/2 - |y|^2/2}/\pi$, where $\mathbb{R}^2$ is identified with $\mathbb{C}$~\cite{ginibre1965statistical}, and the one of the Bessel DPP is $(x, y) \mapsto 2\rho J_1(2|x-y|/\alpha)/(2|x-y|/\alpha)$, where $\rho > 0,~\alpha > 0$ satisfy $\rho < (\pi \alpha^2)^{-1}$ and $J_1$ is the Bessel function of the first kind of order one. Finally, the $1/2$-perturbed lattice~\cite{kim2018effect} is defined as the point process $\{x + V_x + U_x+U \mid x \in \mathbb{Z}^2\}$, where $(U, (U_x)_{x \in \mathbb{Z}^2})$ are i.i.d. uniformly distributed on $[-1/2, 1/2]^2$ and the $(V_x)_{x \in \mathbb{Z}^2}$ are i.i.d. random variables with a symmetric $1/2$-stable distribution, whose characteristic function is given by $k \mapsto \exp(-\sqrt{|k/2|}/2)$ for $k \in \mathbb{R}^2$.
	
	\subsubsection{Choice of the hyper-parameters of the cross-validation}
	To implement the multitaper estimator, we use Hermite tapers and substitute the standard estimator $\widehat{\lambda} = |\Phi \cap W| / |W|$ for the intensity (see Appendix~\ref{sec:with_intens}).
	We consider a family of nodes 
	$$\{k_n = q_n \, (\cos(2\pi/3), \sin(2\pi/3))|~q_n \in \{0, 0.25, 0.5, 0.75, 1, 1.5, 2, 2.5, 3, 3.5, 4, 4.5\}\}.$$
	For the estimation of $S$ at these nodes, we use the data-driven procedure introduced in Section~\ref{sec:dd} to select  the number of tapers, as described below. Then for the estimation of $S$ at an abritrary frequency $k_0$, we choose the number of tapers by interpolating the selected number of tapers between the nodes of closest norm.

	For the cross-validation procedure at a given node $k_n$, we choose:
	$$
	\mathcal{I} := \left\{ I_{i_{\max}} \,\middle|\, i_{\max} \leq 25 \right\}, \quad \text{where } I_{i_{\max}} := \left\{ (i_1, i_2) \in \mathbb{N}^2 \,\middle|\, i_1 < i_{\max},\, i_2 < i_{\max} \right\}.
	$$
	For each taper set $I_{i_{\max}}$, we take $r = R/\sqrt{2i_{\max}+i_{\max}^{2/3}}$ for the dilation factor. The tapers set for the pilot estimator is fixed as $I_p = \{(i_1, i_2) \in \mathbb{N}^2 \mid i_1 < 8,\, i_2 < 8\}$ for the Thomas/Matérn cluster and the exponential LGCP. For the Ginibre, the Bessel DPP and the $1/2$-perturbed lattices, where $S$ takes smaller values near the origin, we instead use a smaller taper set $I_p = \{(i_1, i_2) \in \mathbb{N}^2 \mid i_1 < 2,\, i_2 < 2\}$.
	We consider the frequency pairs $\{(k_{j_1, j_2}, \widetilde{k}_{j_1, j_2})\}$ given by
	\[
	\left\{\left(q_{j_1}\left(\cos\left(\tfrac{2\pi j_2}{n_2}\right), \sin\left(\tfrac{2\pi j_1}{n_2}\right)\right), q_{j_1}\left(\sin\left(\tfrac{2\pi j_2}{n_2}\right), \cos\left(\tfrac{2\pi j_2}{n_2}\right)\right)\right)\right\}_{j_1, j_2 \in [n_1]\times [n_2]},
	\]
	where $n_1 \in \{4, 5\}$ and $n_2$ is chosen such that for all $j_2 \in \{0, \dots, n_2-1\}$, $|k_{j_1, j_2} - k_{j_1, j_2+1}| = 0.02$ if $|k_n| \leq 1$ and $|k_{j_1, j_2} - k_{j_1, j_2+1}| = 0.1$ if $|k_n| \geq 1.$ Moreover, for $k_n \neq 0$, we take $q_{j_1} \in |k_n| + \{-0.2, -0.1, 0, 0.1, 0.2\}$, and for $k_n = 0$, we remove the frequency zero from this last set and consider $q_{j_1} \in \{-0.2, -0.1, 0.1, 0.2\}$. 
	
	\subsubsection{Numerical results}
	
	\begin{figure}[htbp]
		\centering
		\begin{minipage}[t]{0.47\textwidth}
			\centering
			\includegraphics[width=\linewidth]{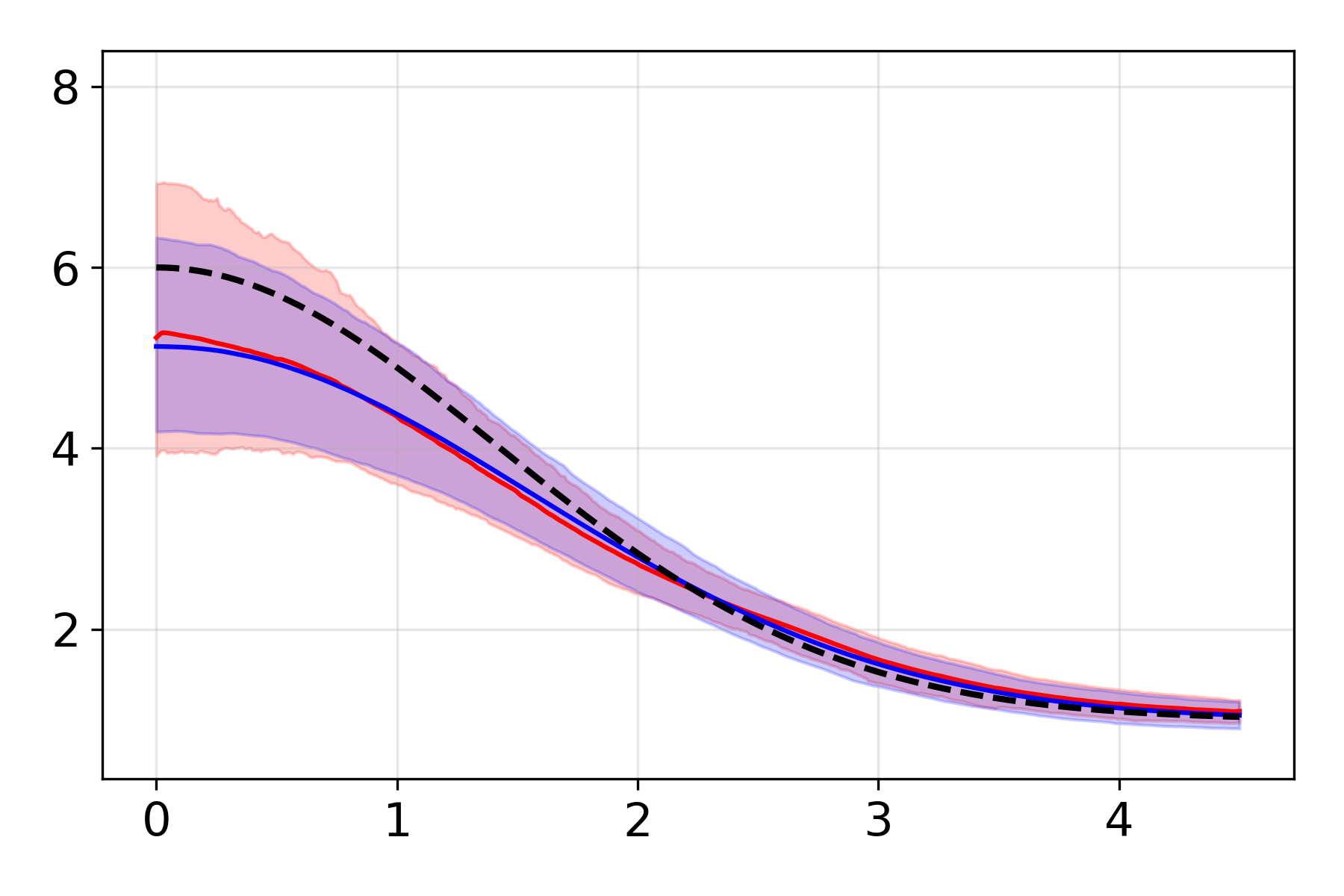}\\
			\small Thomas cluster $(\beta = 2)$
		\end{minipage}
		\hfill
		\begin{minipage}[t]{0.47\textwidth}
			\centering
			\includegraphics[width=\linewidth]{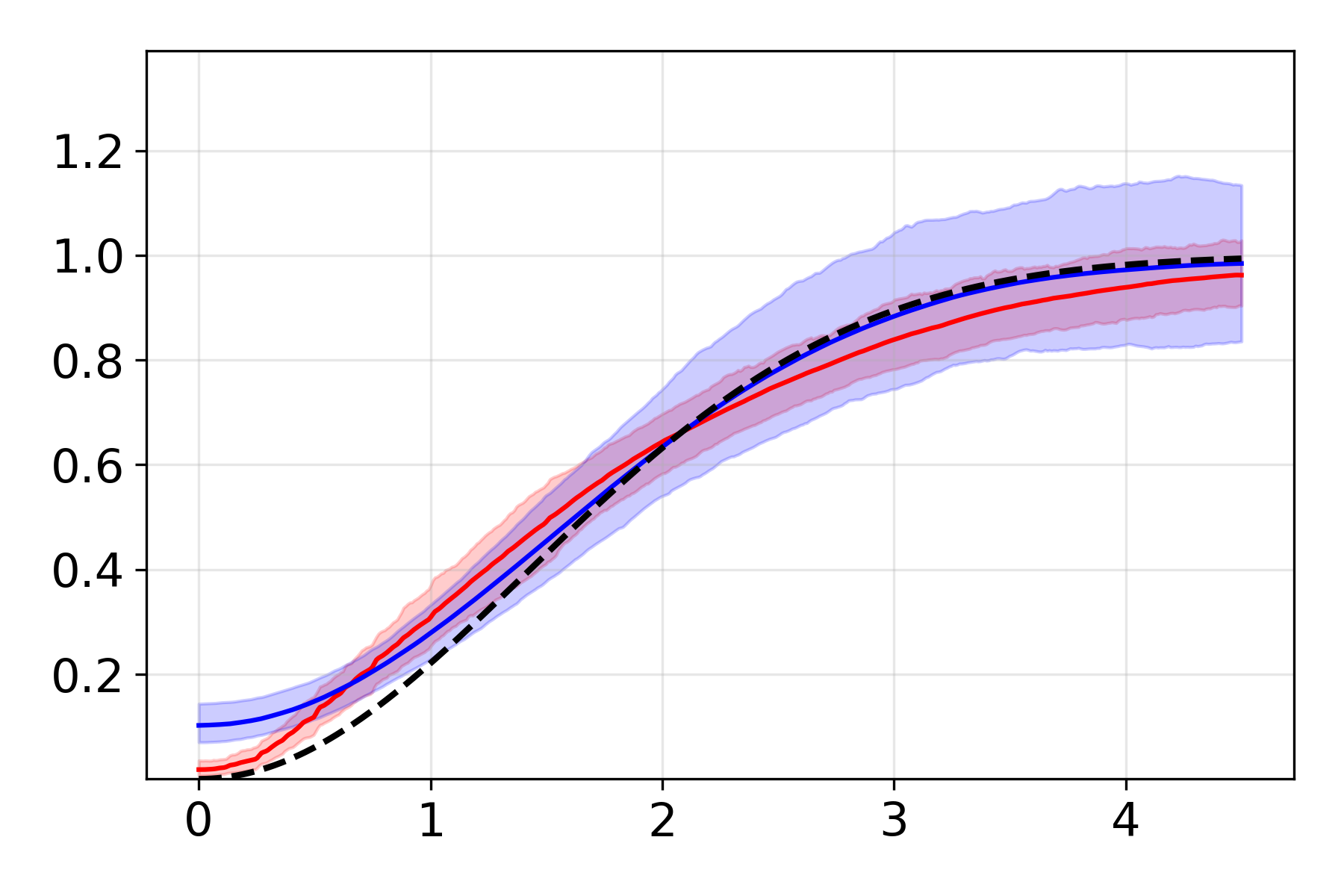}\\
			\small Ginibre $(\beta = 2)$
		\end{minipage}
		\hfill
		\begin{minipage}[t]{0.47\textwidth}
			\centering
			\includegraphics[width=\linewidth]{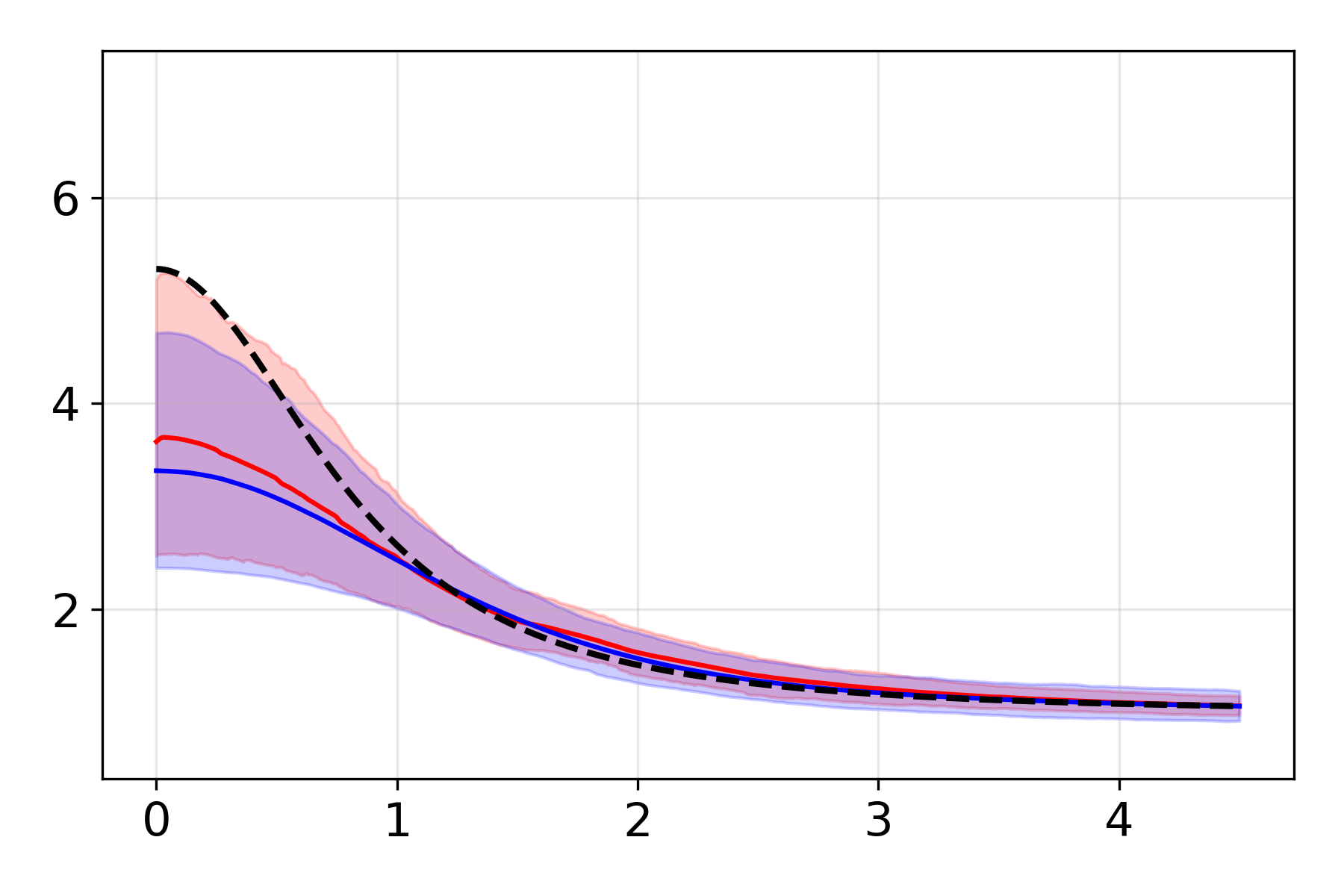}\\
			\small Exponential LGCP $(\beta = 2)$
		\end{minipage}
		\hfill
		\begin{minipage}[t]{0.47\textwidth}
			\centering
			\includegraphics[width=\linewidth]{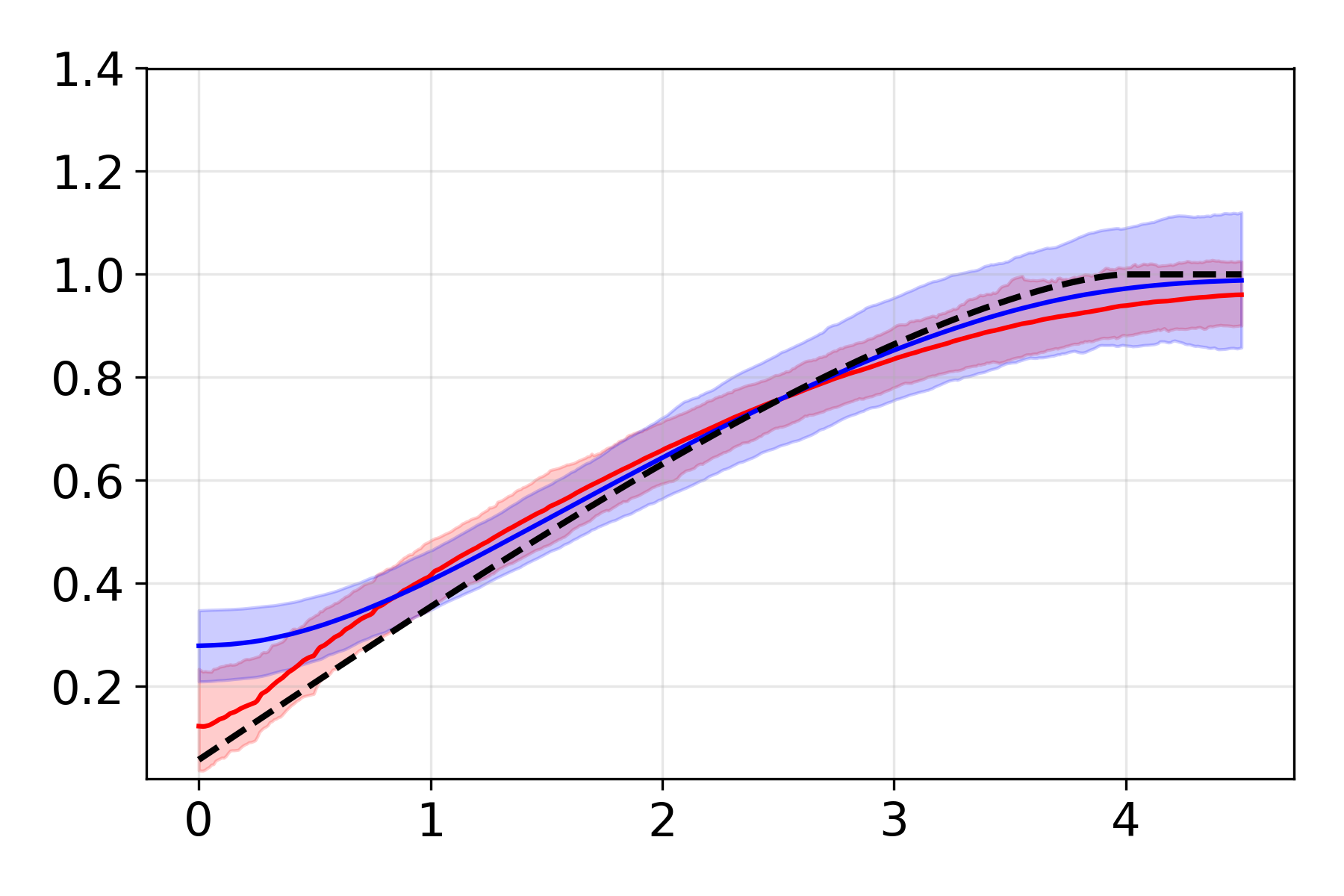}\\
			\small Bessel DPP $(\beta = 1)$
		\end{minipage}
		\hfill
		\begin{minipage}[t]{0.47\textwidth}
			\centering
			\includegraphics[width=\linewidth]{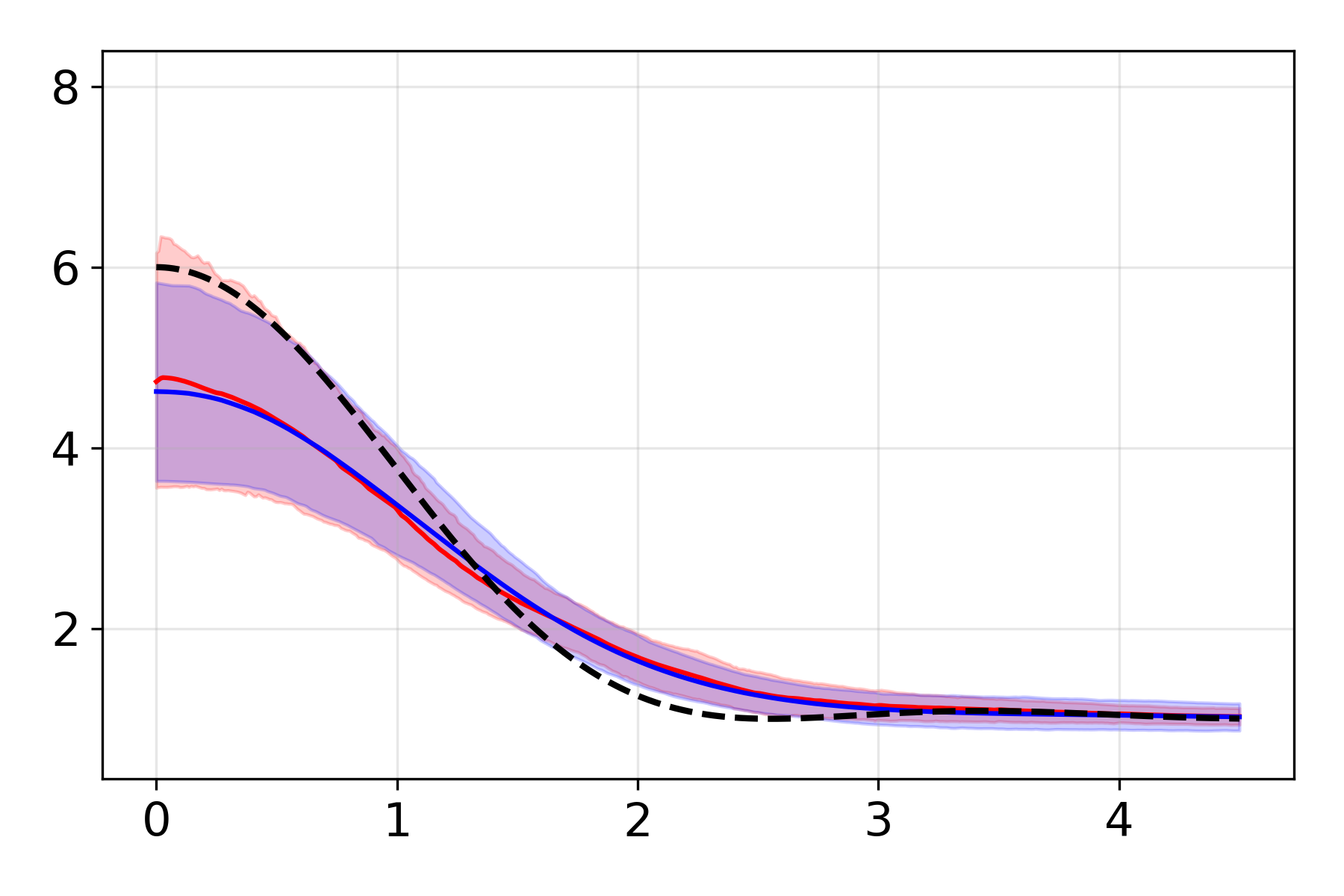}\\
			\small Matérn cluster $(\beta = 2)$
		\end{minipage}
		\hfill
		\begin{minipage}[t]{0.47\textwidth}
			\centering
			\includegraphics[width=\linewidth]{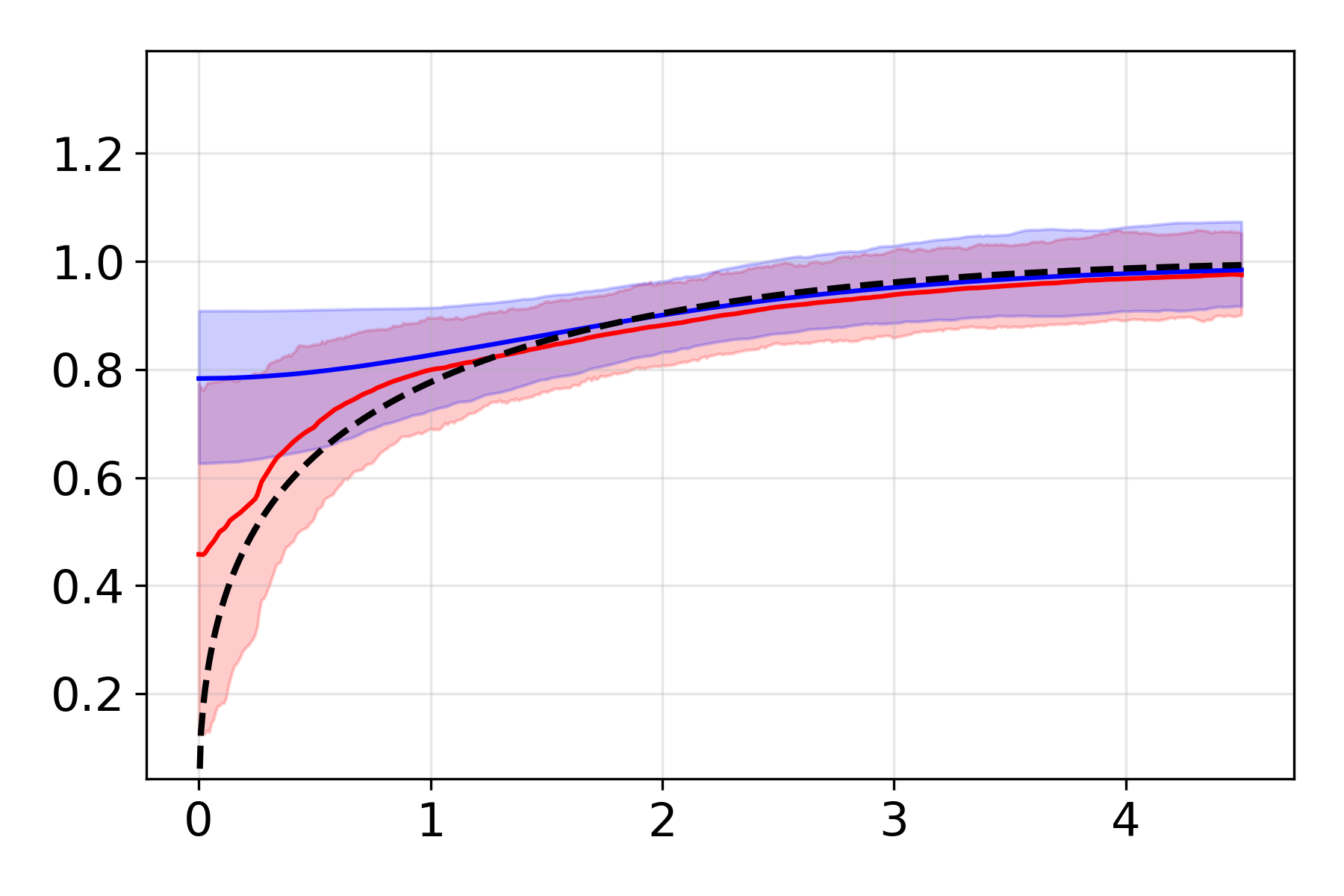}\\
			\small $1/2$-perturbed lattice $(\beta = 1/2)$
		\end{minipage}
		\caption{Each model is simulated $500$ times in the observation window $[-20, 20]^2$. The Hölder regularity $\beta$ of the corresponding structure factors are indicated between parenthesis. The dashed black curve represents the theoretical structure factor. The red color corresponds to the multitaper method, with the number of tapers locally selected via the cross-validation procedure of Section \ref{sec:dd}. The blue color corresponds to the kernel estimator~\cite{yang2024fourier}, with the cross-validation method introduced in~\cite{ding2025pseudo}. For both methods, the solid line represents the average of the structure factor estimates, while the shaded region corresponds to their 0.05--0.95 empirical quantiles.}
		
		\label{fig:taper}
	\end{figure}
	
	Figure~\ref{fig:taper} presents a qualitative study of the performance of the multitaper estimator, using the local cross-validation procedure to select the number of tapers (see the red curves). The blue curves correspond to the kernel estimator considered in~\cite{yang2024fourier}, using the cross-validation approach introduced in~\cite{ding2025pseudo}.
	
	The use of the local cross-validation method for selecting the number of tapers is advantageous, since the optimal number of tapers can vary widely across frequencies and cannot be reliably inferred by merely inspecting the curves for different taper numbers. For example, in simulations not shown here (but available upon request), the optimal number of tapers for the Ginibre point process varies from lower than $10$ at low frequencies to about $400$ at high frequencies. For all models presented in Figure~\ref{fig:taper}, at frequencies that are not too small, the data-driven procedure (corresponding to the red curves) consistently selects an appropriate number of tapers, as predicted by Theorem~\ref{thm:dd}. In the more challenging low-frequency regime, the procedure remains effective for repulsive point processes (Ginibre, Bessel DPP and $1/2$-perturbed lattice), but it encounters more difficulty for clustered point processes (Thomas, Matérn cluster and exponential LGCP). Finally, as shown by the repulsive models, the estimator achieves higher accuracy when the structure factor $S$ exhibits greater Hölder regularity $\beta$, in agreement with theoretical predictions.

	Regarding the comparison with the kernel estimator of $S$ considered in~\cite{yang2024fourier} and \cite{ding2025pseudo} (see the blue curves in Figure~\ref{fig:taper}), we observe that, for clustering point processes, the multitaper estimator with a locally selected number of tapers exhibits lower bias at low frequencies than the kernel estimator. However, it fluctuates more, which, in the case of the Thomas and Matérn cluster point processes, leads to larger mean squared errors. For the LGCP, the reduction in bias at low frequencies compensates for the increase in fluctuations, resulting in smaller mean squared errors overall. At larger frequencies for these three clustering models, the multitaper and kernel estimators perform comparably. For the other repulsive models considered, however, the multitaper method with a locally data-driven selection of the number of tapers outperforms the kernel estimator having a globally selected bandwidth, across both low and high frequencies, highlighting its advantages in these scenarios.

	\section{Proofs of the minimax lower bounds}\label{sec:proof_minimax}
	
	In this section, we prove Theorems \ref{thm:minimax_S} and \ref{thm:minimax_g}. The first key ingredient is Lemma~\ref{lem:chi2}, which provides an upper bound on the $\chi^2$ divergence between an arcsin-Gaussian Cox point process and a unit-intensity stationary Poisson. 
	
	\begin{definition}\label{def:chi2}
		Let $P$ and $Q$ be two probability measures. The $\chi^2$ divergence between $P$ and $Q$ is defined by:
		$$\chi^2(P, Q) := \int \left(\frac{dP}{dQ}\right)^2 dQ -1$$
		if $P$ is absolutely continuous with respect to $Q$ and $+ \infty$ otherwise.
	\end{definition}
	
	Considering the $\chi^2$ divergence is particularly useful in our context, as its quadratic structure enables a symmetrization argument. It allows for reducing the proof of Lemma~\ref{lem:chi2} to bounding the exponential moment of an integral of a random field. 
	
	\begin{lemma}\label{lem:chi2}
		Let $\Phi_{PPP}$ be a Poisson point process of intensity $1$. Let $\sigma \in (0, 1)$, $\rho > 0$ and $(N(x))_{x \in \mathbb{R}^d}$ be a stationary centered Gaussian field with covariance function $c(x) = \sigma^2 c_0(x/\rho)$ where $c_0$ has support included in $B(0, 1/2)$ and $0 \leq c_0 \leq 1$. We consider a Cox point process $\Phi_{Cox}$ with intensity field $(1+ \sin(N(x)))_{x \in \mathbb{R}^d}$ and an observation window $W = [-R, R]^d$ with $R > 0$.
		Assume that $8^d \sigma^4 \rho^{2d}  < 1$ and $R/\rho \in \mathbb{N}.$ Then 
		\begin{align*}
			\chi^2 := \chi^2( \Phi_{Cox} \cap W, \Phi_{PPP} \cap W) & \leq \left(1- 8^d \sigma^4 \rho^{2d} \right)^{-|W|/(2^d\rho^d)}-1
		\end{align*}
	\end{lemma}
	\begin{proof}		
		The density of $\Phi_{Cox} \cap W$ with respect to $\Phi_{PPP} \cap W$ is given by
		\begin{align*}
			\mathbb{E}_{N}\left[\exp\left(-\int_{W}  \sin(N(x)) dx\right) \prod_{y \in \Phi_{PPP} \cap W}\left(1+ \sin(N(y))\right)\right],
		\end{align*}
		where the expectation  $\mathbb{E}_{N}$ is with respect to the Gaussian field $(N(x))_{x \in \mathbb{R}^d}$ \cite{moller2003statistical}. Accordingly, the $\chi^2$ divergence writes 
		\begin{align*}
			\chi^2= \mathbb{E}_{P}\left[\mathbb{E}_{N}\left[\exp\left(-\int_{W}  \sin(N(x)) dx\right) \prod_{y \in \Phi_{PPP} \cap W} \left(1+ \sin(N(y))\right)\right]^2\right]-1,
		\end{align*}
		where $\mathbb{E}_{P}$ denotes the expectation with respect to the Poisson point process $\Phi_{PPP} \cap W$. We first simplify the nested expectations by introducing $(\widetilde{N}(x))_{x \in \mathbb{R}^d}$ an independent copy of $(N(x))_{x \in \mathbb{R}^d}$. Using the Fubini's theorem and the expression for the generating function of the Poisson point process, we obtain
		\begin{align*}
			&\mathbb{E}_{P}\left[\mathbb{E}_{N}\left[\exp\left(-\int_{W}  \sin(N(x)) dx\right) \prod_{y \in \Phi_{PPP} \cap W} (1+ \sin(N(y)))\right]^2\right] \\& = \mathbb{E}_{P}\left[\mathbb{E}_{N, \widetilde{N}}\left[e^{- \int_{W} (\sin(N(x)) + \sin(\widetilde{N}(x))) dx} \prod_{y \in \Phi_{PPP} \cap W} (1+ \sin(N(y))) (1+ \sin(\widetilde{N}(y)))\right]\right] \\& = \mathbb{E}_{N, \widetilde{N}}\left[e^{- \int_{W} (\sin(N(x)) + \sin(\widetilde{N}(x))) dx}  e^{\int_{W} \left((1+\sin(N(x)))(1+ \sin(\widetilde{N}(x)))-1\right) dx}\right] \\& = \mathbb{E}_{N, \widetilde{N}}\left[\exp\left(\int_{W} \sin(N(x))\sin(\widetilde{N}(x)) dx\right)\right].
		\end{align*}
		The next step is to decompose $W$ as an union of families of cubes such that the cubes of each family are well-enough separated to exploit the decorrelation of the fields $N$ and $\widetilde{N}$. Denote, for $\tau \in \mathbb{Z}^d$ and $\varepsilon \in \{0, 1\}^d$, 
		$C_{\varepsilon, \tau} := \prod_{i = 1}^d [(2\tau_i+\varepsilon_i)\rho, (2\tau_i+1+\varepsilon_i) \rho].$
		We decompose $W$ as
		$$W = \bigsqcup_{\varepsilon \in \{-1, 1\}^d} \bigsqcup_{\tau \in \mathbb{Z}^d} C_{\varepsilon, \tau} \cap W$$
		Note that the union is finite. Hölder's inequality yields the following bound 
		\begin{align*}
			\chi^2 &\leq \prod_{\varepsilon \in \{-1, 1\}^d} \mathbb{E}_{N, \widetilde{N}}\left[\exp\left(2^d\int_{\sqcup_{\tau \in \mathbb{Z}^d} C_{\varepsilon, \tau} \cap W} \sin(N(x))\sin(\widetilde{N}(x)) dx\right)\right]^{2^{-d}}-1 \\& =\prod_{\varepsilon \in \{-1, 1\}^d} \mathbb{E}_{N, \widetilde{N}}\left[\exp\left(2^d\sum_{\tau \in \mathbb{Z}^d; C_{\varepsilon, \tau} \cap W \neq \emptyset} X_{\varepsilon, \tau}\right)\right]^{2^{-d}}-1, 
		\end{align*}
		with $X_{\varepsilon, \tau} = \int_{C_{\varepsilon, \tau}} \sin(N(x))\sin(\widetilde{N}(x)) dx$. 
		Recall that the covariance function $c$ of $N$ and $\widetilde{N}$ has support included in $B(0, \rho/2)$. For $x_{\tau} \in C_{\varepsilon, \tau}$ and $x_{\tau'} \in C_{\varepsilon, \tau'}$ with $\tau \neq \tau'$, we have $|x_{\tau} - x_{\tau'}| \geq \rho$, so $c(x_{\tau} - x_{\tau'}) = 0$. The random integrals $(X_{\varepsilon, \tau})_{\tau \in \mathbb{Z}^d}$ are thus independent.~Accordingly,
		\begin{align*}
			\chi^2 &\leq \prod_{\varepsilon \in \{-1, 1\}^d} \prod_{\tau \in \mathbb{Z}^d,~C_{\varepsilon, \tau} \cap W \neq \emptyset} {\underbrace{\mathbb{E}_{N, \widetilde{N}}\left[\exp\left(2^d \int_{C_{\varepsilon, \tau} \cap W} \sin(N(x))\sin(\widetilde{N}(x)) dx\right)\right]}_{=: A_{\varepsilon, \tau}}}^{2^{-d}}-1.
		\end{align*}
		The previous argument have controlled the correlations of the intensity field $1+\sin(N)$. The remaining of the proof concerns its amplitude. Expanding the exponential function in power series and using the Fubini's theorem, we get 
		\begin{align*}
			A_{\varepsilon, \tau} &= \sum_{n \geq 0} \frac{(2^d)^n}{n!} \mathbb{E}\left[\left(\int_{C_{\varepsilon, \tau} \cap W} \sin(N(x))\sin(\widetilde{N}(x)) dx\right)^n\right] \\& = \sum_{n \geq 0} \frac{(2^d)^n}{n!} \int_{(C_{\varepsilon, \tau} \cap W)^n} \mathbb{E}\left[\prod_{i = 1}^n \sin(N(x_i))\right]^2 dx_1 \dots d{x_n}\\&  = \sum_{n \geq 0} \frac{4^{dn}}{(2n)!}\int_{(C_{\varepsilon, \tau} \cap W)^{2n}} \mathbb{E}\left[\prod_{i = 1}^{2n} \sin(N(x_i))\right]^2 dx_1 \dots d{x_{2n}}.
		\end{align*}
		In the last line, we used $\mathbb{E}[N(x)] = 0$. Using $|\sin(x)| \leq |x|$ for all $x \in \mathbb{R}$, the Hölder-inequality and the stationarity of $N$, we obtain
		\begin{align*}
			A & \leq \sum_{n \geq 0} \frac{4^{dn}}{(2n)!} \int_{(C_{j, \tau} \cap W)^{2n}} \prod_{i = 1}^{2n}\mathbb{E}\left[\sin(N(x_i))^{2n}\right]^{1/n} dx_1 \dots d{x_{2n}} \\& \leq \sum_{n \geq 0} \frac{4^{dn}|C_{j, \tau} \cap W|^{2n}}{(2n)!} \mathbb{E}\left[N(0)^{2n}\right]^2 \leq \sum_{n \geq 0} \frac{\left(4^{d} \rho^d\right)^{2n}}{(2n)!} \left(\sigma^{2n} \frac{(2n)!}{n! 2^{n}}\right)^2 \\& \leq  \sum_{n \geq 0}\left(8^{d} \sigma^4 \rho^{2d}\right)^{n} = (1- 8^{d} \sigma^4 \rho^{2d})^{-1},
		\end{align*}
		since $\binom{2n}{n} \leq 2^{2n}$ and $8^d \sigma^4 \rho^{2d}  < 1$. This yields 
		\begin{align*}
			\chi^2( \Phi_{Cox} \cap W, \Phi_{PPP} \cap W) &\leq\left(1- 8^{d} \sigma^4 \rho^{2d}\right)^{-2^{-d}\sum_{\varepsilon \in \{-1, 1\}^d} \sum_{\tau \in \mathbb{Z}^d} \mathds{1}_{C_{\varepsilon, \tau} \cap W \neq \emptyset}}.
		\end{align*}
		To conclude, we use $R/\rho \in \mathbb{N}$ to get $
		\sum_{\varepsilon \in \{-1, 1\}^d} \sum_{\tau \in \mathbb{Z}^d} \mathds{1}_{C_{\varepsilon, \tau} \cap W \neq \emptyset} = |W|/\rho^d.$
	\end{proof}

	The next crucial result is Lemma~\ref{lem:mini_bril_4}, which addresses the behavior of the total variation of the reduced cumulant measure of the Cox processes of Lemma \ref{lem:chi2}. This result is important, as it ensures that the constraint $\max_{m \in \{2, 3, 4\}}|\gamma_{red}^{(m)}|(\mathbb{R}^{d(m-1)}) \leq M$ is satisfied in Theorem~\ref{thm:minimax_S}. It proof consists in transposing the property of the vanishing cumulants of order larger than three for Gaussian fields to these specific Cox point processes. This lemma motivates our choice of considering an underlying intensity field of the form $(1 + \sin(\sigma^2 N_0(x)))_{x \in \mathbb{R}^d}$, where $(N_0(x))_{x \in \mathbb{R}^d}$ is a unit-variance Gaussian field. Using $(1 + \sigma^2 \sin(N_0(x)))_{x \in \mathbb{R}^d}$ would have simplified the proof of the previous lemma, but it would not allow us to control higher-order cumulants in the low-regularity setting $\beta < \frac{d}{2}$.
	
	\begin{lemma}\label{lem:mini_bril_4}
		Let $\sigma \in (0, 1)$ and $\rho > 0$. Let $(N(x))_{x \in \mathbb{R}^d}$ be a stationary centered Gaussian field with covariance function $c(x) =  \sigma^2 c_0(x/\rho)$ where $c_0$ has support included in $B(0, 1/2)$ and $0 \leq c_0 \leq 1$. We consider a Cox point process $\Phi_{Cox}$ with intensity field $(1+ \sin(N(x)))_{x \in \mathbb{R}^d}$.
		Then, the reduced factorial cumulants measures of $\Phi_{Cox}$ satisfies
		\begin{align*}
			|\gamma_{\text{red}}^{(3)}|(\mathbb{R}^{2d})  \leq K \sigma^{4} \rho^{2d},\quad |\gamma_{\text{red}}^{(4)}|(\mathbb{R}^{3d})  \leq K \sigma^{6} \rho^{3d},
		\end{align*}
		where $K < \infty$ is a numerical constant.
	\end{lemma}
	\begin{proof}
		Let $m \geq 2$. The $m$-th order cumulant measure of $\Phi_{Cox}$ is given by \cite{grandell2006doubly} 
		\begin{align*}
			&\gamma^{(m)}(x_1,\dots, x_m) = \kappa_m\left(1+ \sin(N(x_1)), \dots, 1+ \sin(N(x_m))\right).
		\end{align*}
		Using the multi-linearity of the cumulants, we get:
		\begin{align*}
			\gamma^{(m)}(x_1,\dots, x_m) &= \kappa_m\left(\sin(N(x_1)), \dots, \sin(N(x_m))\right). 
		\end{align*}
		Therefore, $
		\gamma_{\text{red}}^{(m)}(y_1, \dots, y_{m-1}) =  \kappa_m\left(\sin(N(y_1)), \dots ,\sin(N(y_{m-1})),\sin(N(0))\right).$ The heuristic for the remainder of the proof is as follows. Assume that $m = 3$. Then, as $\sigma \to 0$, we have
		\begin{align*}
			\gamma_{\text{red}}^{(3)}(y_1, y_2) &=  \kappa_3\left(\sin(N(y_1)),\sin(N(y_2)),\sin(N(0))\right) \\& \simeq \kappa_3\left(1+ N(y_1) + \frac12 N(y_1)^2,1+ N(y_2) + \frac12 N(y_2)^2,1+ N(0) + \frac12 N(0)^2\right).
		\end{align*}
		Then, using the multi-linearity of the cumulants, we obtain a sum of terms, whose main order with respect to $\sigma \to 0$ is:
		$\sigma^{1+1+1}\kappa_3\left(N_1, N_2, N_3\right) + \sigma^{1+1+2}\kappa_3\left(N_1, N_2, N_3^2\right),$
		where $(N_1, N_2, N_3)$ is Gaussian. Accordingly, the first order term $\kappa_3\left(N_1, N_2, N_3\right) $ vanishes, and the scaling $\sigma^4$ appears. Below, we make the previous heuristic rigorous.

		We only consider the case $m = 4$, the case $m = 3$ being similar. Accordingly, in the following, we fix $m = 4$ and denote $y_4 = 0$. Note that if $[4] = L \sqcup L'$ with two disjoint non-empty sets $L$ and $L'$ such that the family of random variables $(\sin(N(y_l)))_{l \in L}$ and $(\sin(N(y_l)))_{l \in L'}$ are independent, then $	\gamma_{\text{red}}^{(4)}(y_1, \dots, y_{m-1}) \equiv 0$ (see, e.g.,~\cite{peccati2011wiener}, Section 3.1. Property (ii)). Since the Gaussian field $(N(x))_{x \in \mathbb{R}^d}$ has covariance function having support included in $B(0, \rho/2)$, it implies that if $\gamma_{\text{red}}^{(m)}(y_1, \dots, y_{m-1}) \neq 0$ then there exists a cycle $\pi : [4] \mapsto [4]$ such that for all $s \in [d]$ and $j \in[4]$, then $|y_j \cdot e_s - y_{\pi(j)}\cdot e_s| \leq \rho$. Accordingly,
		\begin{align}\label{eq:bound_gamma_4_1}
			|\gamma_{\text{red}}^{(4)}(\mathbb{R}^{3d})|& \leq \sum_{\text{cycle }\pi : [4] \mapsto [4]} \int_{\mathbb{R}^{3d}} \prod_{i = 1}^d \prod_{j = 1}^4 \mathds{1}_{|y_{j} \cdot e_i - y_{\pi(j)}\cdot e_i| \leq \rho_i} dy_1 dy_2 dy_2 \sup_{z \in \mathbb{R}^{3d}} |\gamma_{\text{red}}^{(4)}(z)|\nonumber \\& = (4-1)! \left(\prod_{i = 1}^{d} 2\rho\right)^3  \sup_{z \in \mathbb{R}^{3d}} |\gamma_{\text{red}}^{(4)}(z)|.
		\end{align}
		It remains to bound the supremum of $h$. Note that $N = \sigma N_0$ in distribution, where $N_0$ is a Gaussian field with covariance function $c_0$. Moreover, Rolle's theorem implies that for all $j \in [4]$ there exists a random variable $Y_j \in \mathbb{R}$, such that $\sin(N(y_j)) = \sigma N_0(y_j) - \sigma^3 N_0(y_j)^3 \cos(Y_j)/6$. Using the multi-linearity of the cumulants we get
		\begin{align*}
			\gamma_{\text{red}}^{(4)}(y_1, y_2, y_3) = \sum_{P \subset [4]} \sigma^{|P^c| + 3 |P|} \left(-\frac{1}6\right)^{|P|}\kappa_4(N_0(y_l), \cos(Y_{l'}) N_0(y_{l'})^3; l \in P^c, l' \in P),
		\end{align*}
		where the sum runs over all the subsets $P$ of $[4]$. Since the fourth order joint cumulants of a Gaussian field are zero, the sums runs in fact over the non empty subsets $P$. Hence, $|P^c| + 3 |P| = 4 + 2 |P| \geq 6$. Using $\sigma \leq 1$, we have
		\begin{align}\label{eq:bound_gamma_4_2}
			|\gamma_{\text{red}}^{(4)}(y_1, y_2, y_3)| \leq \sigma^{6} \sum_{P \subset [4]} |\kappa_4(N_0(y_l), \cos(Y_{l'}) N_0(y_{l'})^3; l \in P^c, l' \in P)|
		\end{align}
		We denote $Z_l = N_0(x_l)^3$ is $l \in P$ and $Z_l = N_0(x_l)$ if $l \in P^c$. By Hölder-inequality, we have
		\begin{align*}
			|\kappa_4(N_0(y_l), &\cos(Y_{l'}) N_0(y_{l'})^3; l \in P^c, l' \in P)| \leq \sum_{\pi \subset \Pi[4]} (|\pi| - 1)! \prod_{b \in \pi} \prod_{l \in b} \mathbb{E}\left[|Z_l|^{|b|}\right]^{1/|b|},
		\end{align*}
		which is a numerical constant. Recalling \eqref{eq:bound_gamma_4_1} and \eqref{eq:bound_gamma_4_2}, this yields $|\gamma_{\text{red}}^{(4)}|(\mathbb{R}^{3d})  \leq K \sigma^{6} \rho^{3d}$. With similar arguments, we obtain $|\gamma_{\text{red}}^{(3)}|(\mathbb{R}^{2d})  \leq K \sigma^{4} \rho^{2d}$.
	\end{proof}
	
	The next lemma derives the structure factor and the pair correlation function of the arcsin-Cox processes considered in this section.
	
	\begin{lemma}\label{lem:Sc}
		Let $(N(x))_{x \in \mathbb{R}^d}$ be a stationary centered Gaussian field with bounded covariance function $c$ having a compact support. We consider a Cox point process $\Phi_{Cox}$ with intensity field $(1+ \sin(N(x)))_{x \in \mathbb{R}^d}$. 
		Then, the pair correlation function $g_c$ of $\Phi_{Cox}$ is 
		\begin{align*}
			\forall x \in \mathbb{R}^d,~g_c(x) = 1+ \frac{e^{- c(0)}}2\left(e^{c(x)} - e^{- c(x)}\right)
		\end{align*}
		Moreover, the structure factor $S_c$ of $\Phi_{Cox}$ is 
		\begin{align*}
			\forall k \in \mathbb{R}^d,~S_c(k) = 1 + \frac{e^{- c(0)}}2\int_{\mathbb{R}^d} \left(e^{c(x)} - e^{- c(x)}\right) e^{-\bm{i} k\cdot x} dx.
		\end{align*}
	\end{lemma}
	\begin{proof}
		The intensity of $\Phi_{Cox}$ is $\mathbb{E}[1+ \sin(N(0))] = 1$. Moreover, its pair correlation function $g_c$ is~\cite{grandell2006doubly}:
		\begin{align*}
			g_c(x) &=\mathbb{E}\left[(1+\sin(N(x)))(1+ \sin(N(0)))\right] \\& = 1 + 2  \mathbb{E}\left[\sin(N(0))\right] +  \mathbb{E}\left[\sin(N(x))\sin(N(0))\right] = 1+  \mathbb{E}\left[\sin(N(x))\sin(N(0))\right].
		\end{align*}
		The Moivre's formula yields
		\begin{align*}
			\mathbb{E}\left[\sin(N(x))\sin(N(0))\right] &= -\frac12\left(\mathbb{E}\left[e^{\bm{i} (N(x) + N(0))}\right] - \mathbb{E}\left[e^{\bm{i} (N(x) - N(0))}\right]\right)  \\& = \frac{e^{-c(0)}}2\left(e^{c(x)} - e^{-c(x)}\right).
		\end{align*}
		Since $c$ is compactly supported, $g_c - 1 \in L^1(\mathbb{R}^d)$. This concludes the proof, recalling \eqref{eq:def_S}.
	\end{proof}
	
	We are now in position to prove Theorem~\ref{thm:minimax_S}.
	
	\begin{proof}[Proof of Theorem~\ref{thm:minimax_S}]
		Let $W \in \mathcal{W}$. We prove the lower bound when estimating $S$ at frequency zero.
		\begin{multline*}
			r_W:= \inf_{\widehat{S}} \sup_{S \in \Theta(\beta, L)} \sup_{\Phi \sim \mathcal{L}_{(S, M)}} \sup_{k_0 \in \mathbb{R}^d} \mathbb{E} \left[h(|\Phi \cap W|^{\frac{\beta}{2\beta+d}}|\widehat{S}(k_0; \Phi \cap W) - S(k_0)|)\right] \\ \geq 	\inf_{\widehat{S}} \sup_{S \in \Theta(\beta, L)} \sup_{\Phi \sim \mathcal{L}_{(S, M)}} \mathbb{E} \left[h(|\Phi \cap W|^{\frac{\beta}{2\beta+d}}|\widehat{S}(0; \Phi \cap W) - S(0)|)\right].
		\end{multline*}		
		The newt step is to reduce the proof to the square window setting. Since $\mathcal{W}$ is regular, there exists $\eta > 0$ such that for all $W \in \mathcal{W}$ there exists $R_W$ such that $W\subset [-R_W, R_W]^d$ and $|[-R_W, R_W]^d| \leq \eta |W|$. Let $\widehat{S}: \mathbb{R}^d \times \text{Conf}(\mathbb{R}^d) \to \mathbb{R}$ be any estimator. Using the fact that $h$ is increasing and $|[-R_W, R_W]^d| \leq \eta |W|$, we have
		\begin{multline*}
			\sup_{S \in \Theta(\beta, L)} \sup_{\Phi \sim \mathcal{L}_{(S, M)}} \mathbb{E} \left[h(|W|^{\frac{\beta}{2\beta+d}}|\widehat{S}(0; \Phi \cap W) - S(0)|)\right] \\ \geq \sup_{S \in \Theta(\beta, L)} \sup_{\Phi \sim \mathcal{L}_{(S, M)}} \mathbb{E}\left[h_{\eta}(|[-R_W, R_W]^d|^{\frac{\beta}{2\beta+d}}|\widehat{S}_W(0; \Phi \cap [-R_W, R_W]^d) - S(0)|)\right],
		\end{multline*}
		where $h_{\eta} := h(\cdot/\eta^{\frac{2\beta}{2\beta +d}})$ and $\widehat{S}_W(\cdot~; X) := \widehat{S}(\cdot~; X \cap W)$. Note that in the last line, we used that $W \cap [-R_W, R_W]^d = W$. Accordingly,
		\begin{align}\label{eq:mini_int1}
			r_W \geq \inf_{\widehat{S}} \sup_{S \in \Theta(\beta, L)} \sup_{\Phi \sim \mathcal{L}_{(S, M)}} \mathbb{E} \left[h_{\eta}(|[-R_W, R_W]^d|^{\frac{\beta}{2\beta+d}}|\widehat{S}(0; \Phi \cap [-R_W, R_W]^d) - S(0)|)\right].
		\end{align}		
		To derive a lower bound for this quantity, we follow the method presented in Chapter~2 of~\cite{tsybakov2003introduction}, and specifically apply Theorem~2.2 therein.
		
		To alleviate the following computations, we use the notation $R := R_W$. Let $\Phi_{PPP}$ be a Poisson point process of intensity $1$. Let $\phi: \mathbb{R}^d \to \mathbb{R}^+$ be a non identically null, infinitely differentiable function with support included in $B(0, 1/4)$. Then, $c_0 = \phi\ast\phi/\|\phi\ast\phi\|_{\infty}$ is positive definite and has support included in $B(0, 1/2)$. Moreover $0 \leq c_0 \leq 1$. Let $\varepsilon \in (0, 1)$ with $1/\varepsilon \in \mathbb{N}$, that will be chosen small enough during in the proof. Let
		$$\rho := \varepsilon R/\left\lceil \frac{R}{R^{d/(d+2\beta)}} \right\rceil,\quad \sigma^2 = (R^d)^{-(d+\beta)/(2\beta+d)}.$$
		We consider a stationary centered Gaussian field $(N(x))_{x \in \mathbb{R}^d}$ with covariance function $c(x) = \sigma^2 c_0(x/\rho)$ and a Cox point process $\Phi_{Cox}$ with intensity field $(1+ \sin(N(x)))_{x \in \mathbb{R}^d}$. In order to apply Theorem 2.2 of \cite{tsybakov2003introduction} we prove that:
		\begin{enumerate}
			\item\label{pts:min1} The structure factor $S_c$ of $\Phi_{Cox}$ is $\Theta(\beta, L)$,
			\item\label{pts:min2} $|S_c(0) -1| \geq A |[-R, R]^d|^{-\beta/(2\beta +d)}$ for some constant $A > 0$ that does not depend on $R$,
			\item\label{pts:min3} $\chi^2(\Phi_{PPP}\cap [-R, R]^d, \Phi_{Cox} \cap [-R, R]^d) \leq \alpha$ where $\alpha < \infty$ does also not depend on $R$, 
			\item\label{pts:min4} The total variations of the reduced factorial cumulant measures of $\Phi_{Cox}$ satisfies $$\max_{m \in \{2, 3, 4\}}|\gamma_{red}^{(m)}|(\mathbb{R}^{d(m-1)}) \leq M.$$
		\end{enumerate}
		For the moment, assume that all previous points are satisfied. Then, using \eqref{eq:mini_int1}, we obtain
		\begin{align*}
			r_W \geq \inf_{\widehat{S}} \max_{j = 0, 1}\mathbb{E}_{\Phi_j \cap [-R, R]^d} \left[h_{\eta}(|[-R, R]^d|^{\frac{\beta}{2\beta+d}}|\widehat{S}(0; \Phi_j \cap [-R, R]^d) - S(0)|)\right],
		\end{align*}		
		where $\Phi_0 := \Phi_{PPP}$ and $\Phi_1 := \Phi_{Cox}$. Note that the expectations are with respect to the distributions of $\Phi_j \cap [-R, R]^d$. Then, employing the method of Chapter 2, Section 2 of \cite{tsybakov2003introduction}, we obtain $
		r_W \geq c_1 \inf_{\Psi} \max_{j = 0, 1}\mathbb{P}_{\Phi_j \cap [-R, R]^d}[\Psi \neq j],
		$
		where the infimum is over all tests $\Psi: \text{Conf}(\mathbb{R}^d) \to \{0, 1\}$, and $c_1$ is a constant that does not depend on $R$. Using Theorem 2.2 of \cite{tsybakov2003introduction}, we obtain $r_W \geq c_1 \exp(-\alpha)/2> 0$, which does not depend on $W$. This yields \eqref{eq:minimax_S}. 
		
		Now, we prove that points \ref{pts:min1}, \ref{pts:min2}, \ref{pts:min3} and \ref{pts:min4} are satisfied. We consider the point~\ref{pts:min1}. According to Lemma~\ref{lem:Sc}, the structure factor $S_c$ of $\Phi_{Cox}$ is
		\begin{align}\label{eq:Sc_min}
			S_c(k) &= 1 + \frac{e^{-\sigma^2 c_0(0)}}2\int_{\mathbb{R}^d} \left(e^{\sigma^2 c_0(x/\rho)} - e^{-\sigma^2 c_0(x/\rho)}\right) e^{-\bm{i} k\cdot x} dx \nonumber\\& = 1 + \rho^d \frac{e^{-\sigma^2 c_0(0)}}2\int_{B(0, 1)} \left(e^{\sigma^2 c_0(x)} - e^{-\sigma^2 c_0(x)}\right) e^{-\bm{i} \rho k\cdot x} dx.
		\end{align}
		Since $c_0$ has compact support, the Lebesgue dominated convergence theorem ensures that $S_c$ is infinitely differentiable. Let $l \geq 0$ be the smallest integer that is strictly smaller than $\beta$. In particular, $S_c$ is $l$-times differentiable. Moreover, for all $n = (n_1, \dots, n_d) \in \mathbb{N}^d$ such that $n_1 + \dots + n_d = l$, we have for $(k, k') \in \mathbb{R}^{2d}$:
		\begin{align*}
			|\partial_n^l S_c(k) - \partial_n^l S_c(k')| & \leq \rho^{d+l}  \int_{B(0, 1)} \frac12\left(e^{\sigma^2 c_0(x)} - e^{-\sigma^2 c_0(x)}\right) |x|^{l} \left|e^{\bm{i} \rho k\cdot x} - e^{\bm{i} \rho k'\cdot x}\right| dx \\& \leq e \rho^{d+l}  \sigma^2 \int_{B(0, 1)} c_0(x) |x|^{l} 2^{1 - (\beta -l)} \rho^{\beta -l}\left|k - k'\right|^{\beta -l}|x|^{\beta -l} dx \\& \leq 2e\rho^{d+\beta}  \sigma^2 \left(\int_{B(0, 1)} c_0(x) |x|^{\beta} dx\right) |k - k'|^{\beta -l}.
		\end{align*}
		The choices of $\rho$ and $\sigma^2$ ensure:
		\begin{align}\label{eq:check_number_minimax}
			\rho^{d+\beta} \sigma^2 \leq \varepsilon^{d+\beta} \frac{R^{d+\beta}}{\left(R/R^{d/(2\beta+d)}\right)^{d+\beta}} (R^d)^{-(d+\beta)/(2\beta+d)} = \varepsilon^{d+\beta} 
		\end{align}
		By reducing $\varepsilon$ if necessary, we have $S_c \in \Theta(\beta, L)$. Concerning point~\ref{pts:min2}, with ~\eqref{eq:Sc_min}, we get
		\begin{align*}
			|S_c(0) - 1| & = \rho^d \frac{e^{-\sigma^2 c_0(0)}}2\int_{B(0, 1)} \left(e^{\sigma^2 c_0(x)} - e^{-\sigma^2 c_0(x)}\right) dx \\& \geq \rho^d  \sigma^2 e^{-c_0(0)} \int_{B(0, 1)} e^{\sigma^2 c_0(x)} c_0(x) dx \geq \rho^d  \sigma^2 e^{-c_0(0)} \|c_0\|_1.
		\end{align*}
		Note that 
		\begin{align*}
			\rho^d \sigma^2 &\geq \varepsilon^{d}\frac{R^{d}}{\left(R/R^{d/(2\beta+d)}+1\right)^{d+\beta}} (R^d)^{-\frac{d+\beta}{2\beta +d}} 
			\geq 2^{-\beta - d} \varepsilon^d (R^d)^{-\frac{\beta}{2\beta +d}}.
		\end{align*}
		Subsequently, point~\ref{pts:min2} holds with $A = 2^{-\beta - d}\varepsilon^d e^{-c_0(0)} \|c_0\|_1.$ To verify point~\ref{pts:min3}, we use Lemma~\ref{lem:chi2}. If $\varepsilon = 1/n_0$ with $n_0 \in \mathbb{N}$ and $n_0 \geq 5$, then its assumptions are satisfied. Indeed, with computation similar to~\eqref{eq:check_number_minimax}, we get
		$
		8^d \sigma^4 \rho^{2d} \leq 8^d \varepsilon^{2d} (R^d)^{-2\beta/(2\beta +d)} \leq 8^d \varepsilon^{2d} < 1.
		$
		Hence Lemma~\ref{lem:chi2} ensures that 
		\begin{align*}
			&\chi^2(\Phi_{PPP}\cap [-R, R]^d, \Phi_{Cox} \cap [-R, R]^d)  \leq \left(1- 8^d \sigma^4 \rho^{2d} \right)^{-\frac{R^d}{\rho^d}} \\& \quad \leq \left(1- 8^d \varepsilon^{2d}(R^d)^{-\frac{2\beta}{2\beta +d}} \right)^{-\frac{(R^d)^{2\beta/(2\beta +d)}}{\varepsilon^d}} \leq \sup_{x \in [1, \infty)} (1 - 2d 8^d \varepsilon^{2d}/x)^{-x/\varepsilon^d} := \alpha.
		\end{align*}
		Accordingly, point~\eqref{pts:min3} holds. It remains to verify point~\ref{pts:min4}. Firstly,  using $\sinh(x) \leq x e^{x}$ for $x \geq 0$, we have 
		\begin{align*}
			|\gamma_{\text{red}}^{(2)}|(\mathbb{R}^{d}) &= \frac{e^{-\sigma^2 c_0(0)} \rho^d}2 \int_{B(0, 1)} (e^{\sigma^2 c_0(x)} - e^{-\sigma^2 c_0(x)}) dx \\&  \leq \rho^d \sigma^2 e  \int_{B(0, 1)} c_0(x) dx \leq \varepsilon^{2d}e  \int_{B(0, 1)} c_0(x) dx.
		\end{align*}
		Moreover, according to Lemma~\ref{lem:mini_bril_4}, 
		$|\gamma_{\text{red}}^{(3)}|(\mathbb{R}^{2d}) \leq K  \sigma^4 \rho^{2d},$ $|\gamma_{\text{red}}^{(4)}|(\mathbb{R}^{3d}) \leq K  \sigma^6 \rho^{3d},$
		for some numerical constant $K < \infty$. Since,
		$$\sigma^6 \rho^{3d} \leq \varepsilon^{3d} (R^d)^{3d/(2\beta +d)} R^{-3(d+\beta)/(2\beta +d)} \leq \varepsilon^{3d}, \quad \sigma^4 \rho^{2d} \leq \varepsilon^{2d},$$
		then point~\ref{pts:min4} holds, if $\varepsilon$ is small enough. Theorem 2.2. of~\cite{tsybakov2003introduction} concludes the proof.
	\end{proof}
	
	\begin{proof}[Proof of Theorem~\ref{thm:minimax_g}]
		Let $W \in \mathcal{W}$. Arguing as in the beginning of the proof of Theorem \ref{thm:minimax_S}, we obtain
		\begin{align*}
			r &:= \inf_{\widehat{g}} \sup_{g \in \Theta(\beta, L)} \sup_{\Phi \sim \mathcal{L}_{(g, M)}} \sup_{x_0 \in \mathbb{R}^d} \mathbb{E}\left[h(|W|^{\frac{\beta}{2\beta+d}}|\widehat{g}(x_0; \Phi \cap W) - g(x_0)|)\right] \\& \geq \inf_{\widehat{g}} \sup_{g \in \Theta(\beta, L)} \sup_{\Phi \sim \mathcal{L}_{(g, M)}} \mathbb{E}\left[h_{\eta}(|[-R, R]^d|^{\frac{\beta}{2\beta+d}}|\widehat{g}(0; \Phi \cap [-R, R]^d) - g(0)|)\right],
		\end{align*}
		where $\eta$ does not depends on $W$, $h_{\eta} := h(\cdot/\eta^{\beta/(2\beta+d)})$ and $R = R_W > 0$. As in the proof of  Theorem~\ref{thm:minimax_S}, the result will follow from the method of Chapter 2 of \cite{tsybakov2003introduction}.
		
		Let $\Phi_{PPP}$ be a Poisson point process of intensity $1$. Let $\phi: \mathbb{R}^d \to \mathbb{R}^+$ be a non identically null, infinitely differentiable function with support included in $B(0, 1/4)$. Then, $c_0 = \phi\ast\phi/\|\phi\ast\phi\|_{\infty}$ is positive definite and has support included in $B(0, 1/2)$. Moreover $0 \leq c_0 \leq 1$. Let $\varepsilon \in (0, 1)$, that will be chosen small enough during in the proof. Consider
		$$\rho := 1/\left\lceil R^{d/(d+2\beta)} \right\rceil,\quad \sigma^2 = \varepsilon R^{-d\beta/(2\beta+d)}.$$
		We consider a stationary centered Gaussian field $(N(x))_{x \in \mathbb{R}^d}$ with covariance function $c(x) = \sigma^2 c_0(x/\rho)$ and a Cox point process $\Phi_{Cox}$ with intensity field $(1+ \sin(N(x)))_{x \in \mathbb{R}^d}$. In order to apply Theorem 2.2 of \cite{tsybakov2003introduction} we prove that:
		\begin{enumerate}
			\item\label{pts:min1_g} The pair correlation function $g_c$ of $\Phi_{Cox}$ is $\Theta(\beta, L)$,
			\item\label{pts:min2_g} $|g_c(0) -1| \geq A |[-R, R]^d|^{-\beta/(2\beta +d)}$ for some constant $A > 0$ that does not depend on $R$,
			\item\label{pts:min3_g} $\chi^2(\Phi_{PPP}\cap [-R, R]^d, \Phi_{Cox} \cap [-R, R]^d) \leq \alpha$ where $\alpha < \infty$ does also not depend on $R$,  
			\item\label{pts:min4_g} The total variations of the reduced factorial cumulant measures of $\Phi_{Cox}$ satisfies $$\max_{m \in \{2, 3, 4\}}|\gamma_{red}^{(m)}|(\mathbb{R}^{d(m-1)}) \leq M.$$
		\end{enumerate}
		We check point \ref{pts:min1_g}. Recall from Lemma \ref{lem:Sc}, that
		$$\forall x \in \mathbb{R}^d,~g_c(x) = 1+ \frac{e^{-\sigma^2 c_0(0)}}2\left(e^{\sigma^2 c_0(x/\rho)} - e^{-\sigma^2 c_0(x/\rho)}\right).$$
		Let $n = (n_1, \dots, n_d) \in \mathbb{N}^d$ and $l = n_1 + \dots + n_d$. According to Faà-di-Bruno formula, 
		\begin{align*}
			\partial_n^l (g_c(x)-1) = \frac{e^{-\sigma^2c_0(0)}}2 \sum_{\pi \in [l]} \left(e^{\sigma^2 c_0(\frac{x}{\rho})} \prod_{b \in \pi} \frac{\sigma^2}{\rho^{|b|}} \partial_b c_0(\frac{x}{\rho}) + (-1)^{|\pi|}e^{-\sigma^2 c_0(\frac{x}{\rho})} \prod_{b \in \pi} \frac{\sigma^2}{\rho^{|b|}} \partial_b c_0(\frac{x}{\rho})\right),
		\end{align*}
		where in the partition $\pi$, the numbers $1, \dots, n_1$ correspond to a differentiation with respect to $x_1$, and $n_1+1, \dots, n_1+n_2$ to $x_2$ and so on. This yields, using $\sigma \leq 1$, 
		\begin{align*}
			\|\partial_n^l (g_c-1)\|_{\infty} \leq e \frac{\sigma^2}{\rho^l} \sum_{\pi \in [m]} \prod_{b \in \pi} \|\partial_b c_0\|_{\infty}.
		\end{align*}
		Let $l \geq 0$ be the smallest integer that is strictly smaller than $\beta$. Then, for $(x, y) \in \mathbb{R}^{2d}$, 
		\begin{align*}
			|\partial_n^l g_c(x) - \partial_n^l g_c(x)| &\leq (2 \|\partial_n^l (g_c-1)\|_{\infty})^{1-(\beta -l)} (\|\nabla \partial_n^l  (g_c-1)\|_{\infty} |x -y|)^{\beta-l} \\& \leq C_1 \left(\frac{\sigma^2}{\rho^{l}}\right)^{1 - (\beta-l)} \left(\frac{\sigma^2}{\rho^{l+1}}\right)^{\beta-l} |x -y|^{\beta-l} = C_1 \frac{\sigma^2}{\rho^{\beta}}|x -y|^{\beta-l},
		\end{align*}
		for some constant $C_1 < \infty$ that depends only on $c_0$.
		The choices of $\sigma^2$ and $\rho$ ensure that $c_1 \sigma^2/\rho^{\beta} \leq L$ for $\varepsilon$ small enough. This yields point \ref{pts:min1_g}. Concerning point \ref{pts:min2_g}, we have, using $\sinh(x) \geq x$ for $x \geq 0$, 
		\begin{align*}
			g_c(0) -1 =  \frac{e^{-\sigma^2 c_0(0)}}2\left(e^{\sigma^2 c_0(0)} - e^{-\sigma^2 c_0(0)}\right) \geq \frac{e^{-\varepsilon}}2 \sigma^2 c_0(0) = A |[-R, R]^d|^{-\beta/(2\beta +d)},
		\end{align*}
		with $A = \exp(-\varepsilon) \varepsilon^2 c_0(0)/2 > 0$. We deduce point \ref{pts:min3_g} from Lemma \ref{lem:chi2}:
		\begin{align*}
			\chi^2(\Phi_{PPP}\cap [-R, R]^d, \Phi_{Cox} \cap [-R, R]^d) &\leq (1- \varepsilon^2 8^d R^{-(d+\beta)/(2\beta +d)})^{\varepsilon R^{(d+\beta)/(2\beta +d)}} \\& \leq \sup_{x \geq 1} (1 - \varepsilon^2 8^d/x)^{\varepsilon x} =: \alpha,
		\end{align*}
		provided that $\varepsilon^2 8^d < 1$. Finally, since $\sigma$ and $\rho$ are smaller than 1, Lemma \ref{lem:mini_bril_4} ensures that point \ref{pts:min4_g} is satisfied. This concludes the proof, using Theorem 2.2. of \cite{tsybakov2003introduction}.
		
	\end{proof}

	\section*{Acknowledgment}
	
	The author sincerely thanks Frédéric Lavancier for thoroughly reviewing the first versions of the manuscript and for the enriching discussions, as well as Bart{\l}omiej B{\l}aszczyszyn for the stimulating discussions surrounding this work.
	
	\bibliographystyle{acm} 

	
	\appendix
	\section{Proofs related to Section~\ref{sec:l2_risk}}\label{sec:proof_risk}
	
	\subsection{Variance study}
	
	The next lemma provides an upper bound on the variance of the multitaper estimators. Its main argument relies on Bessel's inequality:
	$$
	\sum_{i \in I} |\langle f_i, h \rangle|^2 \leq \|h\|_2^2,
	$$
	for any orthonormal family of function $(f_i)_{i \in I}$ of $L^2(\mathbb{R}^d)$ and $h \in L^2(\mathbb{R}^d)$. This allows us to exploit the orthogonality of the tapers in a non-asymptotic regime. This argument will be used extensively throughout the proofs.

	\begin{lemma}\label{lem:var_S_hat}
		Let $\Phi$ satisfying Assumption \ref{ass_rho_leb} with intensity $\lambda = 1$. Let $W$ be a subset of $\mathbb{R}^d$, $I$ be a discrete subset of $\mathbb{N}^d$ and let $k_0 \in \mathbb{R}^d$. We consider a family of orthonormal functions $(f_i)_{i \in I} \in L^2(\mathbb{R}^d)^{|I|}$. Then, 
		\begin{equation}
			\operatorname{Var}[\widehat{S}(k_0)] \leq \frac{2\|S\|_{\infty}^2}{|I|} \left(1 + \left(\sum_{i \in I} \|f_{i} \mathds{1}_{\mathbb{R}^d \setminus W}\|_2^2\right)^{\frac12} \right)^2 + \left(\frac1{|I|} \sum_{i \in I} \|f_i\|_4^2\right)^2|\gamma_{\text{red}}^{4}|,
		\end{equation} 
		where $
		|\gamma_{\text{red}}^{4}| := \lambda + 7 |\gamma_{\text{red}}^{(2)}|(\mathbb{R}^d) + 6 |\gamma_{\text{red}}^{(3)}|(\mathbb{R}^{2d}) +|\gamma_{\text{red}}^{(4)}|(\mathbb{R}^{3d}).$
	\end{lemma}
	\begin{proof}
		Recall that $\widehat{S}(k_0) = |I|^{-1}\sum_{i \in I} |C_i(k_0)|^2$ where $(C_i(k_0))_{i \in I}$ is defined in \eqref{eq:def_Ti_Ci}. Using the bi-linearity of the covariance, we get
		\begin{align*}
			\operatorname{Var}[\widehat{S}(k_0)] = \frac1{|I|^2} \sum_{(i_1, i_2) \in I^2} \kappa_2(|C_{i_1}(k_0)|^2, |C_{i_2}(k_0)|^2).
		\end{align*}
		According to Corollary~\ref{corol:cumulants_square_rvs}, we have 
		\begin{align*}
			\kappa_2(|C_{i_1}(k_0)|^2, |C_{i_2}(k_0)|^2) = \sum_{\substack{\sigma \in \Pi[4],\\ \sigma \vee \tau = \mathbf{1}_{4}}} \prod_{b \in \sigma} \kappa(D_{i_j}(k_0); j \in b),
		\end{align*}
		where $\tau = 12|34$, $\vee$ denotes the least-upper-bound between partitions (see Section \ref{sec:cum}) and 
		$
		D_{i_1}(k_0) = C_{i_1}(k_0), ~ D_{i_2}(k_0) = \overline{C_{i_1}(k_0)},~
		D_{i_3}(k_0) = C_{i_2}(k_0), ~ D_{i_4}(k_0) = \overline{C_{i_2}(k_0)}.
		$
		Note that we are considering cumulants of centered random variables, so the sum runs over the partitions with blocks of size larger than $2$. Hence:
		\begin{align}\label{eq:lem_var_k4}
			\kappa_2(|C_{i_1}&(k_0)|^2, |C_{i_2}(k_0)|^2) = \sum_{\substack{\sigma \in \Pi_2[4],\\ \sigma \vee \tau = \mathbf{1}_{4}}} \prod_{b \in \sigma} \kappa(E_{i_j}(k_0); j \in b) \nonumber \\& = \kappa(T_{i_1}(k_0), \overline{T_{i_1}(k_0)}, T_{i_2}(k_0), \overline{T_{i_2}(k_0)}) + \sum_{\substack{\sigma \in \Pi_2[4],\\ \sigma \vee \tau = \mathbf{1}_{4}, |\sigma| = 2}} \prod_{b \in \sigma}\operatorname{Cov}\left[E_{i_{b_1}}(k_0), E_{i_{b_2}}(k_0)\right],
		\end{align}
		where $\Pi_2[4]$ denotes the sets of partitions of $[4]$ having only blocks of size larger than $2$ and
		$
		E_{i_1}(k_0) = T_{i_1}(k_0),~ E_{i_2}(k_0) = \overline{T_{i_1}(k_0)},~ E_{i_3}(k_0) = T_{i_2}(k_0), ~ E_{i_4}(k_0) = \overline{T_{i_2}(k_0)}.
		$
		We bound the cumulant corresponding to the block of order $4$ using Lemma \ref{lem:bril_thin}:
		\begin{align}\label{eq:weak_BM}
			|\kappa&(T_{i_1}(k_0), \overline{T_{i_1}(k_0)}, T_{i_2}(k_0), \overline{T_{i_2}(k_0)})| \leq \|f_{i_1}\|_4^2 \|f_{i_2}\|_4^2 \sum_{\pi \in \Pi[4]} |\gamma_{\text{red}}^{(|\pi|)}(\mathbb{R}^{d(|\pi|-1)}).
		\end{align}
		Subsequently,
		\begin{equation}\label{eq:lem_var_k4_final}
			\frac1{|I|^2} \sum_{(i_1, i_2) \in I^2}|\kappa(T_{i_1}(k_0), \overline{T_{i_1}(k_0)}, T_{i_2}(k_0), \overline{T_{i_2}(k_0)})| \leq \left(\frac1{|I|} \sum_{i \in I} \|f_i\|_4^2\right)^2 |\gamma_{\text{red}}^{4}|.
		\end{equation}
		We now focus on the covariance terms. Note that $\tau = 12|34$, so the only two partitions $\sigma \in \Pi_2[4]$ such that $\sigma \vee \tau = \mathbf{1}_{4}$ are $13|24$ and $14|23$. Accordingly, 
		\begin{align*}
			\sum_{\substack{\sigma \in \Pi[4],\\ \sigma \vee \tau = \mathbf{1}_{4}, |\sigma| = 2}} \prod_{b \in \sigma}\operatorname{Cov}\big[E_{i_{b_1}}(k_0), &E_{i_{b_2}}(k_0)\big] = \operatorname{Cov}\left[E_{i_{1}}(k_0), E_{i_{3}}(k_0)\right]\operatorname{Cov}\left[E_{i_{2}}(k_0), E_{i_{4}}(k_0)\right]\\& \qquad  +\operatorname{Cov}\left[E_{i_{1}}(k_0), E_{i_{4}}(k_0)\right]\operatorname{Cov}\left[E_{i_{2}}(k_0), E_{i_{3}}(k_0)\right]
			\\& = \left|\operatorname{Cov}\left[T_{i_1}(k_0), T_{i_2}(k_0)\right]\right|^2 + \left|\operatorname{Cov}\left[T_{i_1}(k_0), T_{i_2}(-k_0)\right]\right|^2.
		\end{align*}
		According to Lemma \ref{lem:bessel_k2}, we have
		\begin{multline}\label{eq:var_cov_same_freq_1}
			\frac1{|I|^2} \sum_{(i_1, i_2) \in I^2} \left(\left|\operatorname{Cov}\left[T_{i_1}(k_0), T_{i_2}(k_0)\right]\right|^2 +\left|\operatorname{Cov}\left[T_{i_1}(k_0), T_{i_2}(-k_0)\right]\right|^2\right)\\ \leq \frac{2\|S\|_{\infty}^2}{|I|} \left(1 + \left(\sum_{i \in I} \|f_i \mathds {1}_{\mathbb{R}^d \setminus W}\|_2^2\right)^{1/2} \right)^2, 
		\end{multline}
		which concludes the proof of this lemma.
		
	\end{proof}
	
	\subsection{Bias study}
	
	Next lemma studies the bias for the multitaper estimator. Note that the factor $\sqrt{d}$ appears only when $\beta > 1$ and is due to the assumption that $S \in \Theta(\beta, L)$, which concerns the coefficient of the gradient of $S$ and not their $\ell^2([d])$ norm.
	\begin{lemma}\label{lem:bias_s_hat}
		Let $\Phi$ satisfying Assumption \ref{ass_rho_leb} with intensity $\lambda = 1$. Let $W$ be a subset of $\mathbb{R}^d$, $I$ be a discrete subset of $\mathbb{N}^d$, $k_0 \in \mathbb{R}^d$ and let $(f_i)_{i \in I}$ be a family of functions with unit $L^2(\mathbb{R}^d)$ norm. Assume that $S \in \Theta(\beta, L)$ with $(\beta, L) \in (0, \infty)^2$ (see Definition~\ref{def:Hölder}). Assume that $\beta \leq 2$. Then, 
		\begin{equation}
			\left|\mathbb{E}\left[\widehat{S}(k_0)\right] - S(k_0)\right| \leq \frac{L \sqrt{d}}{|I|}\sum_{i \in I} \|f_i\|_{\dot{H}^{\beta/2}}^2 + \frac{4\|S\|_{\infty}}{|I|}\sum_{i \in I} \|f_i \mathds{1}_{\mathbb{R}^d \setminus W}\|_2.
		\end{equation}
	\end{lemma}
	\begin{proof}
		Assume first that $\beta \leq 1$. Using the fact that $\|f_i\|_2 = 1$ for all $i \in I$, we get 
		\begin{align*}
			\left|\mathbb{E}\left[\widehat{S}(k_0)\right] - S(k_0)\right| \leq A +B, 
		\end{align*}
		where 
		\begin{align}\label{eq:a_{|b|}iais}
			& A := \frac{1}{|I|}\sum_{i \in I} \int_{\mathbb{R}^d} \left||\mathcal{F}[f_i \mathds{1}_{W}](k)|^2 - |\mathcal{F}[f_i](k)|^2\right|  \left(S(k_0 + k) + S(k_0)\right)dk,
			\\& B =  \frac{1}{|I|}\sum_{i \in I} \int_{\mathbb{R}^d} |\mathcal{F}[f_i](k)|^2 \left|S(k_0 + k) - S(k_0)\right|dk.  \nonumber
		\end{align}
		Since $S$ is $\beta$-Hölder with $\beta \leq 1$, we have
		\begin{align*}
			B \leq \frac{1}{|I|}\sum_{i \in I} \int_{\mathbb{R}^d} |\mathcal{F}[f_i](k)|^2 L |k|^{\beta} dk = \frac{L}{|I|}\sum_{i \in I} \|f_i\|_{\dot{H}^{\beta/2}}^2.
		\end{align*}
		Using the Cauchy-Schwarz inequality and the Plancherel theorem we obtain
		\begin{align*}
			A  &\leq \frac{2\|S\|_{\infty}}{|I|}\sum_{i \in I} \int_{\mathbb{R}^d} \left|\mathcal{F}[f_i \mathds{1}_{\mathbb{R}^d \setminus W}](k)|\right| \left||\mathcal{F}[f_i \mathds{1}_{W}](k)| + |\mathcal{F}[f_i](k)|\right| dk \\& \leq  \frac{2\|S\|_{\infty}}{|I|}\sum_{i \in I} \|\mathcal{F}[f_i \mathds{1}_{\mathbb{R}^d \setminus W}]\|_2 \||\mathcal{F}[f_i \mathds{1}_{W}]| + |\mathcal{F}[f_i]|\|_2 \leq \frac{4\|S\|_{\infty}}{|I|}\sum_{i \in I} \|f_i \mathds{1}_{\mathbb{R}^d \setminus W}\|_2,
		\end{align*}
		which concludes the proof when $\beta \leq 1$. Assume now that $\beta \in (1, 2]$. Hence $S$ is differentiable and using the fact that $k \mapsto  |\mathcal{F}[f_i \mathds{1}_{W}](k)|^2 \nabla S(k_0) \cdot k$ is even, we have
		\begin{align*}
			\left|\mathbb{E}\left[\widehat{S}(k_0)\right] - S(k_0)\right| = \left|\frac{1}{|I|}\sum_{i \in I} \int_{\mathbb{R}^d} |\mathcal{F}[f_i \mathds{1}_{W}](k)|^2 \left(S(k_0 + k) - S(k_0) - \nabla S(k_0) \cdot k \right)dk\right|  \leq A +B', 
		\end{align*}
		where $A$ is defined in~\eqref{eq:a_{|b|}iais} and 
		$$B' := \frac{1}{|I|}\sum_{i \in I} \int_{\mathbb{R}^d} |\mathcal{F}[f_i](k)|^2 \left|S(k_0 + k) - S(k_0) - \nabla S(k_0) \cdot k \right|dk.$$
		Since $S$ is $\beta$-Hölder with $\beta \geq 1$, we have
		\begin{align}\label{eq:biais_beta_s1}
			\left|S(k_0 +k) - S(k_0) - \nabla S(k_0) \cdot k \right| & = \left|\int_{0}^1 \left(\nabla S(k_0 + t k) - \nabla S(k_0)\right) \cdot k dt \right| \nonumber\\& \leq \int_{0}^1 \sqrt{d} L t^{\beta - 1} |k|^{\beta -1}  |k| dt  \leq L \sqrt{d} |k|^{\beta}.
		\end{align}
		Then, concluding as in the proof of the case $\beta \leq 1$, we obtain the result when $\beta \in (1, 2]$.
	\end{proof}
	
	\subsection{Proofs of Theorem~\ref{thm:L2_risk} and Corollary~\ref{cor:L2_risk_herm}}\label{sec:proof_l2_risk_end}
	
	In this section, gathering Lemmas~\ref{lem:var_S_hat} and~\ref{lem:bias_s_hat}, we prove Theorem~\ref{thm:L2_risk}. Then, we specify Theorem~\ref{thm:L2_risk} for the Hermite tapers with Corollary~\ref{cor:L2_risk_herm}.
	
	\begin{proof}[Proof of Theorem~\ref{thm:L2_risk}]
		The bias variance decomposition writes 
		\begin{align*}
			\text{s} := \mathbb{E}\left[\left|\widehat{S}(k_0) - S(k_0)\right|^2\right] &=  \operatorname{Var}[\widehat{S}(k_0)] + \left|\mathbb{E}\left[\widehat{S}(k_0)\right]  - S(k_0)\right|^2.
		\end{align*}
		According to Lemma~\ref{lem:var_S_hat} and Lemma~\ref{lem:bias_s_hat}, we have, by triangular inequality
		\begin{align*}
			&s \leq \Bigg(\frac{2\|S\|_{\infty}^2}{|I|} \left(1 + \left(\sum_{i \in I} \|f_{i} \mathds{1}_{\mathbb{R}^d \setminus W}\|_2^2\right)^{\frac12} \right)^2 + \left(\frac1{|I|} \sum_{i \in I} \|f_i\|_4^2\right)^2|\gamma_{\text{red}}^{4}|(\mathbb{R}^{3d}) \\& \qquad \qquad+ \left(\frac{\sqrt{d}L}{|I|r^{\beta}}\sum_{i \in I} \|f_i\|_{\dot{H}^{\beta/2}}^2 + \frac{4\|S\|_{\infty}}{|I|}\sum_{i \in I} \|f_i \mathds{1}_{\mathbb{R}^d \setminus W}\|_2\right)^2\Bigg)^{\frac12} \\& \leq \frac{\sqrt{2}\|S\|_{\infty}}{|I|^{\frac12}}+ \frac1{|I|} \sum_{i \in I} \|f_i\|_4^2 |\gamma_{\text{red}}^{4}|(\mathbb{R}^{3d})^{\frac12} + \frac{\sqrt{d}L}{|I|}\sum_{i \in I} \|f_i\|_{\dot{H}^{\beta/2}}^2 \\& \qquad\qquad + \sqrt{2}\|S\|_{\infty}\left(\frac1{|I|}\sum_{i \in I} \|f_{i} \mathds{1}_{\mathbb{R}^d \setminus W}\|_2^2\right)^{\frac12} + \frac{4\|S\|_{\infty}}{|I|}\sum_{i \in I} \|f_i \mathds{1}_{\mathbb{R}^d \setminus W}\|_2.
		\end{align*}
		We conclude using $\sqrt{2}+4 \leq 6$ and the Cauchy-Schwarz inequality:
		$$\sum_{i \in I} \|f_i \mathds{1}_{\mathbb{R}^d \setminus W}\|_2 \leq |I|^{\frac12} \left(\sum_{i \in I} \|f_{i} \mathds{1}_{\mathbb{R}^d \setminus W}\|_2^2\right)^{\frac12}.$$
	\end{proof}
	
	\begin{proof}[Proof of Corollary~\ref{cor:L2_risk_herm}]
		The change of variable $x \leftrightarrow x/r$ and Lemma~\ref{lem:tail_herm} applied with $\rho = R/r$ give $
		W\text{-}loc \leq C e^{-c' (R/r)^{3\theta}} \leq C e^{-c' |I|^{\frac{3\theta}{2d}}},
		$
		where $0 < c', C < \infty$ are constants. In the second inequality we used \eqref{cond:risk_herm_rd}. With \eqref{cond:risk_herm_I}, we obtain
		$	W\text{-}loc \leq C e^{-c |W|^{\frac{3\beta\theta}{d(2\beta+d)}}},$
		for constants that do not depend on $r, I$ and $W$. Moreover, Lemma~\ref{lem:norm_H12_herm} and Lemma~\ref{lem:L4_herm} yield
		\begin{align*}
			&\mathcal{F}\text{-}loc \leq \frac{L \sqrt{d}|I|^{\beta/(2d)}}{r^{\beta}}, \qquad 
			\frac1{|I|}\sum_{i \in I} \|\psi_i\|_4^2 \leq  c_{\psi}^d  \frac{\log(|I|^{\frac1d})^{\frac{d}{2}}}{r^{d/2}|I|^{1/4}},
		\end{align*}
		where $c_{\psi} \in (0, \infty)$ is a numerical constant. Condition \eqref{cond:risk_herm_rd} implies that $r \geq R/(\sqrt{2} |I|^{1/(2d)})$.
		Subsequently, we get, with Theorem~\ref{thm:L2_risk} 
		\begin{multline}
			\mathbb{E}\left[\left|\widehat{S}(k_0) - S(k_0)\right|^2\right]^{\frac12}  \leq \\ \frac{\sqrt{2}\|S\|_{\infty}}{|I|^{\frac12}}  + \frac{L \sqrt{d}}{2^{\frac{\beta}2}} \left(\frac{|I|}{|W|}\right)^{\frac{\beta}d} + \left(\frac{c_{\psi}}{\sqrt{2}}\right)^d|\gamma_{\text{red}}^{4}|^{\frac12} \frac{\log(|I|^{\frac1d})^{\frac{d}{2}}}{|W|^{\frac12}} + 6 C\|S\|_{\infty} e^{-c |W|^{\frac{3\beta\theta}{d(2\beta+d)}}}.\label{eq:risk_l2_almost_final}
		\end{multline}
		Using condition \eqref{cond:risk_herm_I} we conclude the proof with elementary computations.
	\end{proof}
	
	\section{Proof related to Section~\ref{sec:nn_asymp}}\label{sec:proof_conc}
	
	The results of Section~\ref{sec:nn_asymp} rely on the control the cumulants of the multitaper estimator. Lemma~\ref{lem:kappa_m_1}, which constitutes a first step for a workable bound on the cumulants of $\widehat{S}(k_0)$, is the core of the proof of Theorem~\ref{thm:conc}. It relies on the Isserlis-Wick formula of Corollary~\ref{corol:cumulants_square_rvs} that allows to express the joint cumulants of square of random variables as a sum of products of joint cumulants of these random variables. The key point of this lemma is to provide a bound on the cumulants of the multitaper estimator where the tapers indexes are summed (see $\Sigma$ defined in \eqref{eq:sigma}). To this end, we make essential use of the combinatorial Lemma~\ref{lem:som_comb}, which crucially exploits the connectivity condition of the partitions involved.
	
	\begin{lemma}\label{lem:kappa_m_1}
		Let $\Phi$ satisfying Assumption \ref{ass_rho_leb} with intensity $\lambda = 1$ and such that $S \in \Theta(\beta, L)$ with $\beta \in (0, 2]$ and $L < \infty.$ Let $W$ be a subset of $\mathbb{R}^d$ and $I$ be a discrete subset of $\mathbb{N}^d$. Assume that $(f_i)_{i \in I}$ is family of orthonormal functions that are $L^1(\mathbb{R}^d) \cap L^{\infty}(\mathbb{R}^d)$. If $\beta > 1$, assume for all $(i, i') \in I^2$, $\mathcal{F}[f_i](k) \mathcal{F}[f_{i'}](k) =  \mathcal{F}[f_i](-k) \mathcal{F}[f_{i'}](-k)$ for all $k \in \mathbb{R}^d$. Let $k_0 \in \mathbb{R}^d$. Then, for all $m \geq 2$, 
		\begin{align}\label{eq:kappa_m_1}
			|\kappa_m(\widehat{S}(k_0))| \leq 2^{m-1}(m-1)! \frac{S(k_0)^m}{|I|^{m-1}} +\sum_{\substack{\sigma \in \Pi[2m],\\ \sigma \vee \tau = \mathbf{1}_{2m}, |\sigma| \geq 2}} \prod_{b \in \sigma} K_{|b|},
		\end{align}
		where, for $\Sigma$ defined in \eqref{eq:sigma}:
		\begin{align}\label{eq:cum_X}
			K_1 = 0,~K_2 = \Sigma, \quad \forall s \geq 3,~K_s = \left(\frac1{|I|} \sum_{i \in I} a_{s}(i)^2\right)^{\frac{s}2} \sum_{\pi \in \Pi[s]} |\gamma_{\text{red}}^{(|\pi|)}|(\mathbb{R}^{d(|\pi|-1)}).
		\end{align}
		Finally, $a_{s}(i) := \|f_i\|_{s}$ for $s \geq 4$ and
		\begin{align*}
			a_{3}(i) := \left(3\sup_{i \in I} \|f_i \mathds{1}_{\mathbb{R}^d \setminus W}\|_3 \sup_{i \in I} \|f_i\|_3^2 +  3\sup_{i \in I} \|\mathcal{F}[f_i] \mathds{1}_{|\cdot| \geq |k_0|/3}\|_1\right)^{\frac13} \wedge  \|f_i\|_{3}.
		\end{align*}
	\end{lemma}
	\begin{proof}
		Using the $m$-linearity of the cumulants we get
		\begin{align*}
			\kappa_m(\widehat{S}(k_0)) = \frac1{|I|^{m}} \sum_{(i_1, \dots, i_m) \in I^m} \kappa(|C_{i_1}(k_0)|^2, \dots, |C_{i_m}(k_0)|^2).
		\end{align*}
		Recall that $(C_i(k_0))_{i \in I}$ are defined in \eqref{eq:def_Ti_Ci}. Then, using Corollary~\ref{corol:cumulants_square_rvs} and arguing as in the proof of Lemma~\ref{lem:var_S_hat}, we obtain
		\begin{align}\label{eq:isserlis_expl}
			\kappa(|C_{i_1}(k_0)|^2, \dots, |C_{i_m}(k_0)|^2) &= \sum_{\substack{\sigma \in \Pi_2[2m],\\ \sigma \vee \tau = \mathbf{1}_{2m}}} \prod_{b \in \sigma} \kappa(E_{i_j}(k_0); j \in b),
		\end{align}	
		where $\Pi_2[2m]$ denotes the sets of partitions of $[2m]$ having only blocks of size larger than $2$, $E_{i_{2p-1}}(k_0) = T_{i_p}(k_0)$ and $E_{i_{2p}}(k_0) = \overline{T_{i_p}(k_0)}$ for $p \in [m]$, and $\tau = 12|\dots|(2m-1)2m.$
		We bound the cumulants corresponding to the blocks of size larger than $3$ using Lemma \ref{lem:bril_thin}:
		\begin{align*}
			\Big|\kappa(E_{i_j}(&k_0); j \in b)\Big| \leq \prod_{j \in b} \|f_{i_{\lceil j/2\rceil}}\|_{|b|} \sum_{\pi \in \Pi[|b|]} |\gamma_{\text{red}}^{(|\pi|)}|(\mathbb{R}^{d(|\pi|-1)}).
		\end{align*}
		Moreover, when $|b| = 3$, Lemma \ref{lem:decor_freq_3} yields
		\begin{multline*}
			\Big|\kappa(E_{i_j}(k_0); j \in b)\Big|  \\ \leq  3 \sum_{\pi \in \Pi[3]} |\gamma_{\text{red}}^{(|\pi|)}|(\mathbb{R}^{d(|\pi|-1)})  \left(\sup_{i \in I} \|f_i \mathds{1}_{\mathbb{R}^d \setminus W}\|_3 \sup_{i \in I} \|f_i\|_3^2 +  \sup_{i \in I} \|\mathcal{F}[f_i] \mathds{1}_{|\cdot| \geq |k_0|/3}\|_1 \right).
		\end{multline*}
		Indeed, the sum of frequencies involved in the functions $f_i$ is of the form
		$k_0  \sum_{s = 1}^3 \epsilon_s,$
		with $(\epsilon_s)_{s\in [3]} \in \{-1, 1\}^3$. Hence 	$|k_0  \sum_{s = 1}^3 \epsilon_s| \geq |k_0|$.
		Accordingly,
		\begin{multline*}
			|	\kappa_m(|C_{i_1}(k_0, r)|^2, \dots, |C_{i_m}(k_0)|^2)| \\\leq \sum_{\substack{\sigma \in \Pi_2[2m],\\ \sigma \vee \tau = \mathbf{1}_{2m}}} \prod_{\substack{b \in \sigma \\ |b| \geq 3}} \prod_{j \in b} a_{|b|}(i_{\lceil j/2\rceil})   \sum_{\pi \in \Pi[|b|]} |\gamma_{\text{red}}^{(|\pi|)}|(\mathbb{R}^{d(|\pi|-1)}) \prod_{\substack{b \in \sigma \\|b| = 2}} |\kappa(E_{i_j}(k_0); j \in b)|.
		\end{multline*}
		We have thus obtained:
		\begin{align*}
			|\kappa_m(\widehat{S}(k_0))| \leq \sum_{\substack{\sigma \in \Pi_2[2m],\\ \sigma \vee \tau = \mathbf{1}_{2m}}}  \left(\prod_{\substack{b \in \sigma \\ |b| \geq 3}}  \sum_{\pi \in \Pi[|b|]} |\gamma_{\text{red}}^{(|\pi|)}|(\mathbb{R}^{d(|\pi|-1)})\right) A_{\sigma},
		\end{align*}
		with 
		\begin{align*}
			A_{\sigma}:= \frac1{|I|^m} \sum_{(i_1, \dots, i_m) \in I^m} \prod_{\substack{b \in \sigma \\ |b| \geq 3}} \prod_{j \in b} a_{|b|}(i_{\lceil j/2\rceil}) \prod_{\substack{b \in \sigma \\|b| = 2}} |\kappa(E_{i_j}(k_0); j \in b)|.
		\end{align*}
		
		To continue the proof, we introduce some notations. Denote by $\mathcal{B}_2(\sigma)$ the blocks of size $2$ of $\sigma$. Let $\varepsilon_p = 1$ if $p$ is even and $-1$ otherwise. Moreover, denote
		$$\mathcal{B}_2^+(\sigma):= \{b = (b_1, b_2) \in \mathcal{B}_2(\sigma)|~\varepsilon_{b_1} \varepsilon_{b_1} = +1\},$$
		and $\mathcal{B}_2^-(\sigma) := \mathcal{B}_2(\sigma) \setminus \mathcal{B}_2^+(\sigma).$ Finally, for $b = (b_1, b_2) \in \mathcal{B}_2(\sigma)$, let
		$\delta_b := \mathds{1}_{i_{\lceil b_1/2 \rceil} = i_{\lceil b_2/2 \rceil}}.$
		We introduce these sets since we anticipate, for $b = (b_1, b_2) \in \mathcal{B}_2^+(\sigma)$, 
		\begin{align*}
			\kappa(E_{i_j}(k_0); j \in b) &= \int_{\mathbb{R}^d} \mathcal{F}[f_{i_{\lceil \frac{b_1}2 \rceil}} \mathds{1}_{W}](k + \varepsilon_{b_1} k_0) \overline{\mathcal{F}[f_{i_{\lceil \frac{b_2}2 \rceil}} \mathds{1}_{W}]}(k + \varepsilon_{b_2} k_0)S(k) dk \\& \simeq S(k_0) \langle f_{i_{\lceil \frac{b_1}2 \rceil}}, f_{i_{\lceil \frac{b_2}2 \rceil}}\rangle = S(k_0) \delta_b,
		\end{align*}
		whereas for $b = (b_1, b_2) \in \mathcal{B}_2^-(\sigma)$, with $b_1$ even and $b_2$ odd,
		\begin{align*}
			\kappa(E_{i_j}(k_0); j \in b) \simeq S(k_0) \langle \mathcal{F}[f_{i_{\lceil \frac{b_1}2 \rceil}}](\cdot + k_0), \mathcal{F}[f_{i_{\lceil \frac{b_2}2 \rceil}}](\cdot - k_0)\rangle.
		\end{align*}
		To take into account this previous heuristic, we denote, for $b \in \mathcal{B}_2^+(\sigma)$, 
		$$\Delta_b := 	\kappa(E_{i_j}(k_0); j \in b) - S(k_0) \delta_b.$$ 
		In the term $A_{\sigma}$, we expand the product over the blocks of size two of $\sigma$ as follows:
		\begin{align*}
			\prod_{\substack{b \in \sigma \\|b| = 2}} |\kappa(E_{i_j}(k_0); j \in b)| &=  \prod_{b \in \mathcal{B}_2^-(\sigma)} |\kappa(E_{i_j}(k_0); j \in b)|  \prod_{b \in \mathcal{B}_2^+(\sigma)} \left(S(k_0) \delta_b + \Delta_b \right) \\& = \sum_{P \in \mathcal{P}(\mathcal{B}_2^+(\sigma))}  \prod_{b \in \mathcal{B}_2^-(\sigma)} |\kappa(E_{i_j}(k_0); j \in b)|  \prod_{b \in P} S(k) \delta_b \prod_{b \in P^c} \Delta_b, 
		\end{align*}
		where $\mathcal{P}(\mathcal{B}_2^+(\sigma))$ denote the subsets of $\mathcal{B}_2^+(\sigma)$.
		This yields:
		$$A_{\sigma} = \frac1{|I|^m} \sum_{P \in \mathcal{P}(\mathcal{B}_2^+(\sigma))} S(k_0)^{|P|} B_{\sigma, P},$$
		where
		$$B_{\sigma, P} := \sum_{(i_1, \dots, i_m) \in I^m} \prod_{\substack{b \in \sigma \\ |b| \geq 3}} \prod_{j \in b} a_{|b|}(i_{\lceil j/2\rceil}) \prod_{b \in \mathcal{B}_2^-(\sigma)} |\kappa(E_{i_j}(k_0); j \in b)|  \prod_{b \in P} \delta_b \prod_{b \in P^c} \Delta_b,$$
		where we recall that $\delta_b$ and $\Delta_b$ depend on indexes $i_1, \dots, i_m$. We distinguish several cases. Assume that $\sigma$ contains only blocks of size $2$ and that $|P| = m$. In that case, according to Lemma \ref{lem:som_comb}, 
		$B_{\sigma, P} = |I|.$
		Otherwise, Lemma \ref{lem:som_comb} yields:
		\begin{align*}
			B_{\sigma, P} \leq \prod_{\substack{b \in \sigma \\ |b| \geq 3}} \left(\sum_{i \in I} a_{|b|}(i)^2\right)^{\frac{|b|}2} \left(\sum_{(i, i') \in I^2} |a_2^{-}(i, i')|^2 \right)^{\frac{|\mathcal{B}_2^-(\sigma)|}2} \left(\sum_{(i, i') \in I^2} |a_2^{+}(i, i')|^2 \right)^{\frac{|P^c|}2},
		\end{align*}
		where
		\begin{align*}
			&a_2^{-}(i, i') := \int_{\mathbb{R}^d} \mathcal{F}[f_{i} \mathds{1}_{W}](k + k_0) \overline{\mathcal{F}[f_{i'} \mathds{1}_{W}]}(k - k_0)S(k) dk,
			\\& a_2^{+}(i, i') := \int_{\mathbb{R}^d} \mathcal{F}[f_{i} \mathds{1}_{W}](k + k_0) \overline{\mathcal{F}[f_{i'} \mathds{1}_{W}]}(k + k_0)S(k) dk - S(k_0) \mathds{1}_{i = i'}.
		\end{align*}
		Note that, using $S(\cdot) = S(-\cdot)$ and changes of variable, we have $|a_2^{\pm}(i, i')| = |a_2^{\pm}(i', i)|$.  If $\sigma$ contains at least one block of size larger than three, we have, using the binomial theorem:
		\begin{align*}
			A_{\sigma} &\leq \frac1{|I|^m}  \prod_{\substack{b \in \sigma \\ |b| \geq 3}} \left(\sum_{i \in I} a_{|b|}(i)^2\right)^{\frac{|b|}2} \left(\sum_{(i, i') \in I^2} |a_2^{-}(i, i')|^2 \right)^{\frac{|\mathcal{B}_2^-(\sigma)|}2} \left(S(k_0) + \left(\sum_{(i, i') \in I^2} |a_2^{+}(i, i')|^2 \right)^{\frac12}\right)^{|\mathcal{B}_2^+(\sigma)|} \\& \leq \prod_{\substack{b \in \sigma \\ |b| \geq 3}} \left(\frac1{|I|}\sum_{i \in I} a_{|b|}(i)^2\right)^{\frac{|b|}2} \prod_{b \in \mathcal{B}_2(\sigma)} \Sigma_0,
		\end{align*}
		where
		\begin{align*}
			\Sigma_0 := \frac{S(k_0)}{|I|} + \frac1{|I|}\left(\sum_{(i, i') \in I^2} |a_2^{+}(i, i')|^2 \right)^{\frac12} +  \frac1{|I|}\left(\sum_{(i, i') \in I^2} |a_2^{-}(i, i')|^2 \right)^{\frac12}.
		\end{align*}
		Indeed, for the term $|I|^m$, we use the fact that
		$2m = 2 |\mathcal{B}_2(\sigma)| + \sum_{b \in \sigma, ~|b| \geq 3} |b|.$
		If $\sigma$ contains only blocks of size $2$, we obtain:
		\begin{align*}
			&A_{\sigma} \leq \frac{S(k_0)^m}{|I|^{m-1}} + \frac1{|I|^m}\left(\sum_{(i, i') \in I^2} |a_2^{-}(i, i')|^2 \right)^{\frac{|\mathcal{B}_2^-(\sigma)|}2} \sum_{P \in \mathcal{P}(\mathcal{B}_2^+(\sigma)), |P| < m} S(k_0)^{|P|}  \left(\sum_{(i, i') \in I^2} |a_2^{+}(i, i')|^2 \right)^{\frac{|P^c|}2} \\&\quad \leq \frac{S(k_0)^m}{|I|^{m-1}} + \prod_{\substack{b \in \sigma \\ |b| \geq 3}} \left(\frac1{|I|}\sum_{i \in I} a_{|b|}(i)^2\right)^{\frac{|b|}2} \prod_{b \in \mathcal{B}_2(\sigma)} \Sigma_0.
		\end{align*}
		Gathering both cases and using the fact that $|\{\sigma \in \Pi[2m]|~|\sigma| = m,~ \sigma \vee \tau = \mathbf{1}_{2m}\}| =  2^{m-1}(m-1)!$ (exercise 3.5 page 92 of~\cite{mccullagh2018tensor}), we get
		\begin{multline*}
			|\kappa_m(\widehat{S}(k_0))| \leq 2^{m-1}(m-1)! \frac{S(k_0)^m}{|I|^{m-1}} \\ + \sum_{\substack{\sigma \in \Pi_2[2m],\\ \sigma \vee \tau = \mathbf{1}_{2m}}}  \left(\prod_{\substack{b \in \sigma \\ |b| \geq 3}}  \sum_{\pi \in \Pi[|b|]} |\gamma_{\text{red}}^{(|\pi|)}|(\mathbb{R}^{d(|\pi|-1)})\right)  \left(\frac1{|I|}\sum_{i \in I} a_{|b|}(i)^2\right)^{\frac{|b|}2} \prod_{b \in \mathcal{B}_2(\sigma)} \Sigma_0,
		\end{multline*}
		where in the second term we use the fact that $\mathds{1}_{|\sigma| = m} + \mathds{1}_{|\sigma| < m} = 1$. The end of the proof consists in bounding $\Sigma_0$ by $\Sigma$ (see \eqref{eq:sigma}). According to Lemma \ref{lem:decor_freq_2}, 
		\begin{align*}
			\frac1{|I|}\left(\sum_{(i, i') \in I^2} |a_2^{-}(i, i')|^2 \right)^{\frac12} \leq 2 \|S\|_{\infty} \left(\left(\frac1{|I|}\sum_{i \in I} \|f_i \mathds{1}_{\mathbb{R}^d \setminus W}\|_2^2\right)^{\frac12} + \left(\frac1{|I|}\sum_{i \in I} \|\mathcal{F}[f_i] \mathds{1}_{|\cdot| \geq |k_0|}\|_2^2\right)^{\frac12} \right).
		\end{align*}
		We also have, according to Lemma \ref{lem:bessel_k2}.
		\begin{align*}
			\frac1{|I|}\left(\sum_{(i, i') \in I^2} |a_2^{-}(i, i')|^2 \right)^{\frac12} \leq \frac{\|S\|_{\infty}}{|I|^{\frac12}} + \|S\|_{\infty}\left(\frac1{|I|}\sum_{i \in I} \|f_i \mathds{1}_{\mathbb{R}^d \setminus W}\|_2^2\right)^{\frac12}.
		\end{align*}
		For the sum related to $a_2^{+}$, we obtain, using the triangular inequality, the Cauchy-Schwarz inequality, the Plancherel theorem and the fact that $(f_i)_{i \in I}$ is orthogonal:
		\begin{align*}
			|a_2^{+}(i, i')| \leq \|S\|_{\infty} \left(\|f_i \mathds{1}_{\mathbb{R}^d \setminus W}\|_2  +\|f_{i'}\mathds{1}_{\mathbb{R}^d \setminus W}\|_2 \right) + |\langle \mathcal{F}[f_i], \mathcal{F}[f_{i'}] (S(\cdot-k_0) - S(k_0)) \rangle|.
		\end{align*}
		Assume that $\beta \leq 1$. Then, using again the fact that $(f_i)_{i \in I}$ is orthogonal we have:
		\begin{align*}
			\frac1{|I|}\left(\sum_{(i, i') \in I^2} |a_2^{+}(i, i')|^2 \right)^{\frac12} \leq L\left(\frac1{|I|^2}\sum_{i \in I} \|f_i \|_{\dot{H}^{\beta}}^2\right)^{\frac12}+ 2 \|S\|_{\infty} \left(\frac1{|I|}\sum_{i \in I} \|f_i \mathds{1}_{\mathbb{R}^d \setminus W}\|_2^2\right)^{\frac12}.
		\end{align*}
		Assume that $\beta \in (1, 2]$. In that case $\mathcal{F}[f_i] \mathcal{F}[f_{i'}]$ is even. Hence $\langle \mathcal{F}[f_i], \mathcal{F}[f_{i'}] \nabla S(k_0) \cdot k \rangle = 0$. Accordingly, arguing as in the proof of Lemma \ref{lem:bias_s_hat}, we obtain
		\begin{align*}
			\frac1{|I|}\left(\sum_{(i, i') \in I^2} |a_2^{+}(i, i')|^2 \right)^{\frac12} \leq \sqrt{d}L\left(\frac1{|I|^2}\sum_{i \in I} \|f_i \|_{\dot{H}^{\beta}}^2\right)^{\frac12}+ 2 \|S\|_{\infty} \left(\frac1{|I|}\sum_{i \in I} \|f_i \mathds{1}_{\mathbb{R}^d \setminus W}\|_2^2\right)^{\frac12}.
		\end{align*}
		The previous bounds yields $\Sigma_0 \leq \Sigma$, which concludes the proof.
	\end{proof}
	
	The next lemma exploits the connectivity condition of partitions arising in the computation of joint cumulants of the squares of the linear statistics associated with the tapers. This condition allows for combinatorial simplifications due to the orthogonality of the tapers. The key point is to provide a bound involving the sums over indices $i, i' \in I$ of the sequences $(a_2^b(i, i'))_{b \in \mathcal{B}_2(\sigma)}$. Indeed, $a_2^b(i, i')$ typically corresponds to a scalar product between the taper $f_i$ and a function depending on $b$. In this context, Bessel's inequality yields a sharper bound on $\sum_{(i, i') \in I^2} |a_2^b(i, i')|^2$ than the trivial bound $\sum_{(i, i') \in I^2} |a_2^b(i, i')|^2 \leq |I|^2 \sup_{(i, i')} |a_2^b(i, i')|^2$, for which the use of the combinatorial structure would not have been necessary. However, this refinement is essential to show that the non-asymptotic correlation term $a_2$ in Theorem~\ref{thm:conc} is negligible at sufficiently large frequencies.
	
	\begin{lemma}\label{lem:som_comb}
		Let $I$ be a discrete subset of $\mathbb{N}^d$.  Let $m \geq 2$ and $\sigma$ be a partition of $[2m]$ having block of size at least $2$. Assume that $\sigma \vee \tau = \mathbf{1}_m$, where $\tau = 12|\dots|(2m-1) 2m$. Denote by $\mathcal{B}_2(\sigma)$ the blocks of size $2$ of $\sigma$. For a partition $b \in \mathcal{B}_2(\sigma)$, we denote $b_1 < b_2$ such that $b = (b_1, b_2)$.
		Let $s \leq m$. Consider $P$ a subset of cardinal $s$ of $\mathcal{B}_2(\sigma)$. 
		Finally consider non-negative numbers $(a_n(i))_{n \geq 3, i \in I}$ and $(a_2^b(i, i'))_{b \in \mathcal{B}_2(\sigma), (i, i')\in I^2}$.
		Denote:
		$$B_{\sigma, P}:= \sum_{(i_1, \dots, i_m) \in I^m} \prod_{\substack{b \in \sigma \\ |b| \geq 3}} \prod_{j \in b} a_{|b|}(i_{\lceil j/2\rceil}) \prod_{b \in \mathcal{B}_2(\sigma) \setminus P} a_2^b(i_{\lceil b_1/2 \rceil}, i_{\lceil b_2/2 \rceil})  \prod_{b \in P}  \mathds{1}_{i_{\lceil b_1/2 \rceil} = i_{\lceil b_2/2 \rceil}}.$$
		Then, $B_{\sigma, P}= |I|$ if $|\sigma| = m = s$. Otherwise, we have 
		\begin{equation}\label{eq:B_2and3}
			B_{\sigma, P} \leq \prod_{\substack{b \in \sigma \\ |b| \geq 3}} \left(\sum_{i \in I} |a_{|b|}(i)|^2\right)^{\frac{|b|}2} \prod_{b \in \mathcal{B}_2(\sigma) \setminus P} \left(\sum_{(i, i') \in I^2} |a_2^b(i, i')|^2 \right)^{\frac12}.
		\end{equation}
	\end{lemma}
	\begin{proof}
		For the proof, we introduce graphs to account for the connectivity condition $\sigma \vee \tau = \mathbf{1}_{[2m]}$. Let $G = G(\sigma)$ be an undirected graph whose vertices are the union of the blocks of $\sigma$ and the blocks of $\tau$. An edge is drawn between two vertices $b$ and $b'$ if $b \cap b' \neq \emptyset$. The condition $\sigma \vee \tau = \mathbf{1}_{[2m]}$ is equivalent to $G$ being connected, that is, having a single connected component (see Section 3.2 of~\cite{mccullagh2018tensor}).
		
		For a vertex $v = (v_1, \dots, v_{|v|})$ in $G$, we call $v_1, \dots, v_{|v|}$ the labels of $v$. We now define $G_2$ as a graph having the same edges as $G$, but where each label $n$ appearing in the blocks of $G$ is replaced by $\lceil n/2 \rceil$. Since the vertices of $G$ corresponding to blocks of the partition $\tau$ yield in $G_2$ vertices where the two labels are equal, we remove such vertices from $G_2$.
		
		Note that the labels of $G_2$ belong to $\{1, \dots, m\}$. Moreover, for any vertex $v$ of $G_2$ with only two labels $v_1$ and $v_2$, we necessarily have $v_1 \neq v_2$. To see this, assume the contrary. If $v_1 = v_2$, then $v$ comes from a vertex of $G$ of the form $(2n-1, 2n)$ for some $n \in [m]$, corresponding to a block of $\sigma$. In that case, the block $(2n-1, 2n)$ is only connected to the vertex of $G$ corresponding to the block $(2n-1, 2n)$ of $\tau$, and conversely. Therefore, $G$ would not be connected, which contradicts the assumption.
		
		The graph $G_2$ satisfies two key properties: $G_2$ is connected (since $G$ is connected) and each number $n \in [m]$ appears exactly twice in the labels of $G_2$.
		
		For the following, we introduce additional notations. Consider a graph $G_0$ whose labels belong to $[m]$. Denote by $\mathcal{V}_2(G_0)$ the set of vertices of $G_0$ with exactly two labels, and by $\mathcal{V}_3(G_0)$ the set of vertices of $G_0$ with more than three labels. We also denote by $P_{\mathcal{V}}(G_0)$ the (possibly empty) set of vertices of $G_0$ corresponding to the blocks of $\sigma$ belonging to $P$.
		
		In the following, products of the form $\prod_{j \in v}$, where $v \in \mathcal{V}_2(G_0) \cup \mathcal{V}_3(G_0)$, iterate over the labels of the vertex $v$. We denote by $\mathcal{L}(G_0)$ the set of distinct labels of $G_0$ (without repetition). Finally, for a vertex $v$ of $G_0$ with exactly two labels, we write $v = (v_1, v_2)$ with $v_1 < v_2$. With the previous notations, we can rewrite $B_{\sigma, P}$ as
		$$B_{\sigma, P} := \sum_{i_l \in I; l \in \mathcal{L}(G_2)} \prod_{v \in \mathcal{V}_3(G_2)} \prod_{j \in v} a_{|v|}(i_j) \prod_{v \in \mathcal{V}_2(G_2) \setminus P_{\mathcal{V}}(G_2)} a_2^v(i_{v_1}, i_{v_2})  \prod_{v \in P_{\mathcal{V}}(G_2)}  \mathds{1}_{i_{v_1} = i_{v_2}}.$$

		Consider the case where $\sigma$ contains at least one block of size larger than three, i.e. $|\sigma| < m$.
		
		We proceed inductively. If $P_{\mathcal{V}}$ is empty, we do nothing. Otherwise, we take into account one constraint $\mathds{1}_{i_{v_1} = i_{v_2}}$ in the sum (for instance, the one such that $(v_1, v_2)$ is the smallest for the lexicographical order). Specifically, we remove the vertex $v =(v_1, v_2)$ in $G_2$ and change the value of the label $v_1$ appearing in the other vertex $v'$ of $G_2$ by $v_2$. We obtain a new graph $G_2'$. Recall that every label appears only twice in $G_2$, so $v_1$ appears in only one block distinct of $(v_1, v_2)$). The case $v' = (v_2, v_1)$ is impossible since $\sigma$ contains at least one block of size larger than three. Indeed, otherwise $v, v'$ would be a connected component of $G_2$, distinct from $G_2$, which contradicts the fact that $G_2$ is connected. Hence, $v$ was connected to two (and only two) distinct vertices $v'$ and $v''$. We add in $G_2'$ an edge between $v'$ and $v''$.
		
		With that procedure, $G_2'$ remains connected. Moreover, $G_2'$ has one fewer block than $G_2$ and contains the labels of $G_2$, except for $v_1$. Since we removed and added once the label $v_2$, every label in $G_2'$ appears exactly twice. Finally, we have $|P_{\mathcal{V}}(G_2')| = |P_{\mathcal{V}}(G_2)| - 1$. The previous discussion yields
		$$
		B_{\sigma, P} = \sum_{i_l \in I,\ l \in \mathcal{L}(G_2')} \prod_{v \in \mathcal{V}_3(G_2')} \prod_{j \in v} a_{|v|}(i_j) \prod_{v \in \mathcal{V}_2(G_2') \setminus P_{\mathcal{V}}(G_2')} a_2^v(i_{v_1}, i_{v_2})  \prod_{v \in P_{\mathcal{V}}(G_2')}  \mathds{1}_{i_{v_1} = i_{v_2}}.
		$$
		
		Note that the number of terms in the product over $\mathcal{V}_2(G_2') \setminus P_{\mathcal{V}}(G_2')$ is unchanged and still equals $|\mathcal{B}_2(\sigma)| - s$. Since $G_2'$ satisfies the same properties as $G_2$ that we used previously, we can iterate this procedure to obtain
		$$
		B_{\sigma, P} = \sum_{i_l \in I,\ l \in \mathcal{L}(G_2^s)} \prod_{v \in \mathcal{V}_3(G_2^s)} \prod_{j \in v} a_{|v|}(i_j) \prod_{v \in \mathcal{V}_2(G_2^s) \setminus P_{\mathcal{V}}(G_2^s)} a_2^v(i_{v_1}, i_{v_2}),
		$$
		where $G_2^s$ is a graph having $m-s$ labels, each appearing exactly twice.
		To upper bound the last quantity, we establish the following intermediate result. Let $N \geq 2$ and $(c^n(i, i'))_{n \leq \mathbb{N}, (i, i') \in I^2}$ a sequence of non-negative numbers. Then
		\begin{align}\label{eq:rec_cs}
			\sum_{(i_1, \dots, i_N) \in I^N} \prod_{n = 1}^N c^n(i_{n}, i_{n+1}) \leq \prod_{n = 1}^N\left(\sum_{(i, i') \in I^2} c^n(i, i')^2\right)^{\frac12}, 
		\end{align}
		with the convention $i_{N+1} = i_1$. We prove it by induction. For $N = 2$, it follows from the Cauchy-Schwarz inequality:
		\begin{align*}
			\sum_{i_1, i_2 \in I^N} c^1(i_1, i_2) c^2(i_2, i_1) &\leq \sum_{i_1 \in I} \left(\sum_{i' \in I} c^1(i_1, i')^2\right)^{\frac12} \left(\sum_{i \in I} c^2(i, i_1)^2\right)^{\frac12} \\& \leq \left(\sum_{(i, i') \in I^2} c^1(i, i')^2\right)^{\frac12} \left(\sum_{(i, i') \in I^2} c^2(i, i')^2\right)^{\frac12}.
		\end{align*}
		Fix $N \geq 2$. Assume that the inequality \eqref{eq:rec_cs} is true when the product of its left-hand side is over $N$ terms. Then, by induction we have
		\begin{align*}
			&\sum_{(i_1, \dots, i_{N+1}) \in I^{N+1}} \prod_{n = 1}^{N+1} c^n(i_{n}, i_{n+1})  = \sum_{(i_1, \dots, i_N) \in I^N}  \prod_{n = 1}^{N-1} c^n(i_{n}, i_{n+1}) \left(\sum_{i_{N+1} \in I} c^{N}(i_N, i_{N+1}) c^{N+1}(i_{N+1}, i_1)\right) \\& \qquad\leq \prod_{n = 1}^{N-1} \left(\sum_{(i, i') \in I^2} c^n(i, i')^2\right)^{\frac12} \left(\sum_{(i, i') \in I^2}  \left(\sum_{i_{N+1} \in I} c^{N}(i, i_{N+1}) c^{N+1}(i_{N+1}, i')\right)^2 \right)^{\frac12} \\& \qquad \leq \prod_{n = 1}^{N+1}\left(\sum_{(i, i') \in I^2} c^n(i, i')^2\right)^{\frac12},
		\end{align*}
		which proves \eqref{eq:rec_cs}. 
		
		We now have the tools to upper bound $B_{\sigma, P}$. We introduce
		\begin{align*}
			\mathcal{L}_{3} := \{l \in \mathcal{L}(G_2^s)|~l \text{ appears only in vertices having more than 3 labels}\}.
		\end{align*}
		Then, $B_{\sigma, P} = B_{\sigma, P}^{3} B_{\sigma, P}^{23}$ where
		\begin{align*}
			& B_{\sigma, P}^{3} := \sum_{i_l; l \in	\mathcal{L}_3} \prod_{v \in \mathcal{V}_3(G_2^s)} \prod_{j \in v} a_{|b|}(i_j)^{\mathds{1}_{j \in  \mathcal{L}_3}},
			\\& B_{\sigma, P}^{23} := \sum_{i_l; l \in \mathcal{L} \setminus \mathcal{L}_3} \prod_{v \in \mathcal{V}_3(G_2^s)} \prod_{j \in v} a_{|v|}(i_j)^{\mathds{1}_{j \in \mathcal{L} \setminus \mathcal{L}_3}} \prod_{v \in \mathcal{V}_2(G_2^s) \setminus P_{\mathcal{V}}(G_2^s)} a_2^v(i_{v_1}, i_{v_2}),
		\end{align*}
		We now bound $B_{\sigma, P}^3$. For each label $l \in \mathcal{L}_3$, denote by $v(l)$ and $v'(l)$ the (possibly identical) vertices of size larger than three in which $l$ appears. Using the Cauchy--Schwarz inequality, we get:
		\begin{align*}
			B_{\sigma, P}^{3} = \prod_{l \in \mathcal{L}_3} \sum_{i \in I} a_{|v(l)|}(i) \, a_{|v'(l)|}(i) \leq \prod_{l \in \mathcal{L}_3} \left(\sum_{i \in I} |a_{|v(l)|}(i)|^2 \right)^{\frac{1}{2}} \left(\sum_{i \in I} |a_{|v'(l)|}(i)|^2 \right)^{\frac{1}{2}}.
		\end{align*}
		Next, consider the following product:
		\[
		p := \prod_{v \in \mathcal{V}_3(G_2^s)} \prod_{j \in v} a_{|v|}(i_j)^{\mathds{1}_{j \in \mathcal{L} \setminus \mathcal{L}_3}}.
		\]
		The number of terms in $p$ is equal to $2m - 2|\mathcal{B}_2(\sigma)| - 2|\mathcal{L}_3|$. Since this number is even, we can group the terms of $p$ two by two. Moreover, because each label appearing in $p$ appears exactly once, we can perform this grouping such that each pair contains distinct labels.
		
		Using the fact that each label of $G_2^s$ appears exactly twice, we can rewrite $B_{\sigma, P}$ in the form of the left-hand side of \eqref{eq:rec_cs}. Applying this result yields:
		\begin{align}\label{eq:bound_B23}
			B_{\sigma, P}^{23} \leq \prod_{l \in \mathcal{L}_{23}} \left(\sum_{i \in I} |a_{|v_3(l)|}(i)|^2 \right)^{\frac{1}{2}} \prod_{v \in \mathcal{V}_2(G_2^s) \setminus P_{\mathcal{V}}(G_2^s)} \left(\sum_{(i, i') \in I^2} a_2^v(i, i')^2\right)^{\frac{1}{2}},
		\end{align}
		where
		\[
		\mathcal{L}_{23} := \left\{ l \in \mathcal{L}(G_2^s) \mid l \text{ appears only in one vertex of size } 2 \right\},
		\]
		and for each $l \in \mathcal{L}_{23}$, $v_3(l)$ denotes the vertex of size larger than three in which it appears. Note that, to obtain \eqref{eq:bound_B23}, we used, after applying \eqref{eq:rec_cs}, the fact that for vertices of size larger than three:
		\[
		\sum_{(i', i) \in I^2} \left( a_{n}(i)^2 \, a_{n'}(i') \right)^2 = \left( \sum_{i \in I} a_n(i)^2 \right) \left( \sum_{i' \in I} a_{n'}(i')^2 \right), \quad \text{for} \ n, n' \geq 3,
		\]
		together with the fact that the grouping of terms in $p$ is performed so that each pair contains distinct labels.

		Since the effect on $\mathcal{V}_2(G_2) \setminus \mathcal{P}(G_2)$ and $\mathcal{V}_3(G_2)$ of the procedure to go from $G_2$ to $G_2^s$ was only to change their labels, we have
		$$
		\prod_{v \in \mathcal{V}_2(G_2^s) \setminus P_{\mathcal{V}}(G_2^s)} \left(\sum_{(i, i') \in I^2} a_2^v(i, i')^2\right)^{\frac12} = \prod_{b \in \mathcal{B}_2(\sigma)} \left(\sum_{(i, i') \in I^2} a_2^b(i, i')^2\right)^{\frac12},
		$$
		and 
		\begin{multline*}
			\prod_{l \in \mathcal{L}_3} \left(\sum_{i \in I} |a_{ |v(l)|}(i)|^2 \right)^{\frac12}  \left(\sum_{i \in I} |a_{ |v'(l)|}(i)|^2 \right)^{\frac12} \prod_{l \in \mathcal{L}_{23}} \left(\sum_{i \in I} |a_{|v_3(l)|}(i)|^2 \right)^{\frac12} \\ =\prod_{\substack{b \in \sigma; \\ |b| \geq 3}} \prod_{j \in b} \left(\sum_{i \in I} a_{|b|}(i)^2\right)^{\frac12} = \prod_{\substack{b \in \sigma; \\ |b| \geq 3}} \left(\sum_{i \in I} a_{|b|}(i)^2\right)^{\frac{|b|}2}.
		\end{multline*}
		It yields \eqref{eq:B_2and3} when $|\sigma| < m$.
		
		Now, assume that $\sigma$ contains only blocks of size two, i.e. $|\sigma| = m$. Then,
		$$B_{\sigma, P} = \sum_{i_l \in I; l \in \mathcal{L}(G_2)} \prod_{v \in \mathcal{V}_2(G_2) \setminus P_{\mathcal{V}}(G_2)} a_2^v(i_{v_1}, i_{v_2})  \prod_{v \in P_{\mathcal{V}}(G_2)}  \mathds{1}_{i_{v_1} = i_{v_2}}.$$
		Using similar connectivity arguments as in the case $|\sigma| < m$, we reduce by $s$ the dimension $|\mathcal{L}(G_2)|$ of the sum over the indices $i_l$ for $l \in \mathcal{L}(G_2)$, by introducing a sequence of graphs $G_2^2, \dots, G_2^s$, provided that $G_2^2, \dots, G_2^{s-1}$ do not reduce to two vertices. This situation cannot occur if $s \leq m-1$, since only one vertex is removed at each step and we start with a graph having $m$ vertices. 	Accordingly, if $s \leq m-1$, using \eqref{eq:rec_cs}, we obtain \eqref{eq:B_2and3}, as in the case $|\sigma| = m$. Finally, when $m = s$, the term $B_{\sigma, P}$ reduces to
		$$
		B_{\sigma, P} = \sum_{i_l \in I,; l \in \mathcal{L}(G_2)} \prod_{v \in G_2} \mathds{1}_{i_{v_1} = i_{v_2}}.
		$$
		Reducing the dimension of the sum by $s-1$, using the previously mentioned method, yields:
		$$
		B_{\sigma, P} = \sum_{i_1 \in I, i_2 \in I} \mathds{1}_{i_1 = i_2} \mathds{1}_{i_2 = i_1} = |I|.
		$$
		This concludes the proof of the lemma.
		
	\end{proof}
	
	The next lemma shows that the Brillinger-mixing assumption on the total variation of reduced factorial cumulant measures extends to the total variation of the cumulant measures of linear statistics.This result is specific to point processes, where the discrete nature of the linear statistics implies that their cumulants of order $m \geq 2$ involve factorial cumulants of orders less than or equal to $m$.

	\begin{lemma}\label{lem:tv_cum_measure}
		Let $s \geq 1$ and $q, \gamma > 0$. Then
		\begin{align*}
			\sum_{\pi \in \Pi[s]} q^{|\pi|-1} ((|\pi|-1)!)^{1+\gamma} \leq 2^s (1+q)^{s-1} ((s-1)!)^{1+\gamma}.
		\end{align*}
	\end{lemma}
	\begin{proof}
		We denote by $S(s, l)$ the Stirling number of the second kind. We have
		\begin{align*}
			\sum_{\pi \in \Pi[s]} q^{|\pi|-1} ((|\pi|-1)!)^{1+\gamma} \leq  ((s-1)!)^{\gamma} \sum_{l = 1}^s (l-1)! q^{l-1} S(s, l).
		\end{align*}
		According to equation (4.5) of \cite{saulis2012limit}, $
		S(s, l) = \sum_{n_1 + \dots + n_l = s} \frac{s!}{n_1! \dots n_l! l!},$
		where the sum is over the decompositions of the integer $s$ into positive integers. Hence,
		\begin{align*}
			\sum_{\pi \in \Pi[s]} q^{|\pi|-1} ((|\pi|-1)!)^{1+\gamma} &\leq  ((s-1)!)^{1+\gamma} s \sum_{l = 1}^s \frac1{l} q^{l-1}\sum_{n_1 + \dots + n_l = s} \frac1{n_1! \dots n_l!} \\& \leq ((s-1)!)^{1+\gamma} s \sum_{l = 1}^s \frac1{l} q^{l-1} \sum_{n_1 + \dots + n_l = s} 1  = ((s-1)!)^{1+\gamma} s \sum_{l = 1}^s \frac1{l} q^{l-1} \binom{s-1}{l-1},
		\end{align*}
		where in the second line we used equation (4.2) of \cite{saulis2012limit}. To conclude, we use 
		\begin{align*}
			s \sum_{l = 1}^s \frac1{l} q^{l-1} \binom{s-1}{l-1} = \sum_{l = 1}^s q^{l-1} \binom{s}{l} \leq 2^s (1+q)^{s-1}.
		\end{align*}
	\end{proof}
	The following lemma controls the cumulants appearing on the right-hand side of~\eqref{eq:kappa_m_1}, by introducing suitable random variables: a Gaussian distribution to handle non-asymptotic correlations between tapers in second-order cumulants, and a Gamma distribution to account for higher-order cumulants. Its key idea is to separate second-order contributions from higher-order ones.
	
	\begin{lemma}\label{lem:kappa_m_2}
		Let $\Phi$ satisfying Assumption \ref{ass_rho_leb} with intensity $\lambda = 1$, $I$ be a discrete subset of $\mathbb{N}$ and $(f_i)_{i \in I}$ be a family of $L^2(\mathbb{R}^d)$. Assume that $\Phi$ satisfies Assumption \ref{def:brill_mix} with parameters $(c, q, \gamma) \in (0, \infty)^3$. Assume \eqref{cond:rI}. We denote
		\begin{equation}
			\mathcal{A} := \sum_{\substack{\sigma \in \Pi[2m],\\ \sigma \vee \tau = \mathbf{1}_{2m}, |\sigma| \geq 2}} \prod_{b \in \sigma} K_{|b|},
		\end{equation}
		where $(K_s)_{s \geq 0}$ is defined in \eqref{eq:cum_X}. If \eqref{ass:cesaro_f} is satisfied with $c_f \in (1, \infty)$, $\eta_1, \eta_2 \geq 0$ and $\rho > 0$, then 
		\begin{align}\label{eq:kappa_m_2}
			\mathcal{A} &\leq  \left(\frac{4^{m} m!}{\sqrt{2\pi}}\right)^{\gamma} \left(\frac{m!}{2\sqrt{2\pi}} (4 \Sigma)^m +   \frac{(m!)^2}{2\sqrt{2\pi}} \left(32\frac{ \Lambda^2 c^{1/2} (1+q)^{3/2} \log(|I|^{\frac1d})^{\frac12}}{\rho^{d/2}|I|^{2\eta_1+\eta_2/2}}\right)^m\right),
		\end{align}
		where $\Lambda$ is defined in \eqref{eq:Lambda}. If \eqref{ass:cesaro_f} is satisfied with $c_f \in (1, \infty)$, $\eta_1, \eta_2 \geq 0$ and $\rho > 0$, then
		\begin{align}\label{eq:kappa_m_2_sf}
			\mathcal{A} &\leq  \left(\frac{4^{m} m!}{\sqrt{2\pi}}\right)^{\gamma} \left(\frac{m!}{2\sqrt{2\pi}} (4 \Sigma)^m +   \frac{(m!)^2}{2\sqrt{2\pi}} \left(32\frac{ \lceil c_f \rceil^2 c^{2/3} (1+q)^{4/3} }{\rho^{d/3}|I|^{2\eta_1}}\right)^m\right).
		\end{align}
	\end{lemma}
	\begin{proof}
		Using the fact that $\Phi$ satisfies Assumption \ref{def:brill_mix}, Lemma \ref{lem:tv_cum_measure}, and the inequality $k_1! \dots k_n! \leq (k_1 + \dots + k_n)!$ for all integers $k_1, \dots, k_n$, we obtain
		\begin{align*}
			\prod_{|b| \geq 3}  \sum_{\pi \in \Pi[|b|]} |\gamma_{\text{red}}^{(|\pi|)}|(\mathbb{R}^{d(|\pi|-1)}) &\leq \prod_{|b| \geq 3} c 2^{|b|}(1+q)^{|b|-1} ((|b|-1)!)^{1+\gamma} \\& \leq ((2m)!)^{\gamma} \prod_{|b| \geq 3} c 2^{|b|}(1+q)^{|b|-1} (|b|-1)!.
		\end{align*}
		We define $\kappa_1 = 0,~\kappa_2 = \Sigma$ and
		\begin{align*}
			&\kappa_3 = c2^4(1+q)^2\left(3 \left(\sup_{i \in I} \|f_i \mathds{1}_{\mathbb{R}^d \setminus W}\|_3 \sup_{i \in I} \|f_i\|_3^2 +  \sup_{i \in I} \|\mathcal{F}[f_i] \mathds{1}_{|\cdot| \geq |k_0|/3}\|_1\right) \wedge \frac{c_f}{\rho^{d/2}|I|^{3 \eta_1}}\right), 
			\\&\forall s \geq 4,~\kappa_s = c_f \frac{c \rho^d \log(|I|^{\frac1d}) }{(1+q) |I|^{\eta_2} } \left(2 \frac{1+q}{\rho^{\frac{d}2}|I|^{\eta_1}}\right)^{s} (s-1)!.
		\end{align*}
		Then, using Assumptions \eqref{ass:cesaro_f} and \eqref{ass:cesaro_f3}, we get
		\begin{align}\label{eq:see_kappa_X}
			\mathcal{A} \leq \left(\frac{4^{m} m!}{\sqrt{2\pi}}\right)^{\gamma} \sum_{\substack{\sigma \in \Pi_2[2m],\\ \sigma \vee \tau = \mathbf{1}_{2m}}} \prod_{b \in \sigma} \kappa_{|b|},
		\end{align}
		where in the last line we used the fact that for $m \geq 2$, $(2m)! \leq 4^m (m!)^2/\sqrt{2\pi}$, obtained by Stirling approximation~\cite{gronwall1918gamma}. 
		
		We prove \eqref{eq:kappa_m_2}. To do so, we consider 
		$$X = \Sigma^{\frac12} \mathcal{N} + 2\frac{(c (1+q)^3)^{1/4} \log(|I|^{\frac1d})^{\frac14}}{\rho^{d/4}|I|^{\eta_1+\eta_2/4}}\Gamma,$$ 
		where $\mathcal{N}$ is a standard Gaussian random variable and $\Gamma$ is an Gamma random variable with parameters 
		$\left(\Lambda, 1\right)$. Moreover, $\mathcal{N}$ and $\mathcal{E}$ are independents. We also denote $\overline{X} = X - \mathbb{E}[X]$. We introduced $\overline{X}$ because $\kappa_s \leq \kappa_s(\overline{X})$, for all $s \geq 1$.
		Indeed, $\kappa_1 = \kappa_1(\overline{X}) = 0$, 
		$$\kappa_2 = \Sigma \leq \Sigma + \Lambda \left(2\frac{(c (1+q)^3)^{1/4}  \log(|I|^{\frac1d})^{\frac14}}{\rho^{d/4}|I|^{\eta_1+\eta_2/4}}\right)^2= \kappa_2(\overline{X}).$$
		For $s = 3$, we have
		\begin{align*}
			\kappa_3 &= \kappa_3 \left(\frac{\rho^{d/4}|I|^{\eta_1+\eta_2/4}}{2(c (1+q)^3)^{1/4} \log(|I|^{\frac1d})^{\frac14}}\right)^3 \frac12 \left(2\frac{(c (1+q)^3)^{1/4} \log(|I|^{\frac1d})^{\frac14}}{\rho^{d/4}|I|^{\eta_1+\eta_2/4}}\right)^3 2! \leq \kappa_3(\overline{X}),
		\end{align*}
		where in the second line we used $c \rho^{d} \log(|I|^{\frac1d}) \geq (1+q) |I|^{\eta_2}$. Finally, for $s \geq 4$, using again $c \rho^{d} \log(|I|^{\frac1d})  \geq (1+q) |I|^{\eta_2}$, we have 
		\begin{align*}
			\kappa_s &= c_f \frac{c \rho^d \log(|I|^{\frac1d})}{(1+q) |I|^{\eta_2} } \left(2 \frac{1+q}{\rho^{\frac{d}2}|I|^{\eta_1}}\right)^{s} (s-1)! \leq  c_f \left(\frac{c \rho^d \log(|I|^{\frac1d})}{(1+q) |I|^{\eta_2} }  \right)^{s/4} \left(2 \frac{1+q}{\rho^{\frac{d}2}|I|^{\eta_1}}\right)^{s} (s-1)! \leq \kappa_s(\overline{X}).
		\end{align*}
		Hence, with equation 3.2.6 of \cite{peccati2011wiener}, we obtain 
		\begin{align}\label{eq:A_int}
			\mathcal{A}\leq \left(\frac{4^{m} m!}{\sqrt{2\pi}}\right)^{\gamma} \sum_{\substack{\sigma \in \Pi_2[2m],\\ \sigma \vee \tau = \mathbf{1}_{2m}}} \prod_{b \in \sigma} \kappa_s(\overline{X}) \leq \left(\frac{4^{m} m!}{\sqrt{2\pi}}\right)^{\gamma} \mathbb{E}[\overline{X}^{2m}].
		\end{align}
		Let $a = \Sigma^{\frac12}$ and $b = 2(c (1+q)^3)^{1/4}\log(|I|^{\frac1d})^{\frac14}/(\rho^{d/4}|I|^{\eta_1+\eta_2/4})$. We have
		\begin{align*}
			\mathbb{E}[\overline{X}^{2m}] &= \mathbb{E}\left[\left(a\mathcal{N} +b\left(\Gamma - \mathbb{E}[\Gamma]\right)\right)^{2m}\right] = \sum_{j = 0}^{2m} \binom{2m}{j} a^j b^{2m-j}\mathbb{E}[\mathcal{N}^j] \mathbb{E}[\left(\Gamma - \mathbb{E}[\Gamma]\right)^{2m-j}] \\& = \sum_{j = 0}^{m} \binom{2m}{2j} a^{2j} b^{2m-2j}\mathbb{E}[\mathcal{N}^{2j}] \mathbb{E}[\left(\Gamma - \mathbb{E}[\Gamma]\right)^{2m-2j}].
		\end{align*}
		Let $(\mathcal{E}_n)_{n \in [\Lambda]}$ be i.i.d. standard exponential random variables. Then, for even $j \geq 0$, we have
		\begin{align*}
			\mathbb{E}[\left(\Gamma - \mathbb{E}[\Gamma]\right)^{j}]^{1/j} &= \mathbb{E}\left[\left(\sum_{n = 1}^{\Lambda} \mathcal{E}_n - \mathbb{E}[\mathcal{E}_n]\right)^{j}\right]^{1/j} \leq \sum_{n = 1}^{\Lambda}  \mathbb{E}\left[\left(\mathcal{E}_n - \mathbb{E}[\mathcal{E}_n]\right)^{j}\right]^{1/j} \\& = \Lambda \left( j! \sum_{n = 0}^{j} \frac{(-1)^n}{n!} \right)^{1/j}\leq \Lambda \left(j!\right)^{1/j}.
		\end{align*}		
		Using the expression of the even moments of the Gaussian distribution, we obtain
		\begin{align*}
			\mathbb{E}[\overline{X}^{2m}] &\leq  \sum_{j = 0}^{m} \binom{2m}{2j} a^{2j} b^{2m-2j}\frac{(2j)!}{2^j j!}\Lambda^{2m-2j} (2m-2j)! = (2m)! (b \Lambda)^{2m}\sum_{j = 0}^m \frac1{j!} d^j,
		\end{align*}
		where $d = a^2/(2b^2\Lambda^2)$. Using the binomial theorem, we get
		\begin{align*}
			\sum_{j = 0}^m \frac1{j!} d^j = \frac1{m!}\int_0^{\infty} \left(d+t\right)^m e^{-t} dt \leq \frac{2^{m-1}}{m!}(d^m+m!). 
		\end{align*}
		Accordingly,
		\begin{align*}
			\mathbb{E}[\overline{X}^{2m}] &\leq \frac{(2m)!}{m!} 2^{m-1} \frac{a^{2m}}{2^m} + (2m)! (b\Lambda)^{2m} 2^{m-1}\\& \leq  \frac{m!}{2\sqrt{2\pi}} (4 \Sigma)^m +   \frac{(m!)^2}{2\sqrt{2\pi}} \left(32\frac{ \Lambda^2 c^{1/2} (1+q)^{3/2}\log(|I|^{\frac1d})^{\frac12}}{\rho^{d/2}|I|^{2\eta_1+\eta_2/2}}\right)^m,
		\end{align*}
		which concludes the proof of \eqref{eq:kappa_m_2}, recalling~\eqref{eq:A_int}. The proof of \eqref{eq:kappa_m_2_sf} is done in a similar way, by considering:
		$$X = \Sigma^{\frac12} \mathcal{N} + 2\frac{(c (1+q)^2)^{1/3}}{\rho^{d/6}|I|^{\eta_1}}\Gamma,$$ 
		where $\mathcal{N}$ is a standard Gaussian random variable and $\Gamma$ is an Gamma random variable with parameters 
		$\left(\lceil c_f \rceil, 1\right)$.
	\end{proof}
	To prove Theorem~\ref{thm:conc}, we combine the previous lemmas with Theorem~\ref{thm:sc}.
	\begin{proof}[Proof of Theorem \ref{thm:conc}]
		
		According to lemma \ref{lem:kappa_m_1} and Lemma \ref{lem:kappa_m_2}, we have
		\begin{align*}
			|\kappa_m(\widehat{S}(k_0))| \leq A_1+ 	A_2 + 	A_3 \wedge A_4,
		\end{align*}
		where
		\begin{align*}
			&A_1 := 2^{m-1}(m-1)! \frac{S(k_0)^m}{|I|^{m-1}}, \quad	
			A_2 := \left(\frac{4^{m} m!}{\sqrt{2\pi}}\right)^{\gamma} \frac{m!}{2\sqrt{2\pi}} (4 \Sigma)^m ,
			\\& A_3 := \left(\frac{4^{m} m!}{\sqrt{2\pi}}\right)^{\gamma}\frac{(m!)^2}{2\sqrt{2\pi}} \left(32\frac{ \Lambda^2 c^{1/2} (1+q)^{3/2} \log(|I|^{\frac1d})^{\frac12}}{\rho^{d/2}|I|^{2\eta_1+\eta_2/2}}\right)^m,
			\\& A_4:= \left(\frac{4^{m} m!}{\sqrt{2\pi}}\right)^{\gamma}\frac{(m!)^2}{2\sqrt{2\pi}} \left(32\frac{ \lceil c_f \rceil^2 c^{2/3} (1+q)^{4/3} }{\rho^{d/3}|I|^{2\eta_1}}\right)^m.
		\end{align*}
		We have, since $m \geq 2$, 
		\begin{align*}
			A_1 \leq \frac{m!}2 \left(\frac{2 S(k_0)}{|I|}\right)^{m-2}\frac{2 S(k_0)^2}{|I|}.
		\end{align*}
		Then, after computations, we get 
		\begin{align*}
			A_2 \leq \left(\frac{m!}2\right)^{1+\gamma} \frac{16}{\sqrt{2\pi}} \left(\frac{32}{\sqrt{2\pi}}\right)^{\gamma} \Sigma^2 (4^{1+\gamma}\Sigma)^{m-2}.
		\end{align*}
		Finally,
		\begin{align*}
			A_3 \wedge A_4 \leq \left(\frac{m!}2\right)^{2+\gamma} \frac{2^{11}}{\sqrt{2\pi}} \left(\frac{32}{\sqrt{2\pi}}\right)^{\gamma} B_{\infty}^{2}  (32B_{\infty})^{m-2}.
		\end{align*}
		Theorem \ref{thm:sc} and numerical computations conclude the proof. 
	\end{proof}
	Using the results of Section \ref{sec:herm}, we specialize Theorem \ref{thm:conc} to Hermite tapers. 
	\begin{proof}[Proof of Corollary \ref{cor:conc_herm}]
		
		In the following, $K$ denotes a finite constant that may change from line to line. We have:
		\begin{align*}
			|\widehat{S}(k_0) - S(k_0)| \leq 	|\widehat{S}(k_0) - \mathbb{E}[\widehat{S}(k_0)]|+ 	|S(k_0)- \mathbb{E}[\widehat{S}(k_0)]|. 
		\end{align*}
		According to Lemma \ref{lem:bias_s_hat} and arguing as in the proof of Corollary \ref{cor:L2_risk_herm}, the bias term is bounded as:
		\begin{align*}
			|S(k_0)- \mathbb{E}[\widehat{S}(k_0)]| \leq \frac{K}{|W|^{\beta/(2\beta +d)}}.
		\end{align*}
		We consider the variance term $|\widehat{S}(k_0) - \mathbb{E}[\widehat{S}(k_0)]|$. According to Lemma \ref{lem:L4_herm}, Assumptions~\eqref{ass:cesaro_f} and \eqref{ass:cesaro_f3} are satisfied with $\rho = r$, $\eta_1 = 1/12$ and $\eta_2 = 1/6$. To apply Theorem~\ref{thm:conc}, we check that condition~\eqref{cond:rI} is satisfied. 
		\begin{align*}
			\frac{(1+q) |I|^{\eta_2}}{c \rho^{d} \log(|I|^{\frac{1}d})} \leq \frac{\sqrt{2} (1+q) |I|^{\frac23}}{c \log(2^{\frac1d})|W|} \leq \frac{\sqrt{2} (1+q) c_3^{\frac23}}{c \log(2^{\frac1d})} |W|^{\frac23 \frac{2\beta}{2\beta +d} -1} = \frac{\sqrt{2} (1+q) c_3^{\frac23}}{c \log(2^{\frac1d})} |W|^{- \frac{2 \beta +3d}{6\beta + 3d}} \leq 1,
		\end{align*}
		provided that $|W|$ is large enough, say $|W| \geq v_0$ for some $v_0 > 0$. Moreover, with similar computations, we obtain
		\begin{align*}
			\frac{1+q}{c \rho^d} \leq 	\sqrt{2} \frac{1+q}{c} |W|^{-\frac{d+\beta}{2\beta+d}} \leq 1,
		\end{align*}
		by increasing $v_0$ if necessary.
		Accordingly, one can apply Theorem \ref{thm:conc}. It remains to bound $\Sigma$ and $B_{\infty}$.With the arguments of the proof of Corollary~\ref{cor:L2_risk_herm} (see Section~\ref{sec:proof_l2_risk_end}), we obtain
		\begin{align*}
			\Sigma & \leq \frac{S(k_0)}{|I|} + K\left( \frac{L}{|I|^{\frac12}}  \left(\frac{|I|}{|W|}\right)^{\frac{\beta}d}  + \|S\|_{\infty}  e^{-c' |I|^{\frac{3\theta}{2d}}}\right) + 2\|S\|_{\infty} \sup_{i \in I} \|\psi_i \mathbf{1}_{|\cdot| \geq |k_0|}\|_2 \\& \leq K\frac{L + \|S\|_{\infty}}{|W|^{2\beta/(2\beta +d)}} + K\|S\|_{\infty} \left( \sup_{i \in I} \|\psi_i \mathbf{1}_{|\cdot| \geq |k_0|}\|_2 \wedge \frac{1}{|W|^{\beta/(2\beta +d)}}\right).
		\end{align*}
		Note that when $\beta \geq 1$, we use the fact that $|I| \geq c_d |\{i \in \mathbb{N}^d|~|i|_{\infty} < i_{\max}\}|$, for a constant $c_d > 0$ depending only on the dimension $d$. Assume that $|k_0| \geq c_3^{\frac1d} 2 |W|^{-\frac{1}{2\beta +d}}$. Then, using assumptions \eqref{cond:rd_herm1} and \eqref{cond:I_herm1}, we prove that $r|k_0| \geq \sqrt{2} |I|^{1/(2d)} \geq \sqrt{i_{\text{max}} + i_{\text{max}}^{1/3+\theta}}$. So, we can apply Lemma \ref{lem:tail_herm}, to get
		\begin{align*}
			\sup_{i \in I} \|\mathcal{F}[f_i] \mathbf{1}_{|\cdot| \geq |k_0|}\|_2& \leq K e^{-c'(r |k_0|)^{3\theta}} \leq K e^{-c'(\sqrt{2} |I|^{\frac1{2d}})^{3\theta}} \leq \frac{K}{|W|^{2\beta/(2\beta +d)}}.
		\end{align*}
		Accordingly, if $|k_0| \geq c_3^{\frac1d} 2 |W|^{-\frac{1}{2\beta +d}}$, then $
		\Sigma \leq K(L + \|S\|_{\infty})/|W|^{2\beta/(2\beta +d)}.$
		Without any assumption on $k_0$, we have $
		\Sigma  \leq K(L+\|S\|_{\infty})/|W|^{\beta/(2\beta +d)}.$
		To upper bound $B_{\infty}$, we start with the case $|k_0| \geq c_3^{\frac1d} 2 |W|^{-\frac{1}{2\beta +d}}$. Using again Lemma \ref{lem:tail_herm} and assumptions \eqref{cond:rd_herm1}, \eqref{cond:I_herm1}, we prove that
		\begin{align*}
			\sup_{i \in I} \|\psi_i \mathds{1}_{\mathbb{R}^d \setminus W}\|_3 \sup_{i \in I} \|\psi_i\|_3^2 +  \sup_{i \in I} \|\psi_i \mathds{1}_{|\cdot| \geq |k_0|/3}\|_1 \leq K e^{-K_1 |W|^{\alpha}},
		\end{align*}
		for positive constants $\alpha > 0$ and $K_1 > 0$. Subsequently, with assumptions \eqref{cond:rd_herm1} and \eqref{cond:I_herm1}, we obtain, for some $\alpha' > 0$,
		\begin{align*}
			\Lambda \leq K \left\lceil c_f \vee e^{-K_1 |W|^{\alpha}} |W|^{\alpha'}\right\rceil \leq K.
		\end{align*}
		This yields, using $2\eta_1 + \eta_2/2 = 1/4$,
		\begin{align*}
			B_{\infty} \leq K \frac{c^{1/2} (1+q)^{3/2} \log(|W|)^{\frac{d}2}}{\rho^{d/2}|I|^{1/4}} \leq K \frac{c^{1/2} (1+q)^{3/2} \log(|W|)^{\frac{d}2}}{|W|^{\frac12}}.
		\end{align*}
		Finally, without any assumption on $k_0$, we get 
		\begin{align*}
			B_{\infty} \leq K \frac{c^{\frac23}(1+q)^{\frac43}}{r^{d/3} |I|^{\frac16}} \leq  K \frac{c^{\frac23}(1+q)^{\frac43}}{|W|^{\frac13}}.
		\end{align*}
		This concludes the proof of this lemma.
	\end{proof}
	
	\section{Proofs related to Section \ref{sec:dd}}\label{sec:dd_proof}
	
	To prove Theorem~\ref{thm:dd}, we introduce the following convenient transformation of the data-driven criterion. With the notations of Section~\ref{sec:dd} and~\ref{def:k0_sep}, we consider, for all $I \subset \mathbb{N}^d$:
	\begin{align}\label{eq:def_R_modif}
		\Delta_I := \widehat{\mathcal{R}}(I) - \frac{1}{N} \sum_{j = 1}^N \left(\widehat{S}_I(k_j) - S_{\frac{1}{2}}(k_j)\right)^2 - \frac{1}{N} \sum_{j = 1}^N \widehat{S}_p(\widetilde{k_j})^2 - 2 \widehat{S}_p(\widetilde{k_j}) S_{\frac{1}{2}}(\widetilde{k_j}).
	\end{align}
	For the following proofs, it is convenient to rewrite $\Delta_I$ as:
	\begin{equation}\label{eq:def_R_modif_fact}
		\Delta_I =\frac{2}N \sum_{j = 1}^N (\widehat{S}_I(k_j) - S_{\frac12}(k_j))(S_{\frac12}(\widetilde{k_j}) - \widehat{S}_p(\widetilde{k_j})),
	\end{equation}
	provided that $S(k_j) = S(\widetilde{k_j})$ for all $j \in [N]$. Theorem~\ref{thm:dd} is a direct consequence of the following lemma.

	\begin{lemma}\label{lem:exp_delta}
		Suppose the assumptions of Theorem~\ref{thm:dd} hold. Then, 
		\begin{align*}
			\mathbb{E}\left[\max_{I \in \mathcal{I}} |\Delta_I|\right] = o\left(|\mathcal{I}|^{\frac{1}{2}} |W|^{-\frac{2\beta}{2\beta + d}}\right),  \text{ as }  |W| \to \infty.
		\end{align*}
	\end{lemma}
	
	The proof of the previous lemma is the main objective of this section. Before proceeding with it, we first show how to deduce Theorem~\ref{thm:dd}.

	\begin{proof}[Proof of Theorem~\ref{thm:dd}]
		By definition, we have, for all $I \in \mathcal{I}$,
		$
		\widehat{\mathcal{R}}(\widehat{I}) \leq \widehat{\mathcal{R}}(I).
		$
		This yields:
		\begin{align*}
			\frac{1}{N} \sum_{j = 1}^N \left(\widehat{S}_{\widehat{I}}(k_j) - S_{\frac{1}{2}}(k_j)\right)^2 &\leq \frac{1}{N} \sum_{j = 1}^N \left(\widehat{S}_I(k_j) - S_{\frac{1}{2}}(k_j)\right)^2 + \Delta_{I} - \Delta_{\widehat{I}} \\
			&\leq \frac{1}{N} \sum_{j = 1}^N \left(\widehat{S}_I(k_j) - S_{\frac{1}{2}}(k_j)\right)^2 + 2 \max_{I' \in \mathcal{I}} |\Delta_{I'}|.
		\end{align*}
		Taking expectations in the previous inequality and applying Lemma~\ref{lem:exp_delta} completes the proof.
	\end{proof}
	
	To prove Lemma~\ref{lem:exp_delta}, we upper bound both the expectation and the variance of $\Delta_I$. We introduce the next notation to highlight the dominant terms in the forthcoming bounds.
	\begin{definition}
		In this section, $\mathcal{E}(|I|, |I_p|)$ denotes a non negative exponentially small quantity such that 
		$$\mathcal{E}(|I|, |I_p|, \theta) \leq K\left(e^{-|I_p|^{3\theta/(2d)}/K} + e^{-|I|^{3\theta/(2d)}/K}\right),$$
		for a constant $K \in (0, \infty)$ and where $\theta \in (0, 2/3)$ is the parameter of equation \eqref{eq:rI_dd}. When the context is clear, we use the abbreviation $\mathcal{E} = \mathcal{E}(|I|, |I_p|, \theta).$
	\end{definition}
	
	The next four lemmas provide upper bounds on the terms appearing in $\mathbb{E}[\Delta_I]$ and $\operatorname{Var}[\Delta_I]$. The following computations rely on the two general underlying ideas, summarized as follows:
	
	\begin{enumerate}
		\item In general, cumulants of order $m \geq 4$ of linear statistics associated with a family of tapers are controlled by the product of the $L^m(\mathbb{R}^d)$-norms of the tapers (see Lemma~\ref{lem:bril_thin}). However, when one of the Hermite tapers $\psi_i$ has an index $i$ appearing only once in a cumulant, and when we sum the square of this cumulant over $i \in I$, we instead use Lemma~\ref{lem:bessel_km} to exploit the orthogonality of the tapers.
		
		\item For cumulants of order two or three, two cases are distinguished. If the cumulant involves tapers index $i \in I$ and/or $i_p  \in I_p$ with $|I|$ and $|I_p|$ small enough we exploit frequency decorrelation via Lemmas~\ref{lem:decor_freq_2} and~\ref{lem:decor_freq_3}. Otherwise, we take advantage of the orthogonality between tapers using Lemmas~\ref{lem:bessel_k2} and~\ref{lem:bessel_km}.
	\end{enumerate}
	
	\begin{lemma}\label{lem:exp_delta_I}
		Suppose that the assumptions of Theorem~\ref{thm:dd} hold. Then, 
		\begin{align*}
			\forall I \in \mathcal{I},~|\mathbb{E}[\Delta_I ]| \leq K\Bigg(\left(\frac{|I||I_p|}{|W|^{2}}\right)^{\frac{\beta}d} + \frac{\log(|W|)^d}{|W|} +\frac{1}{|W| |k_0|^d}  + \mathcal{E}\Bigg),
		\end{align*}
		where $K < \infty$ does not depend on $I$, $r$, $W$, $N$ and $|k|$.
	\end{lemma}
	\begin{proof}
		According to \eqref{eq:def_R_modif_fact},
		\begin{equation*}
			\Delta_I =\frac{2}N \sum_{j = 1}^N (\widehat{S}_I(k_j) - S_{\frac12}(k_j))(S_{\frac12}(\widetilde{k_j}) - \widehat{S}_p(\widetilde{k_j})).
		\end{equation*}
		For $s \in \{1, 2\}$, denote by $C_i^s$ the quantity of equation \eqref{eq:def_Ti_Ci} computed with $\Phi_s$. Hence,
		\begin{equation*}
			\mathbb{E}[\Delta_I] =\frac{2}N \sum_{j = 1}^N \frac1{|I||I_p|}\sum_{(i_1, i_2) \in I \times I_p} \mathbb{E}[(|C_{i_1}^1(k_j)|^2 - S_{\frac12}(k_j))(S_{\frac12}(\widetilde{k_j}) - |C_{i_2}^2(\widetilde{k_j})|^2)].
		\end{equation*}
		Using Corollary \ref{corol:cumulants_square_rvs}, we get
		\begin{equation*}
			|\mathbb{E}[\Delta_I]| \leq \frac{2}N \sum_{j = 1}^N \frac1{|I||I_p|}\sum_{(i_1, i_2) \in I \times I_p} (A + B + C + D),
		\end{equation*}
		where 
		\begin{align*}
			&A := |\mathbb{E}[|C_{i_1}^1(k_j)|^2] - S_{\frac12}(k_j)| |\mathbb{E}[|C_{i_2}^2(\widetilde{k_j})|^2] - S_{\frac12}(\widetilde{k_j})|,
			\\& B := |\mathbb{E}[C_{i_1}^1(k_j) C_{i_2}^2(\widetilde{k_j})]|^2,~ C := |\mathbb{E}[C_{i_1}^1(k_j) \overline{C_{i_2}^2(\widetilde{k_j})}]|^2,
			\\& D:= \kappa(C_{i_1}^1(k_j), \overline{C_{i_1}^1(k_j)}, C_{i_2}^2(\widetilde{k_j}), \overline{C_{i_2}^2(\widetilde{k_j})}).
		\end{align*}
		During the proof, we denote by $K $ a finite constant that may change from line to line. We start by bounding the sum concerning $A$. According to Lemmas \ref{lem:bias_s_hat}, \ref{lem:norm_H12_herm} and \ref{lem:tail_herm}, we have:
		\begin{align*}
			\frac1{|I||I_p|}&\sum_{(i_1, i_2) \in I \times I_p} A \leq  \prod_{\substack{(\rho, J) \in \\  \{(r, I), (r_p, I_p)\}}} \left(\frac{\sqrt{d} L}{2|J|\rho^{\beta}} \sum_{i \in J} \|\psi_i\|_{\dot{H}^{\frac{\beta}2}}^2 + \frac{4}{|J|}\sum_{i \in J} \|\psi_i \mathds{1}_{\mathbb{R}^d \setminus W/\rho}\|_2\right) \\&  \leq  K \prod_{\substack{(\rho, J) \in \\  \{(r, I), (r_p, I_p)\}}}\left(\left(\frac{|J|^{\frac{1}{2d}}}{\rho}\right)^{\beta} +  e^{-c' (R/\rho)^{3\theta}} \right).
		\end{align*}
		In the second line we used assumption \eqref{eq:rI_dd}. Using again this assumption and the fact that $|J|\leq |W|$, we obtain
		\begin{align}\label{eq:exp_crit_A}
			\frac1{|I||I_p|}\sum_{(i_1, i_2) \in I \times I_p} A \leq K \prod_{\substack{(\rho, J) \in \\  \{(r, I), (r_p, I_p)\}}} \left(\left(\frac{|J|}{|W|}\right)^{\frac{\beta}d} + \mathcal{E}\right)   \leq K\left( \frac{(|I||I_p|)^{\frac{\beta}d}}{|W|^{2\frac{\beta}d}} +  \mathcal{E}\right).
		\end{align}
		Now, we consider $B$ and $C$. The idea is to exploit the exponential decorrelation between linear statistics associated with tapers evaluated at distinct frequencies when the number of tapers is small. When this condition does not hold, we instead rely on the variance reduction resulting from the orthogonality of the tapers. Assume that $\rho c_2 |k_0| \geq \sqrt{2} |J|^{\frac{1}{2d}}$ for $(\rho, J) \in \{(r, I), (r_p, I_p)\}$. According to Lemma~\ref{lem:decor_freq_2}, we have
		\begin{align*}
			B^{\frac12} &\leq K \left(\sum_{s = 1}^2 \|\rho_s^{-d/2}\psi_{i_s}(\cdot/\rho_s) \mathds{1}_{\mathbb{R}^d \setminus [-R, R]^d} \|_2 + \|\rho_s^{-d/2}\mathcal{F}[\psi_{i_s}(\cdot/\rho_s)] \mathds{1}_{|\cdot| \geq |k_{j} - \widetilde{k_{j}}|/2} \|_2\right) \\& = K \left(\sum_{s = 1}^2 \|\psi_{i_s} \mathds{1}_{\mathbb{R}^d \setminus [-R/\rho_s, R/\rho_s]^d} \|_2 + \|\mathcal{F}[\psi_{i_s}] \mathds{1}_{|\cdot| \geq \rho_s|k_{j} - \widetilde{k_{j}}|/2} \|_2\right),
		\end{align*}
		where $\rho_1 = r$ and $\rho_2 = r_p$. The $k_0$-allowed assumption ensures that $|k_j - \tilde{k_j}| \geq c_2 |k_0|$. Using Lemma~\ref{lem:tail_herm}, we get, for some $c' > 0$,
		\begin{align*}
			\frac1{|I||I_p|}\sum_{\substack{i_1 \in I \\ i_2 \in I_p}} B \leq K \sum_{\substack{(\rho, J) \in \\ \{(r, I), (r_p, I_p)\}}} e^{-c' (R/\rho)^{3\theta}} + e^{-c' (c_2 \rho |k_0|)^{3\theta}}.
		\end{align*}
		Moreover, using assumption \eqref{eq:rI_dd}, we obtain
		$
		(\rho |k_0|)^2 \geq \sqrt{2} |J|^{\frac{1}{2d}} \rho |k_0| \geq \sqrt{2} |W|^{\frac1d} |k_0|.
		$
		This yields,
		\begin{align}\label{eq:exp_crit_B_exp}
			\frac1{|I||I_p|}\sum_{\substack{i_1 \in I \\ i_2 \in I_p}} B \leq K \left(\frac1{|W||k|^d} +  \mathcal{E}\right),
		\end{align}
		when $\rho c_2 |k_0| \geq \sqrt{2} |J|^{\frac{1}{2d}}$ for $(\rho, J) \in \{(r, I), (r_p, I_p)\}$
		
		Assume now that $\rho c_2 |k_0| \leq \sqrt{2} |J|^{\frac{1}{2d}}$ for $(\rho, J) = (r, I)$ or $(\rho, J ) = (r_p, I_p)$. According to Lemma~\ref{lem:bessel_k2} and Lemma \ref{lem:tail_herm}, we have
		\begin{align*}
			\frac1{|I||I_p|}\sum_{\substack{i_1 \in I \\ i_2 \in I_p}} B &\leq \frac{K}{|I| \vee |I_p|}  +  \frac{K}{|I|}\sum_{i_1 \in I} \|r^{-\frac{d}2}\psi_{i_1}\left(\frac{\cdot}r\right) \mathds{1}_{\mathbb{R}^d \setminus W}\|_2^2+ \frac{K}{|I_p|}\sum_{i_2 \in I_p} \|r_p^{-\frac{d}2}\psi_{i_2}\left(\frac{\cdot}{r_p}\right) \mathds{1}_{\mathbb{R}^d \setminus W}\|_2^2 \\& \leq K \left(\frac{1}{|I| \vee |I_p|} + \mathcal{E}\right).
		\end{align*}
		If $r |k_0| c_2 \leq \sqrt{2} |I|^{\frac{1}{2d}}$, we have, using \eqref{eq:rI_dd}
		$
		\sqrt{2} |I|^{\frac{1}{2d}} \geq c_2 |k_0| |W|^{\frac1d}/|I|^{\frac1{2d}} 
		$
		Subsequently, 
		$
		|I| \vee |I_p| \geq |I| \geq \left(c_2/\sqrt{2}\right)^d |W| |k_0|^d.
		$
		If $r_p |k_0| c_2 \leq \sqrt{2} |I_p|^{\frac{1}{2d}}$, we have the same lower bound for $|I| \vee |I_p|$. Accordingly, the bound \eqref{eq:exp_crit_B_exp} is always valid.
		
		Since we also have $|k_j + \tilde{k_j}| \geq c_2 |k_0|$, the technique for bounding $B$ can be applied to obtain
		\begin{align}\label{eq:exp_crit_C}
			\frac1{|I||I_p|}\sum_{\substack{i_1 \in I \\ i_2 \in I_p}}& C \leq  K \left( \frac{1}{|W| |k_0|^d}   +  \mathcal{E}\right).
		\end{align}
		It remain to consider $D$. According to Lemma \ref{lem:bril_thin}
		\begin{align*}
			|D| \leq K \| r^{-\frac{d}{2}}\psi_{i_1}\left(\frac{\cdot}r\right)\|_4^2 \| r_p^{-\frac{d}{2}}\psi_{i_2}\left(\frac{\cdot}{r_p}\right)\|_4^2 = K (r r_p)^{-\frac{d}{2}} \|\psi_{i_1}\|_4^2 \|\psi_{i_2}\|_4^2.
		\end{align*}
		Using Lemma \ref{lem:L4_herm} and assumption \eqref{eq:rI_dd}, we obtain
		\begin{align}\label{eq:exp_crit_D}
			\frac1{|I||I_p|}\sum_{(i_1, i_2) \in I \times I_p} |D| \leq K \frac{(\log(|I|^{\frac1d})\log(|I_p|^{\frac1d}))^{\frac{d}{2}}}{r^{\frac{d}{2}} |I|^{1/4} r_p^{\frac{d}{2}} |I_p|^{1/4}}  \leq K \frac{(\log(|W|)^d}{|W|}.
		\end{align}
		Gathering the bounds \eqref{eq:exp_crit_A}, \eqref{eq:exp_crit_B_exp}, \eqref{eq:exp_crit_C} and \eqref{eq:exp_crit_D}, we obtain the result.
	\end{proof}
	To upper bound the variance of $\Delta_I$, we proceed in three steps with Lemmas \ref{lem:dd_22}, \ref{lem:dd_13} and~\ref{lem:dd_4}. 
	\begin{lemma}\label{lem:dd_22}
		Suppose the assumptions of Theorem \ref{thm:dd} hold and denote 
		\begin{align*}
			&C_1 := \kappa(\widehat{S}_{I}(k_{j_1}), \widehat{S}_{I}(k_{j_2}))\kappa(\widehat{S}_{0}(\tilde{k_{j_1}}),  \widehat{S}_{0}(\tilde{k_{j_2}})), 
			\\& C_2 := \kappa(\widehat{S}_{I}(k_{j_1}), \widehat{S}_{0}(\tilde{k_{j_2}}))\kappa(\widehat{S}_{0}(\tilde{k_{j_1}}),  \widehat{S}_{I}(k_{j_2})).
		\end{align*}
		Then, we have, for $s \in \{1, 2\}$
		\begin{align}\label{lem:var_crit_22_1}
			\frac1{N^2} \sum_{\substack{1 \leq j_1 \leq N \\ 1 \leq j_2 \leq N}}|C_s| \leq K \left(\frac{1}{|I||I_p|} \left( \left(\frac{(|I_p| \vee |I|)^{\frac1{d}}}{|W|^{\frac1d} |k_0|}\right)^{\zeta}+\frac1N\right) +  \frac{\log(|W|)^{2d}}{|W|^2} + \mathcal{E}(|I|, |I_p|)\right),
		\end{align}
		where $K < \infty$ is a constant and $\zeta = d-1$.
	\end{lemma}
	\begin{proof}
		During the proof, $K$ denotes a constant that may change from line to line.
		We start by proving \eqref{lem:var_crit_22_1} for $s = 1$. By bi-linearity, and using Corollary \ref{corol:cumulants_square_rvs}, we obtain:
		\begin{align}\label{eq:k_SI_SI}
			\kappa(\widehat{S}_{I}(k_{j_1}), \widehat{S}_{I}(k_{j_2}))= \frac1{|I|^2} \sum_{(i_1, i_2) \in I^2}& \kappa(C^1_{i_1}(k_{j_1}), \overline{C^1_{i_1}(k_{j_1})}, C^1_{i_2}(k_{j_2}), \overline{C^1_{i_2}(k_{j_2})}) \nonumber \\& + |\kappa(C^1_{i_1}(k_{j_1}), C^1_{i_2}(k_{j_2}))|^2 + |\kappa(C^1_{i_1}(k_{j_1}), C^1_{i_2}(-k_{j_2}))|^2,
		\end{align}
		where for $s \in \{1, 2\}$, $C_i^s$ denotes the quantity of equation \eqref{eq:def_Ti_Ci} computed with $\Phi_s$. Using Lemma \ref{lem:bril_thin}, Lemma \ref{lem:L4_herm} and assumption \eqref{eq:rI_dd}, we get
		\begin{align}\label{eq:bound_kappa4}
			\frac1{|I|^2} \sum_{(i_1, i_2) \in I^2}& |\kappa(C^1_{i_1}(k_{j_1}), \overline{C^1_{i_1}(k_{j_1})}, C^1_{i_2}(k_{j_2}), \overline{C^1_{i_2}(k_{j_2})})| \leq K \frac{\log(|I|^{\frac1d})^d}{r^d |I|^{\frac12}} \leq K \frac{\log(|W|)^d}{|W|}.
		\end{align}
		To bound the other terms, we will take benefit of the decorrelation between distinct frequencies. We start with the one involving $|\kappa(C^1_{i_1}(k_{j_1}), C^1_{i_2}(k_{j_2}))|^2$. Assume that $|k_{j_1} - k_{j_2}| r \geq \sqrt{2} |I|^{\frac1{2d}}$. According to Lemma \ref{lem:decor_freq_2}, Lemma \ref{lem:tail_herm} and condition \eqref{eq:rI_dd}, we have
		\begin{align*}
			\frac1{|I|^2} \sum_{(i_1, i_2) \in I^2}|\kappa(C^1_{i_1}(k_{j_1}), C^1_{i_2}(k_{j_2}))|^2 \leq K \mathcal{E},
		\end{align*}
		On the opposite, if $|k_{j_1} - k_{j_2}| r \leq \sqrt{2} |I|^{\frac1{2d}}$, then using Lemma \ref{lem:bessel_k2} and Lemma~\ref{lem:tail_herm}, we obtain
		\begin{align*}
			\frac1{|I|^2} \sum_{(i_1, i_2) \in I^2}|\kappa(C^1_{i_1}(k_{j_1}), C^1_{i_2}(k_{j_2}))|^2 \leq K \left(\frac1{|I|} + \mathcal{E}\right).
		\end{align*}
		We obtain the same upper bound for the term of \eqref{eq:k_SI_SI} involving $|\kappa(C^1_{i_1}(k_{j_1}), C^1_{i_1}(-k_{j_2}))|^2$. Accordingly, we have
		\begin{align*}
			|\kappa(\widehat{S}_{I}(k_{j_1}), \widehat{S}_{I}(k_{j_2}))| \leq  K\left( \frac{\log(|W|)^d}{|W|} + \left(\mathds{1}_{|k_{j_1} - k_{j_2}| r \leq \sqrt{2}|I|^{\frac1{2d}}} +  \mathds{1}_{|k_{j_1} + k_{j_2}| r \leq \sqrt{2}|I|^{\frac1{2d}}}\right) \frac1{|I|} + \mathcal{E}\right).
		\end{align*}
		The same line of reasoning also applies to bound $\kappa(\widehat{S}_{0}(\widetilde{k_{j_1}}),  \widehat{S}_{0}(\widetilde{k_{j_2}}))$, yielding 
		\begin{align*}
			|\kappa(\widehat{S}_{0}(\widetilde{k_{j_1}}),  \widehat{S}_{0}(\widetilde{k_{j_2}}))| \leq K\left(\frac{\log(|W|)^d}{|W|} + \left(\mathds{1}_{|\widetilde{k_{j_1}} - \widetilde{k_{j_2}}| r_p \leq \sqrt{2}|I_p|^{\frac1{2d}}} +  \mathds{1}_{|\widetilde{k_{j_1}} + \widetilde{k_{j_2}}| r_p \leq \sqrt{2}|I_p|^{\frac1{2d}}}\right) \frac1{|I_p|} + \mathcal{E}\right).
		\end{align*}
		Subsequently, we get
		\begin{align*}
			\frac{1}{N^2} \sum_{1 \leq j_1, j_2 \leq N} |\kappa(\widehat{S}_{I}(k_{j_1}), \widehat{S}_{I}(k_{j_2}))\kappa(\widehat{S}_{0}(\tilde{k_{j_1}}),  \widehat{S}_{0}(\tilde{k_{j_2}}))| \leq K \left( \frac{\log(|W|)^{2d}}{|W|^2}+  \delta + \mathcal{E}\right),
		\end{align*}
		with 
		\begin{align*}
			&\delta =   \frac{\log(|W|)^d}{|W|} \frac1{|I| N^2}\sum_{1 \leq j_1, j_2 \leq N} \left(\mathds{1}_{|k_{j_1} - k_{j_2}| r \leq \sqrt{2}|I|^{\frac1{2d}}} +  \mathds{1}_{|k_{j_1} + k_{j_2}| r \leq \sqrt{2}|I|^{\frac1{2d}}}\right) \\& \quad  +  \frac{\log(|W|)^d}{|W|} \frac1{|I_p|N^2}\sum_{1 \leq j_1, j_2 \leq N}\left(\mathds{1}_{|\widetilde{k_{j_1}} - \widetilde{k_{j_2}}| r_p \leq \sqrt{2}|I_p|^{\frac1{2d}}} +  \mathds{1}_{|\widetilde{k_{j_1}} + \widetilde{k_{j_2}}| r_p \leq \sqrt{2}|I_p|^{\frac1{2d}}}\right) \\& \quad \qquad + \frac1{|I||I_p|}\frac1{N^2}\sum_{1 \leq j_1, j_2 \leq N}\left(\mathds{1}_{|k_{j_1} - k_{j_2}| \leq \sqrt{2}\frac{|I|^{\frac1{2d}}}{r} \wedge \frac{|I_p|^{\frac1{2d}}}{r_p}} +  \mathds{1}_{|k_{j_1} + k_{j_2}| r_p \leq\sqrt{2}\frac{|I|^{\frac1{2d}}}{r} \wedge \frac{|I_p|^{\frac1{2d}}}{r_p}}\right) \\& =: \delta_1 + \delta_2 + \delta_3.
		\end{align*}
		With the $k_0$-allowed assumption, we obtain
		\begin{align*}
			\delta_1 &\leq \frac{\log(|W|)^d}{|W|} \frac{c_3}{N^2|I|} \left( N^2 \left(\frac{|I|^{\frac1{2d}}}{r |k_0|}\right)^{\zeta} + N\right) \leq K  \frac{\log(|W|)^d}{|W||I|} \left(\left(\frac{|I|^{\frac1{d}}}{|W|^{\frac1d} |k_0|}\right)^{\zeta}+\frac1{N}\right).
		\end{align*}
		Similar computations yield
		\begin{align*}
			&\delta_2\leq K \frac{\log(|W|)^d}{|W||I_p|}\left(\left(\frac{|I_p|^{\frac1{d}}}{|W|^{\frac1d} |k_0|}\right)^{\zeta} + \frac1N\right),
			~ \delta_3 \leq K \frac1{|I||I_p|} \left( \left(\frac{(|I_p| \wedge |I|)^{\frac1{d}}}{|W|^{\frac1d} |k_0|}\right)^{\zeta}+\frac1N\right).
		\end{align*}
		Using the fact that $|I| \vee |I_p| \leq c_5 |W|^{2\beta/(2\beta +d)}$, for upper bounding $\delta_1$ and $\delta_2$, we obtain
		\begin{align*}
			\delta \leq \frac{K}{|I||I_p|} \left( \left(\frac{(|I_p| \vee |I|)^{\frac1{d}}}{|W|^{\frac1d} |k_0|}\right)^{\zeta}+\frac1N\right).
		\end{align*}
		This concludes the proof of \eqref{lem:var_crit_22_1} for $s = 1$. The proof of \eqref{lem:var_crit_22_1} for $s = 2$ follows the same lines, but the details differs. By bi-linearity, and using Corollary \ref{corol:cumulants_square_rvs}, we obtain:
		\begin{align}\label{eq:k_SI_S0}
			\kappa(\widehat{S}_{I}(k_{j_1}), \widehat{S}_{0}(\widetilde{k_{j_2}}))= \frac1{|I| |I_p|} \sum_{\substack{i_1 \in I\\ i_2 \in I_p}}& \kappa(C^1_{i_1}(k_{j_1}), \overline{C^1_{i_1}(k_{j_1})}, C^2_{i_2}(\widetilde{k_{j_2}}), \overline{C^2_{i_2}(\widetilde{k_{j_2}})}) \nonumber \\& + |\kappa(C^1_{i_1}(k_{j_1}), C^2_{i_2}(\widetilde{k_{j_2}}))|^2 + |\kappa(C^1_{i_1}(k_{j_1}), C^2_{i_2}(-\widetilde{k_{j_2}}))|^2.
		\end{align}
		Using Lemma \ref{lem:bril_thin}, Lemma \ref{lem:L4_herm} and assumption \eqref{eq:rI_dd}, we obtain, as in \eqref{eq:bound_kappa4},
		\begin{align*}
			\frac1{|I| |I_p|} \sum_{\substack{i_1 \in I\\ i_2 \in I_p}}& |\kappa(C^1_{i_1}(k_{j_1}), \overline{C^1_{i_1}(k_{j_1})}, C^2_{i_2}(\widetilde{k_{j_2}}), \overline{C^2_{i_2}(\widetilde{k_{j_2}})})| \leq K \frac{\log(|W|)^{d}}{|W|}.
		\end{align*}
		Assume that $|k_{j_1} - \widetilde{k_{j_2}}| \geq \sqrt{2} (|I|^{\frac1{2d}}/r \vee |I_p|^{\frac1{2d}}/r_p) $. Using Lemma \ref{lem:decor_freq_2} and Lemma~\ref{lem:tail_herm} we have
		\begin{align}\label{eq:decor_C1_C2_1}
			\frac1{|I| |I_p|} \sum_{\substack{i_1 \in I\\ i_2 \in I_p}} |\kappa(C^1_{i_1}(k_{j_1}), C^2_{i_2}(\widetilde{k_{j_2}}))|^2~ \leq K \mathcal{E}.
		\end{align}
		Suppose that $|k_{j_1} - \widetilde{k_{j_2}}| \leq \sqrt{2} (|I|^{\frac1{2d}}/r \vee |I_p|^{\frac1{2d}}/r_p)$. Using Lemma \ref{lem:bessel_k2} and Lemma~\ref{lem:tail_herm}, we obtain
		\begin{align}\label{eq:decor_C1_C2_2}
			\frac1{|I| |I_p|} \sum_{\substack{i_1 \in I\\ i_2 \in I_p}} |\kappa(C^1_{i_1}(k_{j_1}), C^2_{i_2}(\widetilde{k_{j_2}}))|^2 \leq K \left(\frac1{|I| \vee |I_p|} + \mathcal{E}\right).
		\end{align}
		We obtain a similar bound for the term involving $|\kappa(C^1_{i_1}(k_{j_1}), C^2_{i_2}(-\widetilde{k_{j_2}}))|^2$, yielding
		\begin{multline*}
			|\kappa(\widehat{S}_{I}(k_{j_1}), \widehat{S}_{0}(\widetilde{k_{j_2}}))|\leq  K \left(\frac{\log(|W|)^{d}}{|W|} +\mathcal{E} \right) \\+  K \left(\mathds{1}_{|k_{j_1} - \widetilde{k_{j_2}}| \leq \sqrt{2} \frac{|I|^{\frac1{2d}}}r \vee \frac{|I_p|^{\frac1{2d}}}{r_p}} +  \mathds{1}_{|k_{j_1} + \widetilde{k_{j_2}}| r \leq \sqrt{2}\frac{|I|^{\frac1{2d}}}r \vee \frac{|I_p|^{\frac1{2d}}}{r_p}}\right) \frac1{|I| \vee |I_p|}.
		\end{multline*}
		By exchanging $j_1$ and $j_2$, we have
		\begin{multline*}
			|\kappa(\widehat{S}_{0}(\widetilde{k_{j_1}}), \widehat{S}_{I}(k_{j_2}))|\leq  K \left(\frac{\log(|W|)^{d}}{|W|} + \mathcal{E} \right) \\+  K \left(\mathds{1}_{|\widetilde{k_{j_1}} - k_{j_2}| \leq \sqrt{2} \frac{|I|^{\frac1{2d}}}r \vee \frac{|I_p|^{\frac1{2d}}}{r_p}} +  \mathds{1}_{|\widetilde{k_{j_1}} + k_{j_2}| r \leq \sqrt{2}\frac{|I|^{\frac1{2d}}}r \vee \frac{|I_p|^{\frac1{2d}}}{r_p}}\right) \frac1{|I| \vee |I_p|}.
		\end{multline*}
		The previous inequalities gives
		\begin{align*}
			\frac{1}{N^2} \sum_{1 \leq j_1, j_2 \leq N} |\kappa(\widehat{S}_{I}(k_{j_1}), \widehat{S}_{0}(\widetilde{k_{j_2}}))\kappa(\widehat{S}_{0}(\widetilde{k_{j_1}}), \widehat{S}_{I}(k_{j_2}))|   &\leq  K \left(\left( \frac{\log(|W|)^{d}}{|W|}+  \mathcal{E}\right)^2 +\delta\right) \\& \leq  K \left( \frac{\log(|W|)^{2d}}{|W|^2}+  \mathcal{E} +\delta\right),
		\end{align*}
		where 
		\begin{multline*}
			\delta \leq \frac{\log(|W|)^{d}}{|W|N^2 |I| \vee |I_p|} \Bigg(\sum_{1 \leq j_1, j_2 \leq N} \mathds{1}_{|\widetilde{k_{j_1}} - k_{j_2}| \leq \sqrt{2} \frac{|I|^{\frac1{2d}}}r \vee \frac{|I_p|^{\frac1{2d}}}{r_p}} +  \mathds{1}_{|\widetilde{k_{j_1}} + k_{j_2}| r \leq \sqrt{2}\frac{|I|^{\frac1{2d}}}r \vee \frac{|I_p|^{\frac1{2d}}}{r_p}} \\\qquad \qquad\qquad \qquad \qquad+ \mathds{1}_{|\widetilde{k_{j_2}} - k_{j_1}| \leq \sqrt{2} \frac{|I|^{\frac1{2d}}}r \vee \frac{|I_p|^{\frac1{2d}}}{r_p}} +  \mathds{1}_{|\widetilde{k_{j_2}} + k_{j_1}| r \leq \sqrt{2}\frac{|I|^{\frac1{2d}}}r \vee \frac{|I_p|^{\frac1{2d}}}{r_p}} \Bigg) \\ \qquad + \frac1{N^2 (|I| \vee |I_p|)^2}\sum_{1 \leq j_1, j_2 \leq N} \mathds{1}_{|\widetilde{k_{j_1}} - k_{j_2}| \leq \sqrt{2} \frac{|I|^{\frac1{2d}}}r \vee \frac{|I_p|^{\frac1{2d}}}{r_p}} +  \mathds{1}_{|\widetilde{k_{j_1}} + k_{j_2}| r \leq \sqrt{2}\frac{|I|^{\frac1{2d}}}r \vee \frac{|I_p|^{\frac1{2d}}}{r_p}}.
		\end{multline*}
		Using the $k_0$-allowed assumption, condition \eqref{eq:rI_dd} and $|I| \leq c_5 |W|^{2\beta/(2\beta +d)}$, we obtain
		\begin{align*}
			|\delta| &\leq  K \left( \frac{\log(|W|)^{d}}{|W||I| \vee |I_p|}+  \frac{1}{(|I| \vee |I_p|)^2} \right) \left(\left(\frac{(|I| \vee |I_p|)^{\frac1d}}{|W|^{\frac1d} |k_0|}\right)^{\zeta} + \frac1N\right) \\& \leq \frac{K}{|I||I_p|} \left( \left(\frac{(|I_p| \vee |I|)^{\frac1{d}}}{|W|^{\frac1d} |k_0|}\right)^{\zeta}+\frac1N\right).
		\end{align*}
		This ends the proof for the case $s = 2$.
	\end{proof}
	
	\begin{lemma}\label{lem:dd_13}
		Suppose the assumptions of Theorem \ref{thm:dd} hold. Denote 
		\begin{align*}
			&B_1 := \mathbb{E}[\widehat{S}_{I}(k_{j_1}) - S_{\frac12}(k_{j_1})]\kappa(\widehat{S}_{0}(\tilde{k_{j_1}}), \widehat{S}_{I}(k_{j_2}), \widehat{S}_{0}(\tilde{k_{j_2}})),
			\\& B_2 :=   \mathbb{E}[\widehat{S}_{0}(\tilde{k_{j_1}}) - S_{\frac12}(\tilde{k_{j_1}})]\kappa(\widehat{S}_{I}(k_{j_1}), \widehat{S}_{I}(k_{j_2}), \widehat{S}_{0}(\tilde{k_{j_2}})),
			\\& B_3 := \mathbb{E}[\widehat{S}_{I}(k_{j_2}) - S_{\frac12}(k_{j_2})]\kappa(\widehat{S}_{I}(k_{j_1}), \widehat{S}_{0}(\tilde{k_{j_1}}), \widehat{S}_{0}(\tilde{k_{j_2}})),
			\\& B_4 = \mathbb{E}[\widehat{S}_{0}(\tilde{k_{j_2}}) - S_{\frac12}(\tilde{k_{j_2}})]\kappa(\widehat{S}_{I}(k_{j_1}), \widehat{S}_{0}(\tilde{k_{j_1}}), \widehat{S}_{I}(k_{j_2})).
		\end{align*}
		Then, for all $s \in \{1, 4\}$,
		\begin{align}\label{lem:var_crit_3_1}
			\frac1{N^2} \sum_{1 \leq j_1, j_2 \leq N} |B_s|\leq  K\left(b(I, I_p) + b(I_p, I) + b+ \mathcal{E}(|I|, |I_p|)\right),
		\end{align}
		for a constant $K \in (0, \infty)$ and where, for $\zeta = d-1$,
		\begin{align*}
			&b(J_1, J_2) :=  \left(\frac{|J_1|}{|W|}\right)^{\frac{\beta}d}\left(\frac1{|J_1| |J_2|^{\frac12}} \left(\left(\frac{ |J_2|^{\frac1d}}{|W|^{\frac1d}|k_0|}\right)^{\zeta}+\frac1N\right) +\frac{\log(|W|)^d}{|W| |J_1|}\right), 
			\\& b:= \frac{\log(|W|)^d}{|W|^2 |k_0|^d} + \frac{\log(|W|)^d}{|W|^{\frac{5}3}} +  \frac1{|k_0|^{2d} |W|^2}.
		\end{align*}
	\end{lemma}
	\begin{proof}
		We only need to prove the upper bound for $|B_1|$ and $|B_2|$, since by exchanging $j_1$ and $j_2$ we have
		\begin{align*}
			\sum_{1 \leq j_1, j_2 \leq N} |B_3| = \sum_{1 \leq j_1, j_2 \leq N} |B_1|, ~\sum_{1 \leq j_1, j_2 \leq N} |B_4| = \sum_{1 \leq j_1, j_2 \leq N} |B_2|.
		\end{align*}
		We start with $B_1$. Using Lemmas \ref{lem:bias_s_hat} and \ref{lem:norm_H12_herm} together with assumption \eqref{eq:rI_dd} we have 
		\begin{align*}
			|\mathbb{E}[\widehat{S}_{I}(k_{j_1}) - S_{\frac12}(k_{j_1})]| \leq K \frac{|I|^{\frac{\beta}{2d}}}{r^{\beta}} \leq K \left(\frac{|I|}{|W|}\right)^{\frac{\beta}d}.
		\end{align*}
		For $s \in \{1, 2\}$, we denote by $C_i^s$ the quantity of equation \eqref{eq:def_Ti_Ci} computed with $\Phi_s$. By multi-linearity, we have
		\begin{align*}
			\kappa(\widehat{S}_{0}&(\tilde{k_{j_1}}), \widehat{S}_{I}(k_{j_2}), \widehat{S}_{0}(\tilde{k_{j_2}})) \\& = \frac1{|I|| |I_p|^2} \sum_{i_2\in I, (i_1, i_3) \in I_p} \kappa(C^2_{i_1}(\widetilde{k_{j_1}}) \overline{C^2_{i_1}(\widetilde{k_{j_1}})}, C^1_{i_2}(k_{j_2}) \overline{C^1_{i_2}(k_{j_2})}, C^2_{i_3}(\widetilde{k_{j_2}}) \overline{C^2_{i_3}(\widetilde{k_{j_2}})}).
		\end{align*}
		We denote
		\begin{align*}
			&A_1 := C^2_{i_1}(\widetilde{k_{j_1}}),~A_2:= \overline{C^2_{i_1}(\widetilde{k_{j_1}})},~ A_3:= C^1_{i_2}(k_{j_2}),~ A_4 := \overline{C^1_{i_2}(k_{j_2})}, ~A_5 := C^2_{i_3}(\widetilde{k_{j_2}}), A_6 := \overline{C^2_{i_3}(\widetilde{k_{j_2}})}.
		\end{align*}
		Using Corollary \ref{corol:cumulants_square_rvs}, we expand $\kappa(A_1 A_2, A_3 A_4, A_5 A_6)$:
		\begin{align*}
			\kappa(A_1 A_2, A_3 A_4, A_5 A_6) = \sum_{\substack{\sigma \in \Pi[6], \\ \sigma \vee \tau_6 = 1_{6}}} \prod_{b \in \sigma} \kappa(A_s; s \in b),
		\end{align*}
		where $\tau_{6} = 12|34|56$ and $1_6 = 123456$. Since the random variables considered are centered, the sum runs only over the partitions having blocks of size at least $2$. In the following, we use the notation 
		\begin{align*}
			\kappa_{\pi} = \prod_{b \in \pi} \kappa(A_s; s \in b).
		\end{align*}
		In the sum over $\sigma$, 5 categories of partitions appear, denoted by $\textbf{I}$, $\textbf{II}$, $\textbf{III}$, $\textbf{IV}$ and $\textbf{V}$. To fix ideas, the partitions involved are as follows (refer to \cite{mccullagh2018tensor}, Table 1):
		\[
		\begin{aligned}
			\textbf{I:} \quad & 123456 \\[0.5em]
			\textbf{II:} \quad & 13|2456,\quad 14|2456,\quad 15|2346,\quad 16|2345, \\
			& 23|1456,\quad 24|1356,\quad 25|1346,\quad 26|1345, \\
			& 35|1246,\quad 36|1245,\quad 45|1236,\quad 46|1235 \\[0.5em]
			\textbf{III:} \quad & 123|456,\quad 124|356,\quad 125|346,\quad 126|345, \\
			& 234|156,\quad 134|256 \\[0.5em]
			\textbf{IV:} \quad & 135|246,\quad 235|146,\quad 145|236,\quad 136|245 \\[0.5em]
			\textbf{V:} \quad & 13|25|46,\quad 13|26|45,\quad 14|25|36,\quad 14|26|35, \\
			& 15|23|46,\quad 15|24|36,\quad 16|23|45,\quad 16|24|35
		\end{aligned}
		\]
		We bound the cumulant of the partition of category $\textbf{I}$. According to Lemma \ref{lem:bril_thin}, 
		\begin{align*}
			|\kappa_{123456}| &\leq K \|r_p^{-\frac{d}2}\psi_{i_1}\left(\frac{\cdot}{r_p}\right)\|_6^2 \|r^{-\frac{d}2}\psi_{i_2}\left(\frac{\cdot}r\right)\|_6^2 \|r_p^{-\frac{d}2}\psi_{i_3}\left(\frac{\cdot}{r_p}\right)\|_6^2 = K r_p^{-\frac{4d}3} r^{-\frac{2d}3} \|\psi_{i_1}\|_6^2 \|\psi_{i_2}\|_6^2 \|\psi_{i_3}\|_6^2. 
		\end{align*}
		Then, using Lemma \ref{lem:L4_herm}, we get
		\begin{align*}
			\frac1{|I|| |I_p|^2} \sum_{\substack{i_2\in I \\ (i_1, i_3) \in I_p}} |\kappa_{123456}| &\leq K \frac{\log(|W|)^d}{r_p^{\frac{4d}3} |I_p|^{\frac{4}{9}}r^{\frac{2d}3} |I|^{\frac{2}{9}}} = K \frac{\log(|W|)^d|I_p|^{\frac29}|I|^{\frac19}}{r_p^{\frac{4d}3} |I_p|^{\frac{2}{3}}r^{\frac{2d}3} |I|^{\frac{1}{3}}} \\& \leq K \frac{\log(|W|^d)|W|^{1/3}}{|W|^2} = K \frac{\log(|W|)^d}{|W|^{\frac53}},
		\end{align*}
		where in the last line we used assumption \eqref{eq:rI_dd}.
		
		For the partitions of class $\textbf{II}$, two types of terms appear, depending on either the block of size $2$ contains a taper of index $i_2 \in I$ or not. 
		In the first case, the corresponding partitions are given by 
		\begin{align*}
			\textbf{II}_a := \{13|2456,~14|2456, ~23|1456,~24|1356, 35|1246,~36|1245,~45|1236,~46|1235\}.
		\end{align*}
		We bound the term $35|1246$ and the computations are similar for the others. By Lemma \ref{lem:bril_thin}, we have
		\begin{align}\label{eq:bound_kappa_4_cst}
			|\kappa_{1246}| \leq K \|r_p^{-\frac{d}2} \psi_{i_1}\left(\frac{\cdot}{r_p}\right)\|_4^2 \|r^{-\frac{d}2} \psi_{i_2}\left(\frac{\cdot}r\right)\|_4 \|r_p^{-\frac{d}2} \psi_{i_3}\left(\frac{\cdot}{r_p}\right)\|_4.
		\end{align} 
		By Cauchy-Schwarz inequality, we have
		\begin{align*}
			&\frac1{|I|| |I_p|^2} \sum_{\substack{i_2\in I \\ (i_1, i_3) \in I_p}} |\kappa_{35|1246}| \\& \leq \frac1{|I_p| r_p^{\frac{d}2}} \sum_{i_1 \in I_p} \|\psi_{i_1}\|_4^2 \sqrt{\frac1{|I| |I_p| } \sum_{\substack{i_2\in I \\ i_3 \in I_p}} \frac{\| \psi_{i_2}\|_4^2 \| \psi_{i_3}\|_4^2}{(r r_p)^{\frac{d}2} }} \sqrt{\frac1{|I| |I_p|} \sum_{\substack{i_2\in I \\ i_3 \in I_p}} |\kappa(C^1_{i_2}(k_{j_2})), C^2_{i_3}(\widetilde{k_{j_2}})|^2 } \\& \leq \frac{K \log(|W|)^d}{|W|} \sqrt{\frac1{|I| |I_p|} \sum_{\substack{i_2\in I \\ i_3 \in I_p}} |\kappa(C^1_{i_2}(k_{j_2})), C^2_{i_3}(\widetilde{k_{j_2}})|^2 },
		\end{align*} 
		where in the last line, we used Lemma \ref{lem:L4_herm} and assumption  \eqref{eq:rI_dd}. Using the fact that the family  $(r^{-\frac{d}2} \psi_i(\cdot/r) e^{-\bm{i} k_{j_2} \cdot})_{i \in \mathbb{N}^d}$ is orthogonal, we have by Lemma \ref{lem:cov_thin}
		\begin{align*}
			\frac1{|I| |I_p|} \sum_{\substack{i_2\in I \\ i_3 \in I_p}} |\kappa(C^1_{i_2}(k_{j_2})), C^2_{i_3}(\widetilde{k_{j_2}})|^2 \leq  K \left(\frac{1}{|I|} + \mathcal{E}\right). 
		\end{align*}
		Accordingly, 
		\begin{align*}
			\frac1{|I|| |I_p|^2} \sum_{\substack{i_2\in I \\ (i_1, i_3) \in I_p}} |\kappa_{35|1246}| \leq K \left(\frac{\log(|W|)^d}{|W||I|} + \mathcal{E}\right).
		\end{align*} 
		By performing identical computations for the other partitions of category $\textbf{II}_a$, we obtain
		\begin{align*}
			\sum_{\sigma \in \textbf{II}_a} |\kappa_{\sigma}| \leq K \left(\frac{\log(|W|)^d}{|W||I|} + \mathcal{E}\right).
		\end{align*}
		We now consider the case where the block of size $2$ does not contain tapers of index $i_2$, i.e. the partitions
		\begin{align*}
			\textbf{II}_b := \{15|2346,~16|2345~25|1346,~26|1345\}.
		\end{align*}
		We detail the bound of the term $15|2346$. Using Lemma \ref{lem:bril_thin}, Cauchy-Schwarz inequality and Lemma~\ref{lem:L4_herm}, we have 
		\begin{align*}
			&2_b := \frac1{|I|| |I_p|^2} \sum_{\substack{i_2\in I \\ (i_1, i_3) \in I_p}} |\kappa_{15|2346}| \\& \leq K \frac1{|I| r^{\frac{d}2}} \sum_{i_2 \in I} \| \psi_{i_2}\|_4^2\sqrt{\frac1{|I_p|^2} \sum_{\substack{i_1 \in I_p \\ i_3 \in I_p}} \frac{\|\psi_{i_2}\|_4^2 \| \psi_{i_3}\|_4^2}{r_p^{d} }} \sqrt{\frac1{|I_p|^2}\sum_{\substack{i_1 \in I_p \\ i_3 \in I_p}} | \kappa(C^2_{i_1}(\widetilde{k_{j_1}}), C^2_{i_3}(\widetilde{k_{j_2}}))|^2} \\& \leq \frac{K \log(|W|)^d}{|W|}  \sqrt{\frac1{|I_p|^2}\sum_{\substack{i_1 \in I_p \\ i_3 \in I_p}} | \kappa(C^2_{i_1}(\widetilde{k_{j_1}}), C^2_{i_3}(\widetilde{k_{j_2}}))|^2}.
		\end{align*}
		Using \eqref{eq:decor_C1_C2_1} and \eqref{eq:decor_C1_C2_2}, we get
		\begin{align*}
			\frac1{|I_p|^2}\sum_{\substack{i_1 \in I_p \\ i_3 \in I_p}} | \kappa(C^2_{i_1}(\widetilde{k_{j_1}}), C^2_{i_3}(\widetilde{k_{j_2}}))|^2 \leq  \frac{K}{|I_p|} \mathds{1}_{|\widetilde{k_{j_1}} - \widetilde{k_{j_2}}| r_p\leq \sqrt{2}  |I_p|^{\frac1{2d}}}  + K \mathcal{E}.
		\end{align*}
		This yields, using the $k_0$-allowed assumption:
		\begin{align*}
			\frac1{N^2} \sum_{1 \leq j_1, j_2 \leq N}|2_b| &\leq   \frac{K \log(|W|)^d}{|W| |I_p|^{\frac12}} \frac{N}{N^2} \left(N \left(\frac{ |I_p|^{\frac1d}}{|W|^{\frac1d}|k_0|}\right)^{\zeta}+1\right)+ K \mathcal{E} \\& \leq K \left(\frac1{|I| |I_p|^{\frac12}} \left(\left(\frac{ |I_p|^{\frac1d}}{|W|^{\frac1d}|k_0|}\right)^{\zeta}+\frac1N\right) \right).
		\end{align*}
		Now, we consider the partitions of categories $\textbf{III}$ and $\textbf{IV}$ such that $i_2$ appears once in each blocks of size $3$, i.e. 
		\begin{align*}
			\{123|456,~ 124|356,~ 126|345,~ 135|246,~ 235|146,~ 145|236,~ 136|245\}.
		\end{align*}
		We detail the computations for the partition $123|456$, that can be directly used for bounding the other terms. Using the Cauchy-Schwarz inequality, we have
		\begin{multline*}
			\sum_{\substack{i_2 \in I \\ (i_1, i_3) \in I_p}} |\kappa_{123|456}| \leq \\	\sum_{(i_1, i_3) \in I_p} \sqrt{\sum_{i_2\in I} \left| \kappa(C^2_{i_1}(\widetilde{k_{j_1}}), \overline{C^2_{i_1}(\widetilde{k_{j_1}})}, C^1_{i_2}(k_{j_2}))\right|^2} \sqrt{\sum_{i_2\in I} \left| \kappa(\overline{C^1_{i_2}(k_{j_2})}, C^2_{i_3}(\widetilde{k_{j_2}}), \overline{C^2_{i_3}(\widetilde{k_{j_2}})})\right|^2}. 
		\end{multline*} 
		Since $(r^{-\frac{d}2} \psi_i(\cdot/r) e^{-\bm{i} k_{j_2} \cdot})_{i \in \mathbb{N}^d}$ is orthogonal, we can apply Lemma \ref{lem:bessel_km} to get 
		\begin{align*}
			\sum_{i_2\in I} \left| \kappa(C^2_{i_1}(\widetilde{k_{j_1}}), \overline{C^2_{i_1}(\widetilde{k_{j_1}})}, C^1_{i_2}(k_{j_2}))\right|^2 \leq K \|r_p^{-\frac{d}2} \psi_{i_2}\left(\frac{\cdot}{r_p}\right)\|_4^4 + K \delta,
		\end{align*}
		with 
		\begin{align*}
			\delta =  |I| \|r_p^{-\frac{d}2} \psi_{i_2}\left(\frac{\cdot}{r_p}\right) \mathds{1}_{\mathbb{R}^d \setminus W}\|_3+ \sum_{i \in I} \|r^{-\frac{d}2} \psi_i\left(\frac{\cdot}r\right) \mathds{1}_{\mathbb{R}^d \setminus W}\|_3.
		\end{align*}
		Using Lemma \ref{lem:tail_herm} and assumption \eqref{eq:rI_dd}, we get, after computations,
		$
		|\delta| \leq K |I| \mathcal{E}
		$.
		Subsequently, 
		\begin{align*}
			\sum_{\substack{i_2 \in I \\ (i_1, i_3) \in I_p}} |\kappa_{123|456}| \leq K\sum_{(i_1, i_3) \in I_p} \left( \frac{\|\psi_{i_1}\|_4^2 \|\psi_{i_3}\|_4^2}{r_p^{d}} +  \mathcal{E} \left(\frac{\|\psi_{i_1}\|_4^2 }{r_p^{\frac{d}2}}+ \frac{\|\psi_{i_3}\|_4^2 }{r_p^{\frac{d}2}} +  \mathcal{E} \right)\right).
		\end{align*} 
		Using Lemma \ref{lem:L4_herm}, we obtain
		\begin{align*}
			\frac1{|I||I_p|^2}\sum_{\substack{i_2 \in I \\ (i_1, i_3) \in I_p}} |\kappa_{123|456}| \leq K \left( \frac{\log(|W|)^d}{|W| |I|} +   \mathcal{E} \right).
		\end{align*} 
		For the partitions of category $\textbf{III}$ such that $i_2 \in I$ appears in only one block of size $3$, i.e.
		\begin{align*}
			\{125|346,~ 126|345,~ 234|156,~ 134|256 \},
		\end{align*}
		we use the decorrelation between frequencies in the cumulants of size $3$ where only the indexes belonging to $I_p$ are present. We detail the computations on $125|346$ and they can be used to bound the other terms. 
		Assume that $|I_p| \leq \sqrt{2} |W| |k_0|^d$. According to Lemma \ref{lem:decor_freq_3}, we have
		
		\begin{align*}
			|\kappa_{125}| &\leq K \sup_{i \in I_p} \left( \|r_p^{-\frac{d}2} \psi_i\left(\frac{\cdot}{r_p}\right) \mathds{1}_{\mathbb{R}^d\setminus W}\|_3 r_p^{-\frac{d}6}+ \|r_p^{-\frac{d}2} \mathcal{F}[\psi_i\left(\frac{\cdot}{r_p}\right)] \mathds{1}_{|\cdot| \geq |\widetilde{k_{j_2}}|/3}\|_1 \right) \\& = K \sup_{i \in I_p} \left( \| \psi_i \mathds{1}_{\mathbb{R}^d\setminus \frac{W}{r_p}}\|_3 r_p^{-\frac{d}3}+ \|\psi_i \mathds{1}_{|\cdot| \geq c_1 r_p|k_0|/3}\|_1 \right),
		\end{align*}
		where we used the fact that $\psi_i$ is bounded by $1$ and of unit $L^2(\mathbb{R}^d)$ norm in the first line. Since $|I_p| \leq \sqrt{2} |W| |k_0|^d$, we can apply Lemma \ref{lem:tail_herm} to get
		\begin{align*}
			|\kappa_{125}| \leq K\left(e^{-c_6|k_0| |W|^{\frac1d}} + \mathcal{E}\right),
		\end{align*}
		for some constant $c_6 > 0.$ Using Lemma \ref{lem:decor_freq_3} again, we prove that $	|\kappa(A_3, A_4, A_6)|$ is bounded by a constant. Hence, 
		\begin{align*}
			\frac1{|I||I_p|^2}\sum_{\substack{i_2 \in I \\ (i_1, i_3) \in I_p}} |\kappa_{125|346}| \leq K\left(e^{-c_6|k_0| |W|^{\frac1d}} + \mathcal{E}\right) \leq  K\left(\frac1{|k_0|^{2d} |W|^2} + \mathcal{E}\right).
		\end{align*}
		On the opposite, assume that  $|I_p| > \sqrt{2} |W| |k_0|^d$. The Cauchy-Schwarz inequality yields
		\begin{align*}
			&\sum_{\substack{i_2 \in I \\ (i_1, i_3) \in I_p}} |\kappa_{125|346}| \\& \leq \sum_{\substack{i_2 \in I \\ i_1 \in I_p}} \sqrt{\sum_{i_3 \in I_p} |\kappa(C^1_{i_1}(\widetilde{k_{j_1}}), \overline{C^2_{i_1}(\widetilde{k_{j_1}})}, C^2_{i_3}(\widetilde{k_{j_2}}))|^2 } \sqrt{\sum_{i_3 \in I_p} |\kappa(C^1_{i_1}(k_{j_2}), \overline{C^1_{i_1}(k_{j_2})}, \overline{C^2_{i_3}(\widetilde{k_{j_2}})})|^2 }.
		\end{align*}
		Using Lemma \ref{lem:bessel_km} and Lemma \ref{lem:norm_H12_herm}, we obtain after computations
		\begin{align*}
			\frac1{|I| |I_p|^2}	\sum_{\substack{i_2 \in I \\ (i_1, i_3) \in I_p}} |\kappa_{125|346}| &\leq K\left(\frac{\log(|W|)^d}{|W| |I_p|} + \mathcal{E}\right) \leq  K\left(\frac{\log(|W|)^d}{|W|^2 |k_0|^d} + \mathcal{E}\right),
		\end{align*}
		since $|I_p| > \sqrt{2} |W| |k_0|^d$. Subsequently, we have proved that in both cases
		\begin{align*}
			\frac1{|I| |I_p|^2}	\sum_{\substack{i_2 \in I \\ (i_1, i_3) \in I_p}} |\kappa_{125|346}| \leq  K\left(\frac{\log(|W|)^d}{|W|^2 |k_0|^d} + \frac1{|k_0|^{2d} |W|^2}  + \mathcal{E}\right),
		\end{align*}
		In the study of $B_1$, it remains to consider the partitions of category~$\textbf{V}$. All these terms are bounded by the following steps, that we details on $13|25|46$. Assume that $|\widetilde{k_{j_1}} - \widetilde{k_{j_2}}| r_p \geq \sqrt{2} |I_p|^{\frac1{2d}}$. Then, using Lemmas \ref{lem:decor_freq_2} and \ref{lem:tail_herm}, we get
		$
		|\kappa(A_2, A_5)| \leq K \mathcal{E}.
		$
		This leads to  
		\begin{align*}
			\frac1{|I| |I_p|^2}	\sum_{\substack{i_2 \in I \\ (i_1, i_3) \in I_p}} |\kappa_{13|25|46}|  \leq K \mathcal{E},
		\end{align*}
		when $|\widetilde{k_{j_1}} - \widetilde{k_{j_2}}| r_p \geq \sqrt{2} |I_p|^{\frac1{2d}}$. Assume now that $|\widetilde{k_{j_1}} - \widetilde{k_{j_2}}| r_p \leq \sqrt{2} |I_p|^{\frac1{2d}}$. Using the Cauchy-Schwarz inequality we obtain
		\begin{align*}
			&\sum_{\substack{i_2 \in I \\ (i_1, i_3) \in I_p}} |\kappa_{13|25|46}| \leq \\& \sum_{\substack{i_1, i_2 \in I_p}} |\kappa(\overline{C^2_{i_1}(\widetilde{k_{j_1}})}, C^2_{i_3}(\widetilde{k_{j_2}}))| \sqrt{\sum_{i_2 \in I} |\kappa(C^2_{i_1}(\widetilde{k_{j_1}}, C^1_{i_2}(k_{j_2})))|^2} \sqrt{\sum_{i_2 \in I} |\kappa(C^2_{i_3}(\widetilde{k_{j_2}}, C^1_{i_2}(k_{j_2}))|^2}.
		\end{align*}
		Then, using twice Lemma \ref{lem:bessel_k2}, we obtain 
		\begin{align*}
			\sum_{\substack{i_2 \in I \\ (i_1, i_3) \in I_p}} |\kappa_{13|25|46}| &\leq K \sum_{\substack{i_1, i_2 \in I_p}} |\kappa(\overline{C^2_{i_1}(\widetilde{k_{j_1}})}, C^2_{i_3}(\widetilde{k_{j_2}}))| \\& \leq K |I_p| \sqrt{\sum_{\substack{i_1, i_2 \in I_p}} |\kappa(\overline{C^2_{i_1}(\widetilde{k_{j_1}})}, C^2_{i_3}(\widetilde{k_{j_2}}))|^2} \leq K |I_p|^{\frac32}.
		\end{align*}
		Accordingly, using the $k_0$-allowed assumption, we get
		\begin{align*}
			\frac1{|I| |I_p|^2 N^2}	\sum_{1 \leq j_1, j_2 \leq N} \sum_{\substack{i_2 \in I \\ (i_1, i_3) \in I_p}} |\kappa_{13|25|46}| &\leq \frac{K}{N^2 |I| |I_p|^{\frac12}} \sum_{1 \leq j_1, j_2 \leq N} \mathds{1}_{|\widetilde{k_{j_1}} - \widetilde{k_{j_2}}| r_p \leq \sqrt{2} |I_p|^{\frac1{2d}}} \\& \leq  K \left(\frac1{|I| |I_p|^{\frac12}} \left(\left(\frac{ |I_p|^{\frac1d}}{|W|^{\frac1d}|k_0|}\right)^{\zeta}+\frac1N\right) \right).
		\end{align*}
		This concludes the study of $B_1$. We obtain the corresponding bound for $B_2$ by exchanging $I_p$ and $I$, and the bounds for $B_3$ and $B_4$ follow by exchanging $j_1$ and $j_2$. 
	\end{proof}

	\begin{lemma}\label{lem:dd_4}
		Suppose the assumptions of Theorem \ref{thm:dd} hold. Then,
		\begin{multline*}
			\frac1{N^2} \sum_{\substack{1 \leq j_1 \leq N \\ 1 \leq j_2 \leq N}} |\kappa(\widehat{S}_{I}(k_{j_1}), \widehat{S}_{0}(\tilde{k_{j_1}}), \widehat{S}_{I}(k_{j_2}), \widehat{S}_{0}(\tilde{k_{j_2}}))| \leq \\ K\left(\frac{1}{|I||I_p|} \left(\left(\frac{(|I| \vee |I_p|)^{\frac1d}}{|W|^{\frac1d} |k_0|}\right)^{\zeta} + \frac1N\right) + \frac{\log(|W|)^d}{|W|^{\frac{5}3}}  + \frac1{|k_0|^{2d} |W|^2} +  \frac{\log(|W|)^{d}}{|W|^2 |k_0|^d}\right) + \mathcal{E},
		\end{multline*}
		for a constant $K \in (0, \infty)$ and $\zeta = d-1$.
	\end{lemma}
	\begin{proof}
		We proceed as in the previous lemma. By multi-linearity and using Corollary~\ref{corol:cumulants_square_rvs}, we have
		\begin{align*}
			\kappa(\widehat{S}_{I}(k_{j_1}), \widehat{S}_{0}(\tilde{k_{j_1}}), \widehat{S}_{I}(k_{j_2}), \widehat{S}_{0}(\tilde{k_{j_2}})) = \frac1{|I|^2 |I_p|^2} \sum_{\substack{(i_1, i_3) \in I^2 \\ (i_2, i_4) \in I_p^2}}\sum_{\substack{\sigma \in \Pi[8], \\ \sigma \vee \tau_8 = 1_{8}}} \prod_{b \in \sigma} \kappa(A_s; s \in b),
		\end{align*}
		where
		\begin{align*}
			&A_1 := C^1_{i_1}(k_{j_1}),~A_2 := C^1_{i_1}(-k_{j_1}),~A_3 := C^2_{i_2}(\widetilde{k_{j_1}}),~A_4 := C^2_{i_2}(-\widetilde{k_{j_1}}),
			\\&A_5 := C^1_{i_3}(k_{j_2}),~A_6 := C^1_{i_3}(-k_{j_2}),~A_7 := C^2_{i_4}(\widetilde{k_{j_2}}),~A_8 := C^2_{i_4}(-\widetilde{k_{j_2}}).
		\end{align*}
		In the sum over the partitions, there are $12$ categories that appear. To fix the idea, one representative of each of them is given by (refer to \cite{mccullagh2018tensor}, Table 1):
		\begin{align*}
			&\textbf{I}~:~12345678~(1), \quad \textbf{II}~: 123457|68~(24), \quad \textbf{III}~: 12345|678~(24), \quad \textbf{IV}~: 13578| 246~(32), 
			\\& \textbf{V}~: 1235|4678~(24), ~ \textbf{VI}~: 1357|2468~(8), \quad \textbf{VII}~: 1235|47|68~(96),~\textbf{VIII}~: 1357|24|68~(48)
			\\& \textbf{IX}~: 123|457|68~(96), ~ \textbf{X}~: 124|567|28~(48),~ \textbf{XI}~: 137|258|46~(96),~\textbf{XII}~: 13|25|47|68~(48).
		\end{align*}
		and the number of partitions inside each category is written between the parenthesis.
		In the following, we use the notation
		\begin{align*}
			\kappa_{\textbf{C}}:= \frac1{|I|^2 |I_p|^2} \sum_{\substack{(i_1, i_3) \in I^2 \\ (i_2, i_4) \in I_p^2}} \sum_{\sigma \in \textbf{C}} \prod_{b \in \sigma} |\kappa(A_s; s \in b)|,
		\end{align*}
		where $\textbf{C}$ is a category of partitions.
		We start by bounding the partitions containing a block of size larger than 6. For the partitions of category $\textbf{I}$, we have, according to Lemma \ref{lem:bril_thin}, 
		\begin{align*}
			|\kappa(A_i; i \in [8])| & \leq K \|r^{-\frac{d}{2}}\psi_{i_1}(\frac{\cdot}{r})\|_8^{2} \|r_p^{-\frac{d}{2}}\psi_{i_2}\left(\frac{\cdot}{r_p}\right)\|_8^{2}\|r^{-\frac{d}{2}}\psi_{i_3}(\frac{\cdot}{r})\|_8^{2}\|r_p^{-\frac{d}{2}}\psi_{i_4}\left(\frac{\cdot}{r_p}\right)\|_8^{2} \\& = r^{-\frac{3d}2} \|\psi_{i_1}\|_8^2 \|\psi_{i_3}\|_8^2~r_p^{-\frac{3d}2} \|\psi_{i_2}\|_8^2 \|\psi_{i_4}\|_8^2.
		\end{align*}
		Using Lemma \ref{lem:L4_herm} and assumption \eqref{eq:rI_dd}, we obtain
		\begin{align*}
			\kappa_{\textbf{I}}  \leq K \frac{\log(|W|)^d}{r^{\frac{3d}2} |I|^{\frac5{12}} r_p^{\frac{3d}2} |I_p|^{\frac5{12}}} \leq \frac{K \log(|W|)^d (|I| |I_p|)^{\frac{1}{3}} }{|W|^3} \leq \frac{K \log(|W|)^d}{|W|^{\frac{7}{3}}}.
		\end{align*}
		For the partitions of category $\textbf{II}$, we bound the cumulant of size two by a constant (using Lemma \ref{lem:cov_thin}) and using Lemma \ref{lem:bril_thin}, we obtain
		\begin{multline*}
			\kappa_{\textbf{II}} \leq K \sup_{J_s} \frac1{|J_1||J_2|}\sum_{i_1 \in J_1, i_2 \in J_2} \|r(J_1)^{-\frac{d}{2}}\psi_{i_1}(\frac{\cdot}{r(J_1)})\|_6 \|r(J_2)^{-\frac{d}{2}}\psi_{i_2}\left(\frac{\cdot}{r(J_2)}\right)\|_6 \\ \frac1{|J_3||J_4|}\sum_{i_3 \in J_3, i_4 \in J_4} \|r(J_3)^{-\frac{d}{2}}\psi_{i_3}(\frac{\cdot}{r(J_3)})\|_6^{2}\|r(J_4)^{-\frac{d}{2}}\psi_{i_4}\left(\frac{\cdot}{r(J_4)}\right)\|_6^{2},
		\end{multline*}
		where $(J_s)_{s \in [4]}$ takes values in $\{I, I_p\}$ and exactly two $J_s$ are equal to $I$. Moreover, $r(I) = r$ and $r(I_p) = r_p$. Note that $J_1$ and $J_2$ are determined by the block of size $2$ of the partitions. Using Lemma \ref{lem:L4_herm} and assumption \eqref{eq:rI_dd}, we get after computations
		\begin{align*}
			\kappa_{\textbf{II}} \leq K \sup_{J_s} \prod_{s \in \{1, 2\}} \frac{\log(|W|)^{d/6}}{r(J_s)^{\frac{d}3} |J_s|^{\frac1{9}}} \prod_{s \in \{3, 4\}} \frac{\log(|W|)^{d/3}}{r(J_s)^{\frac{2d}3} |J_s|^{\frac2{9}}} \leq \frac{K \log(|W|)^d}{|W|^{\frac{5}3}}.
		\end{align*}
		For the partitions of category $\textbf{III}$, the block of size $3$ contains one index $i_s$, $s \in~[4]$ appearing twice. Accordingly, using Lemma \ref{lem:bril_thin}, we bound them as
		\begin{multline*}
			\kappa_{\textbf{III}} \leq K \sup_{J_s} \frac1{|J_1|}\sum_{i_1 \in J_1} \|r(J_1)^{-\frac{d}{2}}\psi_{i_1}(\frac{\cdot}{r(J_1)})\|_3 \|r(J_1)^{-\frac{d}{2}}\psi_{i_1}(\frac{\cdot}{r(J_1)})\|_5 \frac1{|J_2|}\sum_{i_2 \in J_2} \|r(J_2)^{-\frac{d}{2}}\psi_{i_2}(\frac{\cdot}{r(J_2)})\|_3^2 \\ \frac1{|J_3||J_4|}\sum_{i_3 \in J_3, i_4 \in J_4} \|r(J_3)^{-\frac{d}{2}}\psi_{i_3}(\frac{\cdot}{r(i_3)})\|_5^{2}\|r(J_4)^{-\frac{d}{2}}\psi_{i_4}\left(\frac{\cdot}{r(i_4)}\right)\|_5^{2},
		\end{multline*}
		where $(J_s)_{s \in [4]}$ takes values in $\{I, I_p\}$ and exactly two $J_s$ are equal to $I$. Moreover, $r(I) = r$ and $r(I_p) = r_p$. By Cauchy-Schwarz, we have
		\begin{multline*}
			\frac1{|J_1|}\sum_{i_1 \in J_1} \|r(J_1)^{-\frac{d}{2}}\psi_{i_1}(\frac{\cdot}{r(J_1)})\|_3 \|r(J_1)^{-\frac{d}{2}}\psi_{i_1}(\frac{\cdot}{r(J_1)})\|_5\\ \leq \sqrt{\frac1{|J_1|}\sum_{i_1 \in J_1} \|r(J_1)^{-\frac{d}{2}}\psi_{i_1}(\frac{\cdot}{r(J_1)})\|_3^2} \sqrt{\frac1{|J_1|}\sum_{i_1 \in J_1} \|r(J_1)^{-\frac{d}{2}}\psi_{i_1}(\frac{\cdot}{r(J_1)})\|_5^2}.
		\end{multline*}
		Using Lemma \ref{lem:L4_herm} and assumption \eqref{eq:rI_dd}, we get after computations
		\begin{align*}
			\kappa_{\textbf{III}} &\leq K \sup_{J_s} \frac1{r(J_1)^{\frac{d}6}|J_1|^{\frac1{12} }} \left(\frac{\log(|W|)^{2d/5}}{r(J_1)^{\frac{3d}5}|J_1|^{\frac1{5}}}\right)^{\frac12} \frac1{r(J_2)^{\frac{d}3}|J_1|^{\frac1{6}}} \prod_{s \in \{3, 4\}} \frac{\log(|W|)^{2d/5}}{r(J_s)^{\frac{3d}5} |J_s|^{\frac1{5}}}  \leq \frac{K \log(|W|)^d}{|W|^{\frac{7}{4}}}.
		\end{align*}
		For the partitions of category $\textbf{IV}$, the block of size $3$ does not contain indexes $i_s$, $s \in~[4]$ appearing twice. Accordingly, using Lemma \ref{lem:bril_thin}, we bound them as
		\begin{multline*}
			\kappa_{\textbf{IV}} \leq K \sup_{J_s} \prod_{s \in [3]} \frac1{|J_s|}\sum_{i_s \in J_s} \|r(J_s)^{-\frac{d}{2}}\psi_{i_s}(\frac{\cdot}{r(J_s)})\|_3  \|r(J_s)^{-\frac{d}{2}}\psi_{i_s}(\frac{\cdot}{r(J_s)})\|_5 \\ \frac1{|J_4|}\sum_{i_4 \in J_4} \|r(J_4)^{-\frac{d}{2}}\psi_{i_4}\left(\frac{\cdot}{r(i_4)}\right)\|_5^{2},
		\end{multline*}
		where $(J_s)_{s \in [4]}$ takes values in $\{I, I_p\}$ and exactly two $J_s$ are equal to $I$. Moreover, $r(I) = r$ and $r(I_p) = r_p$. Using first the Cauchy-Schwarz inequality, then Lemma \ref{lem:L4_herm} and finally assumption \eqref{eq:rI_dd}, we get after computations
		\begin{align*}
			\kappa_{\textbf{III}} \leq K \sup_{J_s} \frac{\log(|W|)^{2d/5}}{r(J_4)^{\frac{3d}5} |J_4|^{\frac1{5}}}\prod_{s \in [3]} \left(\frac1{r(J_s)^{\frac{d}3}|J_s|^{\frac1{6}}} \frac{\log(|W|)^{2d/5}}{r(J_s)^{\frac{3d}5} |J_s|^{\frac1{5}}} \right)^{\frac12}   \leq \frac{K \log(|W|)^d}{|W|^{\frac{37}{20}}}.
		\end{align*}
		For the partitions of category $\textbf{V}$ or $\textbf{VI}$, all the indexes belong to a block of size $4$. Accordingly, using Lemma \ref{lem:bril_thin}, we bound them as
		\begin{align*}
			\kappa_{\textbf{V}} + \kappa_{\textbf{VI}} \leq K \sup_{J_s} \prod_{s \in [4]} \frac1{|J_s|}\sum_{i_s \in J_s} \|r(J_s)^{-\frac{d}{2}}\psi_{i_s}(\frac{\cdot}{r(J_s)})\|_4^2,
		\end{align*}
		where $(J_s)_{s \in [4]}$ takes values in $\{I, I_p\}$ and exactly two $J_s$ are equal to $I$. Moreover, $r(I) = r$ and $r(I_p) = r_p$. Using Lemma \ref{lem:L4_herm} and assumption \eqref{eq:rI_dd}, we get after computations
		\begin{align*}
			\kappa_{\textbf{V}} + \kappa_{\textbf{VI}}  \leq K \sup_{J_s} \prod_{s \in [4]} \frac{\log(|W|)^{\frac{d}2}}{r(J_s)^{\frac{d}2} |J_s|^{\frac14}} \leq K \frac{\log(|W|)^{2d}}{|W|^2}.
		\end{align*}
		Gathering the previous bounds, we obtain
		\begin{align*}
			\frac1{N^2} \sum_{1 \leq j_1, j_2 \leq N} \kappa_{\textbf{I}} + \dots + \kappa_{\textbf{VI}} \leq \frac{K \log(|W|)^d}{|W|^{\frac{5}3}}.
		\end{align*}
		For the other categories of partitions, we rely on decorrelation between distinct frequencies. We start with the partitions of category $\textbf{VII}$. They differ from the ones of category $\textbf{VIII}$ by the fact that the block of size $4$ contains a block of $12|34|56|78$, i.e. that one index $i_s,~s \in [4]$ appears twice in this block. To fix the idea, the full list of such partitions is given by
		\begin{align*}
			&1235|47|68,~1235|67|48,~1236|47|58,~1236|57|48,~1237|45|68,~1237|46|58,~1238|45|67
			\\&1238|46|57 ,~1245|37|68 ,~1245|67|38 ,~1246|37|58 ,~1246|57|38 ,~1247|35|68 ,~1247|36|58
			\\&1248|35|67 ,~1248|36|57 ,~1257|36|48 ,~1257|46|38 ,~1258|36|47 ,~1258|37|46 ,~1267|35|48
			\\&1267|45|38 ,~1268|35|47 ,~1268|37|45 ,~1345|27|68 ,~1345|67|28 ,~1346|27|58 ,~1346|57|28
			\\&1347|25|68 ,~1347|26|58 ,~1348|25|67 ,~1348|26|57 ,~1356|27|48 ,~1356|47|28 ,~1378|25|46
			\\&1378|26|45 ,~1456|27|38 ,~1456|37|28 ,~1478|25|36 ,~1478|26|35 ,~1567|23|48 ,~1567|24|38
			\\&1568|23|47 ,~1568|24|37 ,~1578|23|46 ,~1578|24|36 ,~1678|23|45 ,~1678|24|35 ,~2345|17|68
			\\&2345|67|18 ,~2346|17|58 ,~2346|57|18 ,~2347|15|68 ,~2347|16|58 ,~2348|15|67 ,~2348|16|57
			\\&2356|17|48 ,~2356|47|18 ,~2378|15|46 ,~2378|16|45 ,~2456|17|38 ,~2456|37|18 ,~2478|15|36
			\\&2478|16|35 ,~2567|13|48 ,~2567|14|38 ,~2568|13|47 ,~2568|14|37 ,~2578|13|46 ,~2578|14|36
			\\&2678|13|45 ,~2678|14|35 ,~3457|16|28 ,~3457|26|18 ,~3458|16|27 ,~3458|17|26 ,~3467|15|28
			\\&3467|25|18 ,~3468|15|27 ,~3468|17|25 ,~3567|14|28 ,~3567|24|18 ,~3568|14|27 ,~3568|17|24
			\\&3578|14|26 ,~3578|16|24 ,~3678|14|25 ,~3678|15|24 ,~4567|13|28 ,~4567|23|18 ,~4568|13|27
			\\&4568|17|23 ,~4578|13|26 ,~4578|16|23 ,~4678|13|25 ,~4678|15|23
		\end{align*}
		Note that the cumulants related to these partitions can be written
		\begin{align*}
			q_{\textbf{VII}} := \sum_{\substack{i_s \in J_s \\ s \in [4]}} \kappa(D_{i_1}, \overline{D_{i_1}},  D_{i_{2}}, D_{i_{3}}) \kappa(\overline{D_{i_2}}, D_{i_4}) \kappa(\overline{D_{i_3}}, \overline{D_{i_4}}),
		\end{align*}
		where $(J_s)_{s \in [4]}$ takes values in $\{I, I_p\}$ and exactly two $J_s$ are equal to $I$. Moreover, if $i \in I$, then $D_i = C_i^1(k)$ for a proper frequency $k \in \{\pm k_{j_1}, \pm k_{j_2}\}$, and $i\in I_p$, then   $D_i = C_i^2(k)$ for a proper frequency $k \in \{\pm \widetilde{k_{j_1}}, \widetilde{\pm k_{j_2}}\}$. Applying three times the Cauchy-Schwarz inequality, we get
		\begin{align*}
			|q_{\textbf{VII}}| &\leq \sum_{\substack{i_s \in J_s \\ s \in [3]}} |\kappa(D_{i_1}, \overline{D_{i_1}},  D_{i_{2}}, D_{i_{3}})| \sqrt{\sum_{i_4 \in J_4} |\kappa(\overline{D_{i_2}}, D_{i_4})|^2} \sqrt{\sum_{i_4 \in J_4} | \kappa(\overline{D_{i_3}}, \overline{D_{i_4}})|^2} \\& \leq \sum_{\substack{i_1 \in J_1}} \sqrt{\sum_{\substack{i_2 \in J_2 \\ i_3 \in J_3}} |\kappa(D_{i_1}, \overline{D_{i_1}},  D_{i_{2}}, D_{i_{3}})|^2 } \sqrt{\sum_{\substack{i_2 \in J_3 \\ i_4\in J_4}} |\kappa(\overline{D_{i_2}}, D_{i_4})|^2} \sqrt{\sum_{\substack{i_3 \in J_3 \\ i_4\in J_4}} | \kappa(\overline{D_{i_3}}, \overline{D_{i_4}})|^2}.
		\end{align*}
		Using Lemmas \ref{lem:bessel_km} and \ref{lem:L4_herm}, we obtain:
		\begin{align*}
			\sum_{\substack{i_1 \in J_1}} \sqrt{\sum_{\substack{i_2 \in J_2 \\ i_3 \in J_3}} |\kappa(D_{i_1}, \overline{D_{i_1}},  D_{i_{2}}, D_{i_{3}})|^2 } & \leq  \sum_{\substack{i_1 \in J_1}} \|r(J_1) \psi_{i_1}(\frac{\cdot}{r(J_1)})\|_6^2 \sqrt{\sum_{\substack{i_s \in J_s}} \|r(J_s) \psi_{i_s}(\frac{\cdot}{r(J_s)})\|_6^2  } + \mathcal{E} \\& \leq K \left( |J_1| (|J_2| \wedge |J_3|)^{\frac12} |J_1|^{\frac19}   (|J_2| \wedge |J_3|)^{\frac1{18}}\frac{\log(|W|)^d}{|W|} + \mathcal{E}\right),
		\end{align*}
		where in the second line $s = 2$ if $|J_3| \geq |J_2|$ and $s = 3$ otherwise, 
		The next step is to note that, due to the connectivity constraint, in all the partitions of category $\textbf{VII}$, at most one block of size $2$ contains frequencies $\epsilon_1 k_{j}$ and $\epsilon_2 \widetilde{k_{j}}$ for $(\epsilon_1, \epsilon_2) \in \{-1, 1\}^2$ and $j \in \{j_1, j_2\}$. The other block of size $2$ contains frequencies associated to $j_1$ and $j_2$. With the $k_0$-allowed assumption and Lemma \ref{lem:bessel_k2}, we get
		\begin{multline*}
			\frac1{N^2} \sum_{1 \leq j_1, j_2 \leq N} \sqrt{\sum_{\substack{i_2 \in J_2 \\ i_4\in J_4}} |\kappa(\overline{D_{i_2}}, D_{i_4})|^2} \sqrt{\sum_{\substack{i_3 \in J_3 \\ i_4\in J_4}} | \kappa(\overline{D_{i_3}}, \overline{D_{i_4}})|^2} \\ \leq K(|J_2| \wedge |J_4|)^{\frac12} \left( \left(\frac{(|J_2| \vee |J_4|)^{\frac1d}}{|W|^{\frac1d} |k_0|}\right)^{\zeta} + \frac1N\right)(|J_3| \wedge |J_4|)^{\frac12} \\ +  K(|J_3| \wedge |J_4|)^{\frac12} \left( \left(\frac{(|J_3| \vee |J_4|)^{\frac1d}}{|W|^{\frac1d} |k_0|}\right)^{\zeta} + \frac1N\right)(|J_2| \wedge |J_4|)^{\frac12} + \mathcal{E}.
		\end{multline*}
		This gives after computations (where we use $a b = (a\vee b)(a \wedge b)$):
		\begin{align*}
			\frac1{N^2} \sum_{1 \leq j_1, j_2 \leq N} |\kappa_{\textbf{VII}}| &\leq \frac{K}{|I|^{\frac{8}9}|I_p|^{\frac{8}9}} \frac{\log(|W|)^d}{|W|}  \left(\left(\frac{(|I| \vee |I_p|)^{\frac1d}}{|W|^{\frac1d} |k_0|}\right)^{\zeta} + \frac1N\right) + \mathcal{E} \\& \leq \frac{K}{|I||I_p|} \left(\left(\frac{(|I| \vee |I_p|)^{\frac1d}}{|W|^{\frac1d} |k_0|}\right)^{\zeta} + \frac1N\right) + \mathcal{E}.
		\end{align*}
		
		Now, we consider the partitions of category \textbf{VIII}. They are given by 
		\begin{align*}
			&1357|24|68,~1357|26|48,~1357|46|28,~1358|24|67,~1358|26|47,~1358|27|46,~1367|24|58
			\\&1367|25|48 ,~1367|45|28 ,~1368|24|57 ,~1368|25|47 ,~1368|27|45 ,~1457|23|68 ,~1457|26|38
			\\&1457|36|28 ,~1458|23|67 ,~1458|26|37 ,~1458|27|36 ,~1467|23|58 ,~1467|25|38 ,~1467|35|28
			\\&1468|23|57 ,~1468|25|37 ,~1468|27|35 ,~2357|14|68 ,~2357|16|48 ,~2357|46|18 ,~2358|14|67
			\\&2358|16|47 ,~2358|17|46 ,~2367|14|58 ,~2367|15|48 ,~2367|45|18 ,~2368|14|57 ,~2368|15|47
			\\&2368|17|45 ,~2457|13|68 ,~2457|16|38 ,~2457|36|18 ,~2458|13|67 ,~2458|16|37 ,~2458|17|36
			\\&2467|13|58 ,~2467|15|38 ,~2467|35|18 ,~2468|13|57 ,~2468|15|37 ,~2468|17|35
		\end{align*}
		We follow the same idea than in the study of the partitions of category~\textbf{VII}, but this times the cumulants appearing are of the form
		\begin{align*}
			p := \sum_{\substack{i_s \in J_s; \\ s \in [4]}} \kappa(D_{i_1}, \overline{D_{i_2}},  D_{i_{3}}, D_{i_{4}}) \kappa(\overline{D_{i_1}}, D_{i_2}) \kappa(\overline{D_{i_3}}, \overline{D_{i_4}}),
		\end{align*}
		where $(J_s)_{s \in [4]}$ takes values in $\{I, I_p\}$ and exactly two $J_s$ are equal to $I$. Moreover, if $i \in I$, then $D_i = C_i^1(k)$ for a proper frequency $k \in \{\pm k_{j_1}, \pm k_{j_2}\}$, and $i\in I_p$, then   $D_i = C_i^2(k)$ for a proper frequency $k \in \{\pm \widetilde{k_{j_1}}, \widetilde{\pm k_{j_2}}\}$. Using the Cauchy-Schwarz inequality, we obtain
		\begin{align*}
			|p| \leq \sqrt{\sum_{\substack{i_s \in J_s; \\ s \in [4]}} |\kappa(D_{i_1}, \overline{D_{i_2}},  D_{i_{3}}, D_{i_{4}})|^2} \sqrt{\sum_{\substack{i_1 \in J_1 \\ i_2\in J_2}} |\kappa(\overline{D_{i_1}}, D_{i_2})|^2} \sqrt{ \sum_{\substack{i_3\in J_3 \\ i_4\in J_4}} |\kappa(\overline{D_{i_3}}, \overline{D_{i_4}})|^2}.
		\end{align*}
		Then, using first Lemma \ref{lem:bessel_km} (with one $J_s$ corresponding to $I$), then Lemma \ref{lem:L4_herm} and finally assumption \eqref{eq:rI_dd} we get
		\begin{align*}
			\sqrt{\sum_{\substack{i_s \in J_s; \\ s \in [4]}} |\kappa(D_{i_1}, \overline{D_{i_2}},  D_{i_{3}}, D_{i_{4}})|^2} \leq K |I_p|^{1+\frac{1}{9}} |I|^{\frac12 + \frac{1}{18}} \frac{\log(|W|)^{d/2}}{|W|} + \mathcal{E}.
		\end{align*}
		As for the partitions of category \textbf{VII}, at most one block of size $2$ contains frequencies $\epsilon_1 k_{j}$ and $\epsilon_2 \widetilde{k_{j}}$ for $(\epsilon_1, \epsilon_2) \in \{-1, 1\}^2$ and $j \in \{j_1, j_2\}$. The other block of size $2$ contains frequencies associated to $j_1$ and $j_2$. Accordingly, we apply the arguments used for the the category \textbf{VII}, to get
		\begin{align*}
			\frac1{N^2} \sum_{1 \leq j_1, j_2 \leq N} |\kappa_{\textbf{VIII}}| \leq \frac{K}{|I||I_p|} \left(\left(\frac{(|I| \vee |I_p|)^{\frac1d}}{|W|^{\frac1d} |k_0|}\right)^{\zeta} + \frac1N\right) + \mathcal{E}.
		\end{align*}
		The partitions of category \textbf{IX}, \textbf{X} and \textbf{XI}, are classed depending on the number of blocks of size $3$ containing a block of $12|34|56|78$ (one for $\textbf{IX}$, two for \textbf{X} and zero for \textbf{XI}). 
		Partitions in category \textbf{X} are the most constrained, since each of the two blocks of size~$3$ contains an index $i_s$, for some $s \in [4]$, that appears twice. We first bound their contribution. The proofs for the other categories follow similar lines and can be readily adapted.
		\begin{align*}
			&123|567|48,~123|568|47,~123|578|46,~123|678|45,~124|567|38,~124|568|37,~124|578|36
			\\&124|678|35 ,~125|347|68 ,~125|348|67 ,~125|378|46 ,~125|478|36 ,~126|347|58 ,~126|348|57
			\\&126|378|45 ,~126|478|35 ,~127|345|68 ,~127|346|58 ,~127|356|48 ,~127|456|38 ,~128|345|67
			\\&128|346|57 ,~128|356|47 ,~128|456|37 ,~134|567|28 ,~134|568|27 ,~134|578|26 ,~134|678|25
			\\&156|347|28 ,~156|348|27 ,~156|378|24 ,~156|478|23 ,~178|345|26 ,~178|346|25 ,~178|356|24
			\\&178|456|23 ,~234|567|18 ,~234|568|17 ,~234|578|16 ,~234|678|15 ,~256|347|18 ,~256|348|17
			\\&256|378|14 ,~256|478|13 ,~278|345|16 ,~278|346|15 ,~278|356|14 ,~278|456|13
		\end{align*}
		The cumulants related the partitions of category $\textbf{X}$ can be written, leveraging the connectivity condition:
		\begin{align*}
			q_{\textbf{X}} := \sum_{\substack{i_s \in J_s \\ s \in [4]}} 
			\kappa(D_{i_1}, \overline{D_{i_1}}, D_{i_2})\,
			\kappa(D_{i_3}, \overline{D_{i_3}}, D_{i_4})\,
			\kappa(\overline{D_{i_2}}, \overline{D_{i_4}}),
		\end{align*}
		where $(J_s)_{s \in [4]}$ takes values in $\{I, I_p\}$, and exactly two of the $J_s$ are equal to $I$. Moreover, if $i \in I$, then $D_i = C_i^1(k)$ for an appropriate frequency $k \in \{\pm k_{j_1}, \pm k_{j_2}\}$, and if $i \in I_p$, then $D_i = C_i^2(k)$ for a suitable frequency $k \in \{\pm \widetilde{k_{j_1}}, \pm \widetilde{k_{j_2}}\}$. Applying the Cauchy–Schwarz inequality twice, we obtain:
		\begin{align*}
			|&q_{\textbf{X}}| \leq \sum_{\substack{i_s \in J_s \\ s \in \{1, 3, 4\}}} \sqrt{\sum_{i_2 \in J_2} |\kappa(D_{i_1}, \overline{D_{i_1}}, D_{i_2})|^2 } |\kappa(D_{i_3}, \overline{D_{i_3}}, D_{i_4})| \sqrt{\sum_{i_2 \in J_2} |\kappa(D_{i_2}, D_{i_4})|^2} \\& \leq \sum_{i_1 \in J_1} \sqrt{\sum_{i_2 \in J_2} |\kappa(D_{i_1}, \overline{D_{i_1}}, D_{i_2})|^2 } \sum_{i_3 \in J_3} \sqrt{\sum_{i_4 \in J_4} |\kappa(D_{i_3}, \overline{D_{i_3}}, D_{i_4})|^2} \sqrt{\sum_{\substack{i_2 \in J_2 \\ i_4 \in J_4}} |\kappa(D_{i_2}, D_{i_4})|^2}.
		\end{align*}
		Using Lemmas \ref{lem:bessel_km} and \ref{lem:L4_herm}, we obtain 
		\begin{align*}
			|q_{\textbf{X}}| &\leq K |J_1| |J_3| \frac{\log(|W|)^{d}}{|W|} \sqrt{\sum_{\substack{i_2 \in J_2 \\ i_4 \in J_4}} |\kappa(D_{i_2}, D_{i_4})|^2} + \mathcal{E}.
		\end{align*}
		Two cases may arise. In the first, the cumulant $\kappa(D_{i_2}, D_{i_4})$ contains frequencies associated with the same index $j_1$ (or $j_2$). In the second, it contains frequencies associated with distinct indices $j_1$ and $j_2$. We begin with the first case. Suppose that
		$|J_2|^{1/(2d)}/r(J_2) \vee |J_4|^{1/(2d)}/r(J_4) \leq \sqrt{2} c_2 |k_0|.$ By applying Lemma~\ref{lem:tail_herm}, we obtain, for some constant $c_6 > 0$,
		
		\begin{align*}
			\frac1{|I_p|^2 |I|^2} \sum_{\substack{(i_1, i_3) \in I, \\ (i_2, i_4) \in I_p}} |q_{\textbf{X}}| &\leq K e^{-c_6 |k_0| |W|^{\frac{1}d}} + \mathcal{E} \leq \frac{K}{|k_0|^{2d} |W|^2} + \mathcal{E}.
		\end{align*}
		Now assume that $|J_2|^{\frac1{2d}}/r(J_2) \vee |J_4|^{\frac1{2d}}/r(J_4) \geq \sqrt{2} c_2 |k_0|$. Subsequently, with assumption~\eqref{eq:rI_dd}, we obtain
		$
		|J_2| \wedge |J_4| \geq c_7 |k_0|^d |W|,
		$
		for $c_7 > 0$. Moreover, according to Lemma~\ref{lem:bessel_km}, we have
		\begin{align*}
			\frac1{|I_p|^2 |I|^2} \sum_{\substack{(i_1, i_3) \in I, \\ (i_2, i_4) \in I_p}} |q_{\textbf{X}}| &\leq K \frac{\log(|W|)^{d}}{|W|} \frac{(|J_2| \wedge |J_4|)^{\frac12}}{|J_2| |J_4|}+ \mathcal{E} \leq K\frac{\log(|W|)^{d}}{|W|^2 |k_0|^d} \frac1{(|I| \wedge |I_p|)^{\frac12}}+ \mathcal{E}.
		\end{align*}
		Now we consider the second case, where the block of size two contains frequencies associated to distinct indices $j_1$ and $j_2$. With the $k_0$-allowed assumption and Lemma \ref{lem:tail_herm}, we obtain
		\begin{align*}
			\frac1N \sum_{1\leq j_1, j_2 \leq N} \frac1{|I_p|^2 |I|^2} \sum_{\substack{(i_1, i_3) \in I, \\ (i_2, i_4) \in I_p}} |q_{\textbf{X}}| \leq K \frac{\log(|W|)^d}{|W| |J_2| |J_4|} \left(\left(\frac{(|J_2| \vee |J_4|)^{\frac1d}}{|W|^{\frac1d} |k_0|}\right)^{\zeta} + \frac1N\right) + \mathcal{E}.  
		\end{align*}
		Since $|J_2| \vee |J_4| \leq |I| \vee |I_p|$, we have
		\begin{align*}
			\frac1{|J_2| |J_4|} \left(\frac{(|J_2| \vee |J_4|)^{\frac1d}}{|W|^{\frac1d} |k_0|}\right)^{\zeta} &\leq \frac1{(|I|\wedge|I_p|)^2}\left(\frac{(|I| \vee |I_p|)^{\frac1d}}{|W|^{\frac1d} |k_0|}\right)^{\zeta}.
		\end{align*}
		This yields, using $|I|\vee |I_p| \leq |W|^{2\beta/(2\beta +d)}$:
		\begin{align*}
			\frac1N \sum_{1\leq j_1, j_2 \leq N} |\kappa_{\textbf{X}}| \leq  K\left(\frac{\log(|W|)^{d}}{|W|^2 |k_0|^d} + \frac{1}{|k_0|^{2d} |W|^2} + \frac{1}{|I| |I_p| }\left(\left(\frac{(|I| \vee |I_p|)^{\frac1d}}{|W|^{\frac1d} |k_0|}\right)^{\zeta} +  \frac1N\right)\right) +\mathcal{E}.
		\end{align*}
		The partitions of category \textbf{IX} are given by
		\begin{align*}
			&123|457|68,~123|458|67,~123|467|58,~123|468|57,~124|357|68,~124|358|67,~124|367|58
			\\&124|368|57 ,~125|367|48 ,~125|368|47 ,~125|467|38 ,~125|468|37 ,~126|357|48 ,~126|358|47
			\\&126|457|38 ,~126|458|37 ,~127|358|46 ,~127|368|45 ,~127|458|36 ,~127|468|35 ,~128|357|46
			\\&128|367|45 ,~128|457|36 ,~128|467|35 ,~134|257|68 ,~134|258|67 ,~134|267|58 ,~134|268|57
			\\&135|278|46 ,~135|478|26 ,~135|678|24 ,~136|278|45 ,~136|478|25 ,~136|578|24 ,~137|256|48
			\\&137|456|28 ,~137|568|24 ,~138|256|47 ,~138|456|27 ,~138|567|24 ,~145|278|36 ,~145|378|26
			\\&145|678|23 ,~146|278|35 ,~146|378|25 ,~146|578|23 ,~147|256|38 ,~147|356|28 ,~147|568|23
			\\&148|256|37 ,~148|356|27 ,~148|567|23 ,~156|237|48 ,~156|238|47 ,~156|247|38 ,~156|248|37
			\\&157|234|68 ,~157|346|28 ,~157|348|26 ,~158|234|67 ,~158|346|27 ,~158|347|26 ,~167|234|58
			\\&167|345|28 ,~167|348|25 ,~168|234|57 ,~168|345|27 ,~168|347|25 ,~178|235|46 ,~178|236|45
			\\&178|245|36 ,~178|246|35 ,~235|478|16 ,~235|678|14 ,~236|478|15 ,~236|578|14 ,~237|456|18
			\\&237|568|14 ,~238|456|17 ,~238|567|14 ,~245|378|16 ,~245|678|13 ,~246|378|15 ,~246|578|13
			\\&247|356|18 ,~247|568|13 ,~248|356|17 ,~248|567|13 ,~257|346|18 ,~257|348|16 ,~258|346|17
			\\&258|347|16 ,~267|345|18 ,~267|348|15 ,~268|345|17 ,~268|347|15
		\end{align*}
		The cumulants related to these partitions can be written (leveraging the connectivity condition $\sigma \vee 12|34|56|78 = 1$),
		\begin{align*}
			q_{\textbf{IX}} := \sum_{\substack{i_s \in J_s \\ s \in [4]}} \kappa(D_{i_1}, \overline{D_{i_1}},  D_{i_{2}}) \kappa(\overline{D_{i_2}}, D_{i_3}, D_{i_4}) \kappa(\overline{D_{i_3}}, \overline{D_{i_4}}),
		\end{align*}
		where $(J_s)_{s \in [4]}$ takes values in $\{I, I_p\}$ and exactly two $J_s$ are equal to $I$. Moreover, if $i \in I$, then $D_i = C_i^1(k)$ for a proper frequency $k \in \{\pm k_{j_1}, \pm k_{j_2}\}$, and $i\in I_p$, then   $D_i = C_i^2(k)$ for a proper frequency $k \in \{\pm \widetilde{k_{j_1}}, \widetilde{\pm k_{j_2}}\}$. Applying three times the Cauchy-Schwarz inequality, we get
		\begin{align*}
			|q_{\textbf{IX}}| &\leq \sum_{\substack{i_s \in J_s;\\ s\in \{1, 3, 4\}}} \sqrt{\sum_{i_2 \in J_2} |\kappa(D_{i_1}, \overline{D_{i_1}}, D_{i_2})|^2}\sqrt{\sum_{i_2 \in J_2} |\kappa(\overline{D_{i_2}}, D_{i_3}, D_{i_4})|^2} |\kappa(D_{i_3}, D_{i_4})| \\& \leq \sum_{\substack{i_1 \in J_1\\ i_4 \in J_4}} \sqrt{\sum_{i_2 \in J_2} |\kappa(D_{i_1}, \overline{D_{i_1}}, D_{i_2})|^2}\sqrt{\sum_{\substack{i_2 \in J_2\\i_3 \in J_3}} |\kappa(\overline{D_{i_2}}, D_{i_3}, D_{i_4})|^2} \sqrt{\sum_{i_3 \in J_3} |\kappa(D_{i_3}, D_{i_4})|^2} \\& \leq \sum_{i_1 \in J_1} \sqrt{\sum_{i_2 \in J_2} |\kappa(D_{i_1}, \overline{D_{i_1}}, D_{i_2})|^2}\sqrt{\sum_{\substack{i_s \in J_s;\\ s\in \{2, 3, 4\}}} |\kappa(\overline{D_{i_2}}, D_{i_3}, D_{i_4})|^2} \sqrt{\sum_{\substack{i_3 \in J_3 \\ i_4 \in J_4}} |\kappa(D_{i_3}, D_{i_4})|^2}.
		\end{align*}
		Then, arguing as in the study of the partitions of category \textbf{X}, we obtain
		\begin{align*}
			\frac1N \sum_{1\leq j_1, j_2 \leq N} |\kappa_{\textbf{IX}}| \leq K\left(\frac{\log(|W|)^{d}}{|W|^2 |k_0|^d} + \frac{1}{|k_0|^{2d} |W|^2} + \frac{1}{|I| |I_p| }\left(\left(\frac{(|I| \vee |I_p|)^{\frac1d}}{|W|^{\frac1d} |k_0|}\right)^{\zeta} +  \frac1N\right)\right) +\mathcal{E}.
		\end{align*}
		We consider the partitions of category \textbf{XI}. They are given by
		\begin{align*}
			&135|247|68,~135|248|67,~135|267|48,~135|268|47,~135|467|28,~135|468|27,~136|247|58
			\\&136|248|57 ,~136|257|48 ,~136|258|47 ,~136|457|28 ,~136|458|27 ,~137|245|68 ,~137|246|58
			\\&137|258|46 ,~137|268|45 ,~137|458|26 ,~137|468|25 ,~138|245|67 ,~138|246|57 ,~138|257|46
			\\&138|267|45 ,~138|457|26 ,~138|467|25 ,~145|237|68 ,~145|238|67 ,~145|267|38 ,~145|268|37
			\\&145|367|28 ,~145|368|27 ,~146|237|58 ,~146|238|57 ,~146|257|38 ,~146|258|37 ,~146|357|28
			\\&146|358|27 ,~147|235|68 ,~147|236|58 ,~147|258|36 ,~147|268|35 ,~147|358|26 ,~147|368|25
			\\&148|235|67 ,~148|236|57 ,~148|257|36 ,~148|267|35 ,~148|357|26 ,~148|367|25 ,~157|236|48
			\\&157|238|46 ,~157|246|38 ,~157|248|36 ,~157|368|24 ,~157|468|23 ,~158|236|47 ,~158|237|46
			\\&158|246|37 ,~158|247|36 ,~158|367|24 ,~158|467|23 ,~167|235|48 ,~167|238|45 ,~167|245|38
			\\&167|248|35 ,~167|358|24 ,~167|458|23 ,~168|235|47 ,~168|237|45 ,~168|245|37 ,~168|247|35
			\\&168|357|24 ,~168|457|23 ,~235|467|18 ,~235|468|17 ,~236|457|18 ,~236|458|17 ,~237|458|16
			\\&237|468|15 ,~238|457|16 ,~238|467|15 ,~245|367|18 ,~245|368|17 ,~246|357|18 ,~246|358|17
			\\&247|358|16 ,~247|368|15 ,~248|357|16 ,~248|367|15 ,~257|368|14 ,~257|468|13 ,~258|367|14
			\\&258|467|13 ,~267|358|14 ,~267|458|13 ,~268|357|14 ,~268|457|13.
		\end{align*}
		The cumulants related to these partitions can be written (leveraging the connectivity condition $\sigma \vee 12|34|56|78 = 1$),
		\begin{align*}
			q_{\textbf{XI}} := \sum_{\substack{i_s \in J_s \\ s \in [4]}} \kappa(D_{i_1}, D_{i_2},  D_{i_{3}}) \kappa(\overline{D_{i_1}}, \overline{D_{i_2}}, D_{i_4}) \kappa(\overline{D_{i_3}}, \overline{D_{i_4}}),
		\end{align*}
		where $(J_s)_{s \in [4]}$ takes values in $\{I, I_p\}$ and exactly two $J_s$ are equal to $I$. Moreover, if $i \in I$, then $D_i = C_i^1(k)$ for a proper frequency $k \in \{\pm k_{j_1}, \pm k_{j_2}\}$, and $i\in I_p$, then   $D_i = C_i^2(k)$ for a proper frequency $k \in \{\pm \widetilde{k_{j_1}}, \widetilde{\pm k_{j_2}}\}$. Applying four times the Cauchy-Schwarz inequality, we get
		\begin{align*}
			|q_{\textbf{XI}}| &\leq \sum_{\substack{i_s \in J_s;\\ s\in \{2, 3, 4\}}} \sqrt{\sum_{i_1 \in J_1} |\kappa(D_{i_1}, D_{i_2}, D_{i_3})|^2}\sqrt{\sum_{i_1\in J_1} |\kappa(\overline{D_{i_1}}, \overline{D_{i_2}}, D_{i_4})|^2} |\kappa(D_{i_3}, D_{i_4})| \\& \leq \sum_{\substack{i_3 \in J_3\\ i_4 \in J_4}} \sqrt{\sum_{\substack{i_1 \in J_1\\ i_2 \in J_2}} |\kappa(D_{i_1}, D_{i_2}, D_{i_3})|^2}\sqrt{\sum_{\substack{i_1 \in J_1\\ i_2 \in J_2}} |\kappa(\overline{D_{i_1}}, \overline{D_{i_2}}, D_{i_4})|^2} |\kappa(D_{i_3}, D_{i_4})| \\& \leq \sum_{i_4 \in J_4} \sqrt{\sum_{\substack{i_s \in J_s;\\ s \in \{1, 2, 3\}}} |\kappa(D_{i_1}, D_{i_2}, D_{i_3})|^2}\sqrt{\sum_{\substack{i_1 \in J_1\\ i_2 \in J_2}} |\kappa(\overline{D_{i_1}}, \overline{D_{i_2}}, D_{i_4})|^2} \sqrt{\sum_{i_3 \in J_3} |\kappa(D_{i_3}, D_{i_4})|^2} \\& \leq 
			\sqrt{\sum_{\substack{i_s \in J_s;\\ s \in \{1, 2, 3\}}} |\kappa(D_{i_1}, D_{i_2}, D_{i_3})|^2}\sqrt{\sum_{\substack{i_s \in J_s;\\ s \in \{1, 2, 4\}}} |\kappa(\overline{D_{i_1}}, \overline{D_{i_2}}, D_{i_4})|^2} \sqrt{\sum_{\substack{i_3 \in J_3\\ i_4 \in J_4}} |\kappa(D_{i_3}, D_{i_4})|^2}.
		\end{align*}
		Arguing as in the case of the partitions of category $\textbf{X}$, we obtain
		\begin{align*}
			\frac1N \sum_{1\leq j_1, j_2 \leq N} |\kappa_{\textbf{XI}}| \leq K\left(\frac{\log(|W|)^{d}}{|W|^2 |k_0|^d} + \frac{1}{|k_0|^{2d} |W|^2} + \frac{1}{|I| |I_p| }\left(\left(\frac{(|I| \vee |I_p|)^{\frac1d}}{|W|^{\frac1d} |k_0|}\right)^{\zeta} +  \frac1N\right)\right) +\mathcal{E}.
		\end{align*}
		It remains to consider the partitions of category \textbf{XII}. They consist in 
		\begin{align*}
			&16|28|35|47,~68|14|27|35,~16|27|48|35,~16|27|38|45,~38|27|15|46,~28|14|36|57
			\\&58|27|13|46 ,~45|67|18|23 ,~16|23|58|47 ,~25|46|38|17 ,~24|68|15|37 ,~14|68|25|37
			\\&38|14|57|26 ,~58|13|47|26 ,~58|46|17|23 ,~16|28|45|37 ,~18|46|23|57 ,~45|68|17|23
			\\&27|15|48|36 ,~38|67|14|25 ,~67|15|23|48 ,~38|24|15|67 ,~16|23|48|57 ,~18|35|47|26
			\\&16|24|37|58 ,~58|14|27|36 ,~67|24|18|35 ,~13|48|57|26 ,~68|15|23|47 ,~38|17|45|26
			\\&45|28|13|67 ,~24|18|36|57 ,~28|13|46|57 ,~58|14|37|26 ,~16|24|38|57 ,~68|13|25|47
			\\&28|15|46|37 ,~68|27|13|45 ,~67|13|25|48 ,~67|28|14|35 ,~24|68|17|35 ,~28|36|15|47
			\\&18|25|36|47 ,~45|18|37|26 ,~58|24|36|17 ,~25|36|17|48 ,~17|48|35|26 ,~18|46|25|37
		\end{align*}
		The cumulants related to these partitions can be written, leveraging the connectivity condition $\sigma \vee 12|34|56|78 = 1$,
		\begin{align*}
			q_{\textbf{XII}} := \sum_{\substack{i_s \in J_s \\ s \in [4]}} 
			\kappa(D_{i_1}, D_{i_2}) \kappa(\overline{D_{i_2}}, D_{i_3}) \kappa(\overline{D_{i_3}}, D_{i_4}) \kappa(\overline{D_{i_4}}, \overline{D_{i_1}}).
		\end{align*}
		where $(J_s)_{s \in [4]}$ takes values in $\{I, I_p\}$, and exactly two of the $J_s$ are equal to $I$. Moreover, if $i \in I$, then $D_i = C_i^1(k)$ for an appropriate frequency $k \in \{\pm k_{j_1}, \pm k_{j_2}\}$, and if $i \in I_p$, then $D_i = C_i^2(k)$ for a suitable frequency $k \in \{\pm \widetilde{k_{j_1}}, \pm \widetilde{k_{j_2}}\}$. Applying the Cauchy–Schwarz inequality four times, we obtain
		\begin{align*}
			|q_{\textbf{XII}}| \leq \sqrt{\sum_{\substack{i_1 \in J_1 \\ i_2 \in J_2}} |\kappa(D_{i_1}, D_{i_2})|^2} \sqrt{\sum_{\substack{i_2 \in J_2 \\ i_3 \in J_3}} |\kappa(\overline{D_{i_2}}, D_{i_3})|^2} \sqrt{\sum_{\substack{i_3 \in J_3 \\ i_4 \in J_4}} |\kappa(\overline{D_{i_3}}, D_{i_4})|^2}  \sqrt{\sum_{\substack{i_1 \in J_1 \\ i_4 \in J_4}} |\kappa(D_{i_1}, D_{i_4})|^2}.
		\end{align*}
		Note that, due to the connectivity condition, at least one block contains two frequencies, one indexed by $j_1$ and the other by $j_2$. Since the partition is connected to $12|34|56|78$, the difference between these two frequencies is at least $|k| c_7$ with $c_7 > 0$. Accordingly, we use the $k_0$-assumption and Lemma \ref{lem:bessel_k2} to get
		\begin{align*}
			\frac1{N^2} \sum_{1 \leq j_1, j_2 \leq N}|q_{\textbf{XII}}| \leq K \prod_{s = 1}^4(|J_s| \wedge |J_{s+1}|)^{\frac12} \left(\left(\frac{(|I| \vee |I_p|)^{\frac1d}}{|W|^{\frac1d} |k|^d}\right)^{\zeta}+\frac1N\right) + \mathcal{E},
		\end{align*}
		with $J_5 := J_1$. Using $ \prod_{s = 1}^4(|J_s| \wedge |J_{s+1}|)^{\frac12} \leq |I| |I_p|$, we obtain
		\begin{align*}
			\frac1{N^2} \sum_{1 \leq j_1, j_2 \leq N}|\kappa_{\textbf{XII}}| \leq K  \frac1{|I| |I_p|}\left(\left(\frac{(|I| \vee |I_p|)^{\frac1d}}{|W|^{\frac1d} |k|^d}\right)^{\zeta}+\frac1N\right) + \mathcal{E}.
		\end{align*}
		Gathering all the previous bounds, we conclude the proof of this lemma.
	\end{proof}
	With the previous lemmas, are now in position to prove Lemma \ref{lem:exp_delta}.
	\begin{proof}[Proof of Lemma \ref{lem:exp_delta}]
		By Jensen's inequality, we have 
		\begin{align*}
			\mathbb{E}\left[\max_{I \in \mathcal{I}} |\Delta_I|\right] & \leq \mathbb{E}\left[\max_{I \in \mathcal{I}} |\Delta_I|^2\right]^{\frac12} \leq \sqrt{|\mathcal{I}| \max_{I \in \mathcal{I}}(\mathbb{E}[\Delta_I]^2 + \operatorname{Var}[\Delta_I])}.
		\end{align*}
		We prove that $\max_{I \in \mathcal{I}}(\mathbb{E}[\Delta_I]^2 + \operatorname{Var}[\Delta_I]) = o(|W|^{-4\beta/(2\beta +d)})$. Throughout the proof, $K$ denotes a generic constant that may vary from line to line. We start with $\max_{I \in \mathcal{I}}\mathbb{E}[\Delta_I]^2$. According to Lemma~\ref{lem:exp_delta_I}, we have
		$$|\mathbb{E}[\Delta_I ]| \leq K\Bigg(\frac{(|I||I_p|)^{\frac{\beta}d}}{|W|^{2\frac{\beta}d}} + \frac{\log(|W|)^d}{|W|} +\frac{1}{|W| |k_0|^d}  + \mathcal{E}(|I|, |I_p|)\Bigg).$$
		Since $|I| \wedge |I_p| \geq \log(|W|)^{2d/(3\theta)+\eta}$, we have
		\begin{align}\label{eq:E}
			\mathcal{E} \leq K e^{-\frac1K \log(|W|)^{1+3\theta\eta/(2d)}} = o\left(|W|^{-\frac{4\beta}{2\beta +d}}\right).
		\end{align}  We always have
		$\log(|W|)^d/|W| =o(|W|^{-2\beta/(2\beta +d)})$. Moreover,
		\begin{align*}
			\frac{1}{|W| |k_0|^d} = A_W^{-d} \frac1{|W| |W|^{-d/(2\beta +d)}} = A_W^{-d} \frac1{|W|^{2\beta/(2\beta +d)}} = o(|W|^{-2\beta/(2\beta +d)}),
		\end{align*}
		since $A_W \to \infty$ as $|W| \to \infty$. Since $|I| \leq c_5|W|^{2\beta/(2\beta +d)}$ and $|I_p| = o(|W|^{2\beta/(2\beta +d)})$, we have
		\begin{align*}
			\frac{(|I||I_p|)^{\frac{\beta}d}}{|W|^{2\frac{\beta}d}} = o \left(\left(\frac{|W|^{\frac{4\beta}{2\beta +d}}}{|W|^2}\right)^{\beta/d}\right) = o(|W|^{-2\beta/(2\beta +d)}).
		\end{align*}
		All previous $o$ are uniform over $I \in \mathcal{I}$. Accordingly $
		\max_{I \in \mathcal{I}}\mathbb{E}[\Delta_I]^2 = o(|W|^{-4\beta/(2\beta +d)}).$
		Concerning the variance, using Corollary \ref{corol:cumulants_square_rvs}, we obtain
		\begin{align*}
			\operatorname{Var}[\Delta_I ] = \frac4{N^2} \sum_{1 \leq j_1, j_2 \leq N} (A + B + C),
		\end{align*}
		with 
		\begin{align*}
			&A := \kappa(\widehat{S}_{I}(k_{j_1}), \widehat{S}_{0}(\tilde{k_{j_1}}), \widehat{S}_{I}(k_{j_2}), \widehat{S}_{0}(\tilde{k_{j_2}})),
			\\&B := \mathbb{E}[\widehat{S}_{I}(k_{j_1}) - S_{\frac12}(k_{j_1})]\kappa(\widehat{S}_{0}(\tilde{k_{j_1}}), \widehat{S}_{I}(k_{j_2}), \widehat{S}_{0}(\tilde{k_{j_2}})) 
			 + \mathbb{E}[\widehat{S}_{0}(\tilde{k_{j_1}}) - S_{\frac12}(\tilde{k_{j_1}})]\kappa(\widehat{S}_{I}(k_{j_1}), \widehat{S}_{I}(k_{j_2}), \widehat{S}_{0}(\tilde{k_{j_2}}))
			\\& \quad + \mathbb{E}[\widehat{S}_{I}(k_{j_2}) - S_{\frac12}(\tilde{k_{j_2}})]\kappa(\widehat{S}_{I}(k_{j_1}), \widehat{S}_{0}(\tilde{k_{j_1}}), \widehat{S}_{0}(\tilde{k_{j_2}}))
			- \mathbb{E}[\widehat{S}_{0}(\tilde{k_{j_2}}) - S_{\frac12}(\tilde{k_{j_2}})]\kappa(\widehat{S}_{I}(k_{j_1}), \widehat{S}_{0}(\tilde{k_{j_1}}), \widehat{S}_{I}(k_{j_2})),
			\\& C := \kappa(\widehat{S}_{I}(k_{j_1}), \widehat{S}_{I}(k_{j_2}))\kappa(\widehat{S}_{0}(\tilde{k_{j_1}}),  \widehat{S}_{0}(\tilde{k_{j_2}})) + \kappa(\widehat{S}_{I}(k_{j_1}), \widehat{S}_{0}(\tilde{k_{j_2}}))\kappa(\widehat{S}_{0}(\tilde{k_{j_1}}),  \widehat{S}_{I}(k_{j_2})).
		\end{align*}
		According to Lemma \ref{lem:dd_22}, 
		\begin{align*}
			\frac4{N^2} \sum_{1 \leq j_1, j_2 \leq N} |C| \leq  K \left(a_1 +  \frac{\log(|W|)^{2d}}{|W|^2} + \mathcal{E} \right),
		\end{align*}
		with 
		\begin{align*}
			a_1 := \frac{1}{|I||I_p|} \left( \left(\frac{(|I_p| \vee |I|)^{\frac1{d}}}{|W|^{\frac1d} |k_0|}\right)^{d-1}+\frac1N\right).
		\end{align*}
		By \eqref{eq:E}, we have $
		\mathcal{E}  = o\left(|W|^{-4\beta/(2\beta +d)}\right).$ Moreover, we have $
		\log(|W|)^{2d}/|W|^2 =  o\left(|W|^{-4\beta/(2\beta +d)}\right).$
		Using the assumptions concerning the cardinal of the tapers index set of $\mathcal{I}$, we get
		\begin{align*}
			a_1 & \leq K\left( \frac1{(|I| \wedge |I_p|)^{1+1/d}} \left(\frac1{|k| |W|^{\frac1d}}\right)^{d-1}+ \frac1{N(|I| \wedge |I_p|)^2}\right) \\& \leq  K\left(\frac{A_{W}^{d-1-\eta(1+1/d)}}{|W|^{\frac{2\beta}{2\beta +d} \frac{d+1}{d}}} \left(\frac{1}{ A_W |W|^{\frac1d - \frac1{2\beta +d}} }\right)^{d-1}+ \frac{A_{W}^{\frac{2d(d-1)}{d+1}-\eta(1+1/d)}}{N|W|^{4\beta/(2\beta +d)}}\right) \\& \leq KA_{W}^{-\eta(1+1/d)} |W|^{-\frac{4\beta}{2\beta +d}} = o\left(|W|^{-\frac{4\beta}{2\beta +d}}\right),
		\end{align*}
		since $N \geq A_{W}^{2d(d-1)/(d+1)}$ and $A_W \to \infty$ as $|W| \to \infty$.
		According to Lemma \ref{lem:dd_13}, 
		\begin{align*}
			\frac4{N^2} \sum_{1 \leq j_1, j_2 \leq N} |B|\leq  K\left(b(I, I_p) + b(I_p, I) + b+ \mathcal{E}(|I|, |I_p|)\right),
		\end{align*}
		where $b(I, I_p), b(I_p, I)$ and $b$ are defined in the statement of Lemma \ref{lem:dd_13}. For $b$, we have, using the fact that $4\beta/(2\beta +d) \leq 4/3 < 5/3$ when $\beta \leq 2$ and $d \geq 2$:
		\begin{align*}
			b &= \frac{\log(|W|)^d}{|W|^2 |k_0|^d} + \frac{\log(|W|)^d}{|W|^{\frac{5}3}} +  \frac1{|k_0|^{2d} |W|^2} \\& = \frac{\log(|W|)^d |W|^{\frac{d}{2\beta +d}}}{|W|^2 A_W^d}+ \frac{|W|^{\frac{2d}{2\beta +d}}}{A_W^{2d} |W|^2} + o(|W|^{-\frac{4\beta}{2\beta+d}}) \\& = \frac{\log(|W|)^d}{|W|^{\frac{d}{2\beta +d}}A_W^d} |W|^{-\frac{4\beta}{2\beta +d}} + A_W^{-2d}|W|^{-\frac{4\beta}{2\beta +d}} + o(|W|^{-\frac{4\beta}{2\beta+d}}) = o(|W|^{-\frac{4\beta}{2\beta+d}}).
		\end{align*}
		We bound $b(I, I_p)$. We have $
		b(I, I_p) \leq K(b_1 + b_2),$
		with 
		\begin{align*}
			&b_1 := \left(\frac{|I|}{|W|}\right)^{\frac{\beta}d} |I_p|^{\frac12} \frac{1}{|I| |I_p|} \left(\frac{|I_p|^{\frac1d}}{|W|^{\frac1d}|k_0|}\right)^{d-1},
			\\& b_2 := \left(\frac{|I|}{|W|}\right)^{\frac{\beta}d} \left(\frac1{|W| \log(|W|)^{\eta}} + \frac{1}{N |I| |I_p|^{\frac12}}  \right).
		\end{align*}
		Using the computations for $a_1$ and
		\begin{align*}
			\left(\frac{|I|}{|W|}\right)^{\frac{\beta}d} |I_p|^{\frac12} \leq K |W|^{-\frac{\beta}d \frac{d}{2\beta +d}} |W|^{\frac{\beta}{2\beta +d}} = K,
		\end{align*}
		we obtain $b_1 = o(|W|^{-4\beta/(2\beta+d)}).$ Concerning $b_2$, we have
		\begin{align*}
			b_2 \leq \frac{K}{|W|^{\frac{3\beta +d}{2\beta +d}}\log(|W|)^{\eta}} + \frac{K}{|W|^{\frac{\beta}{2\beta +d}}} \frac{A_W^{\frac32 \frac{d(d-1)}{d+1}}}{A_W^{\frac{2d(d-1)}{d+1}}|W|^{\frac32 \frac{2\beta}{2\beta +d}}} = o(|W|^{-\frac{4\beta}{2\beta+d}}),
		\end{align*}
		since $\beta \leq d$ and   $A_W \to \infty$ as $|W| \to \infty$. The bound for $b(I_p, I)$ is obtained by exchanging $I$ and $I_p$. It remains to consider $A$. According to Lemma \ref{lem:dd_4}:
		\begin{align*}
			\frac4{N^2}\sum_{1 \leq j_1, j_2 \leq N} |A| \leq K\left(a_1 +b\right) +  \mathcal{E}.
		\end{align*}
		We have already proved that  $a_1+b+\mathcal{E} = o(|W|^{-4\beta/(2\beta+d)})$. 
		Since all the previous $o$ are uniform over $I \in \mathcal{I}$, we have obtained $
		\max_{I \in \mathcal{I}}\operatorname{Var}[\Delta_I]^2 = o(|W|^{-4\beta/(2\beta +d)}).$
		This concludes the proof.
	\end{proof}
	
	\section{Technical results concerning joint cumulants of linear statistics of point processes}\label{sec:lem_lin_stats}
	
	This section gathers bounds on the joint cumulants of linear statistics of point processes. Since Section~\ref{sec:dd} requires handling joint statistics involving thinnings of a given point process, the results presented here are established in this general framework. The structure of the section is as follows. Lemma \ref{lem:cov_thin} is related to second order cumulants. Lemma~\ref{lem:bril_thin} provides a general but coarse bound on high-order cumulants, without using the orthogonality of the tapers or any decorrelation between frequencies. Lemmas~\ref{lem:bessel_k2} and~\ref{lem:bessel_km} exploit the orthogonality of the tapers, while Lemmas~\ref{lem:decor_freq_2} and~\ref{lem:decor_freq_3} address decorrelation between frequencies.

	\begin{lemma}\label{lem:cov_thin}
		Let $\Phi$ be a stationary point process with intensity $\lambda$, pair correlation function $g$ and structure factor $S$. Let $p \in [0, 1]$. Let $\Phi_p$ and $\Phi_{p}'$ be two independent $p$-thinning of $\Phi$. We denote $\Phi_{1- p} := \Phi \setminus \Phi_p$ and $\Phi_{1- p}' := \Phi \setminus \Phi_{p}'$. Let $f_1, f_2 \in L^1(\mathbb{R}^d) \cap L^2(\mathbb{R}^d).$ Then, 
		
		\begin{align*}
			&a := \operatorname{Cov}\left[\sum_{x \in \Phi_p} f_1(x), \sum_{x \in \Phi_{p}} f_2(x)\right] = \lambda p \int_{\mathbb{R}^d} \mathcal{F}[f_1](k) \overline{\mathcal{F}[f_2](k)} (1-p + p S(k)) dk,
			\\&b := \operatorname{Cov}\left[\sum_{x \in \Phi_p} f_1(x), \sum_{x \in \Phi_{1 - p}} f_2(x)\right] = (2\pi)^{\frac{d}2}\lambda^2 p (1-p) \int_{\mathbb{R}^d} \mathcal{F}[f_1](k) \overline{\mathcal{F}[f_2](k)} \mathcal{F}[g-1](k) dk,
			\\&c:=\operatorname{Cov}\left[\sum_{x \in \Phi_p} f_1(x), \sum_{x \in \Phi_{p}'} f_2(x)\right] = \lambda^2 p^2 \int_{\mathbb{R}^d} \mathcal{F}[f_1](k) \overline{\mathcal{F}[f_2](k)} S(k) dk,
			\\&d:=\operatorname{Cov}\left[\sum_{x \in \Phi_p} f_1(x), \sum_{x \in \Phi_{1-p}'} f_2(x)\right] = \lambda^2 p (1-p) \int_{\mathbb{R}^d} \mathcal{F}[f_1](k) \overline{\mathcal{F}[f_2](k)} S(k) dk.
		\end{align*}
	\end{lemma}
	\begin{proof}
		The proof of this lemma proceeds by conditioning on $\Phi$ and applying the law of total covariance. The argument for $a$, which is a standard result, follows the same steps as those used for $b$, $c$, and $d$, and is therefore omitted. We focus on the latter three, which are less classical.
		Note that
		\begin{align*}
			b := \operatorname{Cov}\left[\sum_{x \in \Phi} B_x f_1(x), \sum_{x \in \Phi} (1- B_x)f_2(x)\right],
		\end{align*}
		for i.i.d. Bernouilli random variables $(B_x)_{x \in \Phi}$ of parameter $p$. Using the law of the total covariance, we obtain
		\begin{align*}
			b &= \mathbb{E}\left[\operatorname{Cov}\left[\sum_{x \in \Phi} B_x f_1(x), \sum_{x \in \Phi} (1- B_x)f_2(x)| \Phi\right]\right] + \operatorname{Cov}\left[ \mathbb{E}\left[\sum_{x \in \Phi} B_xf_1(x)|\Phi\right],  \mathbb{E}\left[\sum_{x \in \Phi} (1 -B_x) f_2(x)|\Phi\right] \right] \\& = \mathbb{E}\left[\sum_{x, y \in \Phi} f_1(x) \overline{f_2(x)} \operatorname{Cov}(B_x, 1 - B_y)\right] + p(1 - p) \operatorname{Cov}\left[ \sum_{x \in \Phi} f_1(x), \sum_{x \in \Phi} f_2(x) \right].
		\end{align*}
		According to the Plancherel theorem and using \eqref{e.prop_campbell}, we have
		\begin{align*}
			b &=(2\pi)^{\frac{d}2} \lambda p(1-p) \left(-  \int_{\mathbb{R}^d} \mathcal{F}[f_1](k) \overline{\mathcal{F}[f_2](k)} dk +  \int_{\mathbb{R}^d} \mathcal{F}[f_1]\overline{\mathcal{F}[f_2]} (1+ \lambda \mathcal{F}[g-1]) \right) \\& =  (2\pi)^{\frac{d}2}\lambda^2 p (1-p) \int_{\mathbb{R}^d} \mathcal{F}[f_1](k) \overline{\mathcal{F}[f_2](k)} \mathcal{F}[g-1](k) dk,
		\end{align*}
		which concludes the proof for this term. Concerning $c$, we have
		\begin{align*}
			c := 	\operatorname{Cov}\left[\sum_{x \in \Phi} B_x f_1(x), \sum_{x \in \Phi} B_x'f_2(x)\right],
		\end{align*}
		for i.i.d. Bernouilli random variables $(B_x, B_x')_{x \in \Phi}$ of parameter $p$. In the law of the total covariance, the first term vanishes due to the independence between $(B_x)_{x \in \Phi}$ and $(B_x')_{x \in \Phi}$. Hence
		\begin{align*}
			c = p^2 \operatorname{Cov}\left[\sum_{x \in \Phi} f_1(x), \sum_{x \in \Phi} f_2(x)\right],
		\end{align*}
		which gives the result for $c$ using \eqref{e.prop_campbell}. The proof of $d$ consists in considering i.i.d. Bernouilli $(B_x')_{x \in \Phi}$ of parameter $1-p$ instead of $p$.
	\end{proof}
	
	\begin{lemma}\label{lem:bril_thin}
		Let $\Phi$ be a stationary point process with intensity $\lambda$ and $p \in [0, 1]$.  Let $m \geq 2$. Denote by $(\Phi_p^j)_{1 \leq j \leq m}$ $p$-independent-thinnings of $\Phi$, that are independent between them. Moreover, let $\Phi_{1- p}^j = \Phi \setminus \Phi_p^j$. Consider $(f_i)_{i = 1, \dots, m} \in L^1(\mathbb{R}^d) \cap L^2(\mathbb{R}^d)$. Let $\tau: [m] \to \{p, 1-p\}$ and $\varphi : [m] \to [m]$. Then, 
		\begin{align*}
			\left|\kappa\left(\sum_{x \in \Phi_{\tau(s)}^{\varphi(s)}} f_s(x); s \in [m]\right)\right| \leq  K\prod_{s = 1}^m \|f_s\|_m,
		\end{align*}
		where $K < \infty$ depends only on $p$, $\tau$, $m$, $\lambda$ and $|\gamma_{\text{red}}^{(n)}|(\mathbb{R}^{d(n-1)})$ for $2 \leq n \leq m$.  In particular if $p = 1$, $\tau \equiv 1$ and $\varphi \equiv 1$, we have $K = \sum_{\pi \in \Pi[m]} |\gamma_{\text{red}}^{(|\pi|)}|(\mathbb{R}^{d(|\pi|-1)})$, with the convention $|\gamma_{\text{red}}^{(1)}|(\mathbb{R}^{0}) = \lambda$.
	\end{lemma}
	\begin{proof}
		Let $(B_x^j)_{x \in \Phi, j \in [m]}$ be i.i.d. Bernouilli random variables of parameter $p$. We use the notation, $\tau_s(y) = y \mathds{1}_{\tau_s = p} + (1-y) \mathds{1}_{\tau_s = 1- p}$. According to the law of the total cumulance, we have
		\begin{align*}
			\kappa &:= \kappa\left(\sum_{x \in \Phi_{\tau(s)}^{\varphi(s)}} f_s(x); s \in [m]\right)= \sum_{\pi \in \Pi[m]} \kappa\left(\kappa\left(\sum_{x \in \Phi} \tau_s(B_x^{\varphi(s)})f_s(x)|\Phi; s \in b\right); b \in \pi\right) \\& = \sum_{\pi \in \Pi[m]} \kappa\left(  \sum_{\substack{x_s \in \Phi;\\ s \in b}} \kappa(\tau_s(B_{x_s}^{\varphi(s)}); s \in b) f_s(x_s); b \in \pi\right).
		\end{align*}
		Using the independence of $(B_x^j)_{x \in \Phi}$, for $j \in [m]$, we obtain
		$$\sum_{\substack{x_s \in \Phi;\\ s \in b}} \kappa(\tau_s(B_{x_s}^{\varphi(s)}); s \in b) f_s(x_s) = \mathds{1}_{\#\{\varphi(s)| s \in b\} = 1}\sum_{\substack{x_s \in \Phi;\\ s \in b}} \kappa(\tau_s(B_{x_s}^{\varphi(s_0)}); s \in b) f_s(x_s),$$
		where $s_0 = \min(b)$. Moreover, we have
		\begin{align*}
			\kappa(\tau_s(B_{x_s}^{\varphi(s_0)}); s \in b) &=  (-1)^{\#\{\tau_s = 1| s \in b\}}\kappa(B_{x_s}^{\varphi(s_0)}; s \in b) = (-1)^{\#\{\tau_s = 1| s \in b\}} \mathds{1}_{\#\{x_s|s \in b\}=  1}\kappa_{|b|}(B),
		\end{align*}
		where $B$ is a Bernouilli random variable of parameter $p$. This yields
		\begin{align}\label{eq:kappa_thin_interm}
			\kappa & = \sum_{\pi \in \Pi[m]} A \times B,
		\end{align}	
		with \begin{align*}
			&A:= (-1)^{\#\{\tau_s = 1| s \in [m]\}} \prod_{b \in \pi}  \mathds{1}_{\#\{\varphi(s)| s \in b\} = 1} \kappa_{|b|}(B),
			\quad B := \kappa\left(\sum_{x \in \Phi} \prod_{s \in b} f_s(x); b \in \pi\right).
		\end{align*}
		The quantity $A$ does not depend on the functions $(f_i)_{1 \leq i \leq m}$. Accordingly, it remains to bound $B$. Denote $g_b(x) := \prod_{s \in b} f_s(x)$. Then, 
		\begin{align*}
			B = \sum_{\pi' \in \Pi\{b;~b \in \pi\}} \int_{\mathbb{R}^{d|\pi'|}} \prod_{b' \in \pi'} \prod_{b \in b'} g_b(x_{b'}) \gamma^{(|\pi'|)}(x_{b'}; b' \in \pi') =: \sum_{\pi' \in \Pi\{b;~b \in \pi\}} B_{\pi'},
		\end{align*}
		where $ \Pi\{b;~b \in \pi\}$ denotes the set of partitions of $\{b;~b \in \pi\}$. If $|\pi'| = 1$, we have using the Hölder's inequality,
		\begin{align*}
			|B_{\pi'}| \leq \lambda \int_{\mathbb{R}^{d}} \prod_{i \in [m]} |f_{i}(x)| dx \leq \lambda \prod_{s = 1}^m \|f_s\|_m. 
		\end{align*}
		We use the same argument when $|\pi| > 1$. Let $b_0' \in \pi'$ be a fixed block. By stationarity, and using again the Hölder's inequality, we obtain:
		\begin{align}\label{eq:bound_kappa_thin}
			|B_{\pi'}| & \leq \int_{\mathbb{R}^{d(|\pi'| -1)}}  \int_{\mathbb{R}^{d}}  \prod_{b \in b_0'} |g_{b}(x_{b_0'})| \prod_{b' \in \pi' \setminus\{b_0'\}} \prod_{b \in b'} |g_{b}(x_{b'}+x_{b_0'})|dx_{b_0'} |\gamma_{red}^{(|\pi'|)} |(dx_b'; b \in \pi \setminus\{b_0'\}) \nonumber \\& \leq \int_{\mathbb{R}^{d(|\pi'| -1)}}   \prod_{b \in b_0'} \prod_{s \in b} \|f_s\|_{m} \prod_{b' \in \pi' \setminus\{b_0'\}} \prod_{b \in b'} \prod_{s \in b} \|f_{s}(\cdot+x_b')\|_{m} |\gamma_{red}^{(|\pi'|)} |(dx_b'; b' \in \pi' \setminus\{b_0'\})  \nonumber\\& = |\gamma_{red}^{(|\pi|)} |(\mathbb{R}^{d(|\pi'|-1)}) \prod_{s = 1}^m \|f_s\|_m.
		\end{align}
		This concludes the proof. To obtain the value of the constant when $p =1$, note that in that case $A$ is non zero only if all the blocks of $\pi$ are of size $1$.
	\end{proof}
	
	\begin{lemma}\label{lem:bessel_k2}
		Let $\Phi$ be a stationary point process with intensity $\lambda$ and pair correlation function $g$ and $p \in [0, 1]$. Denote by $(\Phi_p^j)_{j \in [2]}$ two $p$-independent-thinnings of $\Phi$, that are independent between them. Denote $\Phi_{1- p}^j = \Phi \setminus \Phi_p^j$. Let $I_1$ and $I_2$ be two discrete subsets of $\mathbb{N}^d$. Let $(f_i)_{i \in\mathbb{N}^d}$ and $(h_i)_{i \mathbb{N}^d}$ be two orthogonal families of $L^2(\mathbb{R}^d)$. Let $\tau [2] \mapsto \{p, 1-p\}^2$ and $\varphi: [2] \mapsto [2]$. Then, we have for all set $W \subset \mathbb{R}^d$,
		\begin{multline*}
			\frac1{|I_1| |I_2|} \sum_{\substack{i_1 \in I_1 \\ i_2 \in I_2}} \left|\operatorname{Cov}\left(\sum_{x \in \Phi_{\tau(1)}^{\varphi(1)} \cap W} f_{i_1}(x), \sum_{x \in \Phi_{\tau(2)}^{\varphi(2)} \cap W} h_{i_2}(x)\right)\right|^2  \leq \\ \frac{K}{|I_1| \vee |I_2|} \left(1+ \mathds{1}_{|I_1| \geq |I_2|}\left(\sum_{\substack{i_1 \in I_1}} \|f_{i_1} \mathds{1}_{\mathbb{R}^d \setminus W}\|_2^2\right)^{\frac12} + \mathds{1}_{|I_1| < |I_2|}\left(\sum_{\substack{i_2 \in I_2}} \|h_{i_2} \mathds{1}_{\mathbb{R}^d \setminus W}\|_2^2\right)^{\frac12}\right)^2.
		\end{multline*}
		where $K < \infty$ depends only on $p$, $\tau$, $\varphi$, $\lambda$, $\|S\|_{\infty}$ and $\|g-1\|_1$. In particular if $p = 1$, $\tau \equiv 1$ and $\varphi \equiv 1$, we have $K = (\lambda \|S\|_{\infty})^2$. 
	\end{lemma}
	\begin{proof}
		According to Lemma \ref{lem:cov_thin}, we have
		\begin{align*}
			\operatorname{Cov}\left(\sum_{x \in \Phi_{\tau(1)}^{\varphi(1)} \cap W} f_{i_1}(x), \sum_{x \in \Phi_{\tau(2)}^{\varphi(2)} \cap W} h_{i_2}(x)\right) = \int_{\mathbb{R}^d} \mathcal{F}[f_{i_1} \mathds{1}_W](k) \overline{\mathcal{F}[h_{i_2} \mathds{1}_W]}(k) s(k) dk,
		\end{align*}
		where $s$ is bounded. Assume that $|I_1| \geq |I_2|$. We use the decomposition
		\begin{align*}
			I_{i_1, i_2} := \int_{\mathbb{R}^d} \mathcal{F}[f_{i_1} \mathds{1}_W](k) \overline{\mathcal{F}[h_{i_2} \mathds{1}_W]}(k) s(k) dk = \int_{\mathbb{R}^d} \mathcal{F}[f_{i_1}](k) \overline{\mathcal{F}[\mathds{1}_W h_{i_2}]}(k) s(k) dk + \delta,
		\end{align*}
		where 
		\begin{align*}
			\delta = \int_{\mathbb{R}^d} \mathcal{F}[f_{i_1} \mathds{1}_{\mathbb{R}^d \setminus W}](k) \overline{\mathcal{F}[h_{i_2} \mathds{1}_W]}(k) s(k) dk.
		\end{align*}
		Using the Cauchy-Schwarz inequality and the Plancherel theorem, we get
		\begin{align*}
			|\delta| \leq \|s\|_{\infty} \|f_{i_1} \mathds{1}_{\mathbb{R}^d \setminus W}\|_2 \|h_{i_2} \mathds{1}_W\|_2 \leq  \|s\|_{\infty} \|f_{i_1} \mathds{1}_{\mathbb{R}^d \setminus W}\|_2.
		\end{align*}
		Accordingly, using the fact that $(\mathcal{F}[f_i])_{i \in \mathbb{N}^d}$ is orthogonal, we obtain 
		\begin{align*}
			\left(\sum_{\substack{i_1 \in I_1 \\ i_2 \in I_2}} |I_{i_1, i_2}|^{2}\right)^{\frac12} &\leq \left(\sum_{\substack{i_1 \in I_1 \\ i_2 \in I_2}} |\langle \mathcal{F}[f_{i_1}], \mathcal{F}[\mathds{1}_W h_{i_2}] s\rangle|^2\right)^{\frac12} + \left(\sum_{\substack{i_1 \in I_1 \\ i_2 \in I_2}} \delta^2\right)^{\frac12} \\& \leq \left(\sum_{\substack{i_2 \in I_2}} \|\mathcal{F}[\mathds{1}_W h_{i_2}] s\|_2^2\right)^{\frac12} + |I_2|^{\frac12} \|s\|_{\infty} \left(\sum_{\substack{i_1 \in I_1}} \|f_{i_1} \mathds{1}_{\mathbb{R}^d \setminus W}\|_2^2\right)^{\frac12} \\& \leq |I_2|^{\frac12} \|s\|_{\infty}\left(1+ \left(\sum_{\substack{i_1 \in I_1}} \|f_{i_1} \mathds{1}_{\mathbb{R}^d \setminus W}\|_2^2\right)^{\frac12}\right).
		\end{align*}
		This yields
		\begin{align*}
			\frac1{|I_1| |I_2|}\sum_{\substack{i_1 \in I_1 \\ i_2 \in I_2}} |I_{i_1, i_2}|^{2} \leq \frac{\|s\|_{\infty}^2}{|I_1| \vee |I_2|} \left(1+ \left(\sum_{\substack{i_1 \in I_1}} \|f_{i_1} \mathds{1}_{\mathbb{R}^d \setminus W}\|_2^2\right)^{\frac12}\right)^2.
		\end{align*}
		The proof for the case $|I_1| < |I_2|$ follows by symmetry.
	\end{proof}
	
	\begin{lemma}\label{lem:decor_freq_2}
		Let $\Phi$ be a stationary point process with intensity $\lambda$ and pair correlation function $g$ and $p \in [0, 1]$. Denote by $(\Phi_p^j)_{j \in [2]}$ two $p$-independent-thinnings of $\Phi$, that are independent between them. Denote $\Phi_{1- p}^j = \Phi \setminus \Phi_p^j$. Let $f_1, f_2$ be in $L^1(\mathbb{R}^d) \cap L^2(\mathbb{R}^d)$ with $\|f_1\|_2 = \|f_2\|_2 = 1$ and $k_1, k_2 \in \mathbb{R}^d$. Let $\tau [2] \mapsto \{p, 1-p\}^2$ and $\varphi: [2] \mapsto [2]$. Then, we have for all set $W \subset \mathbb{R}^d$,
		\begin{multline*}
			\left|\operatorname{Cov}\left(\sum_{x \in \Phi_{\tau(1)}^{\varphi(1)} \cap W} f_1(x) e^{- \bm{i} k_1 \cdot x}, \sum_{x \in \Phi_{\tau(2)}^{\varphi(2)} \cap W} f_2(x) e^{- \bm{i} k_2 \cdot x}\right)\right| \\ \leq  K   \sum_{s = 1}^2 \|f_s\mathds{1}_{\mathbb{R}^d \setminus W}\|_2+ \|\mathcal{F}[f_s] \mathds{1}_{|\cdot| \geq |k_1 - k_2|/2}\|_{2},
		\end{multline*}
		where $K < \infty$ depends only on $p$, $\tau$, $\varphi$, $\lambda$, $\|S\|_{\infty}$ and $\|g-1\|_1$. In particular if $p = 1$, $\tau \equiv 1$ and $\varphi \equiv 1$, we have $K = \lambda \|S\|_{\infty}$. 
	\end{lemma}
	\begin{proof}
		According to Lemma \ref{lem:cov_thin}, we have
		\begin{multline*}
			\operatorname{Cov}\left(\sum_{x \in \Phi_{\tau(1)}^{\varphi(1)} \cap W} f_1(x) e^{- \bm{i} k_1 \cdot x}, \sum_{x \in \Phi_{\tau(2)}^{\varphi(2)} \cap W} f_2(x) e^{- \bm{i} k_2 \cdot x}\right) \\ = \int_{\mathbb{R}^d} \mathcal{F}[f_1 \mathds{1}_W  e^{- \bm{i} k_1 \cdot}](k) \overline{\mathcal{F}[f_2 \mathds{1}_W  e^{- \bm{i} k_2 \cdot}]}(k) s(k) dk,
		\end{multline*}
		where $s$ is bounded. The first step is to remove the two indicator functions $\mathds{1}_W$. We have
		\begin{align*}
			\int_{\mathbb{R}^d} \mathcal{F}[f_1 \mathds{1}_W  e^{- \bm{i} k_1 \cdot}](k) \overline{\mathcal{F}[f_2 \mathds{1}_W  e^{- \bm{i} k_2 \cdot}]}(k) s(k) dk = \int_{\mathbb{R}^d} \mathcal{F}[f_1  e^{- \bm{i} k_1 \cdot}](k) \overline{\mathcal{F}[f_2  e^{- \bm{i} k_2 \cdot}]}(k) s(k) dk + \delta,
		\end{align*}
		with 
		\begin{multline*}
			\delta = -\int_{\mathbb{R}^d} \mathcal{F}[f_1 \mathds{1}_{\mathbb{R}^d\setminus W}  e^{- \bm{i} k_1 \cdot}](k) \overline{\mathcal{F}[f_2 e^{- \bm{i} k_2 \cdot}]}(k) s(k) dk \\ -\int_{\mathbb{R}^d} \mathcal{F}[f_1 \mathds{1}_W  e^{- \bm{i} k_1 \cdot}](k) \overline{\mathcal{F}[f_2 \mathds{1}_{\mathbb{R}^d\setminus W} e^{- \bm{i} k_2 \cdot}]}(k) s(k) dk.
		\end{multline*}
		By Cauchy-Schwarz inequality, we have
		\begin{align*}
			|\delta| &\leq \|s\|_{\infty} \left( \|\mathcal{F}[f_1 \mathds{1}_{\mathbb{R}^d\setminus W}] \|_2 \|\mathcal{F}[f_2]\|_2 +  \| \mathcal{F}[f_1 \mathds{1}_{W}] \|_2 \|\mathcal{F}[f_2] \mathds{1}_{\mathbb{R}^d\setminus W}\|_2 \right) \\& \leq \|s\|_{\infty} \left( \|\mathcal{F}[f_1 \mathds{1}_{\mathbb{R}^d\setminus W}] \|_2 + \|\mathcal{F}[f_2] \mathds{1}_{\mathbb{R}^d\setminus W}\|_2\right).
		\end{align*}
		It remains to upper bound
		\begin{align*}
			I := \int_{\mathbb{R}^d} \mathcal{F}[f_1  e^{- \bm{i} k_1 \cdot}](k) \overline{\mathcal{F}[f_2  e^{- \bm{i} k_2 \cdot}]}(k) s(k) dk = \int_{\mathbb{R}^d} \mathcal{F}[f_1](k) \overline{\mathcal{F}[f_2]}(k+k_2-k_1) s(k-k_1) dk.
		\end{align*}
		We split $\mathbb{R}^d = A_1\sqcup A_2$ with $A_1 := \{k \in \mathbb{R}^d|~|k - k_1 + k_2| \leq |k_1-k_2|/2\}$ and $A_2 := \mathbb{R}^d \setminus A_1$. Since for $k \in A_1$, we have 
		$|k - k_1 + k_2| \geq |k_1-k_2| - |k| \geq |k_1 - k_2|/2,$
		we obtain using the Cauchy-Schwarz's inequality and the Plancherel theorem
		\begin{align*}
			|I| \leq \|s\|_{\infty} \left(\|\mathcal{F}[f_1] \mathds{1}_{|\cdot| \geq |k_1 - k_2|/2}\|_2 + \|\mathcal{F}[f_2] \mathds{1}_{|\cdot| \geq |k_1 - k_2|/2}\|_2\right),
		\end{align*}
		which concludes the proof.
	\end{proof}
	
	\begin{lemma}\label{lem:decor_freq_3}
		Let $\Phi$ be a stationary point process with intensity $\lambda$ and pair correlation function $g$ and $p \in [0, 1]$. Denote by $(\Phi_p^j)_{j \in [3]}$ three $p$-independent-thinnings of $\Phi$, that are independent between them. Denote $\Phi_{1- p}^j = \Phi \setminus \Phi_p^j$. Let $f_1, f_2, f_3$ be $L^1(\mathbb{R}^d) \cap L^{\infty}(\mathbb{R}^d)$ functions and $k_1, k_2, k_3 \in \mathbb{R}^d$. Let $\tau [3] \mapsto \{p, 1-p\}^3$ and $\varphi: [3] \mapsto [3]$. Then, we have for all set $W \subset \mathbb{R}^d$,
		\begin{multline}\label{eq:lem_kappa_3_2}
			\left|\kappa\left(\sum_{x \in \Phi_{\tau(s)}^{\varphi(s)} \cap W} f_s(x) e^{- \bm{i} k_s \cdot x}; s \in [3]\right)\right| \leq  K   \sum_{s = 1}^3 \|f_s\mathds{1}_{\mathbb{R}^d \setminus W}\|_3 \prod_{s' \in [3]\setminus\{s\}} \|f_{s'}\|_3  \\ + K \sum_{s = 1}^3 \|\mathcal{F}[f_s] \mathds{1}_{|\cdot| \geq |k_1 + k_2 + k_3|/3}\|_{1}  \prod_{s' \in [3] \setminus \{s\}} \|f_{s'}\|_{2},
		\end{multline}
		where $K < \infty$ depends only on $p$, $\tau$, $\varphi$, $\lambda$ and $|\gamma_{\text{red}}^{(m)}|(\mathbb{R}^{d(m-1)})$ for $m \in \{2, 3\}$. In particular if $p = 1$, $\tau \equiv 1$ and $\varphi \equiv 1$, we have $K = \sum_{\pi \in \Pi[3]} |\gamma_{\text{red}}^{(|\pi|)}|(\mathbb{R}^{d(|\pi|-1)})$, with the convention $|\gamma_{\text{red}}^{(1)}|(\mathbb{R}^{0}) = \lambda$.
	\end{lemma}
	\begin{proof}
		Denote by $|\kappa_W|$ the left hand-side of \eqref{eq:lem_kappa_3_2}. According to \eqref{eq:kappa_thin_interm} for $m = 3$, we have
		\begin{align*}
			\kappa_W & = \sum_{\pi \in \Pi[3]}  a_{\pi} \sum_{\pi' \in \Pi\{b;~b \in \pi\}} \int_{\mathbb{R}^{d|\pi'|}} \prod_{b' \in \pi'} \prod_{b \in b'} \prod_{s \in b} \mathds{1}_W(x_{b'})f_s(x_{b'}) \gamma^{(|\pi'|)}(x_{b'}; b' \in \pi'),
		\end{align*}
		with, denoting by $B$ a Bernouilli random variable of parameter $p$,
		$$a_{\pi} := (-1)^{\#\{\tau_s = 1| s \in [m]\}} \prod_{b \in \pi}  \mathds{1}_{\#\{\varphi(s)| s \in b\} = 1} \kappa_{|b|}(B).$$
		The first step is to remove the truncation to $W$. Note that
		$$\mathds{1}_{\mathbb{R}^{d |\pi'|} \setminus W^{|\pi'|}}(x_b'; x_b' \in \pi') \leq \sum_{b_0' \in \pi'} \mathds{1}_{\mathbb{R}^d \setminus W}(x_{b_0'}).$$
		Denote $\Delta := 	|\kappa_W - \kappa_{\mathbb{R}^d}|$. Subsequently,
		\begin{align*}
			\Delta \leq  \sum_{\pi \in \Pi[3]}  |a_{\pi}| \sum_{\pi' \in \Pi\{b;~b \in \pi\}} \sum_{b_0' \in \pi'} \int_{\mathbb{R}^{d|\pi'|}} \mathds{1}_{\mathbb{R}^d \setminus W}(x_{b_0'}) \prod_{b' \in \pi'} \prod_{b \in b'} \prod_{s \in b} |f_s(x_{b'})| \gamma^{(|\pi'|)}(x_{b'}; b' \in \pi').
		\end{align*}
		By stationarity and using the Hölder's inequality, as in \eqref{eq:bound_kappa_thin}, we obtain
		\begin{align*}
			\Delta \leq \sum_{\pi \in \Pi[3]} |a_{\pi}| \sum_{\pi' \in \Pi\{b;~b \in \pi\}} |\gamma_{red}^{(|\pi'|)} |(\mathbb{R}^{d(|\pi'|-1)}) \sum_{b_0' \in \pi'}  \prod_{i \in b_0'} \|f_{i} \mathds{1}_{\mathbb{R}^d \setminus W}\|_3 \prod_{i \notin b_0'} \|f_{i}\|_3.  
		\end{align*}
		Since
		\begin{align*}
			\prod_{i \in b_0'} \|f_{i} \mathds{1}_{\mathbb{R}^d \setminus W}\|_3  \prod_{i \notin b_0'} \|f_{i}\|_3 \leq \sum_{s = 1}^3 \|f_s\mathds{1}_{\mathbb{R}^d \setminus W}\|_3 \prod_{s' \in [3]\setminus\{s\}} \|f_{s'}\|_3,
		\end{align*}
		we have
		\begin{align*}
			\Delta \leq \left( \sum_{\pi \in \Pi[3]} |a_{\pi}| \sum_{\pi' \in \Pi\{b;~b \in \pi\}} |\gamma_{red}^{(|\pi'|)} |(\mathbb{R}^{d(|\pi'|-1)}) |\pi'|\right)  \sum_{s = 1}^3 \|f_s\mathds{1}_{\mathbb{R}^d \setminus W}\|_3 \prod_{s' \in [3]\setminus\{s\}} \|f_{s'}\|_3.
		\end{align*}
		Now, we focus on
		\begin{align*}
			\kappa_{\mathbb{R}^d} & = \sum_{\pi \in \Pi[3]}  a_{\pi} \sum_{\pi' \in \Pi\{b;~b \in \pi\}} \int_{\mathbb{R}^{d|\pi'|}} \prod_{b' \in \pi'} \prod_{b \in b'} \prod_{s \in b} f_s(x_{b'}) \gamma^{(|\pi'|)}(x_{b'}; b' \in \pi') \\& =: \sum_{\pi \in \Pi[3]}  a_{\pi} \sum_{\pi' \in \Pi\{b;~b \in \pi\}} I_{\pi'}.
		\end{align*}
		Each of the terms $I_{\pi'}$ is bounded using the same idea, but the details differ. We start with with the case $|\pi'| = 1$. Using the Fourier inversion theorem \cite{folland2009fourier} we get
		\begin{align*}
			|I_{\pi'}| &= \left|\int_{\mathbb{R}^{d}} \prod_{i = 1}^{3} f_{i}(x) e^{-\bm{i} k_i \cdot x} dx\right| \leq \int_{\mathbb{R}^{2d}} \left|\mathcal{F}[f_1](k) \mathcal{F}[f_2](k') \mathcal{F}[f_3](k_1+k_2 + k_3 - k - k') \right| dk dk' := I_{\infty}.
		\end{align*}
		Denote $\Sigma :=k_1 + k_2 + k_{3}$. Next, we split the integrals over the domains:
		\begin{align*}
			&D_1 := \{|k| \geq |\Sigma|/3\},~
			D_2 = \{|k| \leq |\Sigma|/3,~|k'| \geq |\Sigma|/3\},~
			D_3 := \{|k| \leq |\Sigma|/3,~|k'| \leq |\Sigma|/3\}.
		\end{align*}
		Note that on $D_3$, we have
		$|\Sigma - k - k'| \geq |\Sigma| - |k|-|k'| \geq |\Sigma|/3.$
		Hence, the triangular inequality and then the Young's inequality for the convolution yield
		\begin{align}\label{eq:I123}
			\left| I_{\pi'}\right| \leq I_{\infty} \leq \sum_{s = 1}^3 \|\mathcal{F}[f_s] \mathds{1}_{|\cdot| \geq |\Sigma|/3}\|_{1}  \prod_{s' \in [3] \setminus \{s\}} \|\mathcal{F}[f_{s'}\|_{2}.
		\end{align}
		For the other partitions, we adapt the previous technique, leveraging on the Brillinger-mixing condition. Define
		$$\forall k \in \mathbb{R}^{d(m-1)},~\mathcal{F}[\gamma_{\text{red}}^{(m)}](k):= \int_{\mathbb{R}^{d(m-1)}} e^{-\bm{i} k\cdot x} \gamma_{\text{red}}^{(m)}(dx).$$
		Note that $\|\mathcal{F}[\gamma_{\text{red}}^{(m)}]\|_{\infty} \leq |\gamma_{\text{red}}^{(m)}|(\mathbb{R}^{d(m-1)}) < \infty$ by Brillinger-mixing.  We consider the case when $|\pi'| = 3$. Using the Fourier inversion theorem and the Fubini's theorem, we have 
		\begin{align*}
			I_{\pi'} &= \int_{\mathbb{R}^{3d}} \prod_{i \in [3]} f_{i}(x_{i}) e^{-\bm{i} k_i \cdot x_i} \gamma^{(3)}(dx_1, dx_2, dx_3)\\& = \int_{\mathbb{R}^{3d}}  f_{1}(x_1) e^{-\bm{i} x_1 \cdot k_1} \prod_{s = 2, 3} f_{s}(x_1+x_{s}) e^{-\bm{i} (x_1 +x_s) \cdot k_s} dx_{1} \gamma_{red}^{(3)} (dx_2, dx_3) \\& = \int_{\mathbb{R}^{2d}} \mathcal{F}\left[f_{1} e^{-\bm{i} k_1 \cdot }\right]\left(k'_{2}+k'_3\right)  \prod_{s = 2, 3} \mathcal{F}\left[f_{s} e^{-\bm{i} k_s \cdot }\right](-k'_s) \mathcal{F}[\gamma_{\text{red}}^{(3)}](k'_2, k'_3) dk_2' dk_3'.
		\end{align*}
		By change of variables, we obtain
		\begin{align*}
			I_{\pi'}&=  \int_{\mathbb{R}^{2d}} \mathcal{F}\left[f_{1}\right]\left(k_{1} + k'_{1}+k'_{2}\right) \prod_{s = 2, 3} \mathcal{F}\left[f_{s}\right](k_{s}-k'_s) \mathcal{F}[\gamma_{\text{red}}^{(3)}](k'_2, k'_3) \\& = \int_{\mathbb{R}^{2d}} \mathcal{F}\left[f_{1}\right]\left(\sum_{s = 1}^3 k_{s} -k'_{1} - k'_2\right)  \prod_{s = 2, 3} \mathcal{F}\left[f_{s}\right](k'_s) \mathcal{F}[\gamma_{\text{red}}^{(3)}](-k'_1 + k_1, -k'_2 + k_2) dk_2' dk_3'.
		\end{align*}
		Accordingly, $
		|I_{\pi'}| \leq  |\gamma_{\text{red}}^{(3)}|(\mathbb{R}^{2d}) I_{\infty},$
		and we apply \eqref{eq:I123}. Finally, assume that $|\pi'| = 2$. By renumbering, it suffices to upper bound:
		\begin{align*}
			I_{\pi'} &= \int_{\mathbb{R}^{2d}} f_{1}(x) e^{-\bm{i} k_1 \cdot x} f_{2}(x) e^{-\bm{i} k_2 \cdot x} f_{3}(x') e^{-\bm{i} k_3 \cdot x'} \gamma^{(2)}(dx, dx') \\& = \int_{\mathbb{R}^{d}} \mathcal{F}\left[f_{1} e^{-\bm{i} k_1 \cdot } f_{2} e^{-\bm{i} k_2 \cdot }\right]\left(k\right) \mathcal{F}\left[f_{3} e^{-\bm{i} k_3 \cdot }\right](-k) \mathcal{F}[\gamma_{\text{red}}^{(2)}](k) dk \\& = \int_{\mathbb{R}^{2d}} \mathcal{F}[f_1](k'+k_1) \mathcal{F}[f_2](k + k_2 - k') \mathcal{F}[f_3](-k + k_3)\mathcal{F}[\gamma_{\text{red}}^{(2)}](k) dk dk',
		\end{align*}
		where we used the Fourier inversion theorem. By change of variables, we get $
		|I_{\pi'}| \leq |\gamma_{\text{red}}^{(2)}|(\mathbb{R}^{d})  I_{\infty}.
		$
		Recalling  \eqref{eq:I123}, this concludes the proof. To obtain the value of the constant when $p =1$, note that in that case $A$ is non zero only if all the blocks of $\pi$ are of size $1$.
	\end{proof}
	
	\begin{lemma}\label{lem:bessel_km}
		Let $\Phi$ be a stationary point process with intensity $\lambda$ and $p \in [0, 1]$. Let $m \geq 2$ and denote by $(\Phi_p^j)_{j \in [m]}$ three $p$-independent-thinnings of $\Phi$, that are independent between them. Denote $\Phi_{1- p}^j = \Phi \setminus \Phi_p^j$. Let $(f_i)_{i \in \mathbb{N}^d}, (h_s)_{s = 2, \dots, m}$ be $L^1(\mathbb{R}^d) \cap L^2(\mathbb{R}^d)$ functions. Assume that $(f_i)_{i \in \mathbb{N}^d}$ is orthogonal. Let $\tau [m] \mapsto \{p, 1-p\}^m$ and $\varphi: [m] \mapsto [m]$. Then, we have for all set $W \subset \mathbb{R}^d$,
		\begin{multline*}
			\sum_{i \in \mathbb{N}^d} \left|\kappa\left(\sum_{x \in \Phi_{\tau(1)}^{\varphi(1)} \cap W} f_i(x), \sum_{x \in \Phi_{\tau(s)}^{\varphi(s)} \cap W} h_s(x); 2 \leq s \leq m\right)\right|^2 \\ \leq K \prod_{s =2}^{m} \|h_2\|_{2(m-1)}^2 +  K    \sum_{i \in I} \sum_{s = 1}^m \|h_s\mathds{1}_{\mathbb{R}^d \setminus W}\|_m \prod_{s' \in [m] \setminus \{s'\}} \|h_s\|_m.
		\end{multline*}
		where, we use the notation $h_1 := f_i$ in the second line and where $K < \infty$ depends only on $p$, $\tau$, $\varphi$, $\lambda$ and $|\gamma_{\text{red}}^{(n)}|(\mathbb{R}^{d(m-1)})$ for $2 \leq n \leq m$.
	\end{lemma}
	\begin{proof}
		The first part of the proof is identical to the beginning of the proof of Lemma~\ref{lem:decor_freq_3}. Denote $h_1 = f_i$, $\Delta := | \kappa_W - \kappa_{\mathbb{R}^d}|$, where
		$$\kappa_W := \kappa\left(\sum_{x \in \Phi_{\tau(s)}^{\varphi(s)} \cap W} h_s(x); 1 \leq s \leq m\right).$$
		To remove the truncation to $W$, we use the method of Lemma~\ref{lem:decor_freq_3} to get
		\begin{align*}
			\Delta \leq \left( \sum_{\pi \in \Pi[m]} |a_{\pi}| \sum_{\pi' \in \Pi\{b;~b \in \pi\}} |\gamma_{red}^{(|\pi'|)} |(\mathbb{R}^{d(|\pi'|-1)}) |\pi'|\right)  \sum_{s = 1}^m \|h_s\mathds{1}_{\mathbb{R}^d \setminus W}\|_m \prod_{s' \in [m]\setminus\{s\}} \|h_{s'}\|_m.
		\end{align*}
		Now, we focus on
		\begin{align*}
			\kappa_{\mathbb{R}^d} & = \sum_{\pi \in \Pi[m]}  a_{\pi} \sum_{\pi' \in \Pi\{b;~b \in \pi\}} \int_{\mathbb{R}^{d|\pi'|}} \prod_{b' \in \pi'} \prod_{b \in b'} \prod_{s \in b} h_s(x_{b'}) \gamma^{(|\pi'|)}(x_{b'}; b' \in \pi') =: \sum_{\pi \in \Pi[m]}  a_{\pi} \sum_{\pi' \in \Pi\{b;~b \in \pi\}} I_{\pi'}.
		\end{align*}
		where, denoting by $B$  a Bernouilli random variable of parameter $p$,
		$$a_{\pi} := (-1)^{\#\{\tau_s = 1| s \in [m]\}} \prod_{b \in \pi}  \mathds{1}_{\#\{\varphi(s)| s \in b\} = 1} \kappa_{|b|}(B).$$
		By triangular inequality, we have
		\begin{align*}
			\sum_{i \in I} |\kappa_{\mathbb{R}^d}|^2 \leq \left(\sum_{\pi \in \Pi[m]}  a_{\pi} \sum_{\pi' \in \Pi\{b;~b \in \pi\}} \left(\sum_{i \in I}  |I_{\pi'}|^2 \right)^{1/2} \right)^2.
		\end{align*}
		As in Lemma \ref{lem:decor_freq_3}, we start by considering the partition $\pi'$ of size 1. We have
		\begin{align*}
			I_{\pi'} = \lambda \int_{\mathbb{R}^d} f_i(x) \prod_{s = 2}^m h_s(x) dx = \lambda \langle f_i, \overline{\prod_{s = 2}^m h_s} \rangle.
		\end{align*}
		Since $(f_i)_{i \in \mathbb{N}^d}$ is orthogonal and using the Hölder's inequality, we obtain
		\begin{align*}
			\sum_{i \in \mathbb{N}^d} |I_{\pi'} |^2 \leq \lambda^2 \int_{\mathbb{R}^d}\prod_{s = 2}^m |h_s(x)|^2 dx \leq  \prod_{s =2}^{m} \|h_2\|_{2(m-1)}^2.
		\end{align*}
		For the other partitions, we leverage on the Brillinger-mixing condition to adapt this argument. Let $\pi'$ with $|\pi'| \geq 2$. Denote by $b_0' \in \pi'$ the block containing $1$. By stationarity, we have
		\begin{align*}
			I_{\pi} &=\int_{\mathbb{R}^{d}} f_{i}(x_{b_0'}) \left(\int_{\mathbb{R}^{d(|\pi'| -1)}}   \prod_{b' \in \pi'} \prod_{b \in b'} \prod_{s \in b \setminus\{1\}} h_{s}(x_{b'}+x_{b_0'}) \gamma_{red}^{(|\pi'|)} (dx_b'; b' \in \pi' \setminus\{b_0'\})\right)dx_{b_0'} \\& =: \langle f_i, \overline{F} \rangle_{L^2(\mathbb{R}^d)}.
		\end{align*}
		Since $(f_i)_{i \in \mathbb{N}^d}$ is orthogonal and using the Hölder's inequality, we have
		\begin{align*}
			&\sum_{i \in \mathbb{N}^d} |I_{\pi'}|^2 \leq \int_{\mathbb{R}^d} |F(x)|^2 dx =  \int\displaylimits_{\mathbb{R}^{2d(|\pi'| -1)}}  \int\displaylimits_{\mathbb{R}^{d}} \prod_{b' \in \pi'} \prod_{b \in b'} \prod_{s \in b \setminus\{1\}} h_{s}(x_b'+x_{b_0'}) \overline{h_{s}(y_b'+x_{b_0'})} \\& \qquad \qquad \qquad \qquad \qquad \qquad \qquad \qquad \gamma_{red}^{(|\pi'|)} (dx_b'; b' \in \pi{\setminus\{b_0'\}} ) \overline{\gamma_{red}^{(|\pi'|)} (dy_b'; by \in \pi{\setminus\{b_0'\}} )} \\& \leq \int\displaylimits_{\mathbb{R}^{2d(|\pi'| -1)}} \prod_{b' \in \pi'} \prod_{b \in b'} \prod_{s \in b \setminus\{1\}}\|h_{s}(x_b'+\cdot)\|_{2(m-1)} \|h_{s}(y_b'+\cdot)\|_{2(m-1)}  \\& \qquad\qquad \qquad\qquad\qquad\qquad\qquad\qquad |\gamma_{red}^{(|\pi'|)}|(dx_b'; b' \in \pi{\setminus\{b_0'\}} ) |\gamma_{red}^{(|\pi'|)}| (dy_b'; b' \in \pi_{\setminus\{b_0'\}} ) \\& = |\gamma_{red}^{(|\pi|)}|(\mathbb{R}^{d(|\pi'|-1)})^2 \prod_{s = 2}^m \|h_s\|_{2(m-2)}^2.
		\end{align*}
		Accordingly,
		\begin{align*}
			\sum_{i \in I} |\kappa_{\mathbb{R}^d}|^2 \leq \left(\sum_{\pi \in \Pi[m]}  a_{\pi} \sum_{\pi' \in \Pi\{b;~b \in \pi\}} |\gamma_{red}^{(|\pi|)}|(\mathbb{R}^{d(|\pi'|-1)}) \right)^2 \prod_{s = 2}^m \|h_s\|_{2(m-2)}^2,
		\end{align*}
		which concludes the proof.

	\end{proof}
	
	\section{Generalization to multitaper estimators with plug-in intensity}\label{sec:with_intens}
	
	In this section, we consider the practical setting where the intensity is unknown. In that case, we adapt the multitaper estimators of the structure factor as follows.
	
	\begin{definition}\label{def:multi_tap_real}
		Let $\Phi$ be a stationary point process of $\mathbb{R}^d$,  $W$ be a compact set of $\mathbb{R}^d$ and $I$ be a subset of $\mathbb{N}^d$. Let $k_0 \in \mathbb{R}^d$ and $(f_i)_{i \in I} \in L^2(\mathbb{R}^d)^{|I|}$. The multitaper estimator with plug-in intensity is defined as
		\begin{equation}\label{eq:multi_tap_real}
			\widehat{S}^{\hat{\lambda}}(k_0) := \frac1{\widehat{\lambda}|I|} \sum_{i \in I} \widehat{S}_i^{\hat{\lambda}}(k_0) \mathds{1}_{\widehat{\lambda} > 0} + 0 \mathds{1}_{\widehat{\lambda} = 0}
		\end{equation}
		where
		\begin{align*}
			&\widehat{S}_i^{\hat{\lambda}}(k_0) := \left|\sum_{x \in \Phi \cap W} e^{- \bm{i} k_0.x} f_i(x) - \widehat{\lambda} \int_{W} f_i(x)e^{-\bm{i}k \cdot x} dx\right|^2, ~  \widehat{\lambda} := \frac{\#\{\Phi \cap W\}}{|W|}.
		\end{align*}
	\end{definition}
	
	\begin{remark}
		Note that for estimating the spectra $\lambda S$ instead of the structure factor $S$ (see the discussion following equation \eqref{eq:def_S}), we can use the following estimator
		\begin{equation*}
			\widehat{\lambda S}^{\hat{\lambda}}(k_0) := \frac1{|I|} \sum_{i \in I} \widehat{S}_i^{\hat{\lambda}}(k_0).
		\end{equation*} 
		The results concerning $\widehat{S}^{\hat{\lambda}}$ can be readily adapted to the estimator $\widehat{\lambda S}^{\hat{\lambda}}$.
	\end{remark}

	The following proposition shows that replacing the true intensity \(\lambda\) with its estimator \(\widehat{\lambda}\) in \(\widehat{S}^{\widehat{\lambda}}\) introduces an additional main error term of order \( |I|^{-1/2} \). This stems from the fact that the Fourier transform of the tapers \( (\mathcal{F}[f_i])_{i \in I} \) converge to $\delta_0$, causing the integrals 
	$
	\int_{W} f_i(x) e^{-\bm{i}k \cdot x} \, dx
	$
	to diverge for small frequencies $k$. Note that under the assumptions of Corollary~\ref{cor:L2_risk_herm}, this result implies that \(\widehat{S}^{\widehat{\lambda}}\) still remains minimax optimal for the \(L^1(\mathbb{P})\) risk.

	\begin{proposition}\label{prop:l2_s_plug_s}
		Let $\Phi$ be a stationary point process satisfying Assumption \ref{ass_rho_leb}. Let $W$ be a subset of $\mathbb{R}^d$, $I$ be a discrete subset of $\mathbb{N}^d$. We consider a family of orthonormal functions $(f_i)_{i \in I}$ of $L^1(\mathbb{R}^d) \cap L^2(\mathbb{R}^d)$. Then,
		\begin{align}\label{eq:risk_s_plug}
			\sup_{k \in \mathbb{R}^d} \mathbb{E}\left[|\widehat{S}^{\hat{\lambda}}(k) - S(k)|\right] \leq a_1 \sup_{k \in \mathbb{R}^d}\mathbb{E}&\left[| \widehat{S}^{\lambda}(k) - S(k)|^{2}\right]^{\frac12} + \frac{a_2}{|I|^{\frac12}}+ \frac{a_4}{|W|^{\frac12}}, 
		\end{align}
		where $(a_n)_{n \in [4]}$ are finite constants (see \eqref{eq:a_ai}) provided that
		$|\gamma_{\text{red}}^{4}| := \lambda + 7 |\gamma_{\text{red}}^{(2)}|(\mathbb{R}^d) + 6 |\gamma_{\text{red}}^{(3)}|(\mathbb{R}^{2d}) +|\gamma_{\text{red}}^{(4)}|(\mathbb{R}^{3d})$
		is finite.
	\end{proposition}
	\begin{proof}
		Let $k_0 \in \mathbb{R}^d$. By triangular inequality
		\begin{align*}
			\mathbb{E}&\left[|\widehat{S}^{\hat{\lambda}}(k_0) - S(k_0)|\right]  \leq \mathbb{E}\left[|\widehat{S}^{\lambda}(k_0) - S(k_0)|\right] + \mathbb{E}\left[|\widehat{S}^{\hat{\lambda}}(k_0) - \widehat{S}^{\lambda}(k_0)|\right] \\& \quad \leq \mathbb{E}\left[|\widehat{S}^{\lambda}(k_0) - S(k_0)|\right] + \mathbb{E}\left[|\widehat{S}^{\hat{\lambda}}(k_0) - \widehat{S}^{\lambda}(k_0)|\mathds{1}_{\widehat{\lambda} > 0}\right]+ \mathbb{P}[\widehat{\lambda} = 0] \mathbb{E}[\widehat{S}^{\lambda}(k_0)].
		\end{align*}
		We bound $\mathbb{P}[\widehat{\lambda} = 0] \mathbb{E}[\widehat{S}^{\lambda}(k_0)]$. Using \eqref{e.prop_campbell}, we obtain
		\begin{align*}
			\mathbb{P}(\widehat{\lambda} = 0) \leq \mathbb{P}[|\widehat{\lambda} - \lambda| \geq \lambda] \leq \frac{\mathbb{E}[|\widehat{\lambda} - \lambda|^2]}{\lambda^2}  = \frac1{\lambda |W|^2} \int_{\mathbb{R}^d} |\mathcal{F}[\mathds{1}_W](k)|^2 S(k) dk \leq \frac{\|S\|_{\infty}}{\lambda |W|}.
		\end{align*}
		Moreover, using the Plancherel theorem we get
		\begin{align*}
			\mathbb{E}[\widehat{S}^{\lambda}(k_0)] = \frac{1}{|I|} \sum_{i \in I} \int_{\mathbb{R}^d} |\mathcal{F}[f_i e^{-\bm{i} k_0 \cdot }](k)|^2 S(k) dk  \leq \|S\|_{\infty}.
		\end{align*}
		Accordingly $\mathbb{P}[\widehat{\lambda} = 0] \mathbb{E}[\widehat{S}^{\lambda}(k_0)] \leq \|S\|_{\infty}^2/(\lambda |W|)$.
		We consider $\mathbb{E}\left[|\widehat{S}^{\hat{\lambda}}(k_0) - \widehat{S}^{\lambda}(k_0)|\mathds{1}_{\widehat{\lambda} > 0}\right]$. By writing $\widehat{\lambda} = \widehat{\lambda} - \lambda +\lambda$ in  $\widehat{S}_i^{\widehat{\lambda}}(k_0)$, we obtain when $\widehat{\lambda} > 0$,
		$
		\widehat{S}^{\widehat{\lambda}}(k_0) - \widehat{S}^{\lambda}(k_0) = A + B + C,
		$
		where 
		\begin{align*}
			&A := \left(\frac{\lambda}{\widehat{\lambda}}-1\right) \widehat{S}^{\lambda}(k_0), \qquad
			B = 2\left(\frac{\lambda}{\widehat{\lambda}}-1\right)\frac1{|I|} \sum_{i \in I} \Re\left(C_i(k_0) \cdot \overline{E_i(k_0)}\right),
			\\&C := \frac{(\widehat{\lambda} - \lambda)^2}{\widehat{\lambda}} \frac1{|I|} \sum_{i \in I} |E_i(k_0)|^2, \qquad
			E_i := \int_{W} f_i(x) e^{-\bm{i} k \cdot x} dx.
		\end{align*}
		and where $C_i(k_0)$ is defined in \eqref{eq:def_Ti_Ci}. We start with $C$. Since $(f_i)_{i \in I}$ is orthogonal, we have
		\begin{align}\label{eq:bound_Ei}
			\frac1{|I|} \sum_{i \in I} |E_i(k_0)|^2 \leq \frac1{|I|} \|\mathds{1}_W\|_2^2 = \frac{|W|}{|I|}.
		\end{align}
		Moreover, by Cauchy-Schwarz inequality, we have
		\begin{align*}
			\mathbb{E}\left[\frac{(\widehat{\lambda} - \lambda)^2}{\widehat{\lambda}} \mathds{1}_{\widehat{\lambda} > 0}\right] \leq 	\mathbb{E}\left[(\widehat{\lambda} - \lambda)^4\right]^{\frac12} 	\mathbb{E}\left[\frac1{\widehat{\lambda}^2} \mathds{1}_{\widehat{\lambda} > 0}\right]^{\frac12}.
		\end{align*}
		If $\widehat{\lambda} > 0$ then $\#\{\Phi \cap W\} \geq 1$, whence $\widehat{\lambda} \geq |W|^{-1}$. Accordingly, 
		\begin{align*}
			\mathbb{E}\left[\frac1{\widehat{\lambda}^2} \mathds{1}_{\widehat{\lambda} > 0}\right] &\leq \mathbb{E}\left[(2/\lambda)^2 \mathds{1}_{\widehat{\lambda} \geq \lambda/2}\right] +  \mathbb{E}\left[|W|^2 \mathds{1}_{\widehat{\lambda} \leq \lambda/2}\right] \leq \frac{4}{\lambda^2} + |W|^2 \mathbb{P}[|\widehat{\lambda} - \lambda| \geq \lambda/2]\\&  \leq \frac{4}{\lambda^2} + \frac{16}{\lambda^4} |W|^2 \mathbb{E}[|\widehat{\lambda} - \lambda|^4]. 
		\end{align*}
		We express the centered moment of order four as $
		\mathbb{E}[|\widehat{\lambda} - \lambda|^4] = 3  \mathbb{E}[|\widehat{\lambda} - \lambda|^2]^2 + \kappa_4(\widehat{\lambda}).$ Using Lemma~\ref{lem:bril_thin} we get 
		\begin{align}\label{eq:lbd4}
			\mathbb{E}[|\widehat{\lambda} - \lambda|^4] \leq \frac{3\lambda^2\|S\|_{\infty}^2}{|W|^2} + \frac{|\gamma_{\text{red}}^{4}| \|\mathds{1}_W\|_4^4}{|W|^4} = \frac{3\lambda^2\|S\|_{\infty}^2}{|W|^2} + \frac{|\gamma_{\text{red}}^{4}|}{|W|^3}.
		\end{align}
		Accordingly,
		\begin{align}\label{eq:inv_lbd2}
			\mathbb{E}\left[\frac1{\widehat{\lambda}^2} \mathds{1}_{\widehat{\lambda} > 0}\right] \leq \frac{48 \|S\|_{\infty}^2}{\lambda^2} + \frac{16 |\gamma_{\text{red}}^{4}|}{\lambda^4 |W|} + \frac{4 \|S\|_{\infty}^2}{\lambda^2|W|^2} + \frac{4 |\gamma_{\text{red}}^{4}|}{\lambda^2 |W|^3} := b_1
		\end{align}
		This yields, 
		\begin{align*}
			C &\leq \frac{1}{|I|}\left(3\lambda^2\|S\|_{\infty}^2 + \frac{|\gamma_{\text{red}}^{4}|}{|W|}\right)^{\frac12}b_1^{\frac12}.
		\end{align*}
		Now, we consider $A$. Using the Hölder-inequality we have
		\begin{align*}
			A &\leq 	\|S\|_{\infty}\mathbb{E}\left[\frac{|\widehat{\lambda} - \lambda|}{\widehat{\lambda}} \mathds{1}_{\widehat{\lambda} > 0}\right] + \mathbb{E}\left[|\widehat{S}^{\lambda}(k_0) - S(k_0) | \frac{|\widehat{\lambda} - \lambda|}{\widehat{\lambda}} \mathds{1}_{\widehat{\lambda} > 0}\right] \\& \leq 	\|S\|_{\infty}\mathbb{E}\left[|\widehat{\lambda} - \lambda|^2\right]^{\frac12} 	\mathbb{E}\left[\frac{ \mathds{1}_{\widehat{\lambda} > 0}}{\widehat{\lambda}^2}\right]^{\frac12} + \left[|\widehat{\lambda} - \lambda|^4\right]^{\frac14} 	\mathbb{E}\left[\frac{ \mathds{1}_{\widehat{\lambda} > 0}}{\widehat{\lambda}^4}\right]^{\frac14}\mathbb{E}\left[|\widehat{S}^{\lambda}(k_0) - S(k_0) |^2\right]^{\frac12}.
		\end{align*}
		For the first term, we have using \eqref{eq:inv_lbd2}
		\begin{align*}
			\mathbb{E}\left[|\widehat{\lambda} - \lambda|^2\right]^{\frac12} 	\mathbb{E}\left[\frac{ \mathds{1}_{\widehat{\lambda} > 0}}{\widehat{\lambda}^2}\right]^{\frac12} \leq \frac{\|S\|_{\infty}}{\lambda^{\frac12} |W|^{\frac12}}b_1^{\frac12}.
		\end{align*}
		For the second one, using again the fact that if $\widehat{\lambda} > 0$, then $\widehat{\lambda} \geq |W|^{-1}$, we obtain
		\begin{align*}
			\mathbb{E}\left[\frac1{\widehat{\lambda}^4} \mathds{1}_{\widehat{\lambda} > 0}\right] \leq \frac{16}{\lambda^4}\left(1 + |W|^4 \mathbb{E}[|\widehat{\lambda} - \lambda|^4]\right).
		\end{align*}
		Using equation \eqref{eq:lbd4}, we obtain after computations
		\begin{align*}
			\mathbb{E}\left[\frac1{\widehat{\lambda}^4} \mathds{1}_{\widehat{\lambda} > 0}\right] \mathbb{E}[|\widehat{\lambda} - \lambda|^4] \leq 48 \|S\|_{\infty}^4 + \frac{96 \|S\|_{\infty}^2 |\gamma_{\text{red}}^{4}|}{\lambda^2 |W|} + \frac{48 \|S\|_{\infty}^2}{\lambda^2 |W|^2} + \frac{16  |\gamma_{\text{red}}^{4}|^2}{|W|^2 \lambda^4} + \frac{16  |\gamma_{\text{red}}^{4}|}{\lambda^4 |W|^3}.
		\end{align*} 
		This yields, after numerical evaluations
		\begin{align*}
			A \leq (3 \|S\|_{\infty} + \varepsilon_{W})  \mathbb{E}\left[ |\widehat{S}^{\lambda}(k_0) - S(k_0)|^2\right]^{\frac12} +  \frac{\|S\|_{\infty}}{\lambda^{\frac12} |W|^{\frac12}} b_1^{\frac12}
		\end{align*}
		where
		\begin{align*}
			&\varepsilon_W :=  \frac{4 \|S\|_{\infty}^{\frac14} |\gamma_{\text{red}}^{4}|^{\frac14}}{\lambda^{\frac12} |W|^{\frac14}} + \frac{3 \|S\|_{\infty}^{\frac12}}{\lambda^{\frac12} |W|^{\frac12}} + \frac{2  |\gamma_{\text{red}}^{4}|^{\frac12}}{|W|^{\frac12} \lambda} + \frac{2  |\gamma_{\text{red}}^{4}|^{\frac14}}{\lambda |W|^{\frac34}}.
		\end{align*}
		It remains to bound $B$. Using twice the Hölder-inequality, we obtain
		\begin{align*}
			B &\leq 2 \mathbb{E}\left[\frac1{\widehat{\lambda}^2} \mathds{1}_{\widehat{\lambda} > 0}\right]^{\frac12} \mathbb{E}[|\widehat{\lambda} - \lambda|^4]^{\frac14} \mathbb{E}\left[\left(\frac1{|I|} \sum_{i \in I} |C_i(k_0)| |E_i(k_0)|\right)^4\right]^{\frac14}  \\& \leq 2 \mathbb{E}\left[\frac1{\widehat{\lambda}^2} \mathds{1}_{\widehat{\lambda} > 0}\right]^{\frac12} \mathbb{E}[|\widehat{\lambda} - \lambda|^4]^{\frac14} \mathbb{E}\left[\left(\frac1{|I|} \sum_{i \in I} |C_i(k_0)|^2 \right)^2\right]^{\frac14} \left(\frac1{|I|} \sum_{i \in I} |E_i(k_0)|^2\right)^{\frac12}.
		\end{align*}
		Recalling \eqref{eq:bound_Ei}, \eqref{eq:lbd4} and \eqref{eq:inv_lbd2}, we obtain 
		\begin{align*}
			B &\leq \frac{2\mathbb{E}[\widehat{S}^{\lambda}(k_0)^2]^{\frac14}}{|I|^{\frac12}} \left(3\lambda^2\|S\|_{\infty}^2 + \frac{|\gamma_{\text{red}}^{4}|}{|W|}\right)^{\frac14} b_1^{\frac12} \\& \leq \left(3\lambda^2\|S\|_{\infty}^2 + \frac{|\gamma_{\text{red}}^{4}|}{|W|}\right)^{\frac14}b_1^{\frac12} \left( \frac{2^{\frac32}\|S\|_{\infty}^{\frac12}}{|I|^{\frac12}} + 2^{\frac12}\mathbb{E}[|\widehat{S}^{\lambda}(k_0) - S(k_0)|^2]^{\frac12} + \frac{2^{\frac12}}{|I|}\right).
		\end{align*}
		Gathering the bounds on $A$, $B$ and $C$, we conclude the proof of \eqref{eq:risk_s_plug} with
		\begin{align}\label{eq:a_ai}
			&  b_1 := \frac{48 \|S\|_{\infty}^2}{\lambda^2} + \frac{16 |\gamma_{\text{red}}^{4}|}{\lambda^4 |W|} + \frac{4 \|S\|_{\infty}^2}{\lambda^2|W|^2} + \frac{4 |\gamma_{\text{red}}^{4}|}{\lambda^2 |W|^3}, \quad b_2 := 3\lambda^2\|S\|_{\infty}^2 + \frac{|\gamma_{\text{red}}^{4}|}{|W|}, \nonumber
			\\& a_1:= 1+ 3\|S\|_{\infty} + \varepsilon_W + \sqrt{2} b_2^{\frac14} b_1^{\frac12},\quad 
			a_2 := 2^{\frac32} \|S\|_{\infty}^{\frac12} b_2^{\frac14} b_1^{\frac12}+ \left(b_2^{\frac12}+\sqrt{2} b_2^{\frac14}\right) b_1^{\frac12},
			\\&
			a_4 := \frac{\|S\|_{\infty} b_1^{\frac12}}{\lambda^{\frac12}} + \frac{\|S\|_{\infty}}{\lambda |W|^{\frac12}}.\nonumber
		\end{align}
	\end{proof}
	
	\section{Results concerning the Hermite functions}\label{sec:herm} The next lemma provides bounds on the tail of the Hermite functions.
	\begin{lemma}\label{lem:tail_herm}
		Consider a finite set of indexes $I$ of $\mathbb{N}^d$ and denote $M_I = \max\{i_1, \dots, i_d| i = (i_1, \dots, i_d) \in I\}$. Let $\rho \in (0, \infty)$ be such that $\rho > \sqrt{2(M_I+1)+ (M_I+1)^{1/3+\theta}}$, for $\theta > 0$. Denote $W = [-\rho, \rho]^d$. Then, there exist two constants $0 < c, C < \infty$ such that
		\begin{equation}\label{eq:tail_herm_2}
			\sup_{i \in I} \|\psi_i \mathbf{1}_{W^c}\|_1+\sup_{i \in I} \|\psi_i \mathbf{1}_{W^c}\|_2^2 \leq C e^{-c \rho^{3\theta}},
		\end{equation}
	\end{lemma}
	\begin{proof}
		Recall that the $d$-dimensional Hermite functions are defined as the product of one-dimensional ones. Using the inequality:
		$$\forall x = (x_1, \dots, x_d) \in \mathbb{R}^d,~1 - \prod_{i = 1}^d \mathbf{1}_{|x_l| \leq \rho} \leq \sum_{l = 1}^d \mathbf{1}_{|x_l| > \rho},$$
		and the fact that one-dimensional Hermite functions have a unit $L^2(\mathbb{R}^d)$-norm, we reduce the proof to the case $d = 1$. Let $n \leq M_I$. We start with the case $\rho \geq \sqrt{2}\sqrt{2n+1}$. According to Theorem B of \cite{askey1965mean}, we have for all $x \in \mathbb{R}$ with $|x| \geq \rho$, $|\psi_n(x)| \leq C e^{-c x^2}$, for some finite positive constants $c$ and $C$, that will change from line to line in the following. Using the standard upper bound of the tail of a Gaussian random variable, we obtain:
		$$\|\psi_n \mathbf{1}_{|\cdot|\geq \rho}\|_1+ \|\psi_n \mathbf{1}_{|\cdot|\geq \rho}\|_2^2 \leq  C e^{-c \rho^2}.$$
		Consider now the case $\rho \leq \sqrt{2}\sqrt{2n+1}$. Since $\rho \geq \sqrt{2(n+1)+ (n+1)^{1/3+\theta}}$, according to the bound on equation (2.13) of the proof of Theorem 2.3 of \cite{aptekarev2012asymptotics}, we have 
		$$\|\psi_n \mathbf{1}_{|\cdot|\geq \rho}\|_1+ \|\psi_n \mathbf{1}_{|\cdot|\geq \rho}\|_2^2 \leq  C e^{-c n^{3\theta/2}} \leq C e^{-c \rho^{3\theta}}.$$
		This concludes the proof.
	\end{proof}
	The next lemma studies the Sobolev norms of the Hermite functions. 
	\begin{lemma}\label{lem:norm_H12_herm}
		Assume that $I = \{i \in \mathbb{N}^d|~|i|_{\infty} < i_{\max}\}$ with $i_{\max} \geq 1$. Let $\beta \in (0, 2]$. Then
		\begin{equation}\label{eq:norm_H12_herm_2}
			\frac1{|I|} \sum_{i \in I} \|\psi_i\|_{\dot{H}^{\beta/2}}^2 \leq |I|^{\beta/(2d)}. 
		\end{equation}	
	\end{lemma}
	\begin{proof}
		According to Hölder's inequality, we have, since $\beta \leq 2$,
		\begin{align*}
			\int_{\mathbb{R}^d} |\psi_i(k)|^2 |k|^{\beta} dk &\leq \|\psi_i\|_{2}^{1 - \beta/2} \left(\int_{\mathbb{R}^d} |\psi_i(k)|^2 |k|^2 dk\right)^{\beta/2} \\& = \left(\frac12\sum_{s = 1}^d (2 i_s+1)\right)^{\beta/2} \leq  \frac1{2^{\beta/2}} \sum_{s = 1}^d (2 i_s+1)^{\beta/2},
		\end{align*}
		where we used the inclusion of $\ell^{1}([d])$ in $\ell^{2/\beta}([d])$ (since $2/\beta \geq 1$).
		This yields~\eqref{eq:norm_H12_herm_2}:
		\begin{align*}
			\frac1{|I|} \sum_{i \in I} \|\psi_i\|_{\dot{H}^{\frac12}}^2 &\leq \frac1{2^{\beta/2}|I|} \sum_{0 \leq i_1, \dots, i_d < i_{\max}} \sum_{s = 1}^d (2 i_s+1)^{\beta/2} = \frac{i_{\max}^{d-1}}{2^{\beta/2}|I|} \sum_{s = 1}^{i_{\max}-1} (2 s+1)^{\beta/2} \\& \leq \frac{i_{\max}^{d-1}}{2^{\beta/2}|I|} \sum_{s = 1}^{i_{\max}-1} (2 i_{\max})^{\beta/2}= |I|^{\beta/(2d)}.
		\end{align*}
	\end{proof}
	
	The following lemma provides upper bounds on averages of $L^p(\mathbb{R}^d)$ norms of the Hermite functions. The key point is to control the scaling with respect to both the number of tapers $|I|$ and the order $p$ of the norms.
	
	\begin{lemma}\label{lem:L4_herm}
		Assume that $I = \{i \in \mathbb{N}^d|~|i|_{\infty} < i_{\max}\}$ with $i_{\max} \geq 1$. Then, there exists a numerical constant $c_{\psi} \in (0, \infty)$ such that,
		\begin{equation}\label{eq:L3_herm}
			\frac1{|I|}\sum_{i \in I} \|\psi_i\|_3^2 \leq c_{\psi}^d  \frac{1}{|I|^{\frac16}}.
		\end{equation}
		Moreover, for $p \geq 4$, we have
		\begin{equation}\label{eq:Lp_herm}
			\frac1{|I|}\sum_{i \in I} \|\psi_i\|_p^2 \leq c_{\psi}^d  \frac{\log(|I|^{\frac1d})^{\frac{2d}{p}}}{|I|^{\frac1{3p}+ \frac16}},
		\end{equation}
		\begin{equation}\label{eq:Lp_herm1}
			\frac1{|I|}\sum_{i \in I} \|\psi_i\|_p \leq c_{\psi}^d  \frac{\log(|I|^{\frac1d})^{\frac{d}{p}}}{|I|^{\frac1{6p}+ \frac1{12}}}.
		\end{equation}
	\end{lemma}
	\begin{proof}
		The proof of this lemma relies on the following results (see  equation (3.2) of~\cite{aptekarev2012asymptotics}). For all $p < 4$, there exists $c_p < \infty$ such that
		\begin{equation}\label{eq:Lpetitp_herm_int}
			\forall n \geq 0,~\int_{\mathbb{R}} |\psi_{n}(x)|^p dx \leq \frac{c_p}{(n+1)^{\frac{p}4 - \frac12}},
		\end{equation}
		for $p = 4$, there exists $c_4 < \infty$ such that
		\begin{equation}\label{eq:L4_herm_int}
			\forall n \geq 0,~\int_{\mathbb{R}} |\psi_{n}(x)|^4 dx \leq \frac{c_4 \log(n+1)}{(n+1)^{\frac12}},
		\end{equation}
		and for $p > 4$, there exists $c_p < \infty$ such that
		\begin{equation}\label{eq:Lp_herm_int}
			\forall n \geq 0,~\int_{\mathbb{R}} |\psi_{n}(x)|^p dx \leq \frac{c_p}{(n+1)^{\frac{p}{12} + \frac16}},
		\end{equation}
		We start by proving~\eqref{eq:L3_herm}. According to \eqref{eq:Lpetitp_herm_int} we have
		\begin{align*}
			\frac1{|I|}\sum_{i \in I} \|\psi_i\|_3^2 &\leq c_3^{\frac{2d}{3}} \frac1{|I|} \sum_{i \in I}\prod_{s = 1}^d \frac1{(i_s+1)^{1/6}} =  c_3^{\frac{2d}{3}} \frac1{|I|}  \left(\sum_{s = 0}^{i_{\max}-1} \frac1{(s+1)^{1/6}}\right)^d.
		\end{align*}
		A comparison sum-integral yields the result:
		\begin{align*}
			\frac1{|I|}\sum_{i \in I} \|\psi_i\|_4^2 &\leq c_3^{\frac{2d}{3}} \frac1{|I|}  \left(\int_0^{i_{\max}}\frac{ds}{(s+1)^{1/6}}\right)^d \leq c_4^{\frac{2d}{3}} \frac1{|I|}  \left( \frac{6}{5}(2i_{\max})^{5/6}\right)^d \leq c_{\psi}^d  \frac1{|I|^{1/4}},
		\end{align*}
		where $c_{\psi} \geq 6 2^{5/6} c_2^{2/3}/5$. 
		The subtlety of~\eqref{eq:Lp_herm} is that the constant $c_{\psi}$ does not depend on $p$. Theorem~1 of~\cite{bonan1990estimates} ensures that there exists $D \in (1, \infty)$ such that
		\begin{equation}\label{eq:Lp_herm_unif}
			\forall n \geq 0,~\|\psi_{n}||_{\infty} \leq \frac{D}{(n+1)^{1/12}}.
		\end{equation}
		Hence, for $n\geq 0$ and $p \geq 4$, using~\eqref{eq:L4_herm_int} and~\eqref{eq:Lp_herm_unif}, we have
		\begin{align*}
			\int_{\mathbb{R}} |\psi_{n}(x)|^p dx &= \int_{\mathbb{R}} |\psi_{n}(x)|^{p-4}  |\psi_{n}(x)|^{4}dx \\& \leq D^{p-4} c_2 \frac{\log(n+1)}{(n+1)^{(p-4)/12} (n+1)^{\frac12}} \leq (Dc_2)^p \frac{\log(n+1)}{(n+1)^{p/12+1/6} }. 
		\end{align*}
		Finally, concluding as in the proof of~\eqref{eq:L3_herm} we obtain~\eqref{eq:Lp_herm} and \eqref{eq:Lp_herm1} with $c_{\psi} = (1+D c_2)^2$.
	\end{proof}
	
	\section{Results related to cumulants}\label{sec:cum}	
	
	\subsection{Isserlis-Wick type formulae}
	
	We denote by $\mathbf{1}_n$ the single block partition of $[n]$, i.e. $\mathbf{1}_n  := 12\dots n$. For two partitions $\sigma$ and $\tau$, $\sigma \vee \tau$ denotes their least upper bound (refer to Chapter 3 of~\cite{mccullagh2018tensor} or Chapter 2 of~\cite{peccati2011wiener}). The next result is a direct application of equation (3.3) of~\cite{mccullagh2018tensor}.
	
	\begin{corollary}\label{corol:cumulants_square_rvs}
		Let $Y_1, \dots, Y_m$ be a family of complex valued random variables. Then,
		\begin{equation}
			\kappa_m(|Y_1|^2, \dots, |Y_m|^2) = \sum_{\substack{\sigma \in \Pi[2m],\\ \sigma \vee \tau = \mathbf{1}_{2m}, |\sigma| \geq 2}} \prod_{b \in \sigma} \kappa(Z_i; i \in b),
		\end{equation}
		where the sum runs over the partition of $[2m]$ and for $i \in \{1, \dots, m\}$, $Z_{2i-1} = Y_i$ and $Z_{2i} = \overline{Y_{i}}$ and $\tau = 12|34|\dots|(2m-1)~2m$.
	\end{corollary}
	
	\subsection{Extension of the result of Bentkus and Rudzkis under multiple Statulevi\v{c}ius conditions}
	
	The proof of the concentration inequality stated in Theorem~\ref{thm:conc} relies on the following lemma, which provides a tool to derive non-asymptotic deviation inequalities from cumulant bounds. This result is a straightforward extension of a classical inequality \cite{bentkus1980exponential}, corresponding to the case $J = 1$. Although this generalization is natural and may be known to experts, we have not found it explicitly stated in the literature. Such an extension is required to handle the three distinct contributions to the fluctuations of the multitaper estimator of the structure factor, as discussed in Section~\ref{sec:nn_asymp}.

	\begin{theorem}\label{thm:sc}
		Let $X$ be a real random variable. Suppose that $\mathbb{E}[X] = 0$. Let $J \geq 1$, $(\gamma_i)_{1 \leq i \leq J} \in [0, \infty)^J$, $(H_i, \Delta_i)_{1 \leq i \leq J} \in (0, \infty)^{2J}$. Assume that, for all $m \geq 2$,
		\begin{equation}\label{eq:mStatcond}
			|\kappa_m(X)| \leq \sum_{j = 1}^J \left(\frac{m!}{2}\right)^{1+\gamma_j} \frac{H_j}{\Delta_j^{m-2}}.
		\end{equation}
		Then, for all $x > 0$, we have
		\begin{align}\label{eq:lem_conc}
			\mathbb{P}(X \geq x) \leq C \max_{j = 1, \dots, J} \exp\left(-\frac1{2J} \frac{x^2}{H_j + x^{2 - \alpha_j} /\Delta_j^{\alpha_j}}\right),
		\end{align}
		where $\alpha_j := 1/(1+\gamma_j)$ and $C < \infty$.
	\end{theorem}
	\begin{proof}
		For simplicity, we follow the proof of Theorem 2.5 of Section 5 of \cite{doring2022method} to prove~\eqref{eq:lem_conc}. At a price of more technicality, the following proof could be adapted to prove~\eqref{eq:lem_conc} with $C =1$ (see~\cite{saulis2012limit}, proof of Lemma 2.4, pages 37-41). For $n \geq 1$, we denote the truncated exponential sum as
		$$\exp_n(x) := \sum_{m = 0}^n \frac1{m!} x^m.$$
		Let $h > 0$ and $n \geq 1$ that will be specified during the proof. Since $\exp_{2n}$ is increasing over $\mathbb{R}^+$, we have by Markov's inequality
		$$\mathbb{P}(X \geq x) \leq \frac{\mathbb{E}[\exp_{2n}(hX)]}{\exp_{2n}(hx)}.$$
		Then, according to equation (5.2) of \cite{doring2022method}, 
		$$\mathbb{E}[\exp_{2n}(hX)] \leq \exp_n\left(\sum_{m = 2}^{2n} \frac{h^m}{m!} |\kappa_m(X)|\right) := \exp_n(A).$$
		Let $h_j = x/(H_j + x^{2 - \alpha_j}/\Delta_j^{\alpha_j})$ and $h = (\sum_{j = 1}^J 1/h_j)^{-1}$.
		Using \eqref{eq:mStatcond}, we have
		$$A \leq\frac{h^2}{2}\sum_{j = 1}^J H_j \sum_{m = 2}^{2n} \left(\frac{m!}2\right)^{\gamma_j} \left(\frac{h}{\Delta_j}\right)^{m-2}.$$
		Then, assume that $2n \leq hx$. We use $m!/2 \leq (2n)^{m-2}$ for $m \leq 2n$ (see \cite{doring2022method}), to get 
		$$A \leq\frac{h^2}{2}\sum_{j = 1}^J H_j \sum_{m = 2}^{2n} \left((hx)^{\gamma_j}\frac{h}{\Delta_j}\right)^{m-2} \leq\frac{h^2}{2} \sum_{j = 1}^J H_j \frac1{1 - (hx)^{\gamma_j}\frac{h}{\Delta_j}}.$$
		Note that 
		\begin{align*}
			(hx)^{\gamma_j}\frac{h}{\Delta_j} &= \frac{(xh)^{1+\gamma_j}}{x \Delta_j} \leq \frac{(xh_j)^{1+\gamma_j}}{x \Delta_j} = \left(\frac{x^2}{x^2 + H_j (\Delta_j x)^{1/(1+\gamma_j)}}\right)^{1+\gamma_j} < 1. 
		\end{align*}
		Hence, 
		\begin{align*}
			1 - (hx)^{\gamma_j}\frac{h}{\Delta_j} \geq 1- \frac{x^2}{x^2 + H_j (\Delta_j x)^{1/(1+\gamma_j)}} =  \frac{H_j (\Delta_j x)^{1/(1+\gamma_j)}}{x^2 + H_j(\Delta_j x)^{1/(1+\gamma_j)}} = \frac{H_j h_j}{x}.
		\end{align*}
		Subsequently, 
		$$A \leq \frac{h^2}{2} \sum_{j = 1}^J \frac{x}{h_j} = \frac{hx}{2},$$
		and
		$$\mathbb{P}(X \geq x) \leq \frac{\exp_n(hx/2)}{\exp_{2n}(hx)}.$$
		Arguing exactly as in the end of proof of Theorem 2.5 of \cite{doring2022method} (from equation (5.5)), we get 
		$$\mathbb{P}(X \geq x) \leq C e^{-hx/2} \leq C \max_{j = 1, \dots, J} \exp\left(-\frac1{2J} \frac{x^2}{H_j + x^{2 - \alpha_j} /\Delta_j^{\alpha_j}}\right),$$
		for some numerical constant $C > 0$.
	\end{proof}
\end{document}